\documentclass[12pt]{article}
\usepackage{amsmath,amssymb,amsfonts,amsthm}
\newenvironment{myabstract}{\par\noindent
{\bf Abstract . } \small }
{\par\vskip8pt minus3pt\rm}
\newcounter{item}[section]
\newcounter{kirshr}
\newcounter{kirsha}
\newcounter{kirshb}
\newenvironment{enumroman}{\setcounter{kirshr}{1}
\begin{list}{(\roman{kirshr})}{\usecounter{kirshr}} }{\end{list}}
\newenvironment{enumarab}{\setcounter{kirshb}{1}
\begin{list}{(\arabic{kirshb})}{\usecounter{kirshb}} }{\end{list}}
\newenvironment{athm}[1]{\vskip3mm\par\noindent
{\bf #1 }. \slshape }
{\upshape\par\vskip10pt minus3pt}
\newtheorem{theorem}{Theorem}[section]

\newtheorem{lemma}[theorem]{Lemma}
\newtheorem{corollary}[theorem]{Corollary}
\newenvironment{demo}[1]{\noindent{\bf #1.}\upshape\mdseries}
{\nopagebreak{\hfill\rule{2mm}{2mm}\nopagebreak}\par\normalfont}
\theoremstyle{definition}

\newtheorem{example}[theorem]{Example}
\newtheorem{definition}[theorem]{Definition}

\def\Nr{{\mathfrak{Nr}}}
\def\Fr{{\mathfrak{Fr}}}
\def\Sg{{\mathfrak{Sg}}}
\def\Fm{{\mathfrak{Fm}}}

\def\Cm{{\mathfrak{Cm}}}

\def\K{{\mathfrak{K}}}
\def\CA{{\bf CA}}
\def\K{{\bf K}}
\def\RCA{{\bf RCA}}
\def\SA{{\bf SA}}
\def\Rd{{\ Rd}}
\def\(R)RA{{\bf (R)RA}}
\def\RA{{\bf RA}}
\def\RRA{{\bf RRA}}
\def\Dc{{\bf Dc}}
\def\R{\mathbb{R}}
\def\N{\mathbb{N}}

\def\C{\mathbb{C}}
\def\A{{\mathfrak{A}}}
\def\B{{\mathfrak{B}}}
\def\C{{\mathfrak{C}}}
\def\D{{\mathfrak{D}}}
\def\P{{\mathfrak{P}}}
\def\Rd{{\mathfrak{Rd}}}
\def\Dc{{\bf Dc}}
\def\PA{{\bf PS}}
\def\PEA{{\bf PEA}}
\def\RPEA{{\bf RPEA}}

\def\QRA{{\bf QRA}}
\def\QPA{{\bf QPA}}
\def\SC{{\bf SC}}
\def\SA{{\bf SA}}
\def\Rd{{\mathfrak{Rd}}}
\def\Ra{{\mathfrak{Ra}}}
\def\L{{\mathfrak{L}}}
\def\P{{\mathfrak{P}}}
\def\Ca{{\mathfrak{Ca}}}
\def\(R)RA{{\bf (R)RA}}
\def\RA{{\bf RA}}
\def\RRA{{\bf RRA}}
\def\Dc{{\bf Dc}}
\def\PEA{{\bf PEA}}
\def\PA{{\bf PA}}
\def\R{\mathbb{R}}
\def\N{\mathbb{N}}

\def\C{{\cal{C}}}

\def\At{{\mathfrak{At}}}
\def\GG{\cal {G}}

\def\G{{\mathfrak{G}}}
\def\RDf{{\bold {RDf}}}
\def\Df{{\bold {Df}}}

\def\M{{\bold {M}}}

\def\Dc{{\bf Dc}}

\def\Sg{{\mathfrak{Sg}}}

\def\PA{{\bf PA}}

\def\QEA{{\bf QEA}}
\def\RQPEA{{\bf RQPEA}}
\def\PEA{{\bf PEA}}
\def\PA{{\bf PA}}

\def\L{{\mathfrak{L}}}

\def\SA{{\bf SA}}
\def\Rd{{\mathfrak{Rd}}}
\def\Ra{{\mathfrak{Ra}}}
\def\L{{\mathfrak{L}}}
\def\P{{\mathfrak{P}}}
\def\Ca{{\mathfrak{Ca}}}
\def\(R)RA{{\bf (R)RA}}
\def\RA{{\bf RA}}
\def\RRA{{\bf RRA}}
\def\Dc{{\bf Dc}}
\def\PEA{{\bf PEA}}
\def\PA{{\bf PA}}
\def\R{\mathbb{R}}
\def\N{\mathbb{N}}

\def\C{{\cal{C}}}

\def\Dc{{\bf Dc}}

\def\Sg{{\mathfrak{Sg}}}

\def\PA{{\bf PA}}

\def\RPA{{\bf RPA}}

\def\RDf{{\bf RDf}}

\def\Nr{{\mathfrak{Nr}}}
\def\Fr{{\mathfrak{Fr}}}
\def\Sg{{\mathfrak{Sg}}}
\def\Fm{{\mathfrak{Fm}}}

\def\K{{\mathfrak{K}}}
\def\CA{{\bf CA}}
\def\K{{\bf K}}
\def\RCA{{\bf RCA}}
\def\SA{{\bf SA}}
\def\Rd{{\ Rd}}
\def\(R)RA{{\bf (R)RA}}
\def\RA{{\bf RA}}
\def\RRA{{\bf RRA}}
\def\Dc{{\bf Dc}}
\def\R{\mathfrak{R}}
\def\N{\mathbb{N}}

\def\C{\mathbb{C}}
\def\A{{\mathfrak{A}}}
\def\B{{\mathfrak{B}}}
\def\C{{\mathfrak{C}}}
\def\D{{\mathfrak{D}}}
\def\Ig{{\mathfrak{Ig}}}
\def\M{{\mathfrak{M}}}
\def\Sb{{\mathfrak{Sb}}}
\def\Kf{{\bf Kf}}
\def\Kn{{\bf Kn}}
\def\Kc{{\bf Kc}}

\def\Fm{{\mathfrak{Fm}}}
\def\Cm{{\mathfrak{Cm}}}
\def\Tm{{\mathfrak{Tm}}}

\def\Rd{{\mathfrak{Rd}}}
\def\E{{\mathfrak{E}}}
\def\Dc{{\bf Dc}}
\def\PA{{\bf PA}}
\def\PEA{{\bf PEA}}

\def\F{{\mathfrak F}}

\def\QPA{{\bf QPA}}

\def\QPEA{{\bf QPEA}}

\def\RQEA{{\bf RQEA}}
\def\RK{{\bf RK}}
\def\PEA{{\bf PEA}}

\def\Lf{{\bf Lf}}
\def\LfK{{\bf LfK}}
\def\Cs{{\bf Cs}}
\def\ReK{{\bf ReK}}
\def\SsK{{\bf SsK}}

\def\DcK{{\bf DcK}}
\def\Bb{{\mathfrak{Bb}}}

\def\restr #1{{\restriction_{#1}}}

\def\ws{winning strategy}

\def\ws{winning strategy}

\def \set#1{\{#1\} }

\def\cyl#1{{\sf c}_{#1}}
\def\diag#1#2{{\sf d}_{#1#2}}

\def\c #1{{\cal #1}}

\def\Fm{{\mathfrak{Fm}}}
\def\Cm{{\mathfrak{Cm}}}
\def\Tm{{\mathfrak{Tm}}}

\def\cyl#1{{\sf c}_{#1}}
\def\diag#1#2{{\sf d}_{#1#2}}

\def\c #1{{\cal #1}}

\def\pa{$\forall$}
\def\pe{$\exists$}

\def\nodes{{\sf nodes}}

\def\restr #1{{\restriction_{#1}}}

\def\Ra{{\mathfrak{Ra}}}
\def\Nr{{\mathfrak{Nr}}}
\def\CA{{\bf CA}}
\def\RCA{{\bf RCA}}
\def\c#1{{\mathcal #1}}

\def\set#1{ \{#1\}}

\def\Ca{{\mathfrak Ca}}

\def\pe{$\exists$}
\def\pa{$\forall$}
\def\Cm{{\mathfrak Cm}}
\def\Sg{{\mathfrak Sg}}

\def\At{{\sf At}}

\def\rng{{\sf rng}}
\def\dom{{\sf dom}}

\def\cyl#1{{\sf c}_{#1}}

\def\diag#1#2{{\sf d}_{#1#2}}

\def\ws{winning strategy}

\title{The Finitizability Problem in Algebraic Logic, recent results and developments: From neat embeddings to Erdos' graphs} 
\author{Tarek Sayed Ahmed\\
Department of Mathematics, Faculty of Science,\\ 
Cairo University, Giza, Egypt.
  }
\begin{document}
\maketitle
\begin{myabstract} This is an article 
on the so-called Finitizability Problem in Algebraic Logic. We take a magical tour from the early works of Tarski on relation algebras
in the forties all the way to neat embeddings and recent resuts in algebraic logic using Erdos probabilistic graphs.
Several deep theorems proved for cylindric algebras are surveyed refined and slightly generalized to other algebraisations of first order logic,
like polyadic algebras and diagonal free cylindric algebras.
A hitherto unpublished presentation of this problem in a categorial setting is presented. Techniques from 
stability theory are applied to representation problems in algebraic 
logic.
Philosophical implications are extensively discussed.

\end{myabstract}

Algebraic logic starts from certain special logical considerations, abstracts from them, places them in a general algebraic context
and via this generalization makes contact with other branches of mathematics (like set theory and topology). 
It cannot be overemphasized that algebraic logic is more algebra than logic,
nor more logic than algebra; in this paper we argue that algebraic logic, particularly the theory of cylindric algebras,
has become sufficiently interesting and deep 
to acquire a distinguished status among other subdisciplines of mathematical logic.

The principal ideas of the theory of cylindric algebras which is the
algebraic setting of first order logic were elaborated by Tarski in
cooperation with his students L. H. Chin and F. B. Thompson during
the period 1948 - 1952. This was a natural outcome of Tarski's
formalization of the notion of truth in set theory, for indeed the
prime examples of cylindric algebras are those algebras whose
elements are sets of sequences (i.e., relations) satisfying first
order formulas. Tarski envisaged that cylindric algebras to first order logic, will be like
Boolean algebras to sentential logic.

The idea of solving problems in logic by first translating them to
algebra, then using the powerful methodology of algebra for solving
them, and then translating the solution back to logic, goes back to
Leibnitz and Pascal. Papers on the history of Logic (e.g., Anellis -
Houser \cite{AH1991}, Maddux \cite{Mad1991}) allert us to the fact that this
method was fruitfully applied in the 19th century
 with the work of Boole,
De Morgan, Peirce, Schr\"{o}der, etc., on classical logic, see
\cite{AH1991}. Employing the similarity
between logical equivalence and equality, those pioneers developed logical
systems in which metalogical investigations take on a plainly
algebraic character. Boole's work evolved into the
modern theory of Boolean algebras, and that of De Morgan, Peirce and
Schr\"{o}der led to but did not end with the theory of relation algebras.
From the beginning
of the contemporary era of logic there were two approaches to the subject, one
centered on the notion of logical equivalence and the other, reinforced by Hilbert's work on metamathematics,
centered
on the notions of assertion and inference.
It was not until much later that logicians started to think about
connections between these two ways of looking at logic. Tarski
\cite{Tar1935} gave the precise connection between Boolean algebra
and the classical propositional calculus. His approach builds on
Lindenbaum's idea of viewing the set of formulas as an algebra with
operations induced by the logical connectives. Logical equivalence
is a congruence relation on the formula algebra. This is the
so-called Lindenbaum-Tarski method. 
When Tarski applied the this method to the predicate
calculus, it led him naturally to the concept of cylindric algebras.

Also, we can see that traditionally algebraic logic has focused on
the algebraic investigation of particular classes of algebras of
logic, whether or not they could be connected to some known
assertional system by means of the Lindenbaum- Tarski method.
However, when such a connection could be established, there was
interest in investigating the relationship between various
metalogical properties of the logistic system and the algebraic
properties of the associated class of algebras (obtaining what are
sometimes called "bridge theorems"). For example, it was discovered
that there is a natural relation between the interpolation theorems
of classical, intuitionistic, intermediate propositional calculi,
and the amalgamation properties of varieties of Heyting algebras.
Similar connections were investigated between interpolation theorems
in the predicate calculus and amalgamation results in varieties of
cylindric and polyadic algebras.

Henkin began working with Tarski on the subject of cylindric algebras in the fifties, and
a report of their joint research appeared in 1961. By then Monk had
also made substantial contributions to the theory. The three
planned to write a comprehensive two-volume treatise on the theory
of cylindric algebras. The first volume treated cylindric
algebras from a general algebraic point of view, while the second
volume contained other topics, such as the representation
theory, to which Andr\'eka and N\'emeti contributed a lot,  and connections between cylindric algebras and logic. 
We can find that the theory of cylindric algebras is explicated primarily in three substantial monographs :
Henkin, Monk and Tarski \cite{HMT1}, \cite{HMT2}, and Henkin, Monk,
Tarski, Andreka and Nemeti \cite{HMTAN81}.
This covers the development of the subject till the mid eightees of the last century. This paper surveys and refines later developments of the subject. 
Highlighting the connections with graph theory, model theory, set theory, finite combinatorics, the paper presents topics of broad interest in a 
way that is accessible to a large audience. The paper is not only purely expository, for it contains, in addition, new ideas 
and results and also new approaches to old ones. We hope that this paper also provides rapid dissemination of the latest research in the field.

A cylindric algebra consists of a Boolean algebra endowed with an
additional structure consisting of distinguished elements and
operations, satisfying a certain system of equational axioms. The
introduction and study of these algebras has its motivation in two
parts of mathematics: the deductive systems of first-order logic,
and a portion of elementary set theory dealing with spaces of
various dimensions, better known as cylindric set algebras.

Cylindric set algebras are algebras whose elements are relations of a cer tain pre-assigned arity, endowed with set-theoretic operations
that utilize the form of elements of the algebra as sets of sequences.  Our notation is in conformity with the monograph \cite{HMT1}, \cite{HMT2}.
$\B(X)$ denotes the boolean set algebra $(\wp(X), \cup, \cap, \sim, \emptyset, X)$.
Let $U$ be a set and $\alpha$ an ordinal. $\alpha$ will be the dimension of the algebra.
For $s,t\in {}^{\alpha}U$ write $s\equiv_i t$ if $s(j)=t(j)$ for all $j\neq i$.
For $X\subseteq {}^{\alpha}U$ and $i,j<\alpha,$ let
$${\sf C}_iX=\{s\in {}^{\alpha}U: \exists t\in X (t\equiv_i s)\}$$
and 
$${\sf D}_{ij}=\{s\in {}^{\alpha}U: s_i=s_j\}.$$ 

$(\B(^{\alpha}U), {\sf C}_i, {\sf D}_{ij})_{i,j<\alpha}$ is called the full cylindric set algebra of dimension $\alpha$ 
with unit (or greatest element) $^{\alpha}U$.
Examples of subalgebras of such set algebras arise naturally from models of first order theories. Indeed if $\M$ is a first order structure in a first
order language $L$ with $\alpha$ many variables, then one manufactures a cylindric set algebra based on $\M$ as follows.
Let
$$\phi^{\M}=\{ s\in {}^{\alpha}{\M}: \M\models \phi[s]\},$$
(here $\M\models \phi[s]$ means that $s$ satisfies $\phi$ in $\M$), then the set
$\{\phi^{\M}: \phi \in Fm^L\}$ is a cylindric set algebra of dimension $\alpha$. Indeed 
$$\phi^{\M}\cap \psi^{\M}=(\phi\land \psi)^{\M},$$
and $$^{\alpha}{\M}\sim \phi^M=(\neg \phi)^{\M},$$ 
$${\sf C}_i(\phi^{\M})=\exists v_i\phi^{\M},$$
and finally 
$${\sf D}_{ij}=(x_i=x_j)^{\M}.$$
$\Cs_{\alpha}$ denotes the class of all subalgebras of full set algebras of dimension $\alpha$.
$\CA_{\alpha}$ stands for the class of cylindric algebras of dimension $\alpha$.
This is obtained from cylindric set algebras by a process of abstraction and is defned by a {\it finite} schema
of equations that hold of course in the more concrete set algebra.

\begin{definition} By \textit{a cylindric algebra of dimension} $\alpha$, briefly a
$\CA_{\alpha}$, we mean an
algebra
$$ {\A} = ( A, +, \cdot,-, 0 , 1 , {\sf c}_i, {\sf d}_{ij}
)_{\kappa, \lambda < \alpha}$$ where $(A, +, \cdot, -, 0, 1)$ is a
boolean algebra such that $0, 1$, and ${\sf d}_{i j}$ are
distinguished elements of $A$ (for all $j,i < \alpha$),
$-$ and ${\sf c}_i$ are unary operations on $A$ (for all
$i < \alpha$), $+$ and $.$ are binary operations on $A$, and
such that the following postulates are satisfies for any $x, y \in
A$ and any $i, j, \mu < \alpha$:
\begin{enumerate}
\item [$(C_1)$] $  {\sf c}_i 0 = 0$,
\item [$(C_2)$]$  x \leq {\sf c}_i x \,\ ( i.e., x + {\sf c}_i x = {\sf c}_i x)$,
\item [$(C_3)$]$  {\sf c}_i (x\cdot {\sf c}_i y )  = {\sf c}_i x\cdot  {\sf c}_i y $,
\item [$(C_4)$] $  {\sf c}_i {\sf c}_j x   = {\sf c}_j {\sf c}_i x $,
\item [$(C_5)$]$  {\sf d}_{i i} = 1 $,
\item [$(C_6)$]if $  i \neq j, \mu$, then
 ${\sf d}_{j \mu} = {\sf c}_i
 ( {\sf d}_{j i} \cdot  {\sf d}_{i \mu}  )  $,
\item [$(C_7)$] if $  i \neq j$, then
 ${\sf c}_i ( {\sf d}_{i j} \cdot  x) \cdot   {\sf c}_i
 ( {\sf d}_{i j} \cdot  - x) = 0
 $.
\end{enumerate}
\end{definition}

${\A}\in \CA_{\omega}$ is locally finite, if the dimension set of every element $x\in A$
is finite. The dimension set of $x$, or $\Delta x$ for short, is the set
$\{i\in \omega: {\sf c}_ix\neq x\}.$ 
Tarski proved that every locally finite $\omega$-dimensional
cylindric algebra is representable, i.e. isomorphic
to a subdirect product of set algebra each of dimension $\omega$.
Let $\Lf_{\omega}$ denote the class of 
locally finite cylindric algebras. 
Let $\RCA_{\omega}$ stand for the class of isomorphic copies 
of subdirect products
of set algebras each of dimension $\omega$, 
or briefly, 
the class of $\omega$ dimensional representable
cylindric algebras. Then Tarski's theorem reads 
$\Lf_{\omega}\subseteq \RCA_{\omega}$. 
This representation theorem is non-trivial; in fact it is equivalent
to G\"odel's celebrated Completeness Theorem \cite[\S 4.3]{HMT2}.

Soon in the development of the subject, 
it transpired that the class $\Lf_{\omega}$, the algebraic counterpart
of first order logic,  had 
some serious defects when treated as the 
sole subject of research in an autonomous algebraic theory.
In universal algebra one prefers to deal 
with {\it equational classes} of algebras i.e. classes
of algebras characterized by
systems of postulates, in which 
every postulate has the form of an equation (an identity). 
Such classes are also referred to as {\it varieties}. 

Classes of algebras which are not varieties are often 
introduced in discussions
as specialized subclasses of varieties. 
One often treats fields as a special case of rings. 
This is due to the tradition that in algebra, mainly the equational
language and thus equational logic is used.
Thus, finding an equational form for an algebraic entity is always a value on its own 
right.
Another reason for this preference, is 
the fact that every variety is closed under 
certain general closure operations frequently 
used to construct new algebras from given ones.
We mean here the operations of forming subalgebras, 
homomorphic images and direct products.
By a well known theorem of Garrett Birkhoff, 
varieties are precisely those classes of algebras that have all 
three of these closure properties.  
Local finiteness does not have the form of an identity, 
nor can it be equivalently 
replaced by any identity or system of identities, nor indeed
any set of first order axioms.
This follows from the simple observation that
the ultraproduct of infinitely 
many algebras in $\Lf_{\omega}$ is not, in general, 
locally finite, and a first order
axiomatizable class is necessarily closed under ultraproducts.

When Alfred Tarski introduced cylindric algebras, 
he introduced the class of locally finite cylindric algebras
and proved that $\Lf_{\omega}\subseteq \RCA_{\omega}$.

But some modifications in the definition of Tarski's cylindric algebras
seemed desirable. The definition contains certain assumptions
which considerably restrict the scope of the definition and thus 
can be dispensed 
with. One such assumption is the fixed dimension $\omega$.
The other is local finiteness. 
The restrictive character of these two notions
becomes obvious when we turn our attention to cylindric set algebras.
We find that there are algebras of all dimensions, 
and set algebras that are not locally finite 
are easily constructed. For these reasons the original conception of a cylindric algebra
was extended. The restriction to dimension $\omega$
and local finiteness were removed, 
and the class $\CA_{\alpha}$, of cylindric algebras of dimension $\alpha$, 
where $\alpha$ is any ordinal, finite or tranfinite, was introduced. 

\begin{athm}{Problem}
A central and indeed still (very) active part of research in algebraic logic 
is the vaguely posed frequently 
discussed problem concerning 
improvements of Tarski's representation Theorem. 
\end{athm}

This problem is referred to as 
{\it The Finitizability Problem} by the Budapest group, specifically by
Andr\'eka, N\'emeti  and Sain \cite{N96}
while it is referred to as the {\it Representation Problem} by the London group, 
specifically by Hodkinson and Hirsch \cite{HHbook}.
Here we need to clarify two points.  
First,  strictly speaking, the Finitizability Problem and the 
Representation Problem are not one and the same thing. 
The Finitizability Problem is more restrictive than the 
Representation Problem.
The Finitizability Problem is the Representation Problem 
restricted to the case when in the search for axioms that enforce 
representability, one requires that such axiomatizations are not 
only ``simple, elegant, transparent, decidable,  etc.'' but also {\it strictly finite.} 
In our subsequent discussion, when we refer to the above problem  we 
shall use both words relying on  context.  

Second,  the attribution of names to schools could be 
a little misleading. For example,
the early J\'onsson-Tarski paper \cite{JonTar48} 
was entitled ``{\it Representation problems} for relation algebras''.
Their Representation
Problem is the following: We are given the equationally defined class of 
relation algebras and the
set-theoretically defined class of special relation algebras arising from binary
relations, with the concrete non-boolean operations being composition and forming converses. 
Now,  is every one of the former isomorphic to one of the latter, i.e., is every
relation algebra representable? The famous answer is ``no''. 
So one can say that the 
Representation Problem 
can be traced back to the work of Tarski and his students
on relation algebras
back in the forties of the 20th century. This, however, does not change the fact
that Hirsch and Hodkinson attribute the name ``Representation Problem" 
to the problem above. In the course of our discussion we will give more tangible and 
concrete
forms of the Finitizability Problem.

To start with, the Finitizability Problem
asks for an equationally defined class of algebras, with only finitely many finitary
operations, such that this class can be defined by finitely many equations. In addition, 
every algebra in
this class is representable (assuming that that notion has been suitably defined) 
with the additional (rather vague) property  that the equational theory of that class can
serve as an adequate algebraic version of first-order logic (with or without equality). 
Sain  \cite{Sa98} did
this for logic without equality. It seems that it is unlikely to be done for 
first order logic with equality and there are results in the literature that support this pesimistic viewpoint. 
For attempts to explain what this means, see the original formulations of
Henkin, Monk, and Tarski, for example, in paragraph 1, page 20, of \cite{Monk},
Problem 1 in \cite{HenMonk} and Section 3.5 of \cite{TG}. 

For recent extensive discussions of the Finitizability Problem,
one can of course look at N\'emeti \cite{N96}, or the papers 
of Sain \cite{Sa98} and Simon
\cite{Sim93}. Simon's emphasis \cite{Sim93} is on what the Finitization oroblem is {\it not}.
Here we concentarte on what the Finitizability problem {\it is}.
In what follows we give our reading of the Finitizability Problem 
and, for that matter,
the Representation Problem. 
This problem has, and continues to, invoke extensive amount of research.
To get a grasp of how substantial the problem is, 
let us  start ``from the beginning". 

We intend to find a representation theorem for cylindric algebras, which 
is similar to that for Boolean algebras, due to Stone.
In the latter case, Stone proved that every Boolean algebra 
is isomorphic to a Boolean set algebra.
In analogy with this, we would like to prove that every cylindric 
algebra is isomorphic to some ``concrete cylindric set algebra".
We will explain why the class $\RCA_{\alpha}$ 
of representable cylindric algebras is 
indeed a plausible ``natural'' candidate 
for this. 
It is easily seen that every cylindric set algebra
of given dimension $\alpha<\omega$ is simple (has no proper congruences)
and therefore subdirectly (and directly) indecomposable 
in the sense of the general theory of algebras \cite{HMT1}.
Hence when discussing the problem 
as to which $\CA_{\alpha}$'s are isomorphic to cylindric 
set algebras, 
it is natural to restrict ourselves to subdirectly indecomposable algebras.
On the other hand, as a consequence of a classical 
theorem of Birkhoff, every $\CA_{\alpha}$ is isomorphic to a subdirect product 
of subdirectly indecomposable $\CA_{\alpha}$'s. Therefore
we are naturally led to the problem of characterizing those $\CA_{\alpha}$'s
which are isomorphic to subdirect products of set algebras. 
Henkin, Monk and Tarski
declare that these
are the {\it representable} algebras, thus the notation $\RCA_{\alpha}.$ 

$\RCA_{\alpha}$ plays a role in the theory of cylindric algebras analogous to the role played
by Boolean set algebras in the theory of Boolean algebra, rings of matrices in ring theory and group of permutations in group theory.
The Representation
Problem in Boolean algebras is completely resolved by Stone's Theorem. 
In ring theory we see for example
the Wederburn-Artin Theorem
and Goldie's Theorem, which gives nice intrinsic 
conditions for an abstract ring to be 
isomorphic to a subdirect product
of rings of matrices. As opposed to 
boolean
algebras, the Representation Problem 
for the $\CA$ case proves to be much substantial; indeed it proves to be harder and richer. 

The definition of representability, without any change 
in its formulation, is extended to algebras of infinite dimension. 
In this case, however, an intuitive justification is less clear
since cylindric set algebras of infinite dimension 
are not in general subdirectly indecomposable. 
In fact, for $\alpha\geq \omega$
no intrinsic property is known which singles 
out the algebras isomorphic  to set algebras
among all representable $\CA_{\alpha}$'s, as opposed to the finite dimensional case
where such algebras can be intrinsically characterized by the property of being simple.
But in any case, members of $\RCA_{\alpha}$
can be still represented as algebras consisting of genuine 
$\alpha$-ary relations over
a disjoint union of Cartesian squares, the class consisting of all such algebras
is denoted by ${\bf Gs}_{\alpha}$, with ${\bf Gs}$ standing for 
{\it generalized set algebras.} 

Generalized set algebras thus
differ from the ordinary cylindric set 
algebras in one respect only: the unit of the algebra, i.e., the 
$\alpha$-dimensional Cartesian space $^{\alpha}U$, is 
replaced everywhere in their construction by 
any set that is a disjoint union of 
arbitrarily many pairwise disjoint Cartesian spaces of the same 
dimension. This broadening of the definition makes this class of concrete algebras closed under products, a necessary condition
to be a variety, which it is. At the same time, the class of generalized cylindric set algebras, 
just as that of ordinary cylindric set algebras,
has many features that make it well qualified to represent
$\CA_{\alpha}$. The construction of the algebras in this (bigger) class 
retains its concrete character, 
all the fundamental operations and distinguished elements
are unambiguously defined in set-theoretic terms, and the definitions are uniform
over the whole class; 
geometric intuition underlying the construction 
gives us good insight into the structures of the algebras.
Thus there is (geometric) 
justification that 
$\RCA_{\alpha}$ consists of the standard models
of $\CA$-theory. Its members consist of genuine $\alpha$-ary relations, and the operations are set-theoretically concretely 
defined utilizing the form of these
relations as sets of sequences.

But it soon transpired that the $\CA$ axioms (originating from the (complete) 
axiomatization 
of locally finite algebras, 
do not exhaustively generate all valid principles governing 
$\alpha$-ary relations, when $\alpha>1$. More precisely,
for $\alpha>1$, $\RCA_{\alpha}$ is properly contained in $\CA_{\alpha}$.
$\CA_{\alpha}$, for $\alpha>1$, is only an {\it approximation} 
of $\RCA_{\alpha}$. 
Tarski \cite{HMT1} proved that $\RCA_{\alpha}$ is a variety.
Henkin \cite{HMT2} proved that $\RCA_2$ is finitely axiomatizable. 
However for $\alpha>2$, the class $\RCA_{\alpha}$ cannot be axiomatized 
by a finite schema of equations analogous to that
axiomatizing $\CA_{\alpha}$, a classical result of Monk \cite{M69}. 
Furthermore, for any $\alpha>2$, there is an unavoidable
and inevitable 
degree of complexity 
to any (potential) axiomatization of $\RCA_{\alpha}$, as shown by 
Andr\'eka \cite{An97}.  For example Andreka proved that if $E$ is an equational 
axiomatization of $\RCA_n$ for $2<n<\omega$, there for any natural number $k$ there is an equation in $E$ containing more than $k$ distinct variables 
and all the operation symbols. 
We will refine Andreka's complexity results in the last section in treating the class of representable quasipolyadic equality
algebras, which is a cousin of cylindric algebras.
 
The Finitizability Problem (and for that matter a form of the Representation Problem) 
is thus the attempt to circumvent or sidestep such
complexity.  
If we look at $\RCA_{\alpha}$ as the 
standard models to which the $\CA_{\alpha}$'s aspire, 
the Finitizability Problem
can  thus be rephrased as the attempt to capture the essence 
of the standard models by thorough ``finitary'' 
means. Alternatively, to find 
other broader comprehensible classes  
of ``standard models''  that are sufficiently concrete and tangible.
Most important of all these classes would have to exhaust
the class $\CA_{\alpha}$, or in the worst case possibly a slightly smaller class,
i.e., a variety that is finitely axiomatizable 
(by equations) over 
$\CA_{\alpha}.$

If the class of cylindric set algebras had turned out to be finitely axiomatizable, 
algebraic logic would have evolved along a significantly  different path than it did in the past 
$40$ years. This would have undoubtfully marked the end of the abstract class $\CA_{\alpha}$ as a separate subject of research; 
after all why bother about abstract algebras, if a few nice extra axioms can lead us from those
to concrete algebras consisting of genuine relations, with set theoretic operations uniformy defined over these relations. However, 
due to Monk's non-finitizability result, together with its improvements by various algebraic logicians (from Andr\'eka to Venema) 
$\CA_{\alpha}$ was here to stay and its ``infinite distance' from $\RCA_{\alpha}$
became an important research topic.

\section{Solutions} 

\subsection{Using Twisting}

There has been work in representing 
cylindric algebras using 
quasigroups, cf. \cite{M74}, \cite[p. 91--93]{HMT2}, \cite{Co}
or sheaves \cite{Comer}. Groups were used to represent 
$\CA_3$'s and relation algebras
in \cite{Giv}. A classical result of 
Resek \cite[p. 01]{HMT2}
that is relevant in this connection 
shows that algebras satisfying the $\CA$ axioms plus 
the so-called 
{\it merry-go-round identities}, 
or $MGR$ for short, 
can be represented as {\it relativized} set algebras, 
a primary advance in development of the theory of $\CA$'s, 
as indicated in the introduction of 
\cite{HMT2}. 
Resek's result was polished and 
``finitized'' by Thompson. We refer in this connection to the  Andr\'eka and 
Thompson's paper
\cite{AT}.  
The replacement of $MGR$ by a finite scheme is entirely due to Thompson and
appears in his dissertation; see \cite[3.2.8]{HMT2}. The proof in \cite{AT} is due to 
Andr\'eka. 
It might be not be appropriate to use the word ``finitized'' here, since, if $\alpha<\omega$,
Resek's result already produces a \emph{finite} axiom set, and if $\alpha\geq\omega$, then
Thompson's simplification is still an \emph{infinite} axiom set. 
What is meant
in this context, is that, in case $\alpha\geq\omega$, 
there are originally infinitely many $MGR$ schemata,
and that Thompson reduced this infinity to 2.  
Thompson actually proved much more than
stated here---
he weakened the commutivity of cylindrifications and showed that atoms of
the algebras are represented as orbits of single sequences under  groups 
of permutations of the underlying set. A complete statement (with proofs) of
Resek's and Thompson's theorems can be found in \cite{Maddux}.
We note that the first proof provided for this theorem was more than $100$ pages long. So the result is mentioned in \cite{HMT2} without proof.
Below we give a new sketch of proof of 
the Resek Thompson theorem (due to Ferenczi) using neat embeddings.
Recently, Simon \cite{Sim97} proved that any abstract 
$3$-dimensional cylindric algebra satisfying 
$MGR$ can be obtained from a ${\bf Cs}_3$ by the methods called {\it twisting and
dilation}, studied in
\cite[p. 86--91]{HMT2}. This adds to our understanding of the distance between
the abstract notion of cylindric algebra 
and its concrete one, at least in the case of dimension 3. 
However, Simon had to broaden Henkin's notion of twisting to exhaust the class 
$\CA_3$.  He also showed that Henkin's more restrictive notion of twisting 
does not fit the bill; there are abstract $\CA_3$'s satisfying the $MGR$
that  cannot be obtained by the methods of 
relativization, dilation and twisting, the latter 
understood in the sense of Henkin. Simon's twisting is a stronger ``distortion" of the original algebra, and so its scope
is wider, it can ``reach" more algebras. 
The analogous problem for higher dimension is an intriguing open problem. 
(We heard however, that Thompson proved the result for $\CA_4$).

Now we give an outline of Simon's result. Now instead of asking  
``What is missing from $\CA$'s to be representable?", Henkin turned around the question and asked how much set 
algebras needed to be distorted to provide a representation of all $\CA_{\alpha}$'s. And, strikingly, the anwser is 
``not very much", at least for the lowest value of $\alpha$, 
for which Monk's result and its improvements apply, namely $\alpha=3$.

\begin{definition} Let $\A\in \CA_{\alpha}$. Then ${\sf s}_i^jx={\sf c}_i({\sf d}_{ij}\cdot x)$ if $i\neq j$ and ${\sf s}_i^ix=x$.
$${}_k{\sf s}(i,j)x={\sf s}_i^k{\sf s}_j^i{\sf s}_k^jx.$$
For $x\in A$, 
$$\breve{x}={}_2{\sf s}(0,1)x.$$
Let $k,l,m$ be distinct. Then $MGR_k$ is the equation
$${}_k{\sf s}(l,m){\sf c}_kx={}_k{\sf s}(m,l){\sf c}_kx.$$
\end{definition}
The proof of the following can be destilled from \cite{HMT1} Theorem 1.5.15-1.5.17.
\begin{theorem}
Let $\A\in \CA_3$ and $\{k,l,m\}=3$.
Then 
\begin{enumroman}
\item $\A\models MGR_k$ iff $\A\models {}_k{\sf s}(l,m)_k{\sf s}(l,m){\sf c}_kx={\sf c}_kx$
\item $\A\models MGR_k$ iff $\A\models MGR_l$
\end{enumroman}
\end{theorem}
Then, we may write $MGR$ for any of the equations $MGR_k$ $(k<3)$.
If $\alpha>3$, then $MGR$ consists of two schemas of equations
$${}_k{\sf s}(l,m){\sf c}_kx={}_k{\sf s}(m,l){\sf c}_kx \text { when } |\{k,l,m\}|=3.$$
$${}_k{\sf s}(l,m)_k{\sf s}(m,n){\sf c}_kx={}_k{\sf s}(n,l)_k{\sf s}(l,m){\sf c}_kx \text { when }\{k,l,m,n\}|=4.$$
In both dilation and twising one starts out with a complete and atomic $\CA$, adjoins new elements and /or changes the operations to get a new, 
complete atomic $\CA$ with certain prescribed 
properties. In $\CA$'s of this kind, and actually in all complete and atomic Boolean algebras with operators where the extra non Boolean operations 
distribute over arbitrary joins, the operators are determined by their behaviour on the atoms. Then it is possible, and often even desirable, to work with the 
{\it atom structure} of such an algebra instead of the algebra itself. 
(However, a word of caution is in order; 
this does not always work. It can be proved that there exists $\A$ and $\B$ having the same atom structure; that is $\At\A=\At\B$, 
with $\B\subseteq \A$, while $\B$ is representable and $\A$ is not. What is going on here is that $\A$ 
has more elements and these cannot be represented as 
true relations. Furthermore, $\A$ can be chosen to be the minimal completion of $\B$. More on that later on, theorems \ref{com}, \ref{com1}, \ref{complete}).

In th following definition, the composition $\{(x,y): \exists z (xRz\land zSy)\}$ of binary relations $R$ and $S$ will be denoted by $R|S$.
If $R$ is an $n+1$-ary relation and $X_0, X_1\ldots X_{n-1}$ are sets then $R^*(X_0\ldots X_{n-1})$ stands for the $R-$image 
$\{y:(\exists x_0\in X_0)\ldots (\exists x_{n-1}\in X_{n-1})R(x_0\ldots x_{n-1}y)\}$ of $X_0\ldots X_{n-1}.$

\begin{definition} Let $\alpha$ be an ordinal. A structure
$$\B=(B, T_i, E_{ij})_{i,j<\alpha}$$
with binary relation $T_i$ and unary relations $E_{ij}$ is a cylindric atom structure of dimension $\alpha$ if the following conditions hold for all $i,j,k<\alpha$:
\begin{enumroman}
\item $T_i$ is an equivalence relation on $B$
\item $T_i|T_j=T_j|T_i$
\item $E_{ii}=B$
\item $E_{ij}=T_k(E_{ik}\cap E_{kj})$ if $k\notin \{i,j\}$
\item $T_j\cap (E_{ij}\times E_{ij})\subseteq Id$ if $i\neq j$.
\end{enumroman}
\end{definition}
$\Ca_{\alpha}$ is the class of cylindric atom structures of dimension $\alpha$.
The complex algebra $\Cm\B$ of an atom structure $(B, T_i, E_{i,j})$ is the algebra
$$(\wp(B), \cap, \cup, \sim, T_i^*, E_{ij})_{i,j<\alpha}$$
where for $X\subseteq B$, 
$$T_i^*X=\{y\in B:  (\exists x\in X )(xT_iy)\}$$
The proof of the following is tedious but routine.

\begin{theorem} If $\B\in \Ca_{\alpha}$ iff $\Cm\B\in \CA_{\alpha}$
\end{theorem}
\begin{demo}{Proof} \cite{HMT1} Theorem 2.7.40.
\end{demo}

The idea behind dilation may be expressed vaguely as follows. 
In a $\CA$ iff it is not outright impossible to have an atom in a certain position, then insert a new 
atom there.

\begin{definition}\label{dilate} Let $\B=(B,T_i, E_{ij})_{i,j<\alpha}\in \Ca_{\alpha}$. Let 
$\Psi\subseteq {}^{\alpha}B$ and suppose that for all $\psi\in \Psi$
\begin{equation}
\begin{split}
\psi_i\notin E_{jk} \text { if } |\{i,j,k\}|=3,
\end{split}
\end{equation}
\begin{equation}
\begin{split}
\psi_iT_i|T_j\psi_j \text { for all } i,j<\alpha
\end{split}
\end{equation}
Then the result of dilating $\B$ with $\Psi$ is $\B^{\Psi}=(B^{\Psi}, T_i^{\Psi}, E_{ij}^{\Psi})$
where $B^{\Psi}=B\cup \{v_{\Psi}: \psi\in \psi\}$, $E_{ii}^{\Psi}=\B^{\Psi}$ and $E_{ij}^{\Psi}=E_{ij}$ if $i\neq j$. 
To define $T_i^{\Psi}$ it is convenient to introduce the following notation for $a\in B^{\Psi}$ and $i<\alpha$.
$|a|_i=a$ if $a\in B$ and $|a|_i=\psi_i$ if $a=v_{\psi}$.
Then we let
$$aT_i^{\Psi}b\Leftrightarrow |a|_iT_i|b|_i.$$
\end{definition}
So the new atoms are all outside the diagonals, and each of the is ``coordinatized" by an $\alpha$ sequence of old atoms: the $ith$ element
of this sequence determines how the new atom behaves with respect to $T_i$.
We note that the definition in \cite{HMT2} 3.2.69 allows only dilating with one sequence. So Simon's dilation is "apparently" more general. However, this is not the case, 
for it is not hard to show that
dilating a $\Ca_{\alpha}$ $\B$ with the set $\{\psi_{\beta}: \beta<\mu\}$ (where $\mu$) is an ordinal can be simulated by 
taking the union of $\B_{\beta}$, where $\B_{\beta}$ is $\B_{\delta}^{\psi_{\delta}}$
if $\beta$ is the succesor of $\delta$, or $\bigcup\{\B_{\delta}: \delta<\beta\}$ if $\beta$ is a limit ordinal.
Now if $\B\in \Ca_3$ and $\B^{\Psi}$ 
is the result of dilating $\B$ with $\psi$, then $\Cm\B^{\Psi}\in \Ca_{\alpha}$; furthermore if 
$\Cm\B^{\Psi}\models MGR$ provided $\Cm\B\models MGR$.
So dilation does not take us out from $\CA$'s and it also prerves the $MGR$ identities.
Now twisting, originally, consists of starting from a complete atomic $\CA_\alpha$ $\A$, selecting atoms $a,b\in \A$ 
and an ordinal $k<\alpha$ and then redefining ${\sf c}_k$ on 
$a$ and $b$ by interchanging the action of ${\sf c}_k$ on $a$ and $b$, in part, ``twisting".
Twisting is used to ``'distort" atom structures. It produces $\Ca_{\alpha}$'s from 
$\Ca_{\alpha}$'s, and it typically kills $MGR$. (However, in some circumstances it reproduces $MGR$!).
Henkin's twisting is defined as follows. Suppose we have a $\Ca_{\alpha}$ 
$\B=(B, T_i, E_{ij})_{i,j<\alpha}$, $x,y\in B$, with not $xT_ky$ 
and two partitions $x/T_k=X_0\cup X_1$, $y/T_k=Y_0\cup Y_1$, where $X_0\cap X_1=\emptyset=Y_0\cap Y_1$ 
and the following condition hold; for brevity write $M=x/T_k\cup y/T_k$:
\begin{enumarab}
\item $(M\upharpoonright T_i)\sim Id\subseteq (X_0\times Y_0)\cup (y_0\times X_0)\cup (X_1\times Y_1)\cup (Y_1\times X_1)$ for all $i\in \alpha\sim 
\{k\}$
\item If $i\in \alpha\sim \{k\}$ and $a\in M$, then there is a $b\in M\sim \{a\}$ such that $aT_ib$.
\item If $i\in 2$, and $\lambda,\mu\in \alpha$, then $X_i\cap E_{k\lambda}\cap E_{k\mu}\neq \emptyset$ 
iff $Y_i\cap E_{k\lambda}\cap E_{k\mu}\neq \emptyset$.
\end{enumarab}
Then define
$\B'=(B, T_i, E_{ij})$ as follows
$T_i'=T_i$ if $i\neq k$ while $T_k$' is the equivalence relation on $\B$ wuth equivalence classes
$k/T_k$ for $k/T_k\cap M=\emptyset$ along with $X_0\cup Y_0$ and $X_1\cup Y_0$.
now Simon generalized Henkin's twisting by taking an arbitrary sequence of atoms, instead of just two.
The two notion coinicide when $|I|=2.$

\begin{definition}\label{twist}
Let $\B = (B, T_i, E_{ij})_{i,j< \alpha} \in \Ca_\alpha$, $t \in\alpha$ and $\xi \in {}^I B$ for some set $I$, and suppose that
\begin{equation}
\begin{split}
(\xi_i, \xi_j) \in T_t ~~~~ \textrm{ for all distinct }~~~~~ i, j \in I
\end{split}
\end{equation}
\begin{equation}
\begin{split}\xi_i \notin E_{jk}~~ \textrm{for all}~~ i\in I~~ \textrm{and all distinct} ~~j,k < \alpha~~ \textrm{such that} ~~ t \notin \{j,k \}.
\end{split}
\end{equation}
For $i\in I$, let $\Xi_i$ denote the $T_t$-class of $\xi_i$, let $\pi$ be a permutation of $I$, 
and for all $i\in I$, let $\Xi_i$ be partitioned into $\Xi'_i$ and $\Xi''_i$. 
Assume that for all $i\in I$ and $j < \alpha, j \neq t$,
\begin{equation}
\begin{split}
dom(T_j \cap (\Xi'_i \times \Xi'_{\pi i})) \supseteq \Xi'_i,~~~~~ran(T_j \cap (\Xi'_i \times \Xi'_{\pi i})) \supseteq \Xi'_{\pi i}\\dom(T_j \cap (\Xi''_i \times \Xi''_{\pi i})) 
\supseteq \Xi''_i,~~~~~ran(T_j \cap (\Xi''_i \times \Xi''_{\pi i})) \supseteq\Xi''_{\pi i}
\end{split}
\end{equation}
Then we form a new relational structure $\B' = (B, T'_i,E_{ij})_{i,j < \alpha}$ by letting $T'_i$ be the equivalence 
relation on $B$ with equivalence classes $x/T_t$ for $x\in B \sim\bigcup_{i\in I} \Xi_i$, 
together with the classes $\Xi'_i \cup\Xi''_{\pi i} (i \in I)$, and $T'_i = T_i$ if $i\neq t$. We say that $\B' (\C m \B')$ is a twisted version of $\B (\Cm \B)$.  
\end{definition}
Let $${\bf Tw}K=\{I\K\cup I\{\A: \A \text { is a twisted version of some } \C\in K\}.$$

The final operation on $\CA$'s we consider is that of {\it relativization}.
Relativization of a $\CA_{\alpha}$ is defined in \cite{HMT1} 2.2.4. It is exactly like relativization in Boolean algebras.
That is if $\A\in \CA_{\alpha}$ and $a\in A$, then ${\bf Rl}_a\A=\{x\in \A: x\leq a\}$ and the operations are relativized to $a$.
If $K\subseteq \CA_{\alpha}$, let $${\bf Rl} K=\{{\bf Rl}_a\A: a\in A, \A\in K\}.$$
But unlike twisting, dilation, and for that matter, relativization in Boolean algebras, relativization can get us out of cylindric algebras.
So let $${\bf Rl}^{ca}K=\CA_{\alpha}\cap {\bf Rl}K$$
and let 
$${\bf Srl}^{ca}K=\CA_{\alpha}\cap {\bf SRl}K,$$
where the second ${\bf S}$ stands for the operation of forming subalgebras.
Note that the Resek Thompson result says that $\CA_3\cap Mod(MGR)\subseteq {\bf Srl}^{ca}\Cs_3$.
Now we are ready to recall Simon's amazing theorem:

\begin{theorem}$\CA_3\subseteq {\bf SRl}^{ca}{\bf Tw}{\bf Srl}^{ca}\Cs_3$. That is for every $\A\in \CA_3$ there are 
$\A_1$, $\A_2$, $\A_3\in \CA_3$ and $\A_4\in \Cs_3$ such that
$\A_3\subseteq {\bf Rl}\A_4$, $\A_2$ is a twisted version of $\A_3$, $\A_1={\bf Rl}\A_2$ and $\A\subseteq \A_1$
\end{theorem}
\begin{demo}{Sketch of proof} One starts with a $\CA_3$ $\A$ and checks if $MGR$ holds in it. If it does then by the Resek Thompson theorem, 
$\A$ is a subalgebra of a relativized $\Cs_3$, and we are done. Otherwise 
$\A$ embeds into its canonical extension $\A^{\sigma}$ 
($\A_1$ in the formulation of the theorem) in order to be able to repair the failure of $MGR$ by twisting its atom structure. 
Note that $\A_1$ is complete and atomic, and is in 
$\CA_3$. Before the parameters in twisting is chosen, one has to apply dilation first so that definition \ref{twist} applies. That is where 
${\bf Rl}^{ca}$ in the theorem comes. $\A^{\sigma}$ can be recovered by 
relativising the dilated algebra $\A_2$ with the top element (i.e. the sum of atoms) of 
$\A^{\sigma}.$
The next step is to apply twisting to the dilated algebra and get a $\CA_3$ $\A_3$ in which $MGR$ holds, 
and then use the Resek thompson result to represent the latter as an  ${\bf Srl}^{ca}\Cs_3$. 
So here twising is used in a more constructive way; by twisting an algebra in which $MGR$ does not hold, we get one where
$MGR$ holds.  Since the effect of twisting can always be undone by twisting the twisted algebra, the procedure we have described show that $\A$ can be
obtained from a subalgebra of a relativized $\Cs_3$ by applying twisting, relativization 
and the operation of forming subalgebras.
\end{demo}

A different slant on the Representation Problem  is:
How can we abstract 
away from the subject matter of specific concrete $\alpha$-ary relations 
to arrive at their essential 
forms? For unary relations the answer is given completely by Stone's Theorem:
Every boolean algebra is isomorphic to a set algebra, 
i.e. is representable.
The significance of this is that the (finitely many) axioms of boolean algebras
exactly capture the true properties of unary relations: 
all and only those properties that hold 
in every 
domain of individuals endowed with unary relations are derivable from the axioms.
This is a great achievement for the algebraic viewpoint.
The same situation holds for locally finite cylindric algebras.
The axiomatization provided by Tarski 
is sufficient for representability of locally finite cylindric algebras.
However, if we want to {\it algebraize first order logic}, then we have to 
remove the (non-first order) restriction of local finiteness, 
and keep only equations
or at worst quasi-equations.  According to Blok and Pigozzi, 
\cite{BP89} {\it algebraizable} logics are those 
logics whose algebraic counterparts are quasi-varieties, 
i.e. classes of algebras axiomatized by quasi-equations.
But then infinity seems unavoidable, it backlashes in a strong 
sense:  Any first order  universal axiomatization 
of the class $\RCA_{\alpha}$ for $\alpha>2$ is essentially infinite, and has to be extremely complicated. 
A sample of Andreka's result for $\alpha\geq \omega$, is that if $\Sigma$ is a set of univeral formulas 
axiomatizing $\RCA_{\alpha}$, $k\in \omega$ and $l<\alpha$, then $\Sigma$ contains
infinitely many formulas in which at least one diagonal constant with index $l$, 
more than $k$ cylindrifications, and more than $k$ variables 
occur.

The metalogical aspect of the Finitizability Problem asks
for an algebraizable expansion of first order logic that admits a finite complete 
and sound Hilbert style axiomatization
of the valid formula schemata, involving only valid formula schemata. 
This is formulated as problem 4.16 in \cite{HMT2}.
This form of the Finitizability Problem,  
it seems, is not unrelated to Hilbert's program of finitizing 
metamathematics, and indeed it seems to add
to our knowledge of 
reasoning about reasoning. 

On the other hand, the Representation  Problem (in its algebraic form) 
is indeed non-trivial, as the following 
quotation from Henkin, Monk and Tarski in  \cite{HMT1}
might suggest:
``An outstanding open problem in cylindric algebra theory is that of exhibiting a class
of cylindric algebras which contains an isomorphic image of 
every cylindric algebra and hence serves to represent the class of all these
algebras, and which at the same time is 
sufficiently concrete and simply constructed to qualify
for this purpose from an intuitive point of view. 
It is by no means certain or even highly plausible
that a satisfactory solution of this problem will ever be found!''
(our exclamation mark).

Fortunately today the situation seems to be not as drastic!
However, it still involves some open questions.

\subsection{Without any twisting}

Several different stratagems were evolved to 
get round the obstacle of the non-finite axiomatizability
of the class of representable algebras.
One promulgated by Tarski 
especially was to find elegant intrinsic 
conditions for representability.
For example, certain comprehensible 
subclasses of abstractly defined $\CA_{\alpha}$'s turn out to be representable.
In this connection, examples include locally finite, dimension complemented, 
semisimple, and diagonal algebras of infinite dimension, 
cf. \cite[Thm. 3.2.11]{HMT2}

Another  sample of such results in this direction is the classical result of 
Henkin and Tarski, formulated as Thm. 3.2.14 in \cite{HMT2}, that states that 
any atomic $\CA_{\alpha}$ whose atoms are rectangular
is representable. This was strengthened by Andr\'eka et al. \cite{AGMNS},  by 
looking at dense subsets consisting of rectangular elements 
that are not necessarily atoms. Venema
\cite{V98} extended this result to the diagonal free case. Such a representation theorem allows the introduction of non-orthodox complete
axiomatization of $\L_n$, when $n\geq 3$, which is first order logic restricted to the first $n$ variables.

This approach of finding simple intrinsic sufficient conditions for representability
has continued to the present, and now forms an extensive 
field, cf. \cite{HMT1}, \cite{HMT2}, \cite{N96}.

Another strategy of 
attacking the Finitizability Problem is to define variants of 
$\RCA_{\alpha}$, $\alpha>2$  
that are finitely axiomatizable 
and are still adequate to algebraize first order logic. 
Such an approach originates
with Craig \cite{C} and is further pursued by
Sain \cite{Sa98}, \cite{SaGy96}, Simon \cite{Sim93}
and the present author \cite{ma}, \cite{BL}, \cite{FC}, \cite{th}, \cite{Bulletin}, \cite{SSL}, \cite{Notre}, and \cite{AU}.

The reasoning here is that
maybe the negative results we already mentioned
are merely a historical accident resulting from the particular (far from unique) 
choice of extra non-boolean operations,
namely the cylindrifications and diagonal elements.
This approach typically involves changing the signature of $\CA_{\alpha}$ 
by either taking reducts or expansions. 
Else, perhaps even changing the signature altogether but 
bearing in mind that
cylindrifications and diagonal 
elements are term definable in the new signature, cf. \cite{C} and \cite{Sa98}. 
This could be accompanied by broadening the notion of representability, allowing 
representation on arbitrary
subsets of $\alpha$-ary relations, rather than just (disjoint unions of) 
Cartesian squares \cite{Marx99}.  
The approach of broadening the permissible units is referred 
to in the literature as {\it relativization} or the {\it non-square} approach.
The
term relativization is a term that is already old in logic, 
and has been used in the theory of cylindric algebras from the
very beginning, see \cite[\S 2.2]{HMT1}. It is based on Henkin's ideas of changing the semantics to obtain a completeness theorem for higher
order logics.

The second term
comes from the Amsterdam Group, in fact it is due to Venema \cite{Venema}.
Relativization might involve adding new operations that become no longer
term definable after relativization, 
such as the difference operator \cite{Venema}.
This approach is related to {\it dynamic logic}, cf. \cite{AJN}. 
Modalizing set algebras
yields variants of the $n$-variable fragment of first order logic 
differing from the classical
Tarskian view, because the unit $W$ of the set algebra in question may 
not be of the form of a ``square'' $^nU$, but merely a subset thereof. 
These mutant logics (like the guarded fragments of first order logic) 
are under intensive 
study at the present time and we cite \cite{AJN} and \cite{MLM} 
as sources. 

We should mention, in this connection, that 
\cite{Sa98} provides a solution for first order logic without equality.
Indeed in \cite{Sa98} a stricty finite set of axioms are given for 
a class of representable algebras that is an extension of 
the class of representable quasi-polyadic algebras.  
This provides an algebraizable extension of first order logic {\it without equality}
that admits a finite complete 
and sound Hilbert style axiomatization
of the valid formula schemata, involving only valid formula schemata.
This answers the equality free version of problem 4.16 in \cite{HMT2}. 
Adding equality is problematic so far. 

Another more adventurous approach is to ``stay inside'', so to speak,
the ``$\CA$-$\RCA$ infinite discrepancy'' and to try
to capture the essence of (the equations holding in) 
$\RCA_{\alpha}$ in as  simple a
manner as possible, inspite of Andreka's complexity results, and without resorting to any kind of ``twisting"! 
It is not hard to show that the set of equations 
holding in $\RCA_{\alpha}$ for any countable
$\alpha$ is recursively enumerable. And indeed, using a well-known trick of Craig,
several (recursive) 
axiomatizations of $\RCA_{\alpha}$
exist in the literature, the first 
such axiomatization originating with Monk, building on work of McKenzie \cite{Mc}, 
\cite[p.112]{HMT2} and \cite{M69}.
Robinson's finite forcing in model theory proves 
extremely 
useful here as shown by Hirsch and 
Hodkinson. 
The very powerful recent approach of synthesizing axioms by games 
due to Hirsch and Hodkinson \cite{HH}, 
building on work of Lyndon \cite{Lyndon56}, 
is a typical instance
of giving an intrinsic characterization of the class of representable algebras
by providing an {\it explicit} axiomatization of this class in a step-by-step fashion.
This approach is of a very wide scope; 
using Robinson's finite forcing in the form of games, Hirsch and Hodkinson \cite{HHbook} 
axiomatize, not only the variety of representable algebras, 
but almost all 
pseudo-elementary classes existing in the literature, 
an indeed remarkable achievement.

It turns out, as pointed out by Hirsch and Hodkinson, 
that representations of an algebra can be described in a first 
order 2-sorted language. 
The first sort in a model of this {\it defining theory}
is the algebra itself, while the second sort is a representation of it. 
The defining theory specifies the relation between the two, 
and its axioms depend on what kind of representation is considered.
Thus the representable algebras are those models of the first sort 
of the defining theory, with the second sort providing the representation.
This method has its roots for relation algebras in 
McKenzie's dissertation \cite{Mc}.

The class of all structures that arise as the first sort of a model of a two-sorted 
first order theory is an old venerable notion in model theory introduced 
by Maltsev
in the forties of the 20th century. Ever since it was studied by Makkai and others. 
It is known as a pseudo-elementary class.  
What is meant here is  
a $PC_{\Delta}$ class in the sense of \cite{Ho} 
but expressed in a two sorted language. 
The term pseudo-elementary class strictly 
means $PC_{\Delta}$ 
when the second sort is empty, but the two notions were proved to be 
equivalent by Makkai \cite{Makkai}.

Any elementary class is pseudo-elementary, 
but the converse is not true;  
the class of $\alpha$-dimensional neat reducts
of $\beta$-dimensional cylindric algebras for $1<\alpha<\beta$
is an example; see \cite{BL}, \cite{IGPL}, \cite{FM}, \cite{Qs} and \cite{MQ}. 
Another is the class of strongly representable atom structures 
and the completely representable ones, as proved by Hirsch and 
Hodkinson in \cite{HH2000}.  Extending the latter result of Hirsch and Hodkinson on cylindric algebras, 
it is shown shown that the class of strongly representable atom structures
of many reducts of polyadic algebras, including polyadic algebras, cylindric algebras and diagonal free cylindric algebras
is not first order axiomatizable \cite{complete}.
The constructions used for this purpose employs the probabalistic methods of  of Erd\"os in constructing finite graphs with arbitrarily large 
chromatic number and girth, see theorems \ref{can}, \ref{com}. 

Many classes in algebraic logic can be seen
as pseudo-elementary classes. 
The defining theory is usually finite, simple 
and essentially recursively enumerable. 
According to Hirsch and Hodkinson a 
fairly but not completely general {\it definition} of the notion of representation
is just the second sort of a model of a two-stwo-sorted 
(more often than not recursively enumerable) first order theory, 
where the first sort of the theory
is the algebra. 

Put in this form,  Hirsch and Hodkinson apply model-theoretic finite forcing
to the Representation Problem. 
Model-theoretic forcing, as described in Hodges  
\cite{Ho}, and indeed in 
the proof of the classical Completeness
Theorem by Henkin, and in 
his Neat Embedding Theorem to be recalled below, 
typically involves constructing a model of a first order theory by a {\it game.}

The game builds the model  
step-by-step, elements of the model being produced by the second player
called $\exists$, in his response to criticism by the first 
player, called $\forall.$ 
The approach of Hirsch and Hodkinson is basically to combine the forcing games with 
the pseudo-elementary approach mentioned above to representations.

That is to build the second sort 
of a model of the defining theory whose first sort is  
the algebra the representability of which is at issue.
Taking the defining theory of the pseudo-elementary class to be given, 
this defines the notion of representation to be 
axiomatized. 

We refer the reader to \cite{HH97a} and \cite{HHbook}
for applications of
this technique to axiomatize the 
classes of representable relation and cylindric algebras.
The step-by-step technique in op.cit of building representations, especially when 
viewed as a game,  can be extremely potent and inspiring. 
Not only does it allow
the construction of axiomatizations of relation algebras and 
cylindric algebras and other kinds of related algebras, but close examination of the way
that games can be played on given algebras, 
provides very fine and detailed information about their structure, and makes one delve deep into the analysis. 

\subsection{ Games in action, a case study}

We next apply the game approach to obtain an explicit recursive axiomatization of the class of representable 
polyadic equality algebras of dimension $d$, where $d$ is finite and $d\geq 2$. (No such axiomatization exists in the literature for $d\geq 3$).

We recall the definition of quasi polyadic equality algebras of arbitrary dimension from \cite{ST}.
When the dimension is finite, quasipolyadic equality algebras and polyadic equality algebras are practically the same; we refer to both by polyadic 
equality algebras.

\begin{definition}  By a quasipolyadic equality algebra of dimension $\alpha$, briefly a $\QPEA_{\alpha}$, 
we mean an algebra
$\A=(A,+,-, {\sf c}_i, {\sf s}_j^i, {\sf p}_{ij}, {\sf d}_{ij})_{i,j<\alpha}$ 
where $(A, +, -)$ is a boolean algebra with $+$ denoting the boolean join and $-$ denoting complementation, ${\sf c}_i, {\sf s}_j^i$ and ${\sf p}_{ij}$ 
are unary operations on $A,$
${\sf d}_{ij}$ is a constant (for $i,j<\alpha$) and the following postulates are satisfied for all 
$i, j, k\in \alpha$
\begin{enumerate}
\item ${\sf s}_i^i={\sf p}_{ii}=Id,\text { and }{\sf p}_{ij}={\sf p}_{ji}$
\item $x\leq {\sf c}_ix$ (here $a\leq b$ abbreviates $a+b=b$). 
\item ${\sf c}_i(x+y)={\sf c}_ix+{\sf c}_iy$
\item ${\sf s}_j^i{\sf c}_ix= {\sf c}_ix$
\item ${\sf c}_i{\sf s}_j^ix={\sf s}_j^ix \text { if } i\neq j$
\item ${\sf s}_j^i{\sf c}_kx={\sf c}_k{\sf s}_j^ix \text { if } k\notin \{i,j\}$
\item ${\sf s}_j^i$ and ${\sf p}_{ij}$ are boolean endomorphisms\footnote{Homomorphisms from $(A,+,-)$ to itself.}
\item ${\sf p}_{ij}{\sf p}_{ij}x=x$
\item ${\sf p}_{ij}{\sf p}_{ik}={\sf p}_{jk}{\sf p}_{ij}x   \text { if } |\{i,j,k\}|=3$
\item ${\sf p}_{ij}{\sf s}_i^jx={\sf s}_j^ix$
\item ${\sf s}_j^i{\sf d}_{ij}=1$
\item $x\cdot {\sf d}_{ij}\leq {\sf s}_j^ix$
\end{enumerate}
\end{definition}
In what follows $d$ is finite with $d\geq 2$. In this case we write $\PEA_d$ for $\QPEA_d$ 
and $\RPEA_d$ for $\RQPEA_d$. A full polyadic set algebra of dimension $d$ is an algebra
$$(\wp(^dU), \cap, \cup, \sim, \emptyset ,^{d}U, {\sf C}_i, {\sf D}_{ij}, {\sf P}_{ij})_{i,j<d},$$
where the ${\sf C}_i$'s (cylindrifications) and ${\sf D}_{ij}$'s (the diagonals) are defined like the $\CA$ case and 
$${\sf P}_{ij}X=\{s\circ [i,j]: s\in X\}.$$
$\C$ is representable if it is isomorphic to to subdirect product of set algebras.

Let $\C$ be a polyadic algebra of finite dimension $d\geq 2$. Let $\bar{x}, \bar{y}$ 
be $d$ tuples of elements of some set . We write $x_i$ for the $i$th element of $\bar{x}$, so that $\bar{x}=(x_0,\ldots x_{d-1})$.
For $i,j<d$, we write $\bar{x}\equiv_i\bar{y}$ if $x_k=y_k$ for all 
$k<d$ $k\neq i$. We write $\bar{x}\equiv_{ij}\bar{y}$ if $\bar{x}=\bar{y}\circ [i,j]$.

\begin{definition} Let $\C,d$ be as above. 
\begin{enumarab}
\item A $\C$ prenetwork $N$ is a complete directed $d$ dimensional hypergraph with each hyperedge labelled by an element of $\C$.
Formally $N$ consists of a finite set of nodes, an a map assignining an element of $\C$ to each $d$-tuple of nodes. We use the synbol $N$ to denote the set of nodes,
the mapping, and the graph itself.

\item A $\C$ network or simply a network is a $\C$ prenetwork $N$ satisfying:
\begin{enumroman}
\item for each $i,j<d$, and $d$ tuple $\bar{x}$ from nodes$(N)$ (written, $\bar{x}\in {}^dN)$, if $x_i=x_j$, then $N(\bar{x})\leq {\sf d}_{ij}$.

\item for any $\bar{x}, \bar{y}\in {}^dN$, and any $i<d$, if $\bar{x}\equiv_i\bar{y}$, then $N(\bar{x})\cdot {\sf c}_iN(\bar{y})\neq 0$

\item for any $\bar{x}, \bar{y}\in {}^dN$, and any $i,j<d$, if $\bar{x}\equiv_{ij}\bar{y}$, then $N(\bar{x})\cdot {\sf p}_{ij}N(\bar{y})\neq 0$

\item if $\bar{x}, \bar{y}, \bar{z}\in {}^dN$, $i,j,k<d$ $\bar{x}\equiv_i\bar{y}$, $x_i=z_j$, $y_i=z_k$ and $N(\bar{z})\leq {\sf d}_{jk}$, then
$N(\bar{x})\cdot N(\bar{y})\neq 0$. 
\end{enumroman}

\item We write $N\subseteq N'$ if the nodes of $N'$ include those of $N$, and $N'(\bar{x})\leq N(\bar{x})$ for all $\bar{x}\in {}^dN$.

\end{enumarab}
\end{definition}
\begin{definition} Let $N$ be a prenetwork over the fixed polyadic algebra $\C$, and let $n\leq \omega$. The game $G_n(N,\C)$ is of length
$n$, and $N_0=N$. A play of $G_n(N,\C)$ is a sequence of $\C$ prenetworks $N_0\subseteq \ldots N_n$ if $n$ is finite, 
and $N_0\subseteq N_1\subseteq \ldots $ if $n=\omega$.

In the $t$th round $t<n$, let the last pre-network played be $N_t$. For his move in this round, $\forall$ has two kinds of moves:
\begin{enumarab}
\item Cylindrifier move.
$\forall$ picks $i\leq d$, $\bar{x}\in {}^dN_t$ and $a\in \C$.
\begin{enumroman}
\item  $i<d$.
In this case, $\exists$ must respond to $\forall$'s move with a prenetwork $N_{t+1}\supseteq N_t$ given by:

{\bf reject } $N_{t+1}$ is the same as $N_t$ except that $N_{t+1}(\bar{x})=N_t(\bar{x})\cdot -{\sf c}_ia$

{\bf accept} The nodes of $N_{t+1}$ are those of $N_t$ plus a new node $z$. Let $\bar{z}$ be given by $\bar{z}\equiv_i\bar{x}$, $z_i=z$. The labels of 
$d$ tuples of nodes are given by

$N_{t+1}(\bar{z})=a\cdot \prod_{j,k<d, z_j=z_k}{\sf d}_{jk}$

$N_{t+1}(\bar{x})=N_t(\bar{x})\cdot {\sf c}_ia$

$N_{t+1}(\bar{y})=N_t(\bar{y})$ for all $y\in {}^dN_t\sim \{\bar{x}\}$

$N_{t+1}(\bar{y})=\prod_{j,k<d, y_j=y_k}{\sf d}_{jk}$ for all $y\in {}^dN_{t+1}\sim (\{\bar{z}\}\cup {}^dN_t)$

\item $i=d$, $\exists$ must respond to $\forall$'s move with a prenetwork $N_{t+1}\supseteq N_t$ given by:

{\bf accept} $N_{t+1}(\bar{x})=N_t(\bar{x})\cdot a$ or {\bf reject } $N_{t+1}(\bar{x})=N_t(\bar{x})\cdot -a.$
\end{enumroman}

\item polyadic move

In this case, $\forall$ picks $i, j< d$, $\bar{x}\in {}^dN_t$ and $a\in \C$.

$\exists$ must respond to $\forall$'s move with a prenetwork $N_{t+1}\supseteq N_t$ given by:

{\bf reject } $N_{t+1}$ is the same as $N_t$ except that $N_{t+1}(\bar{x})=N_t(\bar{x})\cdot -{\sf p}_{ij}a$

{\bf accept}The nodes of $N_{t+1}$ are the same as $N_t$.The labels of 
$d$ tuples of nodes are given by

$N_{t+1}(\bar{x})=N_t(\bar{x}\circ [i,j])\cdot {\sf p}_{ij}a$

$N_{t+1}(\bar{y})=N_t(\bar{y})$ for all $y\in {}^dN_t\sim \{\bar{x}\}$

If each $N_t$ $t<n)$ is a $\C$ network, then $\exists$ has won. otherwise $\forall$ won.

\end{enumarab}

\end{definition}

A strategy for a player in a game of the form $G_n(N,\C)$ 
is a set of rules telling the player what move to make in each situation. A strategy for $\forall$ 
will tell him which tuple, which indices, and which algebra element to pick, and one for $\exists$ will tell her whether 
to accept or reject. A strategy is said to be used by a player in a play of the game if that player uses it in every round, 
so his or her moves always accord with what the
strategy suggests. A strategy in the game $G_n(N,\C)$ 
is said to be winning for its owner if the owner wins all matches in which the strategy is used regardless of how the 
opposite player decides to move.
If $a\in \C\sim \{0\}$, let $I_a$ be the $\C$ network with set of nodes $d$, 
given by $I_a(0,1,\ldots d-1)=a$ and $I_a(\bar{x})=\prod_{x_i=x_j}{\sf d}_{ij}$ for each other tuple
$\bar{x}$.
Now given a polyadic equality algebra, it is not always possible to construct a reprsentation jusy be fixing defects one by one, and $\exists$ does not always have a winnin strategy in $G_{\omega}(N,\A)$. If we were confined to thinking in terms of step by step constructions of representations, we might easily say at this point that they just don't work
in general, and give up. With games, however, it is natural to shift the problem from showing that $\exists$ has a winning strategy, to asking {\it when} 
she has a winning startegy. It will turn out that she has such a strategy precisely when the given algebra is representable, 
and we can use this to axiomatize 
$\RPEA_d$, as we illustrate in what follows. 

\begin{theorem} Let $d\geq 2$ be finite. Let $\C$ be a $d$ dimensional polyadic  algebra. 
$\C$ is representable if and only if $\exists$ has a 
winning strategy in the game 
$G_n(I_a,\C)$ for each $n<\omega$ and each 
non zero $a\in \C$.
\end{theorem}
\begin{demo}{Proof} One side is easy. If $\C$ is representable, then $\exists$ can use a representation to 
give her a winning strategy in each of the games. Conversely suppose she can win
each game $G_t(I_a,\C)$ for every $t<\omega$ and every non zero element $a$. Fix 
such an element $a$. We can suppose that $\C$ is countable, for $\RPEA_d$ is a variety. Consider a play of $G_{\omega}(I_a,\C)$ in which $\exists$ 
uses her winning strategy and $\forall$ picks every $d$ tuple of nodes ever constructed, every $i\leq d$, every $i,j<d$, 
and every $a\in \C$ eventually during the game. Let $M$ denote the
set of all nodes introduced during this  game.
Define
$\D=(\wp(^dM), \cup, \sim {\sf C}_i, {\sf D}_{ij}, {\sf P}_{ij})_{i,j<d}$ and a map
$h^*: \C\to \D$ bas follows.
For $r\in C$, let
$$h^*(r)=\{\bar{x}\in {}^dM; \exists t<\omega(\bar{x}\in {}^dN_t\land N_t(\bar{x})\leq r\}.$$
$\forall$ moves of the second kind, when $i=d$, guarantee that for any $d$ -tuple $\bar{x}$ and any $a\in \C$
for sufficiently large $t$ we have either $N_t(\bar{x})\leq a$ or $N_t(\bar{x})\leq -a$. This ensures that $h^*$ 
preserves the Boolean operations. $\forall$ moves of the first kind, 
when $i<d$ ensure that the cylindrifications are respected by $h^*$, while $\forall$ moves in response to the polyadic moves guarantee
that substitutions are preserved.
So $h^*$ is a homomorphism fro the diagonal free reduct of $\C$ into the diagonal free reduct of $\D$.
Now define
$$x\sim y\Leftrightarrow \exists \bar{z}\in {}^dM, \exists k,l<d(z_k=x\land z_l=y\land \bar{z}\in h^*(d_{kl})).$$
Then $\sim$ is an equivalence relation; further $\sim$ is an $h^*$ congruence on $M$, in the sense that if $\bar{x}, \bar{y}\in {}^dM$ and $x_i\sim y_i$ for all $i<d$, 
then $\bar{x}\in h^*(r)$ iff $\bar{y}\in h^*(r)$.
Now let $B=M/\sim$ and define $$h_a(r)=\{(x_0/\sim, \ldots x_{d-1}/\sim): (x_0\ldots x_{d-1})\in h^*(r)\}.$$   
Then it can be checked that $h_a$ is the required representation.
\end{demo}

Now we want to synthesis the above games to obtain a recursive axiomatization of $\RPEA_d$. 
Write $L^d_c$ for the signature $\{0,1,+,-, {\sf c}_i, {\sf d}_{ij}, {\sf p}_{ij}: i,j<d\}$ of $d$ dimensional polyadic algebras.
A $d$ dimensional term network is a pair consisting of a finite non empty set $N$, or nodes$(N)$
of nodes, and a map, also written $N$, assigning an $L_c^d$ term $N(\bar{x})$ to each $d$ tuple $\bar{x}$ of nodes. 

\begin{enumarab}

\item Given a term network $N$, an index $i\leq d$, a $d$ tuple $\bar{x}\in {}^dN$ and an $L_c^d$ term $\tau$, we define two term netwrorks
$Out=Out(N,i,\bar{x},\tau)$ and $In=In(N,i,\bar{x},\tau)$ corrsponding to the two ways $\exists$ can respond to the cylindrifier move 
in the game. 

\begin{enumroman}
\item $i<d$
We define
$In$ and $Out$ as follows

\begin{itemize}

\item $nodes(Out)=nodes(N)$, and $nodes(In)=nodes(N)\cup \{z\}$ for some new node $z$

\item For all $\bar{y}\in {}^dN\sim \{\bar{x}\}$, $In(\bar{y})=Out(\bar{y})=N(\bar{y})$

\item $Out(\bar{x})=N(\bar{x})\cdot-{\sf c}_i\tau$

\item $In(\bar{x})=N(\bar{x})\cdot {\sf c}_i\tau$

\item Let $\bar{z}\equiv_i \bar{x}$ with $z_i=z$. Then $In(\bar{z})=\tau\cdot \prod_{j,k<d, x_k=x_j}{\sf d}_{jk}$

\item For all other $d$ tuples $\bar{y}$ involving $\bar{z}$, $In(\bar{y})=\prod_{j,k<d, y_j=y_k}{\sf d}_{jk}$
\end{itemize}

\item $i=d$. In this case, we define 

\begin{itemize}

\item $nodes(Out)=nodes(N)=nodes(In)$
\item $In(\bar{x})=N(\bar{x})\cdot \tau$
\item $Out(\bar{x})=N(\bar{x})\cdot -\tau$

\item For all $\bar{y}\in {}^dN\sim \{\bar{x}\}$, $In(\bar{y})=Out(\bar{y})=N(\bar{y})$
\end{itemize}
\end{enumroman}
\item Given a term network $N$, an index $i, j< d$, a $d$ tuple $\bar{x}\in {}^dN$ and an $L_c^d$ term $\tau$, we define two term networks
$Out=Out(N,i,j,\bar{x},\tau)$ and $In=In(N,i,j,\bar{x},\tau)$ corresponding to the two ways $\exists$ can respond to the polyadic moves 
in the game. We define
$In$ and $Out$ as follows

\begin{itemize}

\item $nodes(Out)=nodes(N)=nodes(In)$

\item For all $\bar{y}\in {}^dN\sim \{\bar{x}\}$, $In(\bar{y})=Out(\bar{y})=N(\bar{y})$

\item $Out(\bar{x})=N(\bar{x})\cdot-{\sf p}_{ij}\tau$

\item $In(\bar{x})=N(\bar{x}\circ [i,j])\cdot {\sf p}_{ij}\tau.$

\end{itemize}
\end{enumarab}

\begin{definition}
Let $N$ be a $d$ dimensional term network . 
\begin{enumarab}
\item We define the formula $CNet^d(N)$ to be the conjunction of the following four formulas:
$$\bigwedge_{\bar{x}\in {}^dN, i,j<d, x_i=x_j}N(\bar{x})\leq {\sf d}_{ij}$$
$$\bigwedge_{\bar{x}, \bar{y}\in {}^dN, i<d, \bar{x}\equiv_i\bar{y}}N(\bar{x})\cdot {\sf c}_iN(\bar{y})\neq 0$$
$$\bigwedge_{\bar{x}, \bar{y}\in {}^dN, i,j<d, \bar{x}\equiv_{i,j}\bar{y}}N(\bar{x})\cdot {\sf p}_{ij}N(\bar{y})\neq 0$$
$$\bigwedge_{\bar{x},\bar{y},\bar{z}\in {}^dN, i,j,k<n, \bar{x}\equiv_i\bar{y}, x_i=z_j, y_i=z_k}(N(\bar{z})\leq {\sf d}_{jk})\to (N(\bar{x})\cdot N(\bar{y})\neq 0).$$
\item Now we define inductively:
$$\psi_0^d(N)=CNet^d(N)$$
$$\psi_{n+1}^d(N)=\forall y[\bigwedge_{\bar{x}\in {}^dN, i\leq d}\psi_n^d(In(N,i,\bar{x},y))\lor \psi_n^d(out(N,i,\bar{x},y))$$
$$\bigwedge_{\bar{x}\in {}^dN, i,j<d}\psi_n^d(In(N,i,j,\bar{x},y))\lor \psi_n^d(out(N,i,j,\bar{x},y)).]$$ 
\end{enumarab}
\end{definition}

It can be shown by induction on $n<\omega$, that for all assignments $\tau$ of he variables in the terms of $N$ into $\C$, 
we have $\C,\tau\models \psi_n^d(N)$ if and only if $\exists$ has a winning strategy
for $G_n(N^{\tau}, \C)$.
Let $x$ be any variable and define $J_x$ to be the graph with nodes $d$, with $J_x(0,\ldots d-)=x$ and $J_x(\bar{y})=\prod_{y_i=y_j}{\sf d}_{ij}$ for all $\bar{y}\in 
{}^dd\sim \{(0,1,\ldots, d-1)\}$. Then $\C$ is representable, if and only if $\C,\tau\models \psi_n^d(J_x)$ for all $n<\omega$ 
and all assignmentes $\tau: \{x\}\to \C\sim \{0\}$.
Let $\phi_n^d$ be the sentence $\forall x(x\neq 0\to \psi_n^d(J_x))$.

\begin{theorem} For any $3\leq d<\omega$, $\C$ is representable if and only if $\C\models \phi_n^d$ for all $n<\omega$.
\end{theorem}
Since $\PEA_d$ is a variety with discriminator 
term ${\sf c}_0{\sf c}_1\ldots {\sf c}_{d-1}x$, 
then we can convert these universal sentences into equations $\epsilon_n^d$ that axiomatize $\RPEA_D$ withn $\PEA_d$

\begin{theorem}\label{polyadic} For finite $d\geq 3$, $\RPEA_d$ is axiomatized by the equations $\{\epsilon_n^d: n<\omega\}$ together
wth the equations for $\PEA_d.$
\end{theorem}
Call the set of equations in the above theorem $\Sigma$. Let $l<d,$ $k<d$, $k'<\omega$ be natural numbers, 
then by a careful inspection, one can see that $\Sigma$ contains
infinitely equations in which $-$ occurs, one of $+$ or $\cdot$ ocurs  a diagonal or a permutation with index $l$ occurs, more
than $k$ cylindrifications and more than $k'$ variables occur.
We will show that this (degree of complexity) is true for {\it any} axiomatization of $\RPEA_d$ when $d\geq 3$.
{\it There is a prevalent misconception that {\it cylindric algebras} of dimension $d$  are suitable for dealing with $\L_d$ (first order logic restricte to the first $d$
variables)
with its semantical and syntactical notions, whereas in fact, it is the class of polyadic algebras of dimension $d$, that constitute 
the ``real" algebraic counterpart 
of $\L_d$.} However, the theory of (representable) cylindric algebras $(\RCA_d)$ is far more developed than that of polyadic equality algebras.

Since, by our above slogan,  notions of $\L_d$ are reflected by $\RPEA_d$ while syntactical notions, 
like provability in $\L_d$ is reflected by the class 
$\PEA_d$, it is natural and indeed useful to see what results on $\RCA_d$ generalise to $\RPEA_d$ and which do not.
This will be a recurrent theme in what follows.

In \cite{AU}, \cite{IGPL2}, \cite{OT}, \cite{SL}, \cite{IGPL6}, \cite{OTT} and 
\cite{SSL} we use a disguised form of Robinson's 
finite forcing in model theory
to prove interpolation and omitting types for certain infinitary
expansions of first order logic. This involves building models in an essentially 
step-by-step manner, 
although we do not resort to games.
In treating omitting types for extensions of first order logic, a Baire category approach is adopted, see theorems, \ref{covK}, \ref{ZF}, 
and this is utterly unsurprising for one can go from the Baire category approach to games via the Banach Mazur Theorem
\cite{Bull5}, \cite{IGPL2}, \cite{OT}, \cite{Bull5b}, \cite{Bull3}, \cite{TT}, \cite{OTT}.
But in any case, the step-by-step approach, 
whether by games or otherwise, is widely accepted and 
has been used by many authors.

The connection of games to Robinson's finite forcing is well known. 
Indeed the games we played above are called Forcing games. 
Forcing games are also known to descriptive set theorists as Banach-Mazur games.  Model theorists use them as a way of building 
infinite structures with controlled properties. To sketch the idea, we quote Hodges:`` imagine that a countably infinite team of builders are constructing a house $A$. 
Each builder has his or her own task to carry out and  has infinitely many chances to enter the site and add some finite amount of material to the house; 
these slots for the builders are interleaved so that the whole process takes place in a sequence of steps enumerated by the natural numbers. 
To show that the house can be built to order, we need to show that each builder separately can carry out his or her appointed task, 
regardless of what the other builders do. So we imagine each builder as player $\exists$ in a game where all the other 
players are lumped together as another player $\forall$, 
and we aim to prove that $\exists$ has a winning strategy for this game. When we have proved this for each builder separately, 
we can imagine them going to work, each with their own winning strategy. They all win their respective games and the result is one beautiful house.
More technically, the elements of the structure $A$ are fixed in advance, 
but the properties of these elements have to be settled by the play. 
Each player moves by throwing in a set of atomic or negated atomic statements about the elements, 
subject only to the condition that the set consisting of all the statements thrown in so far must be 
consistent with a fixed set of axioms written down before the game. 
At the end of the joint play, the set of atomic sentences thrown in has a canonical model, and this is the structure $A$; 
there are ways of ensuring that it is a 
model of the fixed set of axioms. A possible property $P$ of $A$ is said to be enforceable if a builder who is given the task of making 
$P$ true of $A$ has a winning strategy. 
A central point (due essentially to Ehrenfeucht) is that the conjunction of 
a countably infinite set of enforceable properties is again enforceable."
We shall use such ideas in connection to some stability theory in theorem \ref{stability}.

The name forcing comes from an application of related ideas by Paul Cohen to construct models of set theory in the early 1960s. 
In the mathematical discipline of set theory  forcing is the  technique invented by Paul Cohen for proving  consistency and independence results.
It was first used, in 1962, to prove the independence of the continuum hypothesis and the axiom of choice from Zermelo-Fraenkel set theory.
Forcing was considerably reworked and simplified in the sixties, and has proven to be an extremely 
powerful technique both within set theory and in other areas of  mathematical logic such as descriptive set theory and recursion theory.
We shall use the technique of iterated forcing in showing the independence of a purely algebraic statement involving neat embeddings
in theorem \ref{OT}.

Abraham Robinson adapted the methods of forcing  to make a general method for building countable structures, 
and Martin Ziegler introduced the game setting. The games played above, have infinite lengths, but are determined.
In fact, all the games studied in this paper has the following topological form. Take the Cantor set $^{\omega}2$.
Now $^{\omega}2$, regarded as the product space of the set $2$ with the discrete topology, is a Polish space, that is, 
a topological space that is metrizable with a metric that is both separable
and complete. A condition is a map $p:Y\to \{0,1\}$ where $Y$ is a finite set of $\omega$. 
Let $M_p=\{f\in {}^{\omega}2: p\subseteq s\}$. Then this is a base for the topology on $^{\omega}2$. 
Given a non empty closed set $F\subseteq ^{\omega}2$ and a set $P\subseteq F$, 
players $\forall$ $\exists$ play the following game of length $\omega$.The players choose between them an 
increasing chain $p_0\subseteq p_1\subseteq\ldots$  of 
conditions so that $F\cap M(p_i)\neq \emptyset$. 
Then, it can be shown that  $\exists$ has a winning  strategy iff $P$ is comeager. Player $\forall$ has a winning strategy 
if there exists a condition $p$ such that
$M(p)\sim P$ is comegaer in $M(p)$. 
if $P$ is a Borel set, then the game is determined, that is one of the players have a winning strategy. Using the axiom of choice one can show that 
there is a (non-Borel) set $P$ such that $G(F,P)$ is not determined.
However if one adopts the axiom of determinacy, asserting that such games are determined one is led to an 
extension of $ZF$ contradicting choice.
Recent work of Woodin has revealed that such a theory is equiconsistent with $ZFC$ + the existence of infinitely many Woodin cardinals. 
Not to digress any further, in the two player games we play, one of the 
players has a winning
strategy, i.e. the games we play are determined.
When $\exists$ has a winning strategy in the $\omega$ round game over a given algebra, then this implies
that the algebra in question has a representation, and this game can be truncated to finite ones, these in turn can be translated 
effectively into an axiomatization
of the class of representable algebras.

Another approach initiated by Van Benthem 
and Venema,  consists of 
viewing cylindric set algebras
as subalgebras of complex algebras of Kripke frames 
that have the same signature as atom structures
of cylindric algebras \cite{Venema}, 
thus opening an avenue to techniques and methods 
coming from modal logic.
This typically involves introducing 
Gabbay-style rules on the logic side. 
These extremely liberal Gabbay-style inference systems
correspond to classes that are inductive, i.e., axiomatized by 
$\forall\exists$-formulas.
An example of such a class is the class of rectangularly dense
cylindric algebras, \cite{AGMNS}, \cite{V98}. 
We should mention that this approach is an 
instance, or rather, an application 
of the triple duality, in the sense of Goldblatt \cite{Goldblatt2000}, 
existing between abstract modal logic, Kripke frames 
or relational structures, and boolean algebras with operators.
In this connection, we refer to the article by Venema in  \cite{MLM}
for an explanation and application of 
this duality to relation algebras and to \cite{Venema},
\cite{Marx95}, and \cite{amal} for further elaboration on this duality.
More still can be acheived if we allow relativization. Every weakly associative algebra has a relativized representation; 
this is a classical result of Maddux \cite{Maddux} and is the $\RA$ analogue of the Resek Thompson result for relativized $\CA$'s.
Weakly associative algebras correspond to the so called {\it arrow logic}. 
Kurusz \cite{K} considered arrow logic augmented with various kinds of infinite counting modalities such as ``much more" and 
``many times". 
Adding these modal operators to weakly associative arrow logic result in finitely axiomatizable and decidable logics. 
Arrow logic with projections is extremely expressive. Indeed Tarski and Givant \cite{TG} 
show that the whole of set theory can be built in such a framework, which can be seen as `finitizing' set theory.
Nemeti used such results to show that finite variable fragments have Godel's incompleteness theorem \cite{Bulletin}.
Adding projections \cite{Kb} can even spoil the robust decidability of weakly associative arrow logic: some of these are not even recursively enumerable
using a reduction of unsolvable Diophantine equations. This negative property was shown to be an artifact of the underlying set theory - certain 
non well founded theories
interpret the meaning of projections as to allow for finite axiomatizability even of full arrow logic with projections.
The so called directed cylindric algebras of N\'emeti form a  related topic. 
Sagi \cite{Sagi2000} proves the representation of these in the absence of the axiom of foundation, showing that positive solutions to the Finitizability problem can be obtained
when the underlying set theory is weakened.

There is also a purely categorial approach to the representation
problem. Instead of viewing the representation 
problem as a two-sorted first order theory,
one can present it  in the context of a {\it functor} going from from one sort to 
the other, or rather from one category to the other.  
We have two categories, 
the category of abstract algebras 
and the category of concrete algebras and adjoint functors
between the two categories.
Indeed,  the representation problem can be seen as a typical duality, be it between 
models and theories, or boolean algebras with operators and 
modal logic \cite{Goldblatt89}, 
or quasi-varieties and algebraizable logics \cite{BP89}. 
Indeed it is argued in \cite{Mad} that duality theories, 
representation theorems and adjoint functors are different words for the same 
thing.  At the end of the article, we give another more sophisticated categorial formulation of
the Finitizability problem, where we look at {\it inverses} of the {\it Neat reduct} functor going from one category to another in $\omega$ extra 
dimensions, and try to reflect those in an adjoint situation. A solution to the finitizability problem is thereby presented as an equivalence of two 
categories.

However, it is debatable 
whether any of the intrinsic existing characterizations of the class of the representable algebras
are ``good enough." 
In particular, axiomatizations of this class, existing in the literature, 
see, e.g., \cite{HMT2}, \cite{An97}, \cite{Venema}, \cite{HHbook}, 
seem not to be considered satisfactory from 
the algebraic point of view. Maybe 
a better
description of the present situation 
would be that none 
of these axiomatizations are considered final.  

This also applies to the 
modal  approach initiated by Venema \cite{Venemac} \cite{Venemap}, \cite{V98}.
A similar  situation occurs in modal logic, 
when people say that Gabbay's irreflexitivity rule
\cite{Ga81} introduces
variables by the back door and is {\it inimical} 
to the true nature of modal logic.

Quoting Hirsch and Hodkinson, cf. p.9 in \cite{HHbook}
``The precise objections are hard to pin down, but broadly, it seems that these
axiomatizations are regarded as unsatisfactory in some way: they are too
complicated---or perhaps too trivial, just paraphrasing the original problem without
providing any significant new `algebraic insight'.'' 

There are also some reservations to the approach 
initiated by Hirsch and Hodkinson adopted in \cite{HH97a}
\cite{HHII}, and \cite{HHbook}. 
It is sometimes argued that the use of games is just
a presentational matter. 
In any case the problem to find simpler 
axiomatizations of representable algebras is still open.

\section{Neat embeddings}

A form of the Representation 
Problem is to describe properties of the class
$\RCA_{\alpha}$ and try to give a 
useful characterization of it in {\it abstract} terms.
An old result of Henkin which gives an abstract sufficient and necessary condition
for representability fits here.
An algebra 
$\cal A$ is in $\RCA_{\alpha}$ if and only if for every
$\beta>\alpha$, 
it can be embedded as a {\it neat subreduct} in some cylindric algebra
of dimension 
$\beta$, or, equivalently, using ultraproducts, into 
an algebra
with $\omega$ extra dimensions.

This brings us to the central venerable notion in 
the representation theory of cylindric algebras, 
namely the notion of 
{\it neat reducts.} The notion of neat reducts, which we now recall, 
is also due to Henkin \cite[p.401]{HMT1}.
An old venerable notion in algebraic logic, the notion of neat reducts is now gaining some
momentum, \cite{ma}, \cite{San}, \cite{BL}, \cite{SL}, \cite{IGPL}, \cite{th}, \cite{FM}, \cite{Qs}, 
\cite{Tarski}, \cite{IGPL4}, \cite{note}, \cite{IGPL5}, \cite{IGPL6}, 
\cite{Bull6}, \cite{Bull6b}, and \cite{IGPL7}.
A {\it neat reduct } of a cylindric algebra $\cal A$
is basically a new algebra 
of lesser dimension obtained from $\cal A$ by overlooking some of its 
operations and discarding some of its elements. 
More precisely:

\begin{definition} 
Let $\alpha<\beta$ be ordinals, and let ${\cal A}\in \CA_{\beta}$. Then the 
{\it neat-$\alpha$ reduct} of $\cal A$, in symbols $\Nr_{\alpha}\cal A$, 
is the $\CA_{\alpha}$ whose domain is the set of 
all {\it $\alpha$-dimensional} elements of 
$\cal A$
defined by  
$$Nr_{\alpha}A=\{a\in A: {\sf c}_i^{\cal A} a=a\text { for all } i\in \beta\smallsetminus \alpha\}.$$
The operations of $\Nr_{\alpha}{\cal A}$ are those of 
the $\alpha$-dimensional reduct 
$$\Rd_{\alpha}{\cal A}=
\langle A,-^{\cal A},\ .^{\cal A}\ , {\sf c}_i^{\cal A}, {\sf d}_{ij}^{\cal A}\rangle _{i,j\in \alpha}$$ 
of $\cal A$, restricted to $Nr_{\alpha}A$. 
\end{definition}
When no confusion is likely, we omit the superscript $\cal A$.
$Nr_{\alpha}A$ as easily checked, 
is closed under the indicated
operations and indeed is a $\CA_{\alpha}.$ 
$\Nr_{\alpha}\cal A$ is thus
a special subreduct of $\cal A$, i.e., a special 
subalgebra of a reduct in the universal algebraic 
sense.

For a class $K$, let $SK$ denote the class of all algebras 
embedable into members of $K$.  $\Dc_{\alpha}$ denotes the class of dimension complemented cylindric algebras of dimension $\alpha$.
$\A\in \Dc_{\alpha}$ if $\Delta x\neq \alpha$ for all $x\in \A$. $\Dc_{\alpha}$ us a non-trivial generalization of $\Lf_{\alpha}$, when $\alpha$ is infinite, 
and most results for 
$\Lf_{\alpha}$ generalize to $\Dc_{\alpha}$, see e.g \cite{HMT1} Theorems 2.6.67-71-72, and \cite{HMT2} Thm 4.3.28.
These theorems does not depend on the fact that the dimension set $\Delta x$ is finite in locally finite cylindric algebra, 
rather they depend on the fact that its complement, $\alpha\sim \Delta x$, is 
infinite, a property that holds for $\Dc_{\alpha}$'s. 
If $\A\in \Dc_{\alpha}$ and $\tau\in {}^{\alpha}\alpha$ is a finite transformation then
${\sf s}_{\tau}^{\A}$ denotes the substitution operation as defined in \cite{HMT2}. 
These make $\Dc_{\alpha}$'s actually reducts of quasipolyadic equality 
algebras. The ${\sf s}_{\tau}$'s are defined by composition of finitely many of the 
the operations ${\sf s}_i^j$'s, where ${\sf s}_i^jx={\sf c}_j(x\cdot {\sf d}_{ij})$. Neat reducts 
play a key role in the representation theory of cylindric algebras
as indeed illustrated in the next celebrated theorem of Henkin, 
which is basically a completeness result. This overwhelming result casts its shadow over the entire field.

\begin{theorem}\label{neat}
For any ordinal $\alpha$ and any ${\A}\in \CA_{\alpha}$, 
the following two conditions are equivalent:

\begin{enumroman}

\item ${\A}\in \RCA_{\alpha}$

\item ${\A}\in S\Nr_{\alpha}\CA_{\alpha+\omega}.$

\end{enumroman}

\end{theorem}
\begin{demo}{Proof} One side is trivial. The difficult implication is $(ii)\to (i)$. So let $\A\subseteq \Nr_{\alpha}\B$ and $\B\in \CA_{\alpha+\omega}$.
Let $a\in \A$ be non-zero. Then we can fnd a an ordinal $\beta\geq \alpha+\omega$ and $\B'\in \CA_{\beta}$ such that $\A\subseteq \Nr_{\alpha}\B'$ and $|B'|=|\beta|$.
We can further assume that $\B'\in \Dc_{\beta}$. Abusing notation we write $\B$ for $\B'$.
Then one finds a Boolean  ultrafilter $F$ of $\B$, such that $a\in F$, and $F$ eliminates quantifiers
in the sense that ${\sf c}_kx\in F$, then ${\sf s}_l^kx\in F$ for some $l\notin \Delta x $. Such an ultrafilter, which we will encounter gain, will be called
a Henkin ultrafilter. Then one 
takes $V={}^\alpha\beta^{(\omega)}=\{s\in {}^{\alpha}\beta: |\{i\in \alpha: s_i\neq i\}|<\omega\}$ and finally one defines
$f:\A\to \wp(V)$ by $f(x)=\{\tau\in V: {\sf s}_{\bar{\tau}}^{\B'}x\in F\},$ where $\bar{\tau}=\tau\cup Id\upharpoonright \beta\sim \alpha$.
Then $f$ is a homomorphism into a weak set algebra with $f(a)\neq 0$. The representation corresponding to $F$ in this manner will be called the canonical model 
of $F$.
\end{demo}

Infinity manifests itself in (ii) above, and it does so essentially in 
the case when $\alpha>2$, 
in the sense  that if $\A$ neatly embeds into an algebra in 
finitely many extra dimensions, 
then it might not be representable, as shown by Monk. 
All $\omega$ extra dimensions are needed for representability.
One
cannot truncate 
$\omega$ to any finite ordinal. The $\omega$ extra dimensions 
play the role of added constants or witnesses in Henkin's 
classical proof of the completeness theorem. 
Therefore it is no coincidence that
variations on theorem \ref{neat} lead to metalogical results concerning 
interpolation and omitting types for the corresponding 
logic. Such 
results can be proved by Henkin's 
method of constructing models out of constants. 
In this connection we refer to
\cite{SSL}, \cite{AU}, \cite{OT}, \cite{Am}, \cite{OTT}, \cite{IGPL5},\cite{IGPL6}, \cite{Bull6b}, \cite{complete}, \cite{Bull5}, \cite{amal}, \cite{neet1},\cite{neet2}
and \cite{Bulletin}.
Let us set out from the known property of cylindric algebras that neat
embedding property implies representability, i.e.%
\begin{equation}
\A\in S\Nr_{\alpha}\CA_{\alpha+\varepsilon}\text{\ implies }%
\A\in\mathbf{IGs}_{\alpha} \label{alap}%
\end{equation}
\newline where $\alpha,\varepsilon$ are infinite ordinals.
Ferenczi analysed the property (\ref{alap}) for various classes of algebras.
He introduced some new classes of cylindric like algebras (classes
\textbf{K}$_{\alpha+\varepsilon}^{\alpha}$\textbf{, M}$_{\alpha+\varepsilon
}^{\alpha}$\textbf{, F}$_{\alpha+\varepsilon}^{\alpha}$\textbf{ }e.g.) and
formulated a group of theorems connected with the property (\ref{alap}).
In \cite{00Fer} the following problem is investigated: Is it possible to
replace in (\ref{alap}) the class $\CA$ by a larger class so that the
implicitation in (\ref{alap}) is still true. In other words, can we loosen up the axioms when we get to 
$\omega$ extra dimensions? The answer is affirmative.
Let $C_1-C_7$ denote the cylindric axioms as defined in \cite{HMT1} 1.1.1. 
In \cite{00Fer} the following class \textbf{K}$_{\alpha+\varepsilon}^{\alpha}%
$\textbf{ }is introduced:
Suppose that \textbf{K}$_{\alpha+\varepsilon}^{\alpha}$ is the class for which
\textbf{K}$_{\alpha+\varepsilon}^{\alpha}\vDash\left\{  C_{1},C_{2}%
,C_{3},C_{5},C_{7},C_{4-},C_{6-}\right\}  $\newline where $\beta$ denotes
$\alpha+\varepsilon$ and C$_{4-}$ denotes the pair of the following weakenings
a) and b) of axiom $C_{4}$ $:$

C$_{4-}\;$a) ${\sf c}_{m}{\sf s}_{n}^{j}x={\sf s}_{n}^{j}{\sf c}_{m}x$

C$_{4-}\;$b)$\;{\sf c}_{m}{\sf s}_{n}^{j}{\sf c}_{m}x={\sf s}_{n}^{j}{\sf c}_{m}x$

where $j\in\alpha\;$and $m,n\in\beta,\;m\neq n,$ and $C_{6-}$ denotes the
following weakenings a), b) and c) of axiom $C_{6}:$

C$_{6-}\;$a) ${\sf d}_{mn}={\sf d}_{nm}$

C$_{6-}\;$b) ${\sf d}_{mn}\;{\sf d}_{nj}={\sf d}_{mj}$

C$_{6-}\;$c) ${\sf c}_{m}{\sf d}_{jn}={\sf d}_{jn}\;\;m\notin\left\{  j,n\right\}  .$

\bigskip

\bigskip Notice that C$_{4}-$ and C$_{6}-$ are restrictions of the $\beta
-$dimensional axioms C$_{4}$ and C$_{6}.$

\begin{theorem}\label{T1}Let $\A\in \CA_{\alpha}$. Then $\A\in\mathbf{S\Nr K}_{\alpha+\varepsilon}^{\alpha}$\ if and
only if $\A\in\mathbf{IGs}_{\alpha}$ where $\alpha$ and
$\varepsilon$ are infinite fixed ordinals.
\end{theorem}
\bigskip

It can be proven that in a sense the class \textbf{K}$_{\alpha+\varepsilon
}^{\alpha}$\ is the optimal extension of the class $\mathbf{CA}_{\alpha}$ such
that (\ref{alap}) is still true. Considering the classical logical
aspects of theorem \ref{T1} we can draw the conclusion that in G\"{o}del's
completeness proof we use only a fragement of the complete calculus, i.e. it
is possible to restrict the equality axioms, and to weaken the commutativity
of the quantifiers $\exists x_{i}$ so that G\"{o}del's theorem remains true.
In \cite{08Fer}, the generalization of (\ref{alap}), from
the class \textbf{Gws}$_{\alpha}$ to the class \textbf{Crs}$_{\alpha}%
\cap\mathbf{CA}_{\alpha}$ (i.e., to the class included in the Resek - Thompson
representation Theorem) is investigated. The problem is: Is there a class of cylindric like
algebras instead of \textbf{CA}$_{\alpha}$ such that \textbf{ }(\ref{alap})
is still true when the class \textbf{Gs}$_{\alpha}$ is replaced by the
class \textbf{Crs}$_{\alpha}\cap\mathbf{CA}_{\alpha}$? The answer is again in the affirmative.
A new class \textbf{M}$_{\alpha+\varepsilon}^{\alpha}$ of cylindric like
algebras is introduced in \cite{08Fer}. The character of this class is similar
to that of \textbf{K}$_{\alpha+\varepsilon}^{\alpha}$ i.e. this class
satisfies all the cylindric axioms except for C$_{4}$ and C$_{6}$. Instead of
these axioms it satisfies some concrete weakenings of these axioms.
The following thorem is true for \textbf{M}$_{\alpha+\varepsilon}^{\alpha}$ :
Suppose that $\mathcal{A}\in\mathbf{CA}_{\alpha}$
\begin{theorem}\label{T2} $\mathcal{A}\in\mathbf{ICrs}_{\alpha}\cap\mathbf{CA}_{\alpha}$ if
and only if $\mathcal{A}\in\mathbf{S\Nr M}_{\alpha+\varepsilon}^{\alpha}$ where
$\alpha$ and $\varepsilon$ are infinite fixed ordinals.
\end{theorem}
\bigskip

Resek and Thompson's famous theorem says: $\mathcal{A}\in\mathbf{CA}_{\alpha
}^{M}$ if and only if $\mathcal{A}\in\mathbf{ICrs}_{\alpha}\cap\mathbf{CA}%
_{\alpha}$ where $\mathbf{CA}_{\alpha}^{M}$ denotes the class of cylindric
algebras satisfying the merry-go-round properties. Theorem \ref{T2} allows us to
give a new proof for this classical theorem.
Furthermore it shows that the $NET$ does cast its shadows over the entire field.
Theorem \ref{T2} also has remarkable consequences for logic. In classical first order
logic if the language is extended by new individual variables then the
deduction system obtained is a conservative extension of the old one. This fails to be true for logics with infinitary predicates. But, as a consequence
of the theorem \ref{T2} it can be proven that restricting the commutativity of
quantifiers and the equality axioms in the expanded language and supposing the
merry-go-round properties in the original language, the foregoing extension is
already a conservative one (see \cite{09Fer}).

Coming back to Henkin's Neat Embedding Theorem, we have the following.
On the one hand, 
the characterization established by Henkin is an 
abstract characterization of the class $\RCA_{\alpha}$, and it does help
occasionally to prove that certain subclasses consist 
exclusively of representable algebras.
This occurs when it is easier to 
prove that an algebra embeds neatly into another algebra in 
$\omega$ extra dimensions, and applies, for example,
in the cases of the classes of the so-called 
diagonal cylindric algebras and the semisimple ones, 
\cite[2.6.50]{HMT1} and \cite[3.2.11]{HMT2}.
On the other hand, this characterization is not satisfactory 
from the set-theoretic representation point of view, 
since it is not 
{\it intrinsic} and refers 
to algebras outside the algebra considered, namely, algebras
with $\omega$ extra dimensions.
\footnote{By an intrinsic property of an algebra $\A$ we understand, loosely speaking,
a property which can be expressed entirely in terms of symbols denoting 
the operations of $\A$, and variables ranging exclusively over elements 
of the universe of $\A$, subsets of $\A$, 
relations between elements of $A$, sets of such subsets and relations, etc., and not,
e.g., in case $A$ is a family of sets, in terms of variables ranging over $\cup A.$ 
(Note that the notion of an intrinsic property is of mathematical nature; 
to be made precise it must be relativized to
a well defined formal language).}

In \cite{HMT1} algebras in the class  $S\Nr_{\alpha}\CA_{\alpha+\omega}$ 
are said to have the
{\it neat embedding property}. Therefore, an algebra has the neat embedding property
if it neatly embeds into an algebra with $\omega$ extra dimensions, 
and according to Henkin's Neat Embedding Theorem
the class $\RCA_{\alpha}$ coincides with the class of 
cylindric algebras that have the neat embedding property. 
Since locally finite algebras have the neat embedding property, 
Theorem 3.2 can be seen
as an indeed substantial generalization of Tarski's representation 
Theorem 
pushing it to the limit.
Quoting Henkin-Monk and Tarski \cite{HMT1} p.400
``The notion of the neat embedding property  appears to be 
more suitable for an abstract 
algebraic treatment than that of representability. 
This is the main reason why our
discussion of neat reducts and their subalgebras in the present section will be
comprehensive and detailed."

It is known (cf. \cite{AU} or \cite{SL} ) that the Neat Embedding Theorem, 
or the $NET$ for short, proved
by Henkin in the fifties, is an algebraization of 
Henkin's celebrated proof of the completeness
of first order logic, or rather an extension thereof. 
Indeed it can be viewed as a typical instance 
of Robinson's finite forcing in  model theory \cite{HH97a}. 
Referred to by Hirsch and Hodkinson \cite{HH97b} as one of the 
earliest examples of step-by-step building representations
in algebraic logic, 
variants of the $NET$ have been successfully applied to (algebraically)
prove the completeness of several versions of quantifier logics, that are 
extensions,  variants, or reducts of first order logic. 
Examples include Keisler's logics investigated algebraically in \cite{DM63}, 
and various reducts thereof, like the logics studied in \cite{AGN77} 
and much later in \cite{Sa98}, under the name of {\it typeless
finitary logics of infinitary relations}, 
see also 
\cite[\S4.3]{HMT2} for a systematic treatment
of such logics. Other contexts to which the $NET$ applies to prove completeness
are the higher order logics investigated by 
Sagi \cite{Sagi2000} and Sagi, and Sayed \cite{SS99}.
Variations on the $NET$ gives results on amalgamation, which 
is the algebraic equivalent of interpolation 
in the corresponding logic.  Indeed this theme is pursued in \cite{AU}, 
\cite{IGPL3},\cite{BL}, \cite{Am}, \cite{IGPL4}, \cite{IGPL5},\cite{IGPL6}, 
\cite{Bull6b}, and \cite{AUU}.

\subsection{Variations on the $NET$}

Several other strenghthenings and incarnations of the $NET$ has been investigated by the author for finite dimensions. For a class $K$ let 
$S_cK=\{\A: \exists \B\in K:\A\subseteq \B \text { and for all } X\subseteq \A, \sum^{\A} X=1, \text { then }\sum^{\B}X=1\}$.
We write $\A\subseteq_c \B$, if $\A\in S_c\{\B\}$. Now we let $n<\omega$. It is proved in \cite{SSL} that if $\A\in \CA_n$ 
then $\A$ is completely representable if and only if $\A\in S_c\Nr_n\CA_{\omega}$ and
$\A$ is atomic. While $\RCA_n$ is a variety, it can be shown that the class 
$S_c\Nr_n\CA_{\omega}$ is a pseudo elementary class, 
that is not elementary; furthermore; its elementary closure,  $UpUrS_c\Nr_n\CA_{\omega}$ is not finitely axiomatizable. 
(In fact any $n\geq 3$ and any class $K$ such that $\Nr_n\CA_{\omega}\subseteq K\subseteq S_c\Nr_n\CA_{n+2}$, $K$ is not elementary.)
In \cite{Stability} the following question is investigated. When does $\A\in S_c\Nr_n\CA_{\omega}$ 
posses a cylindric representation preserving a given set of (infinite) 
meets carrying them to set theoretic intersection? If $\A$ has a representation preserving arbitrary meets, then $\A$ is atomic.
Conversely, when $\A$ is countable and atomic then $\A$ has such a representation. Such a representation is called an atomic or complete representation.
A complete representation carries arbitrary joins to unions. That is if $rep:\A\to \wp(^nX)$ is a complete representation, then
$rep(\sum X)=\bigcup_{x\in X}rep(x)$, whenever $\sum X$ exists. We give two examples showing   
that countability is essential and we 
cannot replace $S_c\Nr_n\CA_{\omega}$ by $\Nr_n\CA_{n+k}\cap \RCA_n$ for any finite $k$.

\begin{example}\label{countable}

Here we define an atomic relation algebra $\A$ with uncountably many
atoms.  This algebra will be used to construct cylindric algebras of dimension 
$n$ showing that countability is essential in the above characterization. For undefined terminology the reader is referred 
to \cite{HHbook}. The atoms are $1', \; a_0^i:i<\omega_1$ and $a_j:1\leq j<
\omega$, all symmetric.  The forbidden triples of atoms are all
permutations of $(1',x, y)$ for $x \neq y$, \/$(a_j, a_j, a_j)$ for
$1\leq j<\omega$ and $(a_0^i, a_0^{i'}, a_0^{i^*})$ for $i, i',
i^*<\omega_1$.  In other words, we forbid all the monochromatic
triangles.  Write $a_0$ for $\set{a_0^i:i<\omega_1}$ and $a_+$ for 
$\set{a_j:1\leq j<\omega}$. Call this atom
structure $\alpha$.  Let $\A$ be the term algebra on this atom
structure.  $\A$ is a dense subalgebra of the complex algebra
$\Cm\alpha$. We claim that $\A$ has no complete representation.
Indeed, suppose $\A$ has a complete representation $M$.  Let $x, y$ be points in the 
representation with $M \models a_1(x, y)$.  For each $i<\omega_1$ there is a 
point $z_i \in M$ such that $M \models a_0^i(x, z_i) \wedge a_1(z_i, y)$.  Let 
$Z = \set{z_i:i<\omega_1}$.  Within $Z$ there can be no edges labelled by 
$a_0$ so each edge is labelled by one of the countable number of atoms in 
$a_+$.  Ramsay's theorem forces the existence of three points 
$z^1, z^2, z^3 \in Z$ such that $M \models a_j(a^1, z^2) \wedge a_j(z^2, z^3) 
\wedge a_j(z^3, z_1)$, for some single $j<\omega$.  This contradicts the 
definition of composition in $\A$.

Let $S$ be the set of all atomic $\A$-networks $N$ with nodes
 $\omega$ such that\\ $\set{a_i: 1\leq i<\omega,\; a_i \mbox{ is the label 
of an edge in }
 N}$ is finite.
Then it is straightforward to show $S$ is an amalgamation class, that is for all $M, N 
\in S$ if $M \equiv_{ij} N$ then there is $L \in S$ with 
$M \equiv_i L \equiv_j N.$  
Hence the complex cylindric algebra $\Ca(S)\in \CA_\omega$.
Now let $X$ be the set of finite $\A$-networks $N$ with nodes
$\subset\omega$ such that 
\begin{enumerate}
\item each edge of $N$ is either (a) an atom of
$\c A$ or (b) a cofinite subset of $a_+=\set{a_j:1\leq j<\omega}$ or (c)
a cofinite subset of $a_0=\set{a_0^i:i<\omega_1}$ and
\item $N$ is `triangle-closed', i.e. for all $l, m, n \in nodes(N)$ we
have $N(l, n) \leq N(l,m);N(m,n)$.  That means if an edge $(l,m)$ is
labelled by $1'$ then $N(l,n)= N(m,n)$ and if $N(l,m), N(m,n) \leq
a_0$ then $N(l,n).a_0 = 0$ and if $N(l,m)=N(m,n) =
a_j$ (some $1\leq j<\omega$) then $N(l,n).a_j = 0$.
\end{enumerate}
For $N\in X$ let $N'\in\Ca(S)$ be defined by 
\[\set{L\in S: L(m,n)\leq
N(m,n) \mbox{ for } m,n\in nodes(N)}\]
Then if $N\in X, \; i<\omega$ then $\cyl i N' =
(N\restr{-i})'$.
The inclusion $\cyl i N' \subseteq (N\restr{-i})'$ is clear.
Conversely, let $L \in (N\restr{-i})'$.  We seek $M \equiv_i L$ with
$M\in N'$.  This will prove that $L \in \cyl i N'$, as required.
Since $L\in S$ the set $X = \set{a_i \notin L}$ is infinite.  Let $X$
be the disjoint union of two infinite sets $Y \cup Y'$, say.  To
define the $\omega$-network $M$ we must define the labels of all edges
involving the node $i$ (other labels are given by $M\equiv_i L$).  We
define these labels by enumerating the edges and labelling them one at
a time.  So let $j \neq i < \omega$.  Suppose $j\in nodes(N)$.  We
must choose $M(i,j) \leq N(i,j)$.  If $N(i,j)$ is an atom then of
course $M(i,j)=N(i,j)$.  Since $N$ is finite, this defines only
finitely many labels of $M$.  If $N(i,j)$ is a cofinite subset of
$a_0$ then we let $M(i,j)$ be an arbitrary atom in $N(i,j)$.  And if
$N(i,j)$ is a cofinite subset of $a_+$ then let $M(i,j)$ be an element
of $N(i,j)\cap Y$ which has not been used as the label of any edge of
$M$ which has already been chosen (possible, since at each stage only
finitely many have been chosen so far).  If $j\notin nodes(N)$ then we
can let $M(i,j)= a_k \in Y$ some $1\leq k < \omega$ such that no edge of $M$
has already been labelled by $a_k$.  It is not hard to check that each
triangle of $M$ is consistent (we have avoided all monochromatic
triangles) and clearly $M\in N'$ and $M\equiv_i L$.  The labelling avoided all 
but finitely many elements of $Y'$, so $M\in S$. So
$(N\restr{-i})' \subseteq \cyl i N'$.
Let $X' = \set{N':N\in X} \subseteq \Ca(S)$.
Then the subalgebra of $\Ca(S)$ generated by $X'$ is obtained from 
$X'$ by closing under finite unions.
Clearly all these finite unions are generated by $X'$.  We must show
that the set of finite unions of $X'$ is closed under all cylindric
operations.  Closure under unions is given.  For $N'\in X$ we have
$-N' = \bigcup_{m,n\in nodes(N)}N_{mn}'$ where $N_{mn}$ is a network
with nodes $\set{m,n}$ and labelling $N_{mn}(m,n) = -N(m,n)$. $N_{mn}$
may not belong to $X$ but it is equivalent to a union of at most finitely many 
members of $X$.  The diagonal $\diag ij \in\Ca(S)$ is equal to $N'$
where $N$ is a network with nodes $\set{i,j}$ and labelling
$N(i,j)=1'$.  Closure under cylindrification is given.
Let $\c C$ be the subalgebra of $\Ca(S)$ generated by $X'$.
Then $\A = \Ra(\c C)$.
Each element of $\A$ is a union of a finite number of atoms and
possibly a co-finite subset of $a_0$ and possibly a co-finite subset
of $a_+$.  Clearly $\A\subseteq\Ra(\c C)$.  Conversely, each element
$z \in \Ra(\c C)$ is a finite union $\bigcup_{N\in F}N'$, for some
finite subset $F$ of $X$, satisfying $\cyl i z = z$, for $i > 1$. Let $i_0,
\ldots, i_k$ be an enumeration of all the nodes, other than $0$ and
$1$, that occur as nodes of networks in $F$.  Then, $\cyl 
{i_0} \ldots
\cyl {i_k}z = \bigcup_{N\in F} \cyl {i_0} \ldots
\cyl {i_k}N' = \bigcup_{N\in F} (N\restr{\set{0,1}})' \in \A$.  So $\Ra(\c C)
\subseteq \A$.

$\A$ is relation algebra reduct of $\c C\in\CA_\omega$ but has no
complete representation.
Let $n>2$. Let $\B=\Nr_n \c C$. Then
$\B\in \Nr_n\CA_{\omega}$, is atomic, but has no complete representation.
\end{example}

\begin{example}\label{OTT}
We use a simplified version of a construction in \cite{ANT}. Ultimately we will show that we cannot replace $S_c\Nr_n\CA_{\omega}$ by
$\Nr_n\CA_{n+m}\cap \RCA_n$ for any finite $m$ in completely representing given countable atomic algebras. 
That is for every $m\geq 0$, there exists a countable atomic representable $\A\in \Nr_n\CA_{n+m}$, that has no complete
representation. 
Let $k$ be a cardinal. Let $\E_k=\E_k(2,3)$ denote the relation algebra
which has $k$ non-identity atoms, in which $a_i\leq a_j;a_l$ if $|\{i,j,l\}\in \{2,3\}$
for all non-identity atoms $a_i, a_j, a_k$.
Let $k$ be finite, let $I$ be the set of non-identity atoms of $\E_k(2,3)$ and let $P_0, P_1\ldots P_{k-1}$ be an enumeration of the elements of $I$.
Let $l\in \omega$, $l\geq 2$ and let $J_l$ denote the set of all subsets 
of $I$ of cardinality $l$. Define the symmetric ternary relation on $\omega$ by $E(i,j,k)$ if and only if $i,j,k$ are evenly distributed, that is
$$(\exists p,q,r)\{p,q,r\}=\{i,j,k\}, r-q=q-p.$$
Let $m<n$ be given finite ordinals. We show that there exists $\C\in \Nr_m\CA_n\cap \RCA_m$ and $X\subseteq C$ such that $\prod X=0$
but for any non-zero representation $f:\C\to \mathfrak{D}$ we have  $\bigcap_{x\in X} f(x)\neq 0$.
Now assume that $n>2$, $l\geq 2n-1$, $k\geq (2n-1)l$, $k\in \omega$. Let $\M=\E_k(2,3).$
Then $$(\forall V_2\ldots, V_n, W_2\ldots W_n\in J_l)(\exists T\in J_l)(\forall 2\leq i\leq n)$$
$$(\forall a\in V_i)\forall b\in W_i)(\forall c\in T_i)(a\leq b;c).$$
That is $(J4)_n$ formulated in \cite{ANT} p. 72 is satsified. Therefore, as proved in \cite{ANT} p. 77,
$B_n$ the set of all $n$ by $n$ basic matrices is a cylindric basis of dimension $n$.
But we also have $$(\forall P_2,\ldots ,P_n,Q_2\ldots Q_n\in I)(\forall W\in J_l)(W\cap P_2;Q_2\cap\ldots \cap P_n:Q_n\neq 0)$$
That is $(J5)_n$ formulated on p. 79 of \cite{ANT} holds. According to definition 3.1 (ii) $(J,E)$ is an $n$ blur for $\M$, 
and clearly $E$ is definable in $(\omega,<)$.
Let $C$ be as defined in lemma 4.3 in \cite{ANT}.   
Then, by lemma 4.3, $C$ is a subalgebra of $\Cm\B_n$, hence it contains the term algebra $\Tm\B_n$.
Denote $\C$ by $\Bb_n(\M, J, E)$. Then by theorem 4.6 in \cite{ANT} $\C$ is representable, and by theorem 4.4 in \cite{ANT} 
for $m<n$
$\Bb_m(\M,J,E)=\Nr_m\Bb_n(\M,J,E)$. However $\Cm\B_n$ is not representable hence $\C\in \Nr_m\CA_n\cap \RCA_n$ is atomic, countable, representable, 
but not completely 
representable. 
\end{example}

In \cite{Stability} we also investigate the question of when representations preserve a given (possibly infinite) set of meets.
Here we are touching deep set theoretic waters. 
We show that when the meets are ultrafilters then preservation of $<{}^{\omega}2$ many meets is possible (in $ZFC$), 
while if they are not then we are led to a statement that
is independent of $ZFC$.
In fact we prove the following theorem:

\begin{theorem}\label{covK} Let $\A\in S_c\Nr_n\CA_{\omega}$ be countable. Let $covK$ be the least cardinal such that the real line can be covered by 
$\kappa$ nowhere dense sets.
Let $\kappa<covK$. Let $(X_i:i\in \kappa)$ be a family of subsets of $\A$ such that
$\prod X_i=0$ for all $i\in \kappa$. Then for every $a\in A$ $a\neq 0$, there exists a representation $f:\A\to \wp(^nX)$ such that 
$f(a)\neq 0$ and $\bigcap_{x\in X_i}f(x)=\emptyset$
for all $i\in \kappa$.
\end{theorem}
\begin{demo}{Proof}  Assume that $\A$ is countable with $\A\in S_c\Nr_n\CA_{\omega}.$ Let $a\in \A$ be non-zero. 
Then $\A=\Nr_n\D$ with $\D\in \CA_{\omega}$.
Let $\B=\Sg^{\D}A$. Then $\B\subseteq \D$, $B$ is countable and $\B\in \Lf_{\omega}.$
Futhermore we have $a\in B$ is non-zero and $\prod X_i=0$ in $\B$. 
We have by \cite[1.11.6]{HMT1} that 
\begin{equation}\label{t1}
\begin{split} (\forall j<\alpha)(\forall x\in B)({\sf c}_jx=\sum_{i\in \alpha\smallsetminus \Delta x}
{\sf s}_i^jx.)
\end{split}
\end{equation}
Here $\sum$ denotes supremum  and for distinct $i,j<\beta$, ${\sf s}_i^jx$ is defined by 
${\sf c}_j(x\cdot {\sf d}_{ij})$. ${\sf s}_i^ix$ is defined to be $x$.
If $x$ is a formula, then ${\sf s}_i^jx$ is the 
operation of replacing the free occurrences of variable $v_j$ by $v_i$ 
such that the substitution is free.
Now let $V$ be the weak space $^{\omega}\omega^{(Id)}=\{s\in {}^{\omega}\omega: |\{i\in \omega: s_i\neq i\}|<\omega\}$.
For each $\tau\in V$ for each $i\in \kappa$, let
$$X_{i,\tau}=\{{\sf s}_{\tau}x: x\in X_i\}.$$
Here ${\sf s}_{\tau}$ 
is the unary operation as defined in  \cite[1.11.9]{HMT1}.
${\sf s}_{\tau}$ 
is the algebraic counterpart 
of the metalogical operation of 
the simultaneous  substitution of variables (indexed by the range of $\tau$) 
for variables (indexed by its domain) \cite[1.11.8] {HMT1}.
For each $\tau\in V,$ ${\sf s}_{\tau}$ is a complete
boolean endomorphism on $\B$ by \cite[1.11.12(iii)]{HMT1}. 
It thus follows that 
\begin{equation}\label{t2}\begin{split}
(\forall\tau\in V)(\forall  i\in \kappa)\prod{}^{\A}X_{i,\tau}=0
\end{split}
\end{equation}
Let $S$ be the Stone space of the boolean part of $\A$, and for $x\in \A$, let $N_x$ 
denote the clopen set consisting of all
boolean ultrafilters that contain $x$.
Then form \ref{t1}, \ref{t2}, it follows that for $x\in \A,$ $j<\beta$, $i<\kappa$ and 
$\tau\in V$, the sets 
$$\bold G_{j,x}=N_{{\sf c}_jx}\setminus \bigcup_{i\notin \Delta x} N_{{\sf s}_i^jx}
\text { and } \bold H_{i,\tau}=\bigcap_{x\in X_i} N_{{\sf s}_{\bar{\tau}}x}$$
are closed nowhere dense sets in $S$.
Also each $\bold H_{i,\tau}$ is closed and nowhere 
dense.
Let $$\bold G=\bigcup_{j\in \beta}\bigcup_{x\in B}\bold G_{j,x}
\text { and }\bold H=\bigcup_{i\in \kappa}\bigcup_{\tau\in V}\bold H_{i,\tau.}$$
By properties of $covK$, it can be shown $\bold H$ is a countable collection of nowhere dense sets.
By the Baire Category theorem  for compact Hausdorff spaces, we get that $X=S\sim \bold H\cup \bold G$ is dense in $S$.
Accordingly let $F$ be an ultrafilter in $N_a\cap X$.
By the very choice of $F$, it follows that $a\in F$ and  we have the following 
\begin{equation}
\begin{split}
(\forall j<\beta)(\forall x\in B)({\sf c}_jx\in F\implies
(\exists j\notin \Delta x){\sf s}_j^ix\in F.)
\end{split}
\end{equation}
and 
\begin{equation}
\begin{split}
(\forall i<\kappa)(\forall \tau\in V)(\exists x\in X_i){\sf s}_{\tau}x\notin F. 
\end{split}
\end{equation}
Next we form the canonical representation corresponding to $F$
in which satisfaction coincides with genericity. 
To handle equality we define
$$E=\{(i,j)\in {}^2{\alpha}: {\sf d}_{ij}\in F\}.$$
$E$ is an equivalence relation on $\alpha$.   
$E$ is reflexive because ${\sf d}_{ii}=1$ and symmetric 
because ${\sf d}_{ij}={\sf d}_{ji}.$
$E$ is transitive because $F$ is a filter and for all $k,l,u<\alpha$, with $l\notin \{k,u\}$, 
we have 
$${\sf d}_{kl}\cdot {\sf d}_{lu}\leq {\sf c}_l({\sf d}_{kl}\cdot {\sf d}_{lu})={\sf d}_{ku}.$$
Let $M= \alpha/E$ and for $i\in \omega$, let $q(i)=i/E$. 
Let $W$ be the weak space $^{\alpha}M^{(q)}.$
For $h\in W,$ we write $h=\bar{\tau}$ if $\tau\in V$ is such that
$\tau(i)/E=h(i)$ for all $i\in \omega$. $\tau$ of course may
not be unique.
Define $f$ from $\B$ to the full weak set algebra with unit $W$ as follows:
$$f(x)=\{ \bar{\tau} \in W:  {\sf s}_{\tau}x\in F\}, \text { for } x\in \A.$$ 
Then it can be checked that $f$
is a homomorphism 
such that $f(a)\neq 0$ and 
$\bigcap f(X_i)=\emptyset$ for all $i\in \kappa$, hence the desired.
The natural restriction of $f$ to $\A$ is as desired.
\end{demo}
The above proof depended on the following topological property. If $X$ is a second countable compact Hausdorff space
and $(A_i: i<\kappa)$ is a family of nowhere dense sets then $X\sim \bigcup_{i<\kappa} A_i$ is dense.
The idea of proof is that the (possibly uncountable union) $\bigcup_{i<\kappa}A_i$ can be written as a {\it countable} union of nowhere dense sets
and then a direct application of the Baire category theorem for compact Hausdorff enables one to get the desired ultrafilter $F$.
The question arises as to what happen if we replace $covK$ by $^{\omega}2$. (Recall that it is consistent that they are not equal).
In this case the theorem cannot be proved in $ZFC$. We would need extra (independent) axioms. One possible axiom is Martin's axiom $(MA)$. 
This follows from the fact that  $MA$ implies that if we have a union of nowhere dense sets over an indexing set $I$ with $|I|<{}^{\omega}2$
then it is a countable union. But $MA$ is two strong.
For any ordinal $\alpha$, let $P_{\alpha}$ be the statement: Given a collection $<2^{\omega_{\alpha}}$ subsets of $\omega_{\alpha}$ such that
the intersection of any $<\omega_{\alpha}$ has cardinality $\omega_{\alpha}$, then there is $B\subseteq \omega_{\alpha}$ of cardinality $\omega_{\alpha}$
such that for each element $A$ of the collection $|B-A|<\omega_{\alpha}.$
$P_0$ is the statement : Whenever $\A$ is a family of subsets of $\omega$ such that $|{\A}|<{}^{\omega}2$
and $A_0\cap A_1\cap\ldots A_n$ is infinite, whenever $A_0,A_1\ldots A_n\in \A$, then there is a subset $I$ of $\omega$
such that $I\sim A$ is finite for every $A\in I$ 
which is essentially the combinatorial part of $MA$.
It can be shown that $MA\implies P_0$ and that $MA$ is strictly stronger than $P_0$  
$P_0$ is essentially the combinatorial part of $MA$.
Under $P_0$ the following can be proved: If $X$ is a topological space with countable base, 
then the family of nowheredense sets $J$ has the property that whenever $J_1\subseteq J$ and $|J_1|<{}^{\omega}2$,
there is a countable $J_0\subseteq J$ such that every member of $J_1$ is included in a member of $J_0$ \cite{F}.
And thats all we need.  $P_0$ is equivalent to Martin's axiom 
restricted to the so called $\sigma$-centered partially ordered sets, so it is a restricted form of Martin's 
axiom \cite{F}. But actually what we need is even, yet, a weaker assumption, and that is Martin's
axiom restricted to {it countable} partially ordered sets, called $MA(countable)$. In passing we note that $covK$ is the largest cardinal such that
$MA(countable)$ is true, so that in some exact sense the cardinal $covK$ is the best possible.
In short when we loosen the statement $\kappa<covK$ to $\kappa<{}^{\omega}2$ we are led to an independent statement in set theory.
In fact such a statement, is a consequence of $MA$, and like $MA$ it is independent from $ZFC+\neg CH$.
The consistency of such a statement is proved by showing that 
is a consequence of a combinatorial consequence of Martin's axiom, namely $P_0$.
The independence is proved using iterated forcing. We note that if $\mathfrak{m}$ 
denotes the least cardinal such that $MA$ fails and $\mathfrak{p}$ is the least
cardinal such that $MA$ for $\sigma$ centered partially ordered sets fails, then clearly $\omega_1\leq \mathfrak{m}\leq \mathfrak{p}\leq covK\leq {}^{\omega}2$.
It is consistent that $\mathfrak{m}$ is singular, it is provable  that both $covK$ and $\mathfrak{p}$ are regular, and it is provable 
$covK$ cannot have countable cofinality.
It is also consistent that $\mathfrak{m}<\mathfrak{p}<covK$.
Using Shelah's techniques from stability theory, we also 
investigate preservation of $<{}^{\lambda}2$ many (maximal) meets, where $\lambda$ is a regular uncountable 
cardinal, for uncountable algebras in $S_c\Nr_n\CA_{\omega}$. This will be proved below.

\begin{theorem}
Let $\A\in S_c\Nr_n\CA_{\omega}$ be infinite such that $|A|=\lambda$, $\lambda$ is a a regular cardinal.
Let $\kappa<{}^{\lambda}2$. Let $(X_i:i\in \kappa)$ be a family of non-principal ultrafilters of $\A$.
Then there exists a representation $f:\A\to \wp(^nX)$ such that $\bigcap_{x\in X_i}f(x)=\emptyset$
for all $i\in \kappa$.
\end{theorem}
\begin{demo}{Proof} see theorem \ref{ZF}
\end{demo}
The above theorem, to the best of our knowlege is the first theorem that uses techniques from stability theory in algebraic logic.
Note that, in the countable case, the condition of maximality of ``types" considered shifts us from 
an independent statement, to one that is provable in $ZFC$.

Conversely the $NET$ conjoined 
with some form of Ramsey's theorem 
has been applied to show the essential incompleteness of $\L_n$,
the first order logic restricted to the 
first $n$ variables when $\omega>n>2$,
see, e.g. 
\cite{M71}, \cite{HHM2000}. 
This follows from 
the following classical algebraic result of Monk
that established  the  ``infinite distance'' between $\CA$'s and $\RCA$'s. 
Monk's result marked a turning point in 
the development of the subject, and is considered
one of the most, if not {\it the} most, 
important model-theoretic result concerning
cylindric algebras.

\begin{theorem} Let $\omega>n>2$ and $m\in \omega$. 
Then $\RCA_n$ is properly contained in $S\Nr_n\CA_{n+m}.$
Thus $\RCA_n$ is not finitely axiomatizable.
\end{theorem}
\begin{demo}{Sketch of proof} 
Monk used Ramsey's Theorem to construct 
for each $m\in \omega$ and $2<n<\omega$, an algebra 
$\A_m\in \Nr_n\CA_{n+m}$ that is not representable.
The ultraproduct of the $\A_m$'s constructed by Monk 
(relative to any non-principal ultrafilter on $\omega$) 
is in $S\Nr_{n}\CA_{n+\omega}$,
hence by the $NET$, is representable. 
Using elementary model theory, it follows thus that the class 
$\RCA_n$ for $\omega>n>2$,
is not finitely axiomatizable.
\end{demo}

The $\A_m$'s   are referred to in the 
literature as Monk or Maddux algebras. Both authors used them.
The key idea of the construction of a Monk algebra is not so hard. 
Such algebras are finite, hence atomic, 
more precisely their boolean reduct is atomic.
The atoms are given colors, and 
cylindrifications and diagonals are defined by stipulating that monochromatic triangles 
are inconsistent. 
If a Monk algebra has many more atoms than colors, 
it follows from Ramsey's Theorem that any representation
of the algebra must contain a monochromatic triangle, so the algebra
is not representable. Andr\'eka's splitting as seen in \cite{An97} 
is a variation on the same theme, and indeed leads to a refinement of Monk's result (See below).
Here, splitting refers to splitting an atom
into more atoms ,
that enforces non-representability, 
in which case the original atom before the process of splitting
is no longer an atom after the process of splitting. 
However, Andr\'eka's splitting does not appeal to any form of Ramsey's theorem.
In \cite{IGPL} Andr\'eka's splitting is used to show that the class
of neat $n$ reducts of $\beta$-dimensional cylindric algebras
is not elementary for $2<n<\beta\cap \omega$. (It will be also used below).
We note that Monk established a very interesting connection between 
finite combinatorics and algebraic logic 
\cite{M74}, a recurrent theme in algebraic logic.
A recent use - establishing this link - of Monk algebras with a 
powerful combinatorial result of Erd\"os on probabilistic graphs has 
shown that 
the class of the so-called strongly representable atom structures of relation 
algebras and $n$-dimensional representable cylindric algebras 
is not elementary 
\cite{HH2000}.  
Such counterexamples were used in \cite{SSL}  to 
show that the omitting types theorem fails
for the finite variable fragments of first order logic, 
as long as the number of variables available 
is at least 3. The omitting types theorem  
for variants of first order logic, be it reducts or expansions 
was studied intensely in recent times
\cite{SSL}, \cite{Notre}, \cite{IGPL6}, 
\cite{OT}, \cite{OTT}, \cite{Bull3}, \cite{Bull5b}, \cite{Bull5}, 
and \cite{Bull5b}, see also theorems \ref{o}, \ref{oo}.

We should mention that 
Lyndon's three papers 
\cite{Lyndon50}, \cite{Lyndon56}, and \cite{Lyndon61}
on relation algebras were very influential. 
\cite{Lyndon56}
is the basis for Hirsch-Hodkinson's (step-by-step) 
approach to the representation problem.  
\cite{Lyndon61}, on the other hand, contains results 
of constructing non-representable relation algebras
from projective geometries. 
This led to Monk's famous result that the class of 
representable relation algebras is not finitely axiomatizable \cite{Monk64}. Here Bruck-Ryser theorem
on non existence of projective planes of certain orders was used.
The second key paper in this context is \cite{Monk65} 
where Monk extended his result (of non-finite axiomatizability) to 
$3$ dimensional representable cylindric algebras.
In both of these papers, Monk uses projective geometries. In his 1964 paper \cite{Monk64}
the relation algebras used arise from Lyndon's construction applied to projective
lines, and in the 1965 paper \cite{Monk65}, the algebras dealt with 
are 3-dimensional cylindric algebras constructed 
from Lyndon's relation algebras, which are defined as follows.

\begin{definition} Let $U$ be a set with $|U|\geq 4$ and let $e$ be any element such that $e\notin U$. 
then the Lyndon algebra on $U$, is the $\RA$ type algebra defined to be
$$\L(U)=(\wp(U\cup\{e\}), \cap,\sim, \circ, \breve{}, id)$$
where $id=\{e\}$, for all $X\subseteq U\cup \{e\}$, $\breve{X}=X$ and $\circ$ is the completely additive opeartion on $\wp(U\cup \{e\}$ defined between singletons of $U\cup \{e\}$
as follows. For any $u\neq v\in U$:
$$\{u\}\circ id=id\circ \{u\}=\{u\}$$
$$\{u\}\circ \{u\}=id\cup \{u\}=\{e,u\}$$
$$\{u\}\circ \{v\}=U\sim \{u,v\}$$
$$id\circ id=id$$
Byy a Lyndon algebra we mean a Lyndon algebra on some set $U$
\end{definition}
It is known that a Lyndon algebra on $U$ is representable iff there exists a projective plane whose lines are incident with exactly $|U|$ points.
Thus there are infinitely many $n\in \omega$ such that the Lyndon algebras on $n$ are representable, 
and there are infinitely many $n\in \omega$ such tha the Lyndon algebra on $n$ is not
reprresentable. Monk used this to show that the class $\RRA$ is not finitely axiomatizable.
 
Monk extended his results to cylindric algebras of
dimension larger than 3, in his 1969 paper \cite{M69}, 
using the algebras based on Ramsey's
Theorem.

Had it been otherwise, i.e., if for $\omega>n>2$, $\RCA_n$ 
had turned out to be 
axiomatizable by a finite set of equations $\Sigma$ say,
then this $\Sigma$ would have been probably taken as the standard 
axiomatization of $\CA_n$. Unfortunately this turned out not to be the case. 
Quoting Hirsch and Hodkinson p.8 \cite{HHbook}
``As it seemed, the hopes of workers over a hundred years starting with De Morgan and 
culminating in Tarski's work to produce a (simple, elegant, or at least finite)
set of algebraic properties - or in modern terminology - 
equations that captured exactly the true properties of $n$-ary relations
for $\omega>n>2$ were shattered by Monk's result."
This impasse is still provoking 
extensive research until the present day, in essentially 
two conflicting (but complementary) forms.
To understand the ``essence'' of representable algebras, one often deals
with the non-representable ones, the ``distorted images'' so to speak. 
Simon's result in \cite{Sim97}, of ``representing" non-representable algebras, 
seems to point out that this distortion is, after all, not completely chaotic.   
This is similar to studying non-standard models of arithmetic, 
that do shed light on the standard
model. One form, which we already 
discussed, is to try to circumvent this negative non-finite axiomatizability 
result. The other form is to sharpen it. 
Indeed, Monk's negative result---as far as non-finite axiomatizability is
concerned---stated 
above, was 
refined and strengthened by many authors in many directions,
to mention a few, Andr\'eka \cite{An97}, Biro \cite{Bi92},  Maddux \cite{Maddux89}, 
Sagi \cite{Sagi99}, and Hirsch and Hodkinson \cite{HHII}. 

Maddux \cite{Maddux89} proved that Monk 
algebras can be generated by a single element. 
This is far from being trivial. Besides, this result implies essential incompleteness for finite varaible fragments of $\L_n$ $n\geq 3$
when we have only one binary relation in the language.
Making an algebra one-generated involves increasing 
complexity on its 
automorphism group. 
The structure becomes rigid. Hirsch Hodkinson and Maddux \cite{HHM} used Maddux's algebras together with a combinatorial argument to show that
for $m\geq 3$ the incusions
$\CA_m=S\Nr_m\CA_m\supset S\Nr_m\CA_{m+1}\ldots$ are strict
and all but the first inclusion is finitely axiomatized. This, too, has deep implications concerning the proof theory of $\L_n$ \cite{HHM2000}.
Biro \cite{Bi92} proves that 
$\RCA_n$, $\omega>n>2$,  remains non-finitely axiomatizable, if we add finitely many
{\it first order definable} operations, a result 
that is already implicit in Monk's and Maddux's non-finite
axiomatizability results in \cite{M69} and \cite{Maddux89}. 
The novelty occurring in Biro's result
is making the notion ``first order definable'' explicit.
Andr\'eka \cite{An97}, building on work of Jonsson \cite{Jonsson} for relation
algebras,
proves the same result in case we add other ``kinds'' of operations, 
like for example {\it modalities}, i.e. operations distributing over the boolean join, 
as long as the added operations are finitely many. 
While Biro's result excludes axiomatizations by
a finite set of equations, 
Andr\'eka's, on the other hand,  exclude axiomatizations involving 
universal formulas in which only finitely many variables occur. 
Sagi \cite{Sagi99}, building on work of Lyndon \cite{Lyndon56}, 
addresses the most general formulation of the problem
showing that the Finitizability Problem cannot be solved by 
adding finitely many {\it permutation invariant} operations in 
the sense of Tarski-Givant \cite{TG},
as long as one hopes for particular (universal) axiomatizations involving only finitely
many variables, and he gives a sufficient condition for the refutation of such a problem. 
We recall from \cite{N96} that 
a permutation invariant operation on a set algebra with unit $^nU$ 
is one that is invariant under permutations of $U$.  
Mad\'arasz \cite{Ma} addresses the case when the (finitely many) added 
operations are binary and $L_{\infty \omega}^3$ definable.
One general form of the Finitizability Problem  
for both cylindric algebras and relation algebras 
is to the best of our knowledge still open.  
This more or less concrete form for cylindric algebras 
is the following:

\begin{athm}{Open Problem  (Tarski-Givant-Henkin-Monk-Maddux- N\'emeti)} 
Can we expand the language
of cylindric set algebras of dimension $n$, $\omega> n>2$,  by finitely 
many permutation invariant operations 
so that the interpretation of these newly added operations in the resulting class of algebras
is still of a concrete set-theoretic nature,  and the resulting class 
generates a finitely axiomatizable variety or quasi-variety? 
\end{athm}
For further elaboration on this problem we refer to \cite{Sagi99}, \cite{AU}, \cite{N96}, \cite{SSS}, 
\cite{Sa98}, \cite{SaGy96}, \cite{Sim97} and \cite{Sim93}.
We refer to the above problem as the {\it permutation invariant} 
version of the Finitizability Problem.
The requirement of permutation invariance here is crucial for it corresponds to the 
(meta-logical) fact that isomorphic models 
satisfy the same formulas, a basic requirement in abstract model theory.
Without this requirement there are rather easy solutions to 
the Finitizability Problem due to 
Biro \cite{Bi92}, Maddux \cite{Maddux89}, \cite{Madd}, \cite{Madd1} 
and Simon \cite{Sim97}.

\subsection{ A reduction of the Finitizability Problem}

An important result in \cite{Sagi99} is reducing the Finitizability Problem
for relation algebras addressing (the infinitely many) 
permutation invariant expansions of relation algebras 
to working entirely inside the class of relation algebras. We believe that this could be a breakthrough, and unfortunately, to the best of our knowledge,
this result was not published. It occurs in the first chapter in Sagi's dissertation.
So here we give an outline of Sagi's (important) reduction Theorem. For an algebra $\A$ with Boolean reduct $At(\A)$ denotes the set of atoms of $\A$.

\begin{definition} Let $\A$ be an atomic relation algebra and $2\leq k\in \omega$. 
$1'$ denotes the identity relation and $\circ$ denotes composition. By a $k$ dimensional matrix of $\A$ we understand a function
$f:k\times k\to At(\A)$ satisfying the following conditions
\begin{enumroman}
\item $(\forall i\in k)(f(i,i)\leq 1')$
\item $(\forall i,j\in k)(f(i,j)=f(j,i))$
\item $(\forall i,j, l\in k)(f(i,l)\leq f(i,j)\circ f(j,i))$
\end{enumroman}
\end{definition}
The set of all $k$ dimensional atom matrices of $\A$ is denoted by $\M_k(\A)$. 
If $f\in \M_k(\A)$ and $\sigma:k\to k$, then 
$f\circ \sigma:k\times k\to At(\A)$ such that $(\forall i,j\in k)((f\circ \sigma)(i,j)=f(\sigma(i), \sigma(j))$.
Let $S\subseteq M_k\A$. Then $S$ is a substitutional base for $\A$ iff
$(\forall a\in At(\A))(\exists f\in S)(a=f_{0,1})\text { and } (\forall f\in S)(\forall i,j\in k)(f\circ [i|j]\in S).$
Then one defines an algebra $\Sb(\A)=(\wp(S), \cap, \cup, \sim, {\sf S}_i^j, {\sf D}_{ij})_{i,j<k}$ by
${\sf D}_{ij}=\{f\in S: f_{i,j}\leq 1'\}$ and ${\sf S}_i^jX=\{f\in S: f\circ [i|j]\in X\}$.
Fix an an atomic relation algebra $\A$ 
and let $\C$ be an atomic subalgebra of $\A$. 
Then $\equiv_C$ denotes the relation on $At(\A)$
defined by
$$(\forall a, b\in At(\A))(a\equiv_C b\text { iff } (\forall c\in At(C)(a\leq c\Leftrightarrow b\leq c))$$
If $S$ is a $k$ dimensional substitutional basis of $A$ then $\equiv_C$ can be extended to $S$, the obvious way, that is
$$f\equiv_Cg \text { iff } (\forall i, j\in k)(f_{ij}\equiv_C g_{ij}).$$
Let $S_C$ be the set of equivalence classes of $\equiv_C$ and $(S_C)^*=\{\bigcup X: X\subseteq S_C\}.$
Then it can be checked that the latter is a subalgebra of $\Sb(\A)$.
We let, identifying algebras with their domains,  $\wp(^2U)$ stand for full set relation algebra with unit $^2U$.
For $k\geq 2$, we let $\wp(^kU)$ stand for the algebra $(\wp(^kU), \cup,\sim ,{\sf S}_i^j, {\sf D}_{i,j})_{i,j<k},$
where ${\sf S}_i^jX=\{f\in {}^{\alpha}U: f\circ [i|j]\in X\}$ and ${\sf D}_{ij}=\{s\in {}^kU: s_i=s_j\}$.
When $k=2$ we rely on context to see which algebra we intend. 
Suppose that $\rho:\C\to \wp(^2U)$ is a representation of $\C$, that is, $\rho$ is a one to one homomorphism. 
Then $\rho^*: (S_C)^*\to  \wp(^kU)$ is defined by
$$\rho^*(X)=\{q\in {}^kU: (\exists f\in X)(\forall i,j\in X)(\forall i,j\in k)((q_i,q_j)\in \rho(f_{i,j}/\equiv_C))\}.$$
Now $\rho$ is said to be {\it $S$ rich} iff
$$(\forall s \in {}^kU)(\exists f\in S)(s\in \rho^*(f_{\equiv_C})).$$
$S$ is said to be homogeneous in $\C$ iff
$$(\forall f\in S)(\forall b\in At(\A)(b\equiv_C f_{0,1}\implies \exists g\in S(g_{0,1}=b\land g\equiv_Cf)).$$
We write $\B\subseteq_k \A$ if $\B$ is a subalgebra of $\A$ generated by $k$ elements.

\begin{definition} Let $\A$ be a simple $\RA$ and let $k,n\in \omega$. Then $\A$ satisfies the $k,n$ subalgebra condition iff
$$\forall \B\subseteq_k \A)(\exists \C\subseteq \A)(\exists \text { a rep } \rho^{\B}:\C\to \wp(^2n)(\B\subseteq \C\text { and $\C$ is atomic })$$
$(\rho^{\B}, \B\subseteq_k \A)$ will be called a system of representations
\end{definition}

\begin{theorem} Let $k,n\in \omega$. 
If there exists a simple non representable finite relation algebra which has $n+2$ dimensional substitutional basis $S$ satisfying
\begin{enumroman}
\item the $k,n$ subalgebra condition with representations $(\rho^{\C}: \C\subseteq_k \A_k),$
\item $(\forall C\subseteq_k \A_k)(\rho^{\C}$ is an $S$ rich representation of its domain),
\item $(\forall C\subseteq_k \A_k)$$S$ is homogeneous in the domain of $\rho^{\C}),$ 
\end{enumroman}
then there does not exist a permutation invariant extension of $\RRA$ axiomatizable by univverasl formulas containing $k$
variables.
\end{theorem}
Sagi \cite{Sagi99} actually solves a restricted version of this problem, namely the case when the set of formulas are balanced.
A universal formula $\phi$ in aignature extending that of $\RA$ is balanced if for every subterm $f(t_1\ldots t_n)$ of $\phi$, where $f$ is not an 
$\RA$ operation symbol, every variable occuring in $\phi$
is one of the $t_i$'s. This notion has been investigated by Jonsson, McNulty, and others. intuitively the balanced formulas are the simple ones, 
becuase the new operations can be only used in a simple special way.
The algebras used by Sagi are the Lyndon algebras based on projective geometries.

\begin{theorem} For each $k\in \omega-\{0,1\}$ there exists a simple, finite non representable relation algebra $\A_k$ satisfying the following conditions: there is $n\in \omega$ such that
$\A_k$ has an $n+2$ dimensional substitutional basis $S$ such that
\begin{enumroman}
\item $\A_k$ satisfies the $k,n$ subalgebra condition with edge transitive representation \footnote{This is a kind of rep introduced by Sagi} $(\rho^C, C\\subseteq_k \A_k)$
\item $(\forall C\subseteq_k \A_k)(\rho^{\C}$ is an $S$ rich representation of its domain),
\item $(\forall C\subseteq_k \A_k)$$S$ is homogeneous in the domain of $\rho^{\C}),$ 
\end{enumroman}
\end{theorem}
\begin{demo}{Sketch of proof} Let $k$ be give. Let $U$ be a finite set such that $|U|\geq 2^k(2^k+1)+1$ and
$\L(U)$ is not representable. Let $\A_k=\L(U)$. Let $n=2^{4k}$ and $S=\M_{n+2}(\A_{k})$. Then $\A_k$ is as required.
\end{demo}

The analogous result for cylindric algebras is proved by the present author. 
In other words, it is enough 
to construct a specific countable sequence of non-representable cylindric algebras
to refute the permutation invariant version of the Finitizability Problem.
This of course does not settle the problem completely, instead it transforms it
to a hopefully simpler one.
Furthermore, using Monk's algebras one can show that there does not exists a permutation invariant extension 
of $\RCA_n$ axiomatizable by universal balanced formulas containing $k$ variables.

The current ``belief'' is that 
the answer to the unrestricted 
(permutation invariant) 
form of the problem 
is either negative, or perhaps even independent of 
$ZF$ (Zermelo-Fraenkel set theory). Indeed, in \cite{N97} and \cite{SN96} 
it is proved that several 
versions of the Finitizability Problem are independent from $ZF$ 
minus the axiom of foundation (and adopting other anti-foundation axioms). 
Positive solution exists in non-well founded set theories, becuase one can generate extra infinitely dimensions, forcing a neat embedding theorem,
by digging ``downwards".
This view comes across very much in the case of Nemeti's directed cylindric algebras, invesigated by Sagi \cite{Sagi2000}.
The results
of Hirsch and Hodkinson, in \cite{HHII}, seem to be relevant to the
permutation invariant form 
of the Finitizability Problem.  In \cite{HHbook} 
[ 17.4, p. 625] the problem of axiomatizing the class of relation algebras
with a set of first order sentences using finitely many variables 
is reduced to a problem about (colorings of)
certain graphs. 
On the face of it, this seems to be bad news for graphs, rather than good news for
providing ``simple' axiomatizations (using only finitely many variables) for 
representable relation algebras.

\subsection{Canonicity and strongly representable atom structures, via Erdos graphs}

To summarize, as we have seen the representation problem lies very much at the heart of algebraic logic, and its history dates back to
the early work of Tarski on relation algebras and cylindric algebras.
Algebraic logic arose as a subdiscipline of algebra mirroring constructions
and theorems of mathematical logic. 
It is similar in this respect to such fields as algebraic geometry
and algebraic topology, 
where the main constructions and theorems are algebraic in nature, 
but the main intuitions underlying them are respectively 
geometric and topological. 
The main intuitions 
underlying algebraic logic are, of course, those of formal logic.
Investigations in algebraic logic can proceed in two conceptually different, 
but often (and unexpectedly) closely related ways.
First one tries to investigate the algebraic essence of constructions 
and results in logic, in the hope of gaining more insight
that could add to his understanding, thus his knowledge.
Second, one can study certain ``particular'' algebraic structures 
(or simply algebras) that arise in the course 
of his first kind of investigations
as objects of interest in their own right and 
go on to discuss questions which naturally  arise independently of any connection with 
logic.   But often such  purely algebraic results 
have impact on the logic side.

Examples are the undecidability of the representation problem for  finite relation algebras \cite{finite}, \cite{HHbook} that led to deep results concerning 
undecidability of product modal logics answering problems posed by Gabbay \cite{K5}. This results also implies that the class 
of representable relation algebras cannot be finitely axiomatized in $n$th order logic for any $n$.
A similar situation occurs for $\RCA_3$, so that this class cannot be finitely axiomatized in
$n$ order first order logic. (The analogous result for $\CA_n$ is unkown, for $n\geq 4$).
Another example is the interconnection
of the metalogical notion of Omitting types and algebraic notions of atom canononcity and complete representations, first presented in \cite{Bulletin}
and elaborated upon in \cite{ANT}, see also theorems \ref{o} and \ref{oo} below.
And of course there are the various completeness theorems obtained for variants or modifications of first order logic and multi modal logics
when dealing with (different forms) of the representability problem \cite{HH97b}, \cite{Sa98}, \cite{Venema}, \cite{V98}.
Sometimes certain techniques used first in algebraic logic, prove useful for solving problems in (modal) logic.  An amazing
manifestation of such a phenomena is the use 
the probabalistic methods of  of Erd\"os in constructing finite graphs with arbitrarily large 
chromatic number and girth. 
In his pioneering paper of 1959, Erdos took a radically new approach to construct such graphs: for each $n$ he 
defined a probability space on the set of graphs with $n$ vertices, and showed that, for some carefuuly chosen probability measures, 
the probability that an $n$ vertex graph has these properties is positive for all large enough $n$.
This approach, now called the {\it probabilistic method} has since unfolded into a 
sophisticated and versatile proof technique, in graph theory and in other 
branches of discrete mathematics.
This method was used in algebraic logic to show that the class of strongly representable atom structures of cylindric 
and relation algebras is not elementary  
and that varieties of representable relation algebras are barely canonical.
This result was generalized to the class $\RCA_n,$ for finite $n\geq 3$,  adding to the complexity of potential axiomatizations
of $\RCA_n$,  
for it is proved by the author, that though the class $\RCA_n$ is canonical (i.e. closed under canonical extensions), 
any axiomatization of it must contain infinitely many
non canonical sentences. 
In what follows we give an outline of proof that $\RCA_3$ is barely canonical and that 
the representation problem is undecidable for finite $\CA_3$'s. We follow the notation and terminology of \cite{HV}. 
For a given Boolean algebra with operators $\A$, $(\A_+)$ denotes its ultrafilter frame and $\A^{\sigma}=(\A_+)^+$ 
denotes its canonical extension. 
A variety $V$ is canonical if it is closed under canonical extension. We show that 
the class $\RCA_3$, though canonical, has no canonical axiomatization, and that it is undecidable whether a finite $\CA_3$ has a representation.
These results were proved for relation algebras \cite{HV}. We use a construction of Monk that makes the passage from $\RA$ to $\CA_3$.
Let $\alpha$ be a relation algebra atom structure. Let $F_{\alpha}$ be the set of consistent triples of atoms of $\alpha$. 
Then define a cylindric algebra atom structure $\F=(F_{\alpha}, \equiv_i, {\sf d}_{ij})$ as follows.
For $i<3$, 
$$t\equiv_i s\text { iff } t_i=s_i$$
and
$${\sf d}_{ij}=\{t\in F: t_k\leq 1', k\notin \{i,j\}\}.$$ 
For an atomic algebra $\A$, its atom structure will be denoted by $\At\A$.
For a relation algebra atom structure $\alpha$, $\Ca_3(\alpha)$ denotes the cylindric algebra (of dimension $3$)
atom structure as defined above. 
It turns out, as proved by Monk,  that for a given finite relation algebra $\A$, $[\Ca_3\At \A]^+\in \CA_3.$ 
Another way of obtaining cylindric algebras of finite dimension 
$d\geq 3$ from relation algebras, is due to Maddux \cite{Mad}.
If $\A\in \RA$ posses a $d$ dimensional cylindric basis, then one can construct from this basis a cylindric algebra of dimension $d$.
For a relation algebra atom structure $\alpha$, $\M_d(\alpha)$ denotes the set of all $d$ dimensional basic matrices over $\alpha$.
Next, given a graph $\G$, and a positive integer $d\geq 3$, we define a relation algebra atom structure $\alpha(\G)$ of the form
$(\{1'\}\cup (\G\times d), R_{1'}, \breve{R}, R_;)$.
The only identity atom is $1'$. All atoms are self converse, 
so $\breve{R}=\{(a, a): a \text { an atom }\}.$
The colour of an atom $(a,i)\in \G\times d$ is $i$. The identity $1'$ has no colour. A triple $(a,b,c)$ 
of atoms in $\alpha(\G)$ is consistent if
$R;(a,b,c)$ holds. Then the consistent triples are $(a,b,c)$ where

\begin{itemize}

\item one of $a,b,c$ is $1'$ and the other two are equal, or

\item none of $a,b,c$ is $1'$ and they do not all have the same colour, or

\item $a=(a', i), b=(b', i)$ and $c=(c', i)$ for some $i<d$ and 
$a',b',c'\in \G$, and there exists at least one graph edge
of $G$ in $\{a', b', c'\}$.

\end{itemize}
\begin{theorem} 
\begin{enumarab}
\item $\alpha(\G)$ is a relation atom structure. 
\item The set $\M_d(\alpha(\G))$ of $d$-dimensional basic matrices is an atom structure of a cylindric algebra of dimension $d$.
\end{enumarab}
\end{theorem}
\begin{demo}{Proof} Note that $\alpha(\G)$ is like the relation algebra atom structure defined in \cite{HV}, 
except that we allow $d$ colors instead of just three.
This guarantees that $(2)$ holds.
\end{demo}

We assume that two-player games can be devised to charaterize the class of representable algebras \cite{HHbook}.
Let $L_{ra}$ denote the language of $\RA$'s and $L_{ca}$ denote the language of $\CA_3$.
$\Rd_{df}$ denotes the diagonal free reduct. With a slight abuse of notation we may write $\Ca_3\A$ for the cylindric algebra $[\Ca_3\At \A]^+$.
So that when $\Ca_3$ is applied to an algebra it produces an algebra, while when applied to an atom structure it produces an atom structure.

\begin{theorem}\label{h} There exists a set of first order formulas $\{\sigma_k: k\in \omega\}$ in the language $L_{ra}$  
and $\{\tau_k:k\in \omega\}$ in the language $L_{RA}$ 
such that $\sigma_k$ translates that $\exists$ has a winning  strategy after $k$ rounds in the relation algebra representation game,
and $\tau_k$ translates that $\exists$ has a winning strategy  after $k$ rounds in the $\CA_3$ representation game, 
such that the following hold:
\begin{enumroman}
\item If $\A\in \RA$ then $\A$ is representable if and only if $\A\models \sigma_k$ for al $k\in \omega$
\item if $\A\in \CA_3$ then $\A$ is representable if and only if $\A\models \tau_k$ for all $k\in \omega$.
\item if $\A$ is a finite simple relation algebras, then there exists $k_0\in \omega$, such that for all 
$k\in \omega, k\geq k_0$ 
$$\A\models \sigma_k\Longleftrightarrow \Ca_3\A\models \tau_k.$$
\item $\A$ is an atomic finite relation algebra, then $\A$ is representable iff $\Ca_3\A$ is rep iff $\Rd_{df}\Ca_3\A$ is representable.
\end{enumroman}
\end{theorem}
\begin{demo}{Proof} This can be done by suitable adjusting the games defined in \cite{HHbook}.
(iv) was proved by Monk and Johnson \cite{HMT1}.
\end{demo}
From now on ${\GG}_n$ stands for fixed  games after $n$ rounds as specified above. For relation algebras we follow the games defined in \cite{HV}
while for cylindric algebras we follow an easy  modification of the games defined in \cite{HHbook} so that the above Theorem holds.
The following two definitions are taken from \cite{HV}.
\begin{definition}
\begin{enumarab}
\item We say that a partially ordered set $(I, \leq)$ is directed if every finite subset of $I$ has an upper bound in $I$.
\item An inverse system of $L^a$ structures is a triple
$$D=((I,\leq)) (S_i:i\in I), (\pi_{ij}: i,j\in I, i\leq j)),$$
where $(I,\leq)$ is a directly partially ordered set, each $S_i$ is an $L^a$ structure, and for $i\leq j\in I$, $\pi_{ji}:S_j\to S_i$ is a surjective homomorphism,
such that whenever $k\geq j\geq i$ in $I$ then $\pi_{ii}$ is the identity map and $\pi_{ki}=\pi_{ji}\circ \pi_{kj}$.
\item We say that $D$ in an inverse system of finite structures if each $S_j$ is a finite structure, and an inverse system of bounded morphisms if each $\pi_{ij}$
is a bounded morphism.
\item The inverse limit $lim_{\leftarrow}D$ of $D$ is the substructure of $\prod_{i\in I}S_i$ with domain
$$\{\chi\in \prod_{i\in I}S_i: \pi_{ji}(\chi_j))=\chi(i)\text { whenever } j\geq i \text { in } I\}.$$
\item For any $i\in I$, the projection $\pi_i:lim_{\leftarrow}D\to S_i$ is defined by $\pi_i(\chi)=\chi(i)$
\end{enumarab}
\end{definition}

\begin{definition}
Let $D=((I,\leq) )S_i:i\in I), (\pi_{ji}: i,j\in I, i\leq j)$ be an inverse system of finite structure and bounded morphisms.
Let $I=lim_{\leftarrow} D$. For each $i\in I$ define $\pi_i^+:S_i^+\to I^+$ by $\pi_i^+(X)=\pi_i^{-1}[X]$, for $X\subseteq S_i$.
\end{definition}
Each $\pi_i^+$ is an algebra embedding $: S_i^+\to I^+$, and its range $\pi_i^+(S_i)$ is a finite subalgebra of $I^+$. It follows that
$\A_D=\bigcup_{i\in I}\pi_i^+(S_i^+)$ is a directed union of finite subalgebras of $I^+$, and we have 
$$(\A_D)_+\cong lim_{\leftarrow}D=I$$
The following theorem is proved in \cite{HV} by adapting techniques of Erdos in constructing probabilistic 
graphs with arbitrary large chromatic number and girth.

\begin{theorem} Let $k\geq 2$. There are finite graphs $G_0, G_1\ldots$ and surjective homomorphisms $\rho_i:G_{i+1}\to G_i$ for $i<\omega$
such that for each $i$, $\rho_i$ is a bounded morphism and
\begin{enumarab}
\item for each edge $xy$ of $G_i$ and each $x'\in \rho_i^{-1}(x)$, there is a $y'\in \rho_i^{-1}(y)$ such that
$x'y'$ is an edge of $G_{i+1}$,
\item $G_i$ has no odd cycles of length $\leq i$
\item $\chi(G_i)=k.$
\end{enumarab}
\end{theorem}
\begin{demo}{Proof} \cite{HV}
\end{demo}
Fix integers $k\geq m\geq 2$, and let $H_0, H_1\ldots$ and $\pi_i:H_{i+1}\to H_i$ be graphs and homomorphisms as above. 
Now fix a complete graph $K_m$ with $m$ nodes, and for each $i<\omega$, let $G_i$ be the disjoiny union of $H_i$ and $K_m$. 
For each $i<j<\omega$ define $\rho_{ij}$ to be the identity on $G_i$, and
$\rho_{ji}:G_j\to G_i$ be defined by
\begin{equation*}
\rho_{ji}(x) =
\begin{cases}
\pi_i\circ \ldots \pi_{j-1} , & \hbox{if $x\in H_j$} \\
x,  & \hbox{if $x\in K_m$.} \end{cases}
\end{equation*}
Let $$D_m^k=((\omega, \leq), (G_i:i<\omega), (\rho_{ji}: i\leq j<\omega))$$
Then $lim_{\leftarrow}D_m^k$ is a graph with chromatic number $m$.
Now combining Theorem 2 and the techniques of \cite{HV} theorem 6.8, we are ready for:

\begin{theorem}\label{can} $\RCA_3$ has no canonical axiomatization
\end{theorem}
\begin{demo}{Proof} Let ${\GG}_n$ denote the games defined after $n$ rounds. Assume that $\RCA_3$ has a canonical axiomatization.
Then is an $n_0$ such that for any $n<\omega$, there is $n^*<\omega$, such that for any $\CA_n$ $\A$
if $\exists$ has a winning strategy in ${\GG}_{n^*}(\A)$ and ${\GG}_{n_0}(\A^{\sigma})$ then she has a winning strategy in 
${\GG}_n(\A^{\sigma})$ \cite{HV} pop 5.4.
Let $e:\omega\to \omega$ be defined by $e(n)=2^{dn.4^n}$.
Then if $G$ is a graph and $\chi(G)\geq e(n)$ for some $n<\omega$, then $\exists$ has a winning strategy in ${\GG}_n(\alpha(G)^+)$ \cite{HV} prop 6.4.
For $n<\omega$, let $n'<\omega$ be so large such that any colouring using $dn$ colours , of the edges of a complete graph with $n'$ nodes 
has a monochromatic triangle. Let $u:\omega\to \omega$ be defined ny $u(n)=n'-2+n'(n'-1)(dn+1)$. 
Then if $G$ is a graph with $\chi(G)\leq n<\infty$ and $|G|\geq n'-1$
then $\forall$ has a winning strategy in ${\GG}_{u(n)}(\alpha(G)^+)$ \cite{HV} prop 6.6.
Let $m=e(n_0)$ and $n=u(m)$ . Since the games played are determined,
there is $n^*<\omega$, such that for any cylindric algebra $\A$, such that $\exists$ has a winning startegy in  ${\GG}_{n_0}(\A^{\sigma})$, if $\forall$ 
has a winning strategy in ${\GG}_n(\A^{\sigma})$
then he has a winning strategy in ${\GG}_{n^*}(\A)$.
Let $k=e(n^*)$.
Let $D=D_m^k=((\omega, \leq), (G_i:i<\omega), (\rho_{ij}: i\leq j<\omega))$ be the inverse system as defined above.
We have $G=lim_\leftarrow D$ and $\chi(G_i)=k$ for all $i$ and $\chi(G)=m$.
Let $\alpha(D)=((\omega,\leq), \Ca_3((\alpha(G_i)):i<\omega), (\alpha^{\rho_{ji}}: i\leq j<\omega))$, where 
$\rho_{ji}:\Ca_3(\alpha(G_j))\to \Ca_3(\alpha(G_i))$
is defined by $((a,k), (b,k),(c,k))\mapsto ((\rho_{ji}a,k),(\rho_{ji}b,k) (\rho_{ji}c,k))$. Then $\alpha(D)$ is an inverse system of cylindric algebra
atom structures and 
bounded morphisms. Each $[\Ca_3(\alpha (G_i))]^+$ is a cylindric algebra. Write $\A$ for the algebra $\A_{\alpha(D)}$. 
Then $[\Ca_3(\alpha (G_i))]^+\subseteq \A$ for all $i<\omega$ and $\A$ is the directed union $\bigcup_{i<\omega}[\Ca_3(\alpha (G_i)]^+$.
Then $\A$ is a cylindric algebra of dimension $3$. 
It can be checked that $\A$ is atomic and $\alpha(lim_\leftarrow D)\cong lim_\leftarrow\alpha(D)$, hence $\A^{\sigma}\cong [\Ca_3(\alpha(G))]^+$. 
Note that here $\Ca_3$ is applied to an infinite atom structure. But it can be easily checked that the resulting atom structure is that of a cylindric algebra of dimension $3$.
Now $G$ has chromatic number $m$ and is infinite. Then $\exists$ has a winning strategy 
in ${\GG}_{n_0}(\A^{\sigma})$ while $\forall$ has a winning strategy in 
${\GG}_{n}(\A^{\sigma})$. By choice of $n^*$, $\forall$ also has a winning strategy in ${\GG}_{n*}(\A)$ that 
only uses finitely many elements $W\subseteq \A$. 
We may choose $i<\omega$ such that $W\subseteq \alpha(G_i)^+$. Since the latter is a
subalgebra of $\A$, then this is a winning strategy  for $\forall$ in ${\GG}_{n^*}[\Ca_3(\alpha(G_i))]^+$. 
But $\chi(G_i)=e(n^*)$, then $\exists$ has a winning strategy in this same game. 
This is a contradiction that finishes the proof.
\end{demo}
Monk's construction was used in \cite{HHK} together with the deep result of Hirsch and 
Hodkinson of the undecidability of the representation problem for 
finite relation algebras \cite{un}, to show that the modal logics $K5\times K5\times K5$ are undecidable.
Using the same technique we now show

\begin{theorem}\label{un}  It is undecidable whether a finite simple $\CA_3$ is representable
\end{theorem} 
\begin{demo}{Proof} If there is a decision procedure of deciding whether a finite $\CA_3$ is representable, then this procedure can be implemented to decide that
whether a simple finite relation algebra is representable. For start by the  relation algebra $\A$. Form recursively $\Ca_3\A$. 
Then $\A$ is representable if and only if
$\Ca_3\A$ is representable, and we can decide the latter, so we can decide the former. 

\end{demo}
The results above adds somewhat to the complexity of axiomatizations of $\RCA_3$. 
For example, Theorem \ref{un}, as pointed out by Ian Hodkinson, implies that $\RCA_3$ has no finite axiomatization in $n$th order logic where
$n$ is any number.
Now we discuss our results in two respects. 
For other algebras, like diagonal free cylindric algebras and Halmos polyadic algebras, and for higher dimensions.
For $n=3$, our two main theorems generalise to diagonal free cylindric algebras and polyadic algebras with and without equality
and many reducts in between. This follows from Theorem \ref{h} 
and the fact that one can expand $\Ca\A$ to polyadic equality algebras by swapping coordinates.
An important reduct is that of Pinter's substitution algebras.
Now for higher dimensions, the proofs of theorems \ref{h} and  \ref{un} go 
through for diagonal free cylindric algebras. This is worthwhile formulating separately

\begin{theorem} Let $n\geq 3$ be finite. Then $\RDf_n$ cannot be axiomatized by canonical equations, and it is undecidable whether
 a finite $\Df_n$ is representable.
\end{theorem}
\begin{demo}{Proof} Given an $\RA$ $\A$ one defines a diagonal free cylindric algebra 
$\Rd_{df}\Ca_3\A$ and the extra cylindrifications are defined as the identity (This is actually done in \cite{HHK}) .  
Then a complete analogue of Theorem \ref{h} holds
and we are done.
\end{demo}
For cylindric algebras of higher dimensions, we show in \cite{ca} that there is a recursive function $g:\omega\to \omega$, such that
$g(k)\geq k$ eventually, and if $\exists$ has a winning strategy in ${\GG}_n(\A)$, then she has a winning strategy in ${\GG}_{g(n)}(\M_d(\At \A)).$
That being said, on replacing $\Ca_3\alpha$ by $\M_{d}(\alpha)$ in the proof of theorem \ref{can}, and undergoing the obvious modifications, 
would finish the proof.
This will show that $\RCA_n$ for $n\geq 3$ has no canonical axiomatization.
So putting together this result with the results of Andreka, we obtain the following ``reasonable" result 
conjectured (but not proved) recently by the author.

\begin{theorem}Let $n\geq 3$ be finite. 
Let $\Sigma$ be a set of equations axiomatizing $\RCA_n$. Let $l<n$ $k<n$, $k'<\omega$ be natural numbers. 
Then $\Sigma$ contains
infinitely many non-canonical equations in which $-$ occurs, one of $+$ or $\cdot$ ocurs  a diagonal with index $l$ occurs, more
than $k$ cylindrifications and more than $k'$ variables occur.
\end{theorem}
On the one hand, such techniques deepens the connections between algebraic logic and graph theory, and on the other
sheds more light on the complexity of axiomatizations of the class of representable algebras, showing that the variety $\RCA_n$ is really  ``wild". 
Another result using Erdos graphs is that the class of strongly representable atom structures is not elementary.
Here we extend this result to other algebas.
From now on we follow closely \cite{HHstrong}. In \cite{HHstrong} definition 3.5, 
the authors define a cylindric atom structure based on a graph $\Gamma$.
We enrich this atom structure by the relations corresponding to the polyadic operations:

\begin{definition} We define an atom structure $\eta(\Gamma)=(H, D_{ij}, \equiv_i, P_{ij})$
as follows.
\begin{enumarab}
\item $H$ is the set of all pairs $(K, \sim)$ where $K:n\to \Gamma\times n$ is a partial map and $\sim$ is an equivalence relation
on $n$ satisfying the followng conditions

(a) If $|n/\sim|=n$ then $dom(K)=n$ and $rng(K)$ is not  independent subset of $n$. 

(b) If $|n/\sim|=n-1$, then $K$ is defined only on the unique $\sim$ class $\{i,j\}$ say of size $2$
and $K(i)=K(j)$

(c) If $|n/\sim|\leq n-2$, then $K$ is nowhere defined.

\item $D_{ij}=\{(K,\sim)\in H: i\sim j\}$

\item $(K,\sim)\equiv_i (K', \sim')$ iff $K(i)=K'(i)$ and $\sim\upharpoonright (n\setminus \{i\})=\sim'\upharpoonright (n\setminus \{i\})$

\item $(K.\sim)\equiv_{ij}(K',\sim')$ iff $K(i)=K'(j)$ and $K(j)=K'(i),$ $K\upharpoonright n\sim \{i,j\}=K'\upharpoonright n\sim \{i,j\}$ 
and if $i\sim j$ then $\sim=\sim'$, if not, then $\sim'$ is related to $\sim$ as follows
For all $k\notin [i]_{\sim}\cup [j]_{\sim}$ $[k]_{\sim'}=[k]_{\sim}$ $[i]_{\sim'}=[j]_{\sim}\setminus \{j\}\cup \{i\}$ and 
$[j]_{\sim'}=[i]_{\sim}\setminus \{i\}\cup \{j\}.$ 
\end{enumarab}
\end{definition}

\begin{definition} Let $\C(\Gamma)$ be the complex algebra of polyadic type of the above atom structure.
That is $\C(\Gamma)=(\B(\eta(\Gamma)), {\sf c}_i, {\sf s}_i^j, {\sf p}_{ij}, {\sf d}_{ij})_{i,j<n}$ with extra non-Boolean operations
defined by: 
$${\sf d}_{ij}=D_{ij}$$
$${\sf c}_iX=\{c: \exists a\in X, a\equiv_ic\}.$$
$${\sf p}_{ij}X=\{c:\exists a\in X, a\equiv_{ij}c\}$$
and $${\sf s}_i^jx={\sf c}_j(x\cap d_{ij}).$$
\end{definition}
For $\A\in \CA_n$ and $x\in A$, recall that $\Delta x,$ the dimension set of $x$, is the set $\{i\in n: {\sf c}_ix\neq x\}$.

\begin{theorem} For any graph $\Gamma$, $\C(\Gamma)$ is a simple $\PEA_n$, that is generated by the set
$\{x\in C: \Delta x\neq n\}$.
\end{theorem}
\begin{demo}{Proof} $\Rd_{ca}\C(\Gamma)$ is a simple $\CA_n$ by \cite{HHstrong} lemma 5.1, 
hence if we prove that $\C(\Gamma)$ is a polyadic equality algebra, then as a polyadic equality algebra it will be simple. 
This follows from the simple observation that any polyadic ideal in $\C(\Gamma)$ is a cylindric ideal.
Furthermore for any atom $x=\{(K, \sim)\}$ of $\C$ we have $x={\sf c}_0x\cap {\sf c}_1x\ldots\cap {\sf c}_nx$.
We need to check the polyadic axioms. 
Since $\Rd_{ca}\A$ is a $\CA_n$ we need to show that 
the following hold for all $i,j,k\in n$:
\begin{enumarab}
\item ${\sf p}_{ij}$'s  are boolean endomorphisms
\item ${\sf p}_{ij}{\sf p}_{ij}x=x$
\item ${\sf p}_{ij}{\sf p}_{ik}={\sf p}_{jk}{\sf p}_{ij}x   \text { if } |\{i,j,k\}|=3$
\item ${\sf p}_{ij}{\sf s}_i^jx={\sf s}_j^ix$
\end{enumarab}
These properties, follow from the definitions, and are therefore left to the reader.
\end{demo}
Let $\Gamma=(G,E)$ be a graph. Then A set $X\subseteq G$ is independent if $E\cap (X\times X)=\emptyset$.
The chromatic number $\chi(\Gamma)$ of $\Gamma$ is the least  $k<\omega$ such that $G$ can be partitioned into $k$ independent sets, and
$\infty$ if there is no such set.

\begin{theorem} 
\begin{enumroman}
\item Suppose that $\chi(\Gamma)=\infty$. Then $\C(\Gamma)$ is representable  as a polyadic equality algebra.
\item If $\Gamma$ is infinite and $\chi(\Gamma)<\infty$ then $\Rd_{df}\C(\Gamma)$ is not representable.
\end{enumroman}
\end{theorem}
\begin{demo}{Proof} (i) We have $\Rd_{ca}\C(\Gamma)$ is representable by \cite{HHstrong} proposition 5.2. 
Let $J=\{x\in C: \Delta x\neq n\}$. Then $\C(\Gamma)$ is generated from $J$ using infinite intersections and complementation.
Let $f$ be an isomorphism of of $\Rd_{ca}\C(\Gamma)$ onto a cylindric set algebra with base $U$. Since the ${\sf p}_{ij}$'s  
distribute over arbitrary (unions and) intersections, 
it suffices to  show that $f{\sf p}_{kl}x={\sf p}_{kl}fx$ for all $x\in J$. Let $\mu\in n\setminus \Delta x$.
If $k=\mu$ or $l=\mu$, say $k=\mu$, then using the polyadic axioms we have
$$f{\sf p}_{kl}x=f{\sf }{\sf p}_{kl}{\sf c}_kx=f{\sf s}_l^kx={\sf s}_l^kfx={\sf s}_l^k{\sf c}_kfx= {\sf p}_{kl}fx.$$ 
If $\mu\neq k,l$ then again using the polyadic axioms we get
$$f{\sf p}_{kl}x=f{\sf s}_{\mu}^l{\sf s}_l^k {\sf s}_k^{\mu}{\sf c}_{\mu}x=
{\sf s}_{\mu}^l{\sf s}_l^k {\sf s}_k^{\mu}{\sf c}_{\mu}fx={\sf p}_{kl}f(x)$$

(ii) Note that $\Rd_{ca}\C(\Gamma)$ is generated by $\{x\in C: \Delta x\neq n\}$ using infinite intersections and complementation.
\end{demo}
Recall that an atom structure is strongly representable if the complex algebra over this atom structure is representable \cite{HHstrong}.
We now have:
\begin{theorem}\label{st} Let $t$ be any signature between 
$\Df_n$ and $\PEA_n$. Then the class of strongly representable atom structures of type $t$
is not elementary.
\end{theorem}
\begin{demo}{Proof} \cite{HHstrong} theorem 6.1. 
By a famous theorem of Erdos, for every $k<\omega$, there is a finite graph $G_k$ with $\chi(G_k)>k$ 
and with no cycles of length $<k$.
Let $\Gamma_k$ be the disjoint union of of the $G_l$ for $l>k$. Then $\chi(\Gamma_k)=\infty$. Thus, by the previous theorem 
$\C(\Gamma_k)\in {\bf RPEA_n}$. In fact, being simple, $\C(\Gamma_k)$ is actually a  polyadic set algebra.
Let $\Gamma$ be a non principal ultraproduct $\prod_D \Gamma_k$. So $\Gamma$ has no cycles, and so $\chi(\Gamma)\leq 2$.
It follows, again from the previous theorem,  that $\Rd_{df}\C(\Gamma)$ is not representable.
From $\prod_D\C(\Gamma_k)\cong \C(\prod_D\Gamma_k)$ we are done.
\end{demo}

\begin{corollary}\label{com} Let $\K\in \{\CA, \Df, \PA, \PEA\}$ and $2<n<\omega$. Then the following hold:
\begin{enumarab}

\item There exist two atomic 
algebras in $\K_n$  with the same atom structure, only one of which is 
representable.

\item $\RK_n$ 
is not closed under completions and is not atom-canonical. 

\item There exists a non-representable $\K_n$ with a dense representable
subalgebra.

\item $\RK_n$ 
is not Sahlqvist axiomatizable. 

\item There exists an atomic representable $\K_n$ with no complete representation.

\end{enumarab}
\end{corollary}
\begin{demo}{Proof} We prove it for $\PEA$'s. The rest is the same. Let $H$ be a weakly representable atom structure that is not strongly 
representable. Let $H$ be the  atom structure of $\A$.  Then we have: 
\begin{enumarab}
\item $\Tm H$ and $\Cm H$ have the same atom structure.
$\Tm H$ is representable and $\Cm H$ is not.
\item  $\Cm H$ is the completion of $\Tm H$.
Then  $\Cm(At \RK_n)$ is not contained in $\RK_n$.
Thus $\RK_n$ is not atom-canonical.
\item $\Tm H$ is dense in $\Cm H$.

\item $\RK_n$ is a conjugated variety 
that is not closed under completions, hence by \cite{Venema} it is not Sahlqvist axiomatizable.

\item $\Tm H$ has no complete represention; 
else $\Cm H$ would be representable. (A complete representation \cite{HH} is one that preserves infinitary meets and joins whenever defined).
\end{enumarab}
\end{demo}

\begin{corollary} For each finite $n\geq 3$, there exists a simple countable atomic representable polyadic equality 
algebra of dimension $n$ whose $\Df$ reduct is not completely representable
\end{corollary}
\begin{demo}{Proof} Let $\A$ be a countable atomic representable polyadic algebra, that is not necessarily simple, satisfying that its $Df$ reduct is not 
completely representable. Consider the elements $\{{\sf c}_na: a\in At\A'\}$. Then every simple component $S_a$ of $\A$ can be
obtained by relativizing to ${\sf c}_na$ for an atom $a$. 
Then one of the $S_a$'s should have no complete representation . Else for each atom $a$ $S_a$ has a complete representation $h_a$.
From those one constructs a complete representation for $\A$. 
The domain of the representation will be the disjoint union of the domains of $h_a$, and now represent $\A$ by
$$h(\alpha)=\bigcup_{a\in AtA}\{h(\alpha\cdot {\sf c}_na)\}.$$
\end{demo}  

\begin{theorem} Let $n\in\omega$. Let $K$ be a signature between $\Df_n$ and $\QEA_n$. 
Then the class of completely representable $K$ algebras 
is elementary
if and only if $n\leq 2$, in which case this class coincides with the (elementary) class of atomic representable algebras.
\end{theorem} 
\begin{demo}{Proof} \cite{MS}
\end{demo}

Also we have:

\begin{theorem}\label{com1} Let $n\in\omega$. Let $K$ be a signature between $\Df_n$ and $\QEA_n$. 
Then the class of representable $K$ algebras 
is closed under completions 
if and only if $n\leq 2$.
\end{theorem}
\begin{demo}{Proof} For $\Df$, $\QPA$, $\SC$, the class of representable algebras of dimension $\leq 2$ coincides with the class of algebras
that is axiomatized by  Sahlqvist equations. For $\CA$ and $\QEA$ 
the class of representable algebras of dimension $\leq 2$ is also finitely axiomatizable
by Sahlqvist equations. The case $n>2$ follows from the proof of  corollary 22.
\end{demo}

Some historical remarks are in order. It was proved by Stone in the
1930s that every Boolean algebra $\B$ can be embedded into a complete and atomic
Boolean set algebra, namely the Boolean algebra of the class of all
subsets of the set
of ultrafilters in $\B$. This canonical extension of $\B$ was characterized
algebraically
by Jonsson and Tarski in 1951. They developed a theory of Boolean algebras
with operators (similar in spirit to the theory of groups with
operators), proved
that every Boolean algebra with operators $\B$ can be embedded into a canonical
complete and atomic Boolean algebra with operators $\A$ of the same
similarity type
as $\B$, and established a number of preservation theorems concerning
equations and
universal Horn sentences that are preserved under the passage from $\B$ to $\A$. They
concluded that the canonical extension of every abstract relation
algebra is again a
relation algebra, and similarly for cylindric algebras and other
related structures.
They did not settle the question of whether, e.g., the canonical extension of a
representable relation algebra (a relation algebra that is isomorphic
to a concrete
algebra of binary relations on a set) is again representable, since
the equations
that characterize representable relation algebras are in general not
preserved under the passage to canonical extensions. 
As sketched above any  
axiomatization of representable relation and cylindric algebras, must involve infinitely many non -canonical 
sentences.

However, this problem was
settled in the affirmative
sometime in the 1960s, by Monk (unpublished).
This implies that the class of representable relation algebras is barely canonical.
An analogous result for the class of representable cylindric algebras hold.
MacNeille and Tarski showed in the 1930s that every Boolean algebra has 
another natural complete extension - and in fact it is a minimal complete
extension
that is formed using Dedekind cuts. In 1970, Monk developed the theory
of minimal
complete extensions of Boolean algebras with complete operators in analogy with
the Jonsson-Tarski theory of canonical extensions of Boolean algebras
with operators. In particular, he proved an analogous preservation theorem, and
concluded
that the minimal complete extension of every relation algebra is again
a relation
algebra, and similarly for cylindric algebras and other related
structures. He was
unable to settle the question of whether the minimal complete
extension of a representable relation algebra or a representable finite-dimensional
cylindric algebra
is representable. This question was finally settled negatively for
relation algebras
and finite-dimensional cylindric algebras of dimension $>2$ by Hodkinson in 1997 \cite{Hod97}; his proof used a
back-and-forth game-theoretic argument, and was rather complicated. 
In theorem \ref{complete},  we give a simpler proof of Hodkinson's theorem, and we  extend Hodkinson's negative results to other kinds of
algebras of logic, for instance to quasi-polyadic algebras with and
without equality
(developed by Halmos), to substitution algebras, and to diagonal-free
cylindric algebras. This simplified method of proof also uses a back-and-forth
game-theoretic
argument based on structures defined from certain graphs.
We will basically show that there exists a weakly representable atom structure, that is not strongly representable.
Note that theorem \ref{st} is a strengthening of this result, because, unlike the class of strongly representable atom structures,
the class of weakly reprsentable atom structures
is elementary.

\subsection{ Back to where we started; Monk's result}

We note that the technique in theorem \ref{st} is a dual to Monk's non finite axiomatizability result.
To further elaborate on this, becuase $\RCA_n$ is a variety, an atomic algebra $\A$ will be in $\RCA_n$ iff all equations defining $\RCA_n$ holds in $\A$.
From the point of view of of $\At\A$, each equation corresponds to a certain universal monadic second-order statement, 
where the universal quantifiers are restricted 
to ranging over the set of atoms that lie underneath elements of $\A$. Such a statement fails in $\A$ iff $At\A$ 
be partitioned into finitely many $\A$-definable sets with certain 
`bad' properties. Call this a {\it bad partition}. 
A bad partition of a graph is a finite colouring: a partition of its sets of nodes into finitely many independent sets. 
This idea can be 
used to reprove Monk's non finite axiomatizability result, 
that $\RCA_n$ cannot be finitely axiomatized, by finding a sequence of atom structures, 
each having some sets that form a bad partition, but with minimal number of sets in a bad
partition increasing as we go along the sequence.
This boils down, to finding graphs of finite chromatic numbers $\Gamma_i$, having an ultraproduct
$\Gamma$ with infinite chromatic number. 
So the above construction can be used to prove the famous non-finite axiomatizability results of Monk and Johnson of $\RCA_n$, $\RPEA_n$
and $\RDf_n$ for $2<n<\omega$. Curiously the above problem is a reverse of this. An atom structure is strongly representable iff it 
has no bad partition using any sets at all. So, here, we want to find atom structures, with no bad partitions, 
with an ultraproduct that does have a bad partition.
From a graph we can create an atom structure that is strongly rep iff the graph
has no finite colouring. So the problem is to find a sequence of graphs with no finite colouring, 
with an ultraproduct that does have a finite colouring. We want graphs of infinite chromatic numbers, having an ultraproduct
with finite chromatic number. It is not obvious, a priori, that such graphs actually exist.
And here is where Erdos' methods offer solace. Indeed, graphs like this can be found using the probabilistis methods of Erdos, for those methods
render finite graphs of arbitrarily large chormatic number and girth. 
By taking disjoint union we obtain graphs of infinite chromatic number (no bad partitions) and arbitarly large girth. A non principal 
ultraproduct of these
has no cycles, so has chromatic number 2 (bad partition).
Using these probabilistic techniques of Erd\"os in  \cite{GHV} it is also shown 
that there exist continuum-many canonical equational classes of Boolean algebras 
with operators that are not generated by the complex algebras of any first-order 
definable class of relational structures. And yet again, using a variant of this construction  the authors resolve the vexing 
long-standing question of Fine that baffled logicians for some time, 
by exhibiting a bimodal logic that is valid in its canonical frames, but is not sound and complete for any first-order definable class of 
Kripke frames.

\section{Stronger forms, connections with systems of varieties}

The class of representable cylindric algebras cannot be axiomatized by a set of universal formulas containing finitely many variables \cite{An97},
the class or (representable) cylindric algebras fails to have the amalgamation property \cite{P}, 
and the class of completely representable cylindric algebras
is not elementary \cite{HH97b}. All of those results switch positive when we go to the polyadic paradighm.
There is a finite schema that axiomatizes the class of representable polyadic algebras,
this class has the superamalgamation property and indeed atomic algebras are completely representable, as we proced to show.
A similar result is proved in \cite{complete}. But there the context was countable and therefore it was possible to appeal to the Baire category 
theorem. In our next theorem, the proof is still topological, but we do something different:

\begin{theorem} Let $\alpha$ be an infinite ordinal. Let $\A\in \PA_{\alpha}$ be atomic. Then $\A$ is completely representable.
That is for all non-zero element $a\in \A$ there exists a polyadic set algebra $\B$ and a homomorphism $f:\A\to \B$ such that $f(a)\neq 0$ and
$f(\sum X)=\bigcup_{x\in X}f(x)$ whenever $\sum X$ exists in $A$. 
In particular, the class of completely representable $\PA_{\alpha}$'s is elementary.
\end{theorem}
\begin{demo}{Proof} Let $a\in A$ be non-zero. Let $\mathfrak{m}$ be the local degree of $\A$, $\mathfrak{c}$ its effective cardinality and 
$\mathfrak{n}$ be any cardinal such that $\mathfrak{n}\geq \mathfrak{c}$ 
and $\sum_{s<\mathfrak{m}}\mathfrak{n}^s=\mathfrak{n}$. Then by \cite{DM63} there exists $\B\in \PA_{\mathfrak{n}}$ 
such that $\A\subseteq \Nr_{\alpha}\B$ and $A$ generates $\B$. 
Being a minimal dilation of $\A$, the local degree of $\B$ is the same as that of $\A$, 
in particular each $x\in \B$ admits a support of cardinality $<\mathfrak{m}$.
We have for all $Y\subseteq A$, $\Sg^{\A}Y=\Nr_{\alpha}\Sg^{\B}Y$. Without loss of generality, we assume that $\A=\Nr_{\alpha}\B$.
Hence $\A$ is first order interpretable in $\B$. In particular, any first order sentence (e.g. the one expressing that $\A$ is atomic)
of the language of $\PA_{\alpha}$ translates effectively to a sentence $\hat{\sigma}$ of the language of $\PA_{\beta}$ such that for all 
$\C\in \PA_{\beta}$, we have
$\Nr_{\alpha}\C\models \sigma\longleftrightarrow \C\models \hat{\sigma}$. Since $\A=\Nr_{\alpha}\B$ and $\A$ is atomic, it 
follows that $\B$ is also atomic.
Let $\Gamma\subseteq \alpha$ and $p\in \A$. Then in $\B$ we have 
\begin{equation}\label{tarek1}
\begin{split}
{\sf c}_{(\Gamma)}p=\sum\{{\sf s}_{\bar{\tau}}p: \tau\in {}^{\alpha}\mathfrak{n},\ \  \tau\upharpoonright \alpha\sim\Gamma=Id\}.
\end{split}
\end{equation}
Let $X$ be the set of atoms of $\A$. Since $\A$ is atomic, then  $\sum^{\A} X=1$. Now $\A=\Nr_{\alpha}\B$, then
for all $\tau\in {}^{\alpha}\mathfrak{n}$ we have
\begin{equation}\label{tarek2}
\begin{split}
\sum {\sf s}_{\bar{\tau}}^{\B}X=1.
\end{split}
\end{equation}
Let $X^*$ be the set of principal ultrafilters of $\B$.  
These are isolated points in the Stone topology. So we have $X^*\cap T=\emptyset$ for every
nowhere dense set $T$ (since principal ultrafilters lie outside nowhere dense sets). 
Now  for all $\Gamma\subseteq \alpha$ and all $p\in A$,
\begin{equation}\label{tarek3}
\begin{split}
G_{(\Gamma,p)}=N_{{\sf c}_{(\Gamma)}p}\sim \bigcup N_{s_{\bar{\tau}}p}
\end{split}
\end{equation}
is nowhere dense, and so is
\begin{equation}\label{tarek4}
\begin{split}
G_{X, \tau}=S\sim \bigcup_{x\in X}N_{s_{\bar{\tau}}x}.
\end{split}
\end{equation}
for every $\tau\in {}^{\alpha}\mathfrak{n}$. 
Let $F$ be a principal ultrafilter of $S$ containing $a$. 
This is possible since $\B$ is atomic, so there is an atom $x$ below $a$, just take the 
ultrafilter generated by $x$.
Then $F\in X^*$, so $F\notin G_{\Gamma, p}$, $F\notin G_{X,\tau}$ 
for every $\Gamma\subseteq \alpha$ and $p\in A$
and $\tau\in {}^{\alpha}\mathfrak{n}$.
Now define for $c\in A$
$$f(c)=\{\tau\in {}^{\alpha}\mathfrak{n}: {\sf s}_{\bar{\tau}}^{\B}c\in F\}.$$
Then $f$ is a homomorphism from $\A$ to the full
set algebra $\C$ with unit $^{\alpha}\mathfrak{n}$ such that $f(a)\neq 0$ cf. \cite{super}. 
Furthermore $f$ is a complete representation.
Indeed we have 
$$f(x)=\bigcup\{f(b): b\text { is an atom }\leq x\}.$$
Let $X$ be a subset of $\A$ such that $\sum X$ exists. Then $s\in
f(\sum X)$ iff $s\in f(b)$ for some atom $b\leq \sum X$ iff $s\in
f(b)$ for some atom $b$  some $x$ with $b\leq x\in X,$ iff $s\in
f(x)$ iff $s\in \bigcup f(X).$

\end{demo}

However polyadic algebras are viewed as non satisfactory, in the recursive sense,  because they have uncountably many operations.
A stronger form of the Finitizability problem is to search for varieties that are not only well behaved from the axiomatic point of view
but also has other desirable properties like for instance, the amalgamation property.
In this section, we wish to analyze this dichotomy between the cylindric paradigm and polyadic one 
and try to draw a border line that separates the two paradigms.
So one form of the Finitizability problem, is to find a well behaved variety $V$ of representable algebras that enjoy the positive properties
of both paradigms.

\subsection{Cylindric paradighm}

We start from the cylindric one. One form of capturing the essence of 
the cylindric paradigm is that of {\it systems of varieties definable by schemes}.
Such systems provide a unifying framework for almost all 
cylindric-like algebraic logics existing in the literature. The idea of a system of varieties definable by schemes is simple, 
and indeed it transforms the work on universal (algebraic) logic, 
to the realm of universal algebras.
An example, and indeed the prime source of such systems,  is the system $(\CA_{\alpha}: \alpha \text { an infinite ordinal })$
where $\CA_{\alpha}$ is the variety of cylindric algebras of dimension $\alpha$. 
$\A\in \CA_{\alpha}$ if $\A=(\B, {\sf c}_i, {\sf d}_{ij})_{i,j\in \alpha}$
 where $\B$ is a Boolean algebra and the ${\sf c}_i$'s are unary operations of cylindrifications and ${\sf d}_{ij}$'s are diagonal elements.
$\CA_{\alpha}$ is defined by a finite schema of equation, one such schema is $\sigma:={\sf c}_i{\sf c}_jx={\sf c}_j{\sf c}_ix, i,j\in \alpha$, 
reflecting the fact 
that cylindrifications commute. Now assume that $\eta\in {}^{\omega}\alpha$ is one to one, and assume that $i,j\in \omega$,  
then we define $\eta(\sigma)={\sf c}_{\eta(i)}{\sf c}_{\eta(j)}x={\sf c}_{\eta(j)}{\sf c}_{\eta(i)}$.
Similarly if $E$ is a set of equations then $\eta(E)$ is defined as $\{\eta(\sigma): \sigma\in E\}$.
Now consider the system of varieties $(\CA_{\alpha}: \alpha\in Ord)$. Then this system is definable by a finite schema of equations meaning
the following. There is a set $E$ of equations in the language of $\CA_{\omega}$, and indeed a finite one, such that
$$\CA_{\alpha}=Mod\{\eta(E): \eta\in {}^{\omega}\alpha \text { is one to one }\},$$
and the point is that the same $E$ works for all $\alpha.$
The exact definition of systems of varieties is given in \cite{HMT2}  definition 5.6.12, by abstracting away from $\CA$'s:

\begin{definition}
\begin{enumroman}
\item A type schema is a quadruple $t=(T, \delta, \rho,c)$ such that
$T$ is a set, $\delta$ maps $T$ into $\omega$, $c\in T$, and $\delta c=\rho c=1$.
\item A type schema as in (i) defines a similarity type $t_{\alpha}$ for each $\alpha$ as follows. The domain $T_{\alpha}$ of $t_{\alpha}$ is
$$T_{\alpha}=\{(f, k_0,\ldots k_{\delta f-1}): f\in T, k\in {}^{\delta f}\alpha\}.$$
For each $(f, k_0,\ldots k_{\delta f-1})\in T_{\alpha}$ we set $t_{\alpha}(f, k_0\ldots k_{\delta f-1})=\rho f$.
\item A system $(\K_{\alpha}: \alpha\geq \omega)$ of classes 
of algebras is of type schema $t$ if for each $\alpha\geq \omega$ $\K_{\alpha}$ is a class of algebras of type 
$t_{\alpha}$.
\end{enumroman}
\end{definition}

\begin{definition} Let $t$ be  type schema. 
\begin{enumroman}
\item With each $\alpha$ we associate a language $L_{\alpha}^t$ of type $t_{\alpha}$: for each $f\in T$ and 
$k\in {}^{\delta f}\alpha,$ we have a function symbol $f_{k0,\ldots k(\delta f-1)}$ of rank $\rho f$ 
\item Let $\eta\in {}^{\beta}\alpha$. We associate with each term $\tau$ of $L_{\beta}^t$ a term $\eta^+\tau$ of $L_{\alpha}^t$.
For each $\kappa,\omega, \eta^+ v_k=v_k$. if $f\in T, k\in {}^{\delta f}\alpha$, and $\sigma_1\ldots \sigma_{\rho f-1}$ are terms of $L_{\beta}^t$, then
$$\eta^+f_{k(0),\ldots k(\delta f-1)}\sigma_0\ldots \sigma_{\rho f-1}=f_{\eta(k(0)),\ldots \eta(k(\delta f-1))}\eta^+\sigma_0\ldots \eta^+\sigma_{\rho f-1}.$$
Then we associate with each equation $\sigma=\tau$ of $L_{\beta}^t$ the equation $\eta^+\sigma=\eta^+\tau$ of $L_{\alpha}^t$.
\item Let $E$ be a set of equations of $L_{\omega}^t$. a system $(\K_{\alpha}:\alpha\geq \omega)$ of type schema $t$ 
is definable by $E$ if for every $\alpha\geq \omega$, we have
$$\K_{\alpha}=Mod\{\eta^+e: e\in E: \eta\in {}^{\omega}\alpha, \eta \text { one to one }\}.$$
\end{enumroman}
\end{definition}
Examples of such systems include Pinter's substitution algebras $\SC$'s, and quasi- $\PA$'s ($\QPA$) and quasi- $\PEA$'s $(\QPEA)$. 
Such algebras are all defined in \cite{HMT2}. 
For such systems, general notions like locally finite algebras, dimension complemented algebras
and neat reducts can be formulated.

\begin{definition} Let $(\K_{\alpha}:\alpha\geq \omega)$ be a type of schema $t$. For $\A\in \K_{\alpha}$, 
$a\in A,$
let $\Delta a=\{i\in \alpha: {\sf c}_ia\neq a\}.$
\begin{enumroman}
\item  $\A\in \K_{\alpha}$ is locally finite if $\Delta x$ is finite for every $x\in A.$ $\Kf_{\alpha}$ is the class of locally finite $\A\in \K_{\alpha}$.
\item $\A\in \K_{\alpha}$ is dimension complemented if for every finite $X\subseteq A$,
$\alpha\sim \bigcup_{x\in X}\Delta x$ is infinite. $\Kc_{\alpha}$ is the class of dimension complemented $\A\in \K_{\alpha}$.
\item We assume that for any $f\in T$, and $x_0,\ldots x_{\rho f-1}\in \A\in \K_{\beta}$, 
if $\alpha\leq \beta$, and $\Delta x_i\leq \alpha$ and $k\in {}^{\delta f}\alpha$, 
we have $\Delta [f_{k0,\ldots, k(\delta f-1)}(x_0\ldots x_{\rho f-1})] \subseteq \alpha$.   
Now Suppose $\omega\leq \alpha\leq \beta$ and $\A\in \K_{\beta}$.
Then $\Nr_{\alpha}\B$ is the subalgebra of $\Rd_{\alpha}\B$ with universe $\{x\in B: \Delta x\subseteq \alpha\}$.
This is well defined. 
$\Nr_{\alpha}\K_{\alpha+\omega}=\{\Nr_n\A: \A\in \K_{\alpha+\omega}\}$ and $\Kn_{\alpha}$ is the class obtained by forming subalgebras
of $\Nr_{\alpha}\K_{\alpha+\omega}$, in short $\Kn_{\alpha}=S\Nr_{\alpha}\K_{\alpha+\omega}$. 
\end{enumroman}

\end{definition}
The class $\Nr_{\alpha}\K_{\beta}$ is the class of neat $\alpha$ reducts.
The class of neat reducts is important for other other algebraic logics, as well \cite{Bulletin}. 
The following is proved by Andr\'eka and N\'emeti:

\begin{theorem}  Let $(\K_{\alpha}:\alpha\geq \omega)$ be definable by a schema, and let $\alpha\geq \omega$. Then
$\Kf_{\alpha}\subseteq \Kc_{\alpha}\subseteq \Kn_{\alpha}=SUp\Kn_{\alpha}$.
\end{theorem}
\begin{demo}{Proof} \cite{HMT2} Thm 5.6.15
\end{demo}

We note that in most algebraic logics the class $\Kn_{\alpha}$ coincide with the class of representable 
algebras and so cannot be axiomatized by a finite schema, and indeed 
a certain complexity is inevitable in any such axiomatization, 
witness the results of Andr\'eka in \cite{An97}.
Andreka's results generalize to quasi-polyadic equality algebras. 
We follow the notation of \cite{ST} in treating quasipolyadic equality algebras.
In particular, we view $\B\in \QPEA_{\alpha}$ to be of the form $(\A, {\sf p}_{ij})_{i,j<\alpha}$, where $\A\in \CA_{\alpha}$
and ${\sf p}_{ij}$'s are substitutions corresponding only to transpositions.
The class of neat reducts is defined in complete analogy to the $\CA$ case.
We mention two deep results proved recently by the author:

\begin{theorem} Let $n\geq 3$. Let $l\geq 2$. Then the class $S\Nr_n\QPEA_{n+l}$ is not finitely axiomatizable by a set of quantifier free
formulas containing finitey many variables.
\end{theorem}
\begin{demo}{Sketch of Proof} Let $m>2^k$ $m<\omega$ 
and let $(U_i:i<n)$ be a system of disjoint sets each of cardinality $m$ such that $U_0=\{0,\ldots m-1\}$.
Let  $$U=\bigcup\{U_i:i<n\}.$$
$$q\in \times_{i<\alpha}U_i=\{s\in {}^{\alpha}U:  s_i\in U_i\text { for all } i<\alpha\},$$
and let $$R=\{z\in \times _{i<\alpha}U_i:|\{i<\alpha: z_i\neq q_i\}|<\alpha\}.$$
Let
$$F=\{s\in {}^nU^{(q)}: s_0, s_1\in U_0, s_1=s_0+1(mod(m)\}.$$
Let $\A'$ be the polyadic equality algebra generated by $R$ and $F$. 
Let $\A$ be the algebra we obtain from $\A$ by splitting ${\sf s}_{\tau}R$ into ${\sf s}_{\tau}R_j$ 
for $j\leq m$ and $\tau\in Per$. That is $\A$ is an algebra
such that
\begin{enumarab}
\item $\A'\subseteq \A$, and the Boolean part of $\A$ is a Boolean algebra,
\item ${\sf s}_{\tau}R_j$, are pairwise distinct atoms of $\A$ for each $\tau\in Tr$ and $j\leq m$ and 
${\sf c}_i{\sf s}_{\tau}R_j={\sf c}_i{\sf s}_{\tau}R$ for all $i<n$ and all $\tau\in Tr,$
\item each element of $\A$ is a join of element of $\A'$ and of some ${\sf s}_{\tau}R_j$,'s
\item ${\sf c}_i$ distributes over joins,
\item The ${\sf s}_{\tau}$'s are Boolean endomorphisms such 
that ${\sf s}_{\tau}{\sf s}_{\sigma}a={\sf s}_{\tau\circ \sigma}a$.
\end{enumarab}
Then \cite{An97} p. 160-161 $\Rd_{ca}\A_k$ is not representable
\footnote{Strictly speaking, Andreka proves this only for the finite dimensional case, but the proof works for the infinite dimensional case 
as well.}, and by the technique in \cite{p} all $k$ generated subalgebras are representable.
The idea is that a representation of even the $\CA$ reduct of $\A$ would force that $|U_0|\geq m+1$, which is not the case.
Now let $(e_m: m<\omega)$ be the second set of equations in Remark 2 p.162 of \cite{An97}. Then
by \cite{An97} p.163-166
$S\Nr_n\QEA_{n+2}\models e_m.$  Finally it is not the case that $\A_k\models e_m$
where $\A_k $ is the algebra for which $|U_0|=m.$
The conclusion now follows by the argument of Andreka in \cite{An97} p.163.
\end{demo}

It can be easily shown \cite{neet} extending  results of Hirsch and Hodkinson to the quasipolyadic case, 
that for $3\leq m<\omega$, each of the inclusions 
$\QEA_m=S\Nr_m\QEA_m\supset S\Nr_m\QEA_{m+1}\ldots$ is strict and all but the first inclusion is not finitely axiomatized.
Furthermore $\RQEA_m$ is not finitely axiomatizable over $S\Nr_m\QEA_{m+k}$ for all $k\in \omega$.
Indeed, following Hirsch and Hodkinson \cite{HHbook}, we define 
relation algebras $\A(n,r)$ having two parameters $n$ and $r$ with $3\leq n<\omega$ and $r<\omega$.
Let $\Psi$ satisfy $n,r\leq \Psi<\omega$. We specify the stom structure of $\A(n,r)$.
\begin{itemize}
\item The atoms of $\A(n,r)$ are $id$ and $a^k(i,j)$ for each $i<n-1$, $j<r$ and $k<\psi$.
\item All atoms are self converse.
\item We can list te forbidden triples $(a,b,c)$ of atoms of $\A(n,r)$- those such that
$a.(b;c)=0$. Those triples that are not forbidden are the consistent ones. This defines composition: for $x,y\in A(n,r)$ we have
$$x;y=\{a\in At(\A(n,r)); \exists b,c\in At\A: b\leq x, c\leq y, (a,b,c) \text { is consistent }\}$$
Now all permutations of the triple $(Id, s,t)$ will be inconsistent unless $t=s$.
Also, all permutations of the following triples are inconsistent:
$$(a^k(i,j), a^{k'}(i,j), a^{k''}(i,j')),$$
if $j\leq j'<r$ and $i<n-1$ and $k,k', k''<\Psi$.
All other triples are consistent.
\end{itemize}

\begin{theorem} We have $\A(n,r)\in S\Ra \PEA_n$. 
\end{theorem}
\begin{demo}{Proof} 
The set $H_n^{n+1}(\A(n,r), \Lambda)$ aff all $(n+1)$ wide $n$ 
dimensional $\Lambda$ hypernetworks over $\A(n,r)$ is an $n+1$ wide $n$ 
dimensional {\it symmetric} $\Lambda$ hyperbasis. 
$H$ is symmetic, if whenever $N\in H$ and $\sigma:m\to m$, then $N\circ\sigma\in H$.
Hence $\A(n,r)\in S\Ra \PEA_n$.
In \cite{HHbook} it is proved that $\A(n,r)\in S\Ra\CA_n$, but observing that $H_n^{n+1}(\A(n,r),\Lambda)$ is symmetric
gives our stronger result.  
\end{demo}
Next we define certain polyadic equality algebras based on the relation algebras we defined:
Let $3\leq m\leq n$,  
$$\C_r=\Ca(H_m^{n+1}(\A(n,r),  \omega)).$$
Note that $\C_r$ depends on $n$ and $m$, but we omit reference to those not to clutter notation. 
Since $H_m^{n+1}(\A(n,r), \omega))$ is symmetric this defines a polyadic algebra of dimension $m$.
We can also prove that

\begin{theorem} For any $r$ and $3\leq m\leq n<\omega$, we 
have $\C_r\in \Nr_m\PEA_n$. 
\end{theorem}
\begin{demo}{Proof} $H_n^{n+1}(\A(n,r), \omega)$ is a wide $n$ dimensional $\omega$ symmetric hyperbases, so $\Ca H\in \PEA_n.$
But $H_m^{n+1}(\A(n,r),\omega)=H|_m^{n+1}$.
Thus
$$\C_r=\Ca(H_m^{n+1}(\A(n,r), \omega))=\Ca(H|_m^{n+1})\cong \Nr_m\Ca H$$
\end{demo}
\begin{theorem} $\Rd_{ca}\C_r\notin S\Nr_m\CA_{n+1}$, and for any non principal ultrafilter on $\omega$,
if $3\leq m<n<\omega$, we have $\prod \C_r/F\in {\bf RPEA}_m$.
\end{theorem}
\begin{demo}{Proof} The first part is like the $\CA$ case proved in Corollary 15.10 in \cite{HHbook}.
The second part is identical to exercise 2 on p. 484 of \cite{HHbook}
\end{demo}
We now have

\begin{theorem}
\begin{enumarab}
\item Each of the inclusions $S\Ra\PEA_3\supset S\Ra\PEA_4\supset \ldots$ is strict
\item All but the first inclusion above can be finitely axiomatized
\item For $3\leq m<\omega$, each of the inclusions $\PEA_m=S\Nr_m\PEA_m\supset S\Nr_m\PEA_{m+1}\ldots$ is strict
\item All but the first inclusion is not finitely axiomatized
\end{enumarab}
\end{theorem}

Now we prove a complexity result for $\RPEA_n$ when $n$ is finite $n\geq 3.$

\begin{theorem}\label{c} Let $n\geq 3$. 
Let $\Sigma$ be a set of equations axiomatizing $\RPEA_n$. Let $l<n$ $k<n$, $k'<\omega$ be natural numbers. 
Then $\Sigma$ contains
infinitely equations in which $-$ occurs, one of $+$ or $\cdot$ ocurs  a diagonal or a permutation with index $l$ occurs, more
than $k$ cylindrifications and more than $k'$ variables occur

\end{theorem} 
\begin{demo}{Sketch of Proof} This is proved in \cite{p}. Here we give an outline of the proof that 
renders the gist of the techniques of Andreka in \cite{An97} but only for finite $n$. We shall construct
a non-representable algebra $\A$, such that its $k$ generated subalgebras are representable, and
  \begin{enumarab}
\item The complementation free reduct 
$\A^{-}$ of $\A$ is a homomorphic image of a subalgebra $\C$ of the complemention free reduct of $\P^{-}$ of a 
$\P$ in $\RPEA_n$. In fact this $\P$ is a set algebra with infinite base.
\item $\A^{-} \notin \RPEA_n^{-}$, the complementation free reduct of $\RPEA_n$.
\item $\A$ can be represented as a polyadic set algebra such that every operation except for $\cup$ and $\cap$ are the natural ones.
\item  There is an infinite set $W$, such that for all $\mu<n$, there is an embedding $h:\A\to (\B(^nW), {\sf c}_i, {\sf p}_{ij}, {\sf d}_{ij})_{i,j<n}$ 
such that $h$ is a homomorphism preserving all operations except for ${\sf c}_{\mu}.$
\item There is an infinite set $W$, such that there is an embedding $h:\A\to (\B(^nW), {\sf c}_i, {\sf p}_{ij}, {\sf d}_{ij})_{i,j<n}$ 
such that $h$ is a homomorphism preserving all opeartions except for ${\sf d}_{ij}, {\sf p}_{ij},$ $i\neq j$, with $l\in \{i,j\}$.
\end{enumarab}
This will prove the theorem as indicated in \cite{An97} p.195.
Now we proceed with the construction. Let
$2^{{k.n!+1}}\leq K(n-1)$ and let $m=K(n-1)$.
Let $\{U_i:i\leq n\}$ be a system of disjoint sets such that $|U_0|=m$ and $|U_i|\geq \omega$ for $0<i\leq n$.
Let $f:U_0\to U_0$ be a bijection such that the orbits of $f$ have cardinality $K$.
Let 
$$U=\bigcup\{U_i: i\leq n\}$$
$$R=\prod_{i<n} U_i$$
$$F=\{s\in {}^nU: s_0, s_1\in U_0 \text{ and } s_1=f(s_0)\}.$$
Let $\A'$ be the subalgebra of $\wp(^nU)$ generated by $R$ and $F$, and let $\A$ be the algebra obtained 
from $\A'$ by splitting into $m+1$ (distinct) atoms as defined above. 
Now it turns out that $\Rd_{ca}\A'$ is not representable, for again such a  representation
would force that $|U_0|\geq m+1$ which is not the case. 
To show that various reducts of $\A'$ are representable, we first transform $U_0$ into a set $W_0$ such that
$|W_0|=m+1$. In more detail, let $u\notin U$, let $(W_i:i\leq n)$ be such that 
$W_0=U_0\cap \{u\}$, $|U_0|=m$, $W_i=U_i$ for $0<i\leq n$ and let
$W=\bigcup\{W_i:I\leq n\}=U\cup \{w\}$.
Let $R'=\prod_{i<n}W_i$ and let $R_j', j\leq m$ be a partition of $R'$ such that
${\sf c}_i{\sf s}_{\tau}R_j'={\sf c}_i{\sf s}_{\tau}R'$ for all $i<n$, $j\leq m$. This partition, now, exists by $|W_i|\geq m+1$ for all $i\leq n$
and \cite{An97} lemma 2.
Let $f:U_0\to U_0$ be a bijection such that a all orbits of $f$ have cardinality $K$. Such an $f$ clearly exists. Extend this 
permutation to a permutation of $W$
such that $f$ permutes $W_0$, all orbits of $f\upharpoonright W_0$ are of size $K$, and $f$ is the identity on $W\sim W_0$.  
We denote this extension by $f$ as well. Now we follow \cite{An97} p. 199.
Let $e$ be the smallest equivalence relation containing $f$. Fix some $u\in U_0$, $w\in W\sim U$ and let
$w/e=\{v: vew\}$, $u/e=\{v: veu\}$. Let $\delta:w/e\to u/e$ be such that
$\delta$ preserve $f$ and let
$$\sigma=\delta\cup Id\upharpoonright U_0, \ \ \pi=\delta\cup \delta^{-1}\cup Id\upharpoonright (U_0\sim u/e).$$
Then $\sigma:W_0\to U_0$ and $\pi:W_0\to W_0$. Define $s(w|u)=\sigma\circ s$. Let
$R'=\prod_{i<n}W_i$ and let $R_j: j\leq n$ be a partition of $R'$ such that ${\sf c}_i{\sf s}_{\tau}R_j'={\sf c}_i{\sf s}_{\tau}R'$ forall $i<n$ $j\leq m$
and $\tau\in Tr$.
Let $G$ be the set of all permutations of $U$ that leave $R$ and $F$ fixed.
Let $$D=\{s\in {}^nW: u/e\cap Rng(s)\neq \emptyset, w/e\cap Rng(s)\neq \emptyset\}$$
$$B'=\{x\subseteq {}^nW: x=\pi x, (x\cap {}^nU)=G(x\cap {}^nU) \text { and } (\forall s\in x\cap D(s(w|u)\in x\}$$
So far we are following \cite{An97} verbatim, with the sole exception of allowing substitutions in the splitting.
Now, we allow substitutions also in $B$,  we modify the definition of $B$ as follows
$$B=\{\bigcup\{{\sf s}_{\tau}R_j':\tau\in H, j\in J\}\cup {\sf s}_{\sigma}x: J\subseteq m+1, \sigma\in Tr, H\subseteq Tr, 
x\in B'\}.$$
It can be checked that (this modified) $B$ is closed under the operations of 
$$\wp(^nW), \cup,\cap, \emptyset, {}^nW, {\sf c}_i, {\sf p}_{ij}, {\sf d}_{ij})_{i,j<n}.$$
Let $\B=(B,\cup,\cap,\emptyset ,^nW, {\sf c}_i, {\sf p}_{ij}, {\sf d}_{ij})_{i,j<n}$. Then $\B$ is a subalgebra of the complementation free reduct
$\P^-$ of $\P=(\wp(^nW), {\sf c}_i, {\sf p}_{ij}, {\sf d}_{ij})$.
Now we show that the complementation free reduct $\A^-$ of $\A$ is embeddable into a homomorphic image of $\B$. 
We first define an algebra $\C$ and a homomorphism 
of $\B$ into $\C$. Let
$$V={}^{n}W\sim D$$
$$C=\wp(V),\ \ {\sf c}_i^{\C}x={\sf c}_ix\cap V, {\sf p}_{ij}^{\C}x={\sf p}_{ij}x\cap V, {\sf d}_{ij}^{\C}={\sf d}_{ij}\cap V.$$
$$\C=(C, \cup, \cap, \emptyset, V, {\sf c}_i^{\C}, {\sf p}_{ij}^{\C}, {\sf d}_{ij}^{\C})_{i,j<n}.$$
Let $g:\B\to \C$ be defined by $g(x)=x\cap V$.
Then we need to check that $g$ is a homomorphsim. This is done by Andreka for all the operations except substitutions, so we need to check those.
Let $i,j<n$ and $x\in \B$. We want to show that $g({\sf p}_{ij}^{\B}x)={\sf p}_{ij}^{\C} g(x)$, i.e. that
$({\sf p}_{ij}^{W}x)\cap V={\sf p}_{ij}^{W}(x\cap V)\cap V$. 
But this follows from definitions and from the fact that if $s\in V$, then $s\circ [i,j]\in V$.
Let $M=\{x\cap V: x\in B\}$ and $\M=(M,\cup,\cap,\emptyset ,V, {\sf c}_i^{\C}, {\sf p}_{ij}^{\C}, {\sf d}_{ij}^{\C})_{i,j<n}$
Then $\M$ is a homomorphic image of $\B$.

We now show that $\A^{-}$ is embeddable into $\M$.
Now any element of $A$ is of the form $\sum\{{\sf s}_{\tau}R_j: \tau\in H\}+a$ where $H\subseteq Tr$, $J\subseteq m+1$ 
$a\in \A'$ and $a\cap s_{\tau}R=\emptyset$, because 
$\A$ is obtained from $\A'$ by splittng $R$ into $m+1$ 
parts $R_j$, $j\leq m$. Call an element $a\in \A$ normal, if there is a 
single $\tau$ such that $a=\sum\{{\sf s}_{\tau}R_j: j\in J\subseteq m+1\}.$
Then for any $a\in A$, we have $a=\sum a_i$ for some normal $a_i$'s.
Let $\pi=[u,v]$. We first define $h$ on the normal $a$'s by
$$h(\sum\{{\sf s}_{\tau}R_j:j\in J\}+a)=\bigcup\{{\sf s}_{\tau}R_j': j\in J\}\cup a\cup \pi a.$$
Now $h$ is well defined, one to one and $h(a)\in M$ for all normal $a\in A$ by the reasoning of Andreka\cite{An97}p.197.
Then for any element $a\in A$, with $a=\sum a_i$ define $\bar{h}(a)=\sum h(a_i).$
Then $\bar{h}$ is as required.

Now $\A^{-}\notin \RPEA_n^{-}$ is the same as the proof of Andreka \cite{An97} p.199.

Now we represent the $\cup$ and $\cap$ free reduct of $\A$.
We define a mapping
$h:\A\to \wp(^nU)$. We first define $f$ on the normal elements.
Let $a$ be normal, $a=\sum\{{\sf s}_{\tau}R_i: i\in J\}+a'$. Let $z\in {\sf s}_{\tau}R\sim {\sf s}_{\tau}(R_0'\cup R_1')$ be fixed and let
$\eta:\{J\subseteq m+1, 0\in J\}\to \{G: G\subseteq {\sf s}_{\tau}R\sim ({\sf s}_{\tau}(R_0\cup R_1): z\in G\}$ be an arbitrary injection.
$$h(\sum \{{\sf s}_{\tau}R_j:j\in J\}={\sf s}_{\tau}R_0'\cup \eta(J), \text { if } 0\in J, J\neq m+1$$
$$h(\sum \{{\sf s}_{\tau}R_j:j\in J\}=R\sim (R_0'\cup \mu(m+1)\sim J)), \text { if } 0\notin J, J\neq \emptyset$$
$$h(\sum \{{\sf s}_{\tau}R_j:j\in J\}=0, \text { if } J=\emptyset$$
$$h(\sum \{{\sf s}_{\tau}R_j:j\in J\}=R, \text { if } J=m+1$$ 

Now define $$h(a)=(a\sim {\sf s}_{\tau}R)\cup h(a\cap {\sf s}_{\tau}R).$$
Then define $\bar{h}$ by extending $h$ as above.
Then it is easy to check that $\bar{h}$ is as required.

Now we show that $\A$ becomes representable if we drop any of the cylindrifications. 
We shall use the following fact that is easy to check.
Recall that $\A$ was obtained from $\A'$ by splitting into $m+1$ atoms. 
Now assume that $h:\A'\to (\B(^{n}U), {\sf c}_i, {\sf p}_{ij}, {\sf d}_{ij})_{i,j<n}$ is a Boolean embedding
and $h(R)=\prod_{i<n}U_i$ such that $(U_i:i<n)$ is a system of disjoint sets each having cardinality 
$\geq m+1$. Then $h$ can be extended to 
$\bar{h}:\A\to (\B(^{n}U), {\sf c}_i, {\sf p}_{ij}, {\sf d}_{ij})_{i,j<n}$ such that $\bar{h}$ preserves the same operations that 
$h$ preserves.

Let $\mu<n$. Let $U_i, W_i: i<n$ be as above. Let $f:U_0\to U_0$ be a bijection such that all orbits of $f$ have cardinality $K$.
Let $U$, $R$ and $F$ and $\A'$ as above. Recall that $\A$ is the algebra obtained from $\A'$ by splitting $R$ into $m+1$ parts.
Extend $f$ 
permutation to a permutation of $W$
such that $f$ permutes $W_0$, all orbits of $f\upharpoonright W_0$ are of size $K$, and $f$ is the identity on $W\sim W_0$.  
Let $e$ denote the equivalence relation on $W_0$ with blocks the orbits of $f$.
Let $S$ be a binary relation on the blocks of $e$ such that $S$ contains the identity relation and each block is in relation with exactly $n-1$ blocks.
Then $S$ can be viewed as a binary relation on $W_0$ satisfying certain properties \cite{An97} p. 180. On p. 180 of \cite{An97}, 
an equivalence
relation $\equiv$ is defined on sequences of different lengths, so that
$$s\equiv z\implies (s\in x\text { iff } z\in x), x\in A'.$$
Proceeding like the proof of Andreka on p.180-181, \cite{An97}, we define a function $g:{}^nW\to {}^nU$.
Let $s\in {}^nW$. Let $\Omega=\{i<n: s_i\in W_0\}$. Let $I=n\sim \{\mu\}$.
Assume that $\mu\notin \Omega$, i.e $\Omega\subseteq I$. let $s'\in {}^{\Omega}U_0$ such that 
$s'\equiv s\upharpoonright \Omega$. Such $s'$ exists by $|\Omega|<n.$
For function $f,g$, let
$$f[\Omega|g]=f\upharpoonright (Dom f\sim \Omega)\cup f'\upharpoonright G\Omega$$
Define
$$g(s)=s[\Omega|s'].$$ 
Now assume that $\mu\in \Omega$. Let $\Omega'=\{i\in \Omega: s_iSs_{\mu}\}$ and $\Omega''=\Omega\sim \Omega'$.
let
$s'\in {}^{\Omega'}U_0$ be such that
$s'\equiv s\upharpoonright\Omega'$ and let $s''\in {}^{\Omega''}(U_n\sim Ranges)$ be such that
$ker(s')=ker(s\upharpoonright \Omega'')$.
Then set
$$g(s)=s[\Omega'|s'][\Omega''|s''].$$ 
Proved by Andreka to preserve all operations except for substitutions, we need to check that $$h(x)=\{s\in {}^{n}W: g(s)\in x\}$$
defined on $\A'$ preserves substitutions, too.
But this follows from the simple observation that
$$s[\Omega|s']\circ [i,j]=s\circ [i,j][\Omega|s']$$
and  
$$s[\Omega'|s'][\Omega''|s'']\circ [i,j]=s\circ [i,j][\Omega'|s'][\Omega|s''].$$
Since $\A$ is obtained by splitting $\A'$ then $h$ can be extended to $\A$
and $h$ is as desired.

Finally, we can assume that $l=0$, see \cite{An97}p.176-177. 
We show that there is an embedding $h:\A\to (\B(^nW), {\sf c}_i, {\sf p}_{ij}, {\sf d}_{ij})_{i,j<n}$ 
such that $h$ is a homomorphism w.r.t all operations of $\A$ except 
${\sf d}_{0i}, {\sf d}_{i0}, {\sf p}_{i0}, {\sf p}_{0i}.$ The proof is like that of Claim 6 in \cite{An97} p.176.
 
Let $W$ and $U$ as above. Define $t,r:W\to U$ as on p. 176.
Define for $1\leq i\leq n$, $t_0:W\to U$, by $t_0(x)=t(x)$ 
and for $i>0$, $t_i(x)=r(x)$.  Set $g(s)_i=t_i(s_i)$. Then define $h$ for $x\in A'$ by
$$h(x)=\{s\in {}^nW: g(s)\in x\}.$$
Then we leave it to the reader to check the required.
Then $h$ extends to a mapping $\bar{h}$ on $\A$ with the required properties.
That is $$\bar{h}:\A\to (\B(^nW), {\sf c}_i, {\it p}_{ij}, {\sf d}_{ij})_{i,j<n}$$ 
such that $h$ is a homomorphism preserving all opeartions except 
for ${\sf d}_{ij}, {\sf p}_{ij},$ $i\neq j$, with $0\in \{i,j\}.$

The difficult part in this result, is to show that the $k$ generated subalgebras of 
$\A$ are representable. Let $G$ be given such that $|G|\leq k$. The idea is to use $G$ and define a ``small" subalgebra of $\A$ that contains $G$
and is representable. Let $T=\{[i,j] : i<j<n\}$. 
Define $R_i\equiv R_j$ iff
$$\forall g\in G\forall \tau\in T({\sf s}_{\tau}R_i\leq g\Longleftrightarrow {\sf s}_{\tau}R_j\leq g)$$
Then $\equiv$ is an equivalence relation on $\{R_j: j\leq m\}$ which has $\leq 2^{k.n!}$ 
blocks by $|G|\leq k$
and $|T|\leq n!$. 
Let $p$ denote the number of blocks of $\equiv$, that is $p=|R_j/\equiv:j\leq m\}|\leq 2^k\leq m$.
Let
$$B=\{a\in A_k: (\forall i,j\leq m)(\forall \tau\in T)(R_i\equiv R_j\text { and }{\sf s}_{\tau}R_i\leq a\implies {\sf s}_{\tau}R_j\leq a\}.$$
We show that $B$ is closed under the operations of $\A$. Let $i<l<n$
Clearly $B$ is closed under the Boolean operations.
${\sf d}_{il}\in \B$ since ${\sf s}_{\tau}R_j\nleq {\sf d}_{il}$ for all $j\leq m$ and $\tau\in T$.
Also $A'\subseteq B$ since ${\sf s}_{\tau}R$ is an atom of $\A'$ and ${\sf c}_ia\in A'$ for all $a\in A$.
Thus ${\sf c}_ib\in B$ for all $b\in B$.
Assume that $a\in B$ and let $\tau\in T$. Suppose that $R_i\equiv R_j$ and ${\sf s}_{\sigma}R_i\leq {\sf s}_{\tau}a$. Then 
${\sf s}_{\tau}{\sf s}_{\sigma}R_i\leq a$, so ${\sf s}_{\tau\circ \sigma}R_i\leq a$. Since $a\in \B$ we get
that ${\sf s}_{\tau\circ \sigma}R_j={\sf s}_{\tau}{\sf s}_{\sigma}R_j\leq a$, and so ${\sf s}_{\sigma}R_j\leq{\sf s}_{\tau}a$.
Thus $\B$ is also closed under substitutions.
Let $\B\subseteq \A$ be the subalgebra of $\A$ with universe $B$. Since $G\subseteq B$
it suffices to show that $\B\in \RQEA_{\alpha}$
Let $(y_j;j<p\}=\{\sum(R_j/\equiv):j\leq m\}$. 
Then $\{y_j:j<p\}$ is a partition of $R$ in $\B,$ ${\sf c}_iy_j={\sf c}_iR$ for all $j<p$ and 
$i<\alpha$ and every element of $\B$ is a join of some element of $\A'$ and of finitely many of 
${\sf s}_{\tau}y_j$'s.
We now split $R$ into $m$ ``real"  atoms using Andr\'eka's method \cite{An97} lemma 2 p. 167.
Let  $\{R_0'', \ldots  R_{m-1}''\}$ be a partition of $R$ such that ${\sf c}_iR_j''=
{\sf c}_iR$ for all $i<\alpha$.
Let $\A''$ be the subalgebra of 
$\langle \B(^{\alpha}U), {\sf c}_i, {\sf d}_{ij}, {\sf s}_{\tau}\rangle_{i,j<\alpha, \tau\in Per}$
generated by $R_0'',\ldots R_{m-1}''.$
Let 
$$\R=\{{\sf s}_{\sigma}R_j'':\sigma\in Per, j<m\}.$$
Let $$H=\{a+\sum X: a\in A', X\subseteq_{\alpha}\R\}.$$
Clearly $H\subseteq A''$ and $H$ is closed under the boolean operations. Also
because Tranformations considered are bijections we have
$${\sf c}_i{\sf s}_{\sigma}R_j={\sf c}_i{\sf s}_{\sigma}R\text { for all $j<m$ and }\sigma\in Per$$
Thus $H$ is closed under ${\sf c}_i.$ Also $H$ is closed under substitutions
Finally ${\sf d}_{ij}\in A'\subseteq H.$
We have proved that $H=A''$. This implies that every element of $\R$ is an atom of $\A''$.
We now show that $\B$ is embeddable in $\A''$, and hence will be representable.
Define for all $j<p-1$,
$$R_j'=R_j'',$$
and
$$R_{p-1}'=\bigcup\{R_j'': p-1\leq j<m\}.$$
Then define
$h(a+\sum {\sf s}_{\tau}y_j)=a+\sum\{{\sf s}_{\tau}R_j': j<p\}$
That is if $b=a+\sum {\sf s}_{\tau}y_j$
Then
$$h(b)=(b-\sum {\sf s}_{\tau}y_j)\cup\bigcup\{{\sf s}_{\tau}R_j', j<p, {\sf s}_{\tau}y_j\leq b\}$$
It is clear that $h$ is one one, 
preseves the boolean operations and the diagonal elements and is the identity on $A'$.
Now we check cylindrifications and substitutions.
$${\sf c}_ih(b)={\sf c}_i[(b-\sum {\sf s}_{\tau}y_j)\cup\bigcup\{{\sf s}_{\tau}R_j', j<p, {\sf s}_{\tau}y_j\leq b\}]$$
$${\sf c}_ih(b)={\sf c}_i(b-\sum {\sf s}_{\tau}y_j)\cup\bigcup\{{\sf c}_i{\sf s}_{\tau}R_j', j<p, {\sf s}_{\tau}y_j\leq b\}$$
$${\sf c}_ih(b)={\sf c}_i(b-\sum {\sf s}_{\tau}y_j)\cup\bigcup\{{\sf c}_i{\sf s}_{\tau}y_j, j<p, {\sf s}_{\tau}y_j\leq b\}$$
$$={\sf c}_i[(b-\sum {\sf s}_{\tau}y_j)\cup\bigcup\{{\sf s}_{\tau}y_j, j<p, {\sf s}_{\tau}y_j\leq b\}]={\sf c}_ib$$
On the other hand
$$h{\sf c}_i(b)=({\sf c}_ib-\sum {\sf s}_{\tau}y_j)\cup\bigcup\{R_j': {\sf s}_{\tau}y_j\leq {\sf c}_ib\}={\sf c}_ib$$
Preservation of substitutions follows from the fact that the substitutions are Boolean endomorphisms.

\end{demo}
Now which of the above results generalize to algbras without diagonal elements like Pinter's substitution algebras
and Halmos polyadic algebras (without equality).
We follow \cite{HHM}. We show that their construction proves more. Let $4\leq n\leq m<\omega$.
Then a set with $n+m$ elements $B_n^m$, that will constitute the set of atoms in the future relation algebra, is defined on p.201.
In this page, the forbidden triples are also specified and the relation algebar  $\A_n^m$ 
is defined as the complex algebra of the resulting  atom structure.  
The set of all $n$ by $n$ basic matrices is actually symmetric, and so they are a symmetric hyperbases, and so $\A\in \Ra\PEA_n.$
However the identity free reduct of $\A$ is not in $S\Ra \SC_{n+1}$.
Indeed assume that $\A\subseteq \Ra\C$ where $\C\in \SC_{n+1}$ \footnote{Let $\A=(A,+,\cdot,-,0,1, {\sf c}_i, {\sf s}_i^j)_{i,j<n}$ be an $\SC_m$ with $m\geq 3$. then we define
$$\Ra\A=(Nr_2\A,+,\cdot,-,0,1,  \breve, ;)$$
where $Nr_2A=\{x\in \A: {\sf c}_ix=x$ for all $i\geq 2\}$ and for any $x,y\in Nr_2\A$
$$x;y={\sf c}_2({\sf s}_2^1x\cdot {\sf s}_2^0y)$$
$$\breve{x}={}_2{\sf s}(0,1)x$${}
Here $_k{\sf s}(i,j)x={\sf s}_i^k{\sf s}_j^i{\sf s}_k^jx.$
For $\K\subseteq \SC_m$, $\Ra\K=\{\Ra\B: \B\in \K\}$. 
In the above definition we mimicked the way how relation algebras are obtained from cylindric algebras \cite{HMT2} 5.3.7.}
, then the proof of Theorem 8 in \cite{HHM}
goes through, for in the proof one can easily check that the authors 
are using the equations collected in fact 9 pages 204-205, and all these are valid in $\SC_{n+1}$.
In other words, following the same proof we arrive at the same contradition. 

Now let $\K\in \{\CA, \SC, \PA, \PEA\}.$ Then the following Theorem holds:

\begin{theorem} For any finite $n\geq 3$, and any $k\in \omega$, we have $S\Nr_n\K_{n+k}\supset S\Nr_n\K_{n+k+1}$
\end{theorem}
 \begin{demo}{Proof} \cite{HHM} Cor 2: Assume that $S\Nr_n\K_{n+k}=S\Nr_n\K_{n+k+1}$.
Then $$S\Ra \K_{n+k}=S\Ra\Nr_n\K_{n+k}=S\Ra S\Nr_n\K_{n+k}$$
$$=S\Ra S\Nr_n\K_{n+k+1}=S\Ra\Nr_n\K_{n+k+1}=S\Ra \K_{n+k+1}.$$ 
But this cannot happen because the appropriate reduct of the algebra $\A_n^m$ distinguishes beween these two classes.
\end{demo}
  
However, it is not known whether  results concerning the complexity of axiomatizations of $\RCA_n$ and $\RPEA_n$, like theorem \ref{c},
extend to $\SC$'s and $\PA$'s, and for that matter $\Df$'s.

We note that $\RPEA_n^{-}$ is a universal class that is not finitely axiomatizable; this can be proved exactly like the $\CA$ case
proved by Comer.
Now let us go deeper into the analysis of the problem of amalgamation. We will show that the notions of axiomatizability
and amalgamation are not entirely unrelated. In fact, the purpose of this discussion is to stress that what distinguishes the two paradigms 
are finite axiomatizability and
amalgamation, for some reason, they come together (in the polyadic paradigm) and they fail together 
(in the cylindric paradigm, the syntactical part of which is reflected
by systems of varieties definable by schemas.)
To analyse this we recall a recent result proved by the author connecting neat embeddings to amalgamation in a very general setting:
\begin{definition} 
\begin{enumroman}
\item Let $K$ be a class of algebras having a boolean reduct. 
$\A_0\in K$ is in the amalgamation base of $K$ if for all $\A_1, \A_2\in K$ and 
monomorphisms $i_1:\A_0\to \A_1,$ $i_2:\A_0\to \A_2$ 
there exist $\D\in K$
and monomorphisms $m_1:\A_1\to \D$ and $m_2:\A_2\to \D$ such that $m_1\circ i_1=m_2\circ i_2$. 
\item If in addition, $(\forall x\in A_j)(\forall y\in A_k)
(m_j(x)\leq m_k(y)\implies (\exists z\in A_0)(x\leq i_j(z)\land i_k(z) \leq y))$
where $\{j,k\}=\{1,2\}$, then we say that $\A_0$ lies in the super amalgamation base of $K$. Here $\leq$ is the boolean order.
$K$ has the (super) amalgamation property $((SUP)AP)$, if the (super) amalgamation base of $K$ coincides with $K$.
\end{enumroman}
\end{definition}
The super amalgamation property was introduced by Maksimova, and it was recently studied in Algebraic logic by Sagi and Shelah \cite{Shelah}
\begin{definition} 
\begin{enumroman}
\item Let $\A\in \Kn_{\alpha}$. Then $\A$ has the $UNEP$ (short for unique neat embedding property) 
if for all $\A'\in \K_{\alpha}$, $\B$, $\B'\in \K_{\alpha+\omega},$
isomorphism $i:\A\to \A'$, embeddings  $e_A:\A\to \Nr_{\alpha}\B$ and $e_{A'}:\A'\to \Nr_{\alpha}\B'$ 
such that $\Sg^{\B}e_A(A)=\B$ and $\Sg^{\B'}e_{A'}(A)'=\B'$, there exists
an isomorphism $\bar{i}:\B\to \B'$ such that $\bar{i}\circ e_A=e_{A'}\circ i$. 
\item Let $\A\in \Kn_{\alpha}$. Then $\A$ has the strong neat embedding property $SNEP$,
if for all $\B\in \K_{\alpha+\omega}$ if $\A\subseteq \Nr_{\alpha}\B$ and $A$ generates $\B$ then 
$\A=\Nr_{\alpha}\B$. 
\end{enumroman}
\end{definition}
The following is proved in \cite{universal}:

\begin{theorem} Let $\K=(\K_{\alpha}: \alpha\geq \omega)$ be a system of varieties. Let 
$\bold M=\{\A\in \K_{\alpha+\omega}:\Sg^{\A}\Nr_{\alpha}\A=\A\}.$
Assume that $\bold M$ has $SUPAP$, and that for any $\A,\B\in \bold M$ and isomorphism $f:\Nr_{\alpha}\A\to \Nr_{\alpha}\B$
there exists an isomorphism $\bar{f}:\A\to \B$ such that $f\subseteq \bar{f}$.
Then the following hold for any $\C\in \Kn_{\alpha}$.
\begin{enumroman}
\item  $\C$ has $UNEP$ if and only if $\C\in APbase(\Kn_{\alpha})$.
\item $\C$ has $UNEP$ and $SNEP$ if and only if $\C\in SUPAPbase(\Kn_{\alpha})$.
\end{enumroman}
\end{theorem}
   \begin{demo}{Proof} Assume that $\C$ has $UNEP$. Let $\A,\B\in \Kn_{\alpha}$. Let $f:\C\to \A$ and $g:\C\to \B$ be monomorphisms.
Then there exist $\A^+, \B^+, \C^+\in \K_{\alpha+\omega}$, $e_A:\A\to \Nr_{\alpha}\A^+$ 
$e_B:\B\to  \Nr_{\alpha}\B^+$ and $e_C:\C\to \Nr_{\alpha}\C^+$.
We can assume that $\Sg^{\A^+}e_A(A)=\A^+$ and similarly for $\B^+$ and $\C^+$.
Let $f(C)^+=\Sg^{A^+}e_A(f(C))$ and $g(C)^+=\Sg^{B^+}e_B(g(C)).$
Since $\C$ has $UNEP$, there exist $\bar{f}:\C^+\to f(C)^+$ and $\bar{g}:\C^+\to g(C)^+$ such that 
$(e_A\upharpoonright f(C))\circ f=\bar{f}\circ e_C$ and $(e_B\upharpoonright g(C))\circ g=\bar{g}\circ e_C$.
Now $\bold M$ as $SUPAP$, hence there is a $\D^+$ in $\bold M$ and $k:\A^+\to \D^+$ and $h:\B^+\to \D^+$ such that
$k\circ \bar{f}=h\circ \bar{g}$. Then $k\circ e_A:\A\to \Nr_{\alpha}\D^+$ and
$h\circ e_B:\B\to \Nr_{\alpha}\D^+$ are one to one and
$k\circ e_A \circ f=h\circ e_B\circ g$.
Now for the converse. It suffices to show that if $\A\in APbase(\Kn_{\alpha})$, if $i_1:\A\to \Nr_{\alpha}\B_1$, $i_2:\A\to \Nr_{\alpha}\B_1$ 
such that $i_1(A)$ generates $\B_1$ and $i_2(A)$ generates
$\B_2$, then there is an isomorphism $f:\B_1\to \B_2$ auch that $f\circ i_1=i_2$.
By assumption, there is a $\D\in \Kn_{\alpha}$, $m_1:\Nr_{\alpha}\B_1\to \D$, $m_2:\Nr_{\alpha}\B_2\to \D$ such that
$m_1\circ i_1=m_2\circ i_2$. We can assume that $m_1:\Nr_{\alpha}\B\to \Nr_{\alpha}\D^+$ for some $\D^+\in \bold M$, and similarly for $m_2$.
By hypothesis, Let $\bar{m_1}:\B_1\to \D^+$ and $\bar{m_2}:\B_2\to \D^+$ be isomorphisms extending $m_1$ and $m_2$.  
Then since $i_1A$ generates $\B_1$ and $i_2A$ generates $\B_2$, then $\bar{m_1}\B_1=\bar{m_2}\B_2$. It follows that 
$f=\bar{m}_2^{-1}\circ \bar{m_1}$ is as desired.
Now we prove (ii). Assume that $\C$ has $UNEP$ and $SNEP$. We obtain (using the notation in the first part)
$\D\in \Nr_{\alpha}\K_{\alpha+\omega}$ 
and $m:\A\to \D$ $n:\B\to \D$
such that $m\circ f=n\circ g$.
Here $m=k\circ e_A$ and $n=h\circ e_B$.  Denote $k$ by $m^+$ and $h$ by $n^+$.
Suppose that $\C$ has $SNEP$. We further want to show that if $m(a) \leq n(b)$, 
for $a\in A$ and $b\in B$, then there exists $t \in C$ 
such that $ a \leq f(t)$ and $g(t) \leq b$.
So let $a$ and $b$ be as indicated. 
We have  $m^+ \circ e_A(a) \leq n^+ \circ e_B(b),$ so
$m^+ ( e_A(a)) \leq n^+ ( e_B(b)).$
Since $\bold M$ has $SUPAP$, there exist $ z \in C^+$ such that $e_A(a) \leq \bar{f}(z)$ and
$\bar{g}(z) \leq e_B(b)$.
Let $\Gamma = \Delta z \sim \alpha$ and $z' =
{\sf c}_{(\Gamma)}z$. (Note that $\Gamma$ is finite, by the generating condition.) So, we obtain that 
$e_A({\sf c}_{(\Gamma)}a) \leq \bar{f}({\sf c}_{(\Gamma)}z)~~ \textrm{and} ~~ \bar{g}({\sf c}_{(\Gamma)}z) \leq
e_B({\sf c}_{(\Gamma)}b).$ It follows that $e_A(a) \leq \bar{f}(z')~~\textrm{and} ~~ \bar{g}(z') \leq e_B(b).$ Now by hypothesis
$$z' \in \Nr_\alpha \C^+ = \Sg^{\Nr_\alpha \C^+} (e_C(C)) = e_C(C).$$ 
So, there exists $t \in C$ with $ z' = e_C(t)$. Then we get
$e_A(a) \leq \bar{f}(e_C(t))$ and $\bar{g}(e_C(t)) \leq e_B(b).$ It follows that $e_A(a) \leq e_A \circ f(t)$ and 
$e_B \circ g(t) \leq
e_B(b).$ Hence, $ a \leq f(t)$ and $g(t) \leq b.$
Now assume that $\A\in SUPAPbase(\Kn_{\alpha})$. Then $\A$ is in the $APbase(\Kn_{\alpha})$ and so by the first part $\A$ has $UNEP$.
We want to show that $\A$ has $SNEP$. If not, then $\A\subseteq \Nr_{\alpha}\B$, $\B\in K,$ $A$ generates $\B$ and $A\neq \Nr_{\alpha}\B$.
Then $\A$ embeds into $\Nr_{\alpha}\B$ via the incusion map $i$ . Let $\C=\Nr_{\alpha}\B$.
Since $\A\in SUPAPbase$, there is a $\D\in \Kn_{\alpha}$ and $m_1$, $m_2$ monomorphisms 
from $\C$ to $\D$ such that $m_1(\C)\cap m_2(\C)=m_1\circ i(\A)$. Let $y\in \C\sim A$. 
Then $m_1(y)\neq m_2(y)$ for else $d=m_1(y)=m_2(y)$ will be in
$m_1(\C)\cap m_2(\C)$ but not in $m_1\circ i(\A)$. Assume that 
$\D\subseteq \Nr_{\alpha}\D^+$ with $\D^+\in K$.
There exist $\bar{m_1}:\B\to \D^+$ and $\bar{m_2}:\B\to \D^+$ extending $m_1$ and $m_2$. 
But $A$ generates $\B$ and so $\bar{m_1}=\bar{m_2}$.
Thus $m_1y=m_2y$ which is a contradiction.   
Note that the last part of the proof shows that if $\A\subseteq \Nr_{\alpha}\B$ , $A$ generates $\B$ and $A\neq \Nr_{\alpha}\B$, 
then the inclusion $\A\subseteq \Nr_{\alpha}\B$ cannot be strongly amalgamated in $\Kn_{\alpha}.$
\end{demo}

The $SAPbase(K)$ for a class $K$ is defined to be those algebras that are in the amalgamation base of $K$, 
and further the amalgam is strong. 

\begin{corollary} The following are equivalent for $\Kn_{\alpha}$ and $\A\in \K_{\alpha}$.
\begin{enumroman}
\item $\A$ has $UNEP$ and $SNEP$
\item $\A\in SUPAPbase(\Kn_{\alpha})$
\item $\A\in SAPbase(\Kn_{\alpha})$
\end{enumroman}
\end{corollary}
\begin{demo}{Proof}
One proves that $(i)\implies (ii)$ like the previous proof , $(ii)\implies (iii)$ is obvious and $(iii)\implies (i)$
is actually what we proved in the previous theorem.
\end{demo}
Note that by the techiques of Sagi and Shelah in \cite{Shelah}, it is not hard to extend their result to $\RPEA_n$. For every finite 
$n\geq 3$, there is a finitely axiomatizable variety $V\subseteq \RPEA_n$ that has $SAP$ but not $SUPAP.$
However, if an algebra $\A$ strongly amalgamates with all representable algebras, then it superly amalgamate with all such algebras as well. 
Using the above charaterization, we can go deeper into the analysis,
it is easy to show that $SAP$ fails in the class of representable algebras.
We do it for $\CA$'s, the other cases are completely analogous. It
It is enough to show that there exists a representable algebra that is not in $\Nr_{\alpha}\CA_{\alpha+\omega}$.
(This is equivalent to showing that the latter is not closed under forming subalgebras).
For suppose that $\A$ is such. Then $\A\subseteq \Nr_{\alpha}\B$ and $\B\in \CA_{\alpha+\omega}$. Let $\B'=\Sg^{\B}A$, then $\B'\in \bold M$
as defined above, furthemore the inclusion $\A\subseteq \Nr_{\alpha}\B'$ cannot be strongly amalgamated  
in the class of representable algebras. Such examples exist in the literature \cite{N83}. Furthermore these algebras can be chosen to be 
diagonal cylindric algebras in the sense of \cite{P}, so that this class
does not have the amalgamation property with respect to the class of representable algebras. This answers a question of Pigozzi in \cite{P}.
(Different solutions of this and other open questions of Pigozzi's can be found in \cite{MStwo}.
For other algebras, a similar construction can be found in  \cite{SL}.) Finally it is shown in \cite{MStwo} 
that several distinguished classes of cylindric algebras lie in the $SUPAPbase(\RCA_{\alpha})$
like algebras of positive characteristic and monadic generated algbras.

So this is the magic connection between amalgamation and representability, 
the notion of {\it neat embeddings}. Algebras that have the neat embedding propery are representable, algebras that have the unique neat embedding 
amalgamate and algebras that have the unique neat embedding property and strong neat embedding property superamalgamate.
Recall, that in the finite dimensional case, atomic algebras that have the {\it complete} neat embedding property 
are those algebras that are {\it completely} representable.

\subsection{The polyadic paradigm}

Now let us go to the polyadic paradigm. There are other systems of varieties that do not conform to the notion ``definable by schemes"
and these are Halmos' polyadic algebras $\PA_{\alpha}$
and their reducts studied by Sain \cite{Sa98}, in the context of finitizing first
order logic. 
Surprisingly for $\PA_{\alpha}$ we have $\PA_{\alpha}=\Nr_{\alpha}\PA_{\beta}$ for all $\alpha<\beta$ and \cite{super}
$$\bold R\PA_{\alpha}=\PA_{\alpha}=APbase(\PA_{\alpha})=SUPAPbase(\PA_{\alpha}).$$
Here $\bold R\PA_{\alpha}$ stands for the class of representable $\PA_{\alpha}$'s. The same 
can be said about the classes of algebras studied by Sain \cite{Sa98}. However, these algebras do not fit in the framework adapted herein, 
although they are a system of varieties, they are {\it not } defined by a schema in the above sense.

{\it An interesting question is whether there {\it is} a system of varieties definable by schemes for which  
$$\Kn_{\alpha}=\K_{\alpha}=APbase(\K_{\alpha})=SUPAPbase(\K_{\alpha}).$$}
This question is strongly related, to the finitizability problem, for in the known algebraic logics existing in the literature,
the distance between $\Kn_{\alpha}$ and $\K_{\alpha}$ is essentially infinite \cite{An97}, and one form of the 
finitizability problem, though admittedly never put in this form,  is how to ``finitize" this gap \cite{Bulletin}.
Note that the above question has three essentially distinct statements: 
\begin{enumarab}
\item $\K_{\alpha}=\Kn_{\alpha}$ 
\item $\K_{\alpha}=APbase(\K_{\alpha})$
\item $\K_{\alpha}=SUPAPbase(\K_{\alpha}).$
\end{enumarab}
These are not entirely independent for clearly (3) implies (2).

In her solution to the Finitizability problem, Sain \cite{Sa98} introduced a system of varieties in which this gap can be finitized, and it was 
further proved by the present author \cite{AU} that this class has $SUPAP$. So another question arises: Suppose that we can finitize this gap, 
that is suppose we can find a finite schema  equations that
define $\Kn_{\alpha}$ above $\K_{\alpha}$, does it follow then that $\Kn_{\alpha}$ has $SUPAP$? Does it necessarily have $AP$?
In other words, how (un)related are (1)(2) and (3) of the above item. 
We believe that these are difficult questions, that touch upon crucial issues in universal algebraic logic, and that they 
definitely deserve to be dealt with in a general framework.
{\it Let us see what is happening here. The Finitizability problem is crudely the attempt to capture infinitely many extra dimensions in a finitary way. 
One way to do that is to force a neat embedding theorem, as done by Sain \cite{Sa98}.
But when we force a neat embedding theorem, this in turn, forces $UNEP$ and $SNEP$ so $AP$ and $SUPAP$ comes as well.
We do not know of a framework in which this chain is broken at some point.}
(See also the last paragraph of the article, where we return to this point, in a slightly different context).

{\it Now how to explore those two paradigms in one context?}

\subsection{A solution in the Form of equivalence of two categories} 

We can regard $(\K_{\alpha}: \alpha\geq \omega)$ as a system of concrete categories synchronized by the 
the neat reduct functor. This view now encompasses the cases $\PA_{\alpha}$ and Sain's algebras $\SA_{\alpha}$ 
studied in \cite{Sa98} and \cite{AU}, as well as the notion of systems of varieties definable by schemas.
That is, for $\beta>\alpha$, we can regard $\Nr_{\alpha}:\K_{\beta}\to \K_{\alpha}$ 
as a functor with $\Nr_{\alpha}\A$ being the neat $\alpha$ reduct of $\A$ and
for a morphism $f$, $\Nr_{\alpha}(f)=f\upharpoonright \alpha$.
In particular $\Nr_{\alpha}:\K_{\alpha+\omega}\to \K_{\alpha}$ is a functor. 

In several concrete case, it has been shown, that 
when the latter functor has an inverse, that is there exists a functor
$\bold F:\K_{\alpha}\to \K_{\alpha+\omega}$ such that $F\circ \Nr_{\alpha}$ is naturally isomorphic to the identity functor then
$\Kn_{\alpha}=\K_{\alpha}$ and $SUPAP$ follows. This has been shown to be the case 
for $\PA_{\alpha}$ \cite{super} and $\SA_{\alpha}$ \cite{AU}. 
Actually in these previous cases, it turns out 
that the category $\K_{\alpha}$ is equivalent to $\K_{\alpha+\omega}.$

{\it The equivalence of these two categories, says that the gap can be finitized, or rather in fact, it does not exist at all!}

In fact both categories are equivalent to the category $\bold M=\{\A\in \K_{\alpha+\omega}: \A=\Sg^{\A}\Nr_{\alpha}\A\}$.
Surprisingly, even for the cylindric algebras the last category is well behaved,
for example it has $SUPAP$, 
but the point is, it is not equivalent to $\K_{\alpha}$ nor $\K_{\alpha+\omega}$
in this particular case. In fact for $\CA$'s $\bold M\subseteq \Dc_{\alpha+\omega}$ and a lot of properties of the latter class do not even generalize
to the class of representable algebras see theorem \ref{amal}.

This viewpoint has not been studied much; it was only touched upon in \cite{amal}, where two techniques of proving the amalgamation
property for various classes of algebras are unified, and both presented as adjoint situations. One is due to Pigozzi, and the other is due to Nemeti.
The unification consists of presenting both techniques as transforming a diagram of algebras to be strongly amalgamated
into certain saturated representations of these algebras that can be strongly amalgamated, and then returning to the original diagram using an inverse operator. Both can be described functorially by an adjoint situation making the noton of inverse involved 
precise. In the case of Pigozzi it is the neat reduct functor (an inverse to a neat embedding functor taking an algebra into one in 
$\omega$ extra dimensions, i.e a clasical representation), while in Nemeti's case it is basically the operation of forming atom 
structures that is an inverse
of taking an algebra to its canonical extension (which can be seen as a modal representation). 
This takes the representation problem expresses by a two sorted defining theory a step further, asking that the second sort be a
saturated representation.   

We conclude that finding a system of varieties definable by schemas for which  
$$\Kn_{\alpha}=\K_{\alpha}=APbase(\K_{\alpha})=SUPAPbase(\K_{\alpha}).$$
and the algebras in $\Kn_{\alpha}$ are representable, 
would provide a strong solution to the Finitizability problem (asking for finite schema) in algebraic logic.
Note that usually algebras in $\Kn_{\alpha}$ are representable, when $\Kn_{\alpha}=SUp\Kf_{\alpha}$, and the class
of locally finite algebras are representable.  Note too, that the notion of representability is not incuded in the definition of such systems, it comes from 
``outside", since it has to do with semantics and not with syntax.

Viewing a solution to the Finitizability problem as the existence of an equivalence between two categories is a novel approach.
The equivalence of the categories $\K_{\alpha}$ and $\K_{\alpha+\omega}$ to $\bold M$ says roughly that any 
algebra in $\K_{\alpha}$ contains
infinitely many hidden extra dimensions. These unfold to force a neat embedding theorem.
When an algebra can be {\it neatly} embedded in $\omega$ extra dimensions, this (usually, but not always) 
force representability of the operations. The real technical difficulty that come up here, is that when we expand our languages, and add axioms
to code extra dimensions somehow, in the hope of obtaining a neat embedding theorem, 
then usually we succeed in representing the already existing operations; the difficult problem 
is that the new operations turn out representable as well! (This comes up across in the case of Sain's algebras in \cite{Sa98}).

In any case, we believe that the problem of finding simple (finite) 
axiomatizations for the class of 
representable algebras remains an open philosophical question.  

\subsection{Neat embeddings, amalgamation and a problem of Henkin Monk and Tarski}

In this section we state another result that sets the two paradigms apart. 
We start by quoting Henkin, Monk and Tarski in \cite{HMT1}:``It will be shown in Part II that for each $\alpha$, $\beta$
such that $\beta\geq \alpha\geq \omega$ there is a $\CA_{\alpha}$ $\A$ and a $\CA_{\beta}$ $\B$ such that 
$\A$ is a generating subreduct of $\B$ different from $\Nr_{\alpha}\B$; in fact, both $\A$ and $\B$ can be taken to
be representable. Thus $\Dc_{\alpha}$ cannot be replaced by $\CA_{\alpha}$ in Theorem 2.6.67 (ii); 
it is known that this replacement
also cannot be made in certain consequences of 2.6.67, namely 2.6.71 and 2.6.72."
And we quote Henkin and Monk in the introduction of \cite{HMT2}: 
``Throughout Part I various ``promises" were made about 
material which would be found in Part II. These are located in this volume at the appropriate places, 
with the following exceptions,  which mainly concern results whose proofs could not be reconstructed."
It turns out that these are $5$ (unfulfilled) items, cf. \cite{HMT2}. Item (5) in op.cit. reads:

  ``Cf. Part 1 page 426. We do not know whether, if $\omega\leq \alpha <\beta$, there is a $\CA_{\alpha}$
$\A$ and a $\CA_{\beta}$ $\B$ such that $\A$ is a generating subreduct of $\B$ different from $\Nr_{\alpha}\B$."

To the best of our knowledge counterexamples to generalizations of 2.6.71-72 in \cite{HMT1} are also unknown. 
We now show that in the above quoted theorems, $\Dc_{\alpha}$ cannot be replaced by $\RCA_{\alpha}$
confirming what seems to have been  a conjecture of Tarski's, the proof of which could not be reconstructed by his co-authors Henkin 
and Monk.  
In what follows, we use the notation of the monograph \cite{HMT1}, often without warning, with the following exception.
We write $f\upharpoonright A$ instead of $A\upharpoonright f$ to denote the restriction of $f$ to $A$.

\begin{lemma}\label{neat} If $ \alpha < \beta$ are any ordinals and $L \subseteq \CA_\alpha$, 
then, in the sequence of conditions 
(1) - (5) below, (1) - (4) implies the immediately following one:
\begin{enumarab}
\item For any $\A \in L$ and $ \B \in \CA_\beta$ with $ \A \subseteq \Nr_\alpha \B$, 
for all $X \subseteq A$ we have $\Sg^\A X = \Nr_\alpha \Sg^\B X$.
\item For any $\A \in L$ and $ \B \in \CA_\beta$ with $\A \subseteq \Nr_\alpha \B$, if $ \Sg^\B A = \B$, 
then $ \A = \Nr_\alpha \B$.
\item For any $\A \in L$ and $ \B \in \CA_\beta$ with $\A \subseteq \Nr_\alpha \B$, if $ \Sg^\B A = \B$, 
then for any ideal $I$ of $\B$, $\Ig^\B (A \cap I) = I$.
\item If whenever $\A \in L$, there exists $x \in {}^{|A|} A$ such that if 
$ \rho = \langle \Delta x_i : i < |A| \rangle,$ $\D = \Fr_{|A|} ^ \rho\CA_\beta$ and $g_\xi=\xi/Cr_{|A|}^{\rho}\CA_{\beta}$, 
then $\Sg ^{\Rd_\alpha \D} \{ g_\xi : \xi < |A| \} \in L$, then the following hold:
For $\A, \A' \in L$, $\B, \B' \in \CA_\beta$ with embeddings $e_A:\A \to \Nr_\alpha \B$ and $e_{A'}:\A' \to \Nr_\alpha \B'$ 
such that $ \Sg^\B e_A(A) = \B$ and $ \Sg^{\B'} e_{A'}(A) = \B'$, whenever $ i : \A \longrightarrow\A'$ 
is an isomorphism, 
then there exists an isomorphism $\bar{i} : \B \longrightarrow \B'$such that $\bar{i} \circ e_A =e_{A'} \circ i$.
\item Assume that $\beta=\alpha+\omega$. Then $L$ has the amalgamation property with respect to $\RCA_{\alpha}$. That is for all $\A_0\in L,$
${\A}_1$ and ${\A}_2\in \RCA_\alpha$, and
all monomorphisms $i_1$ and $i_2$ of ${\A}_0$ into ${\A}_1$, ${\A}_2$,
respectively, there exists ${\A}\in \RCA_{\alpha}$,
a monomorphism $m_1$ from ${\A}_1$ into ${\A}$ and a monomorphism $m_2$ from
${\A}_2$ into $\A$ such that $m_1\circ i_1=m_2\circ i_2$.

\end{enumarab}
\end{lemma}

\begin{demo}{Proof}
(1) implies (2) is trivial. Now we prove (2) implies (3).  
From the premise that $\A$ is a generating subreduct of $\B$ we easily infer that
$|\Delta x\sim \alpha|<\omega$ for all $x\in B$. We now have $\A=\Nr_{\alpha}\B$. Now clearly 
$\Ig^{\B}(I\cap A)\subseteq I$. Conversely let $x\in I$. 
Then ${\sf c}_{(\Delta x\sim \alpha)}x$ is in $\Nr_{\alpha}\B,$
hence in $\A$. Therefore ${\sf c}_{(\Delta x\sim \alpha)}x\in A\cap I$. 
But $x\leq {\sf c}_{(\Delta x\sim \alpha)}x$,
hence the required.
We now prove (3) implies (4). 
Let  $\A, \A' \in L$, $\B, \B' \in \CA_{\beta}$ and assume that $e_A, e_{A'}$ are embeddings from $\A, \A'$ into   $\Nr_\alpha \B,
\Nr_\alpha \B'$, respectively, such that
$ \Sg^\B (e_A(A)) = \B$
and $ \Sg^{\B'} (e_{A'}(A')) = \B',$
and let $ i : \A \longrightarrow \A'$ be an isomorphism. 
We need to ``lift" $i$ to $\beta$ dimensions.
Let $\mu=|A|$. Let $x$ be a  bijection  from $\mu$ onto $A$ that satisfies the premise of $(4)$.   
Let $y$ be a bijection from $\mu$ onto $A'$,
such that $ i(x_j) = y_j$ for all $j < \mu$.
Let $\rho = \langle \Delta ^{(\A)} x_j : j < \mu\rangle$, $\D = \Fr^{(\rho)}_{\mu} \CA_{\beta}$, $g_\xi =
\xi/Cr^{(\rho)}_{\mu} \CA_{\beta}$ for all  $\xi< \mu$
and $\C = \Sg^{\Rd_\alpha \D} \{ g_\xi : \xi < \mu \}.$
Then $\C \subseteq \Nr_\alpha \D,\  C \textrm{ generates } \D
~~\textrm{and by hypothesis }~~ \C \in L.$
There exist $ f \in Hom (\D, \B)$ and $f' \in Hom (\D,
\B')$ such that
$f (g_\xi) = e_A(x_\xi)$ and $f' (g_\xi) = e_{A'}(y_\xi)$ for all $\xi < \mu.$
Note that $f$ and $f'$ are both onto. We now have 
$e_A \circ i^{-1} \circ e_{A'}^{-1} \circ ( f'\upharpoonleft \C) = f \upharpoonleft \C.$
Therefore $ Ker f' \cap \C = Ker f \cap \C.$
Hence by $(3)$ $\Ig(Ker f' \cap \C) = \Ig(Ker f \cap \C).$
So, $Ker f'  = Ker f.$
Let $y \in B$, then there exists $x \in D$ such that $y = f(x)$. Define $ \hat{i} (y) = f' (x).$
The map is well defined and is as required.
We now prove that (4) implies (5). Let $\C\in L$. Let $\A,\B\in \RCA_{\alpha}$. Let $f:\C\to \A$ and $g:\C\to \B$ be monomorhisms.
Then by the Neat Embedding Theorem, there exist $\A^+, \B^+, \C^+\in \CA_{\alpha+\omega}$ and embeddings 
$e_A:\A\to \Nr_{\alpha}\A^+$ 
$e_B:\B\to  \Nr_{\alpha}\B^+$ and $e_C:\C\to \Nr_{\alpha}\C^+$.
We can assume that $\Sg^{\A^+}e_A(A)=\A^+$ and similarly for $\B^+$ and $\C^+$.
Let $f(C)^+=\Sg^{\A^+}e_A(f(C))$ and $g(C)^+=\Sg^{\B^+}e_B(g(C)).$
Then by (4)  there exist $\bar{f}:\C^+\to f(C)^+$ and $\bar{g}:\C^+\to g(C)^+$ such that 
$(e_A\upharpoonright f(C))\circ f=\bar{f}\circ e_C$ and $(e_B\upharpoonright g(C))\circ g=\bar{g}\circ e_C$.
Let $K=\{\A\in \CA_{\alpha+\omega}: \A=\Sg^{\A}\Nr_{\alpha}\A\}$. Then $\A^+$, $\B^+$ and $\C^+$ are all in $K$. Now by \cite{P} 2.2.12 $K$ has the amalgamation property, 
hence there is a $\D^+$ in $K$ and monomorphisms $k:\A^+\to \D^+$ and $h:\B^+\to \D^+$ such that
$k\circ \bar{f}=h\circ \bar{g}$. Let $\D=\Nr_{\alpha}\D^+$. Then $k\circ e_A:\A\to \Nr_{\alpha}\D$ and
$h\circ e_B:\B\to \Nr_{\alpha}\D$ are one to one and
$k\circ e_A \circ f=h\circ e_B\circ g$. By this the proof is complete.

\end{demo}
\begin{theorem} Let $\alpha>1$. Let $\beta = \alpha + \omega$.  
Then   (1) - (5) in Lemma 1 are false for $L=\RCA_\alpha$.
 \end{theorem}
\begin{demo}{Proof} Using lemma \ref{neat} upon noting that $\RCA_{\alpha}$ fails to have the amalgamation property \cite{P} and that
$\RCA_{\alpha}$ satisfies the premise of (4) in lemma \ref{neat} when $\beta=\alpha+\omega$.
\end{demo} 
We readily conclude:

{\bf We cannot replace $\Dc_{\alpha}$ in 2.6.67 (ii), 2.6.71-72 of \cite{HMT1} by $\RCA_{\alpha}$ when $\alpha\geq \omega.$}
In more detail we have 
\begin{theorem}\label{amal} For $\alpha\geq \omega$, the following hold:
\begin{enumroman}
\item There are non-isomorphic representable cylindric algebras of dimension $\alpha$ each of 
which is a generating subreduct of the same $\alpha+\omega$ dimensional cylindric algebra.
\item There exist $\A\in \RCA_{\alpha}$, a $\B\in \CA_{\alpha+\omega}$ and  an ideal $J\subseteq \B$, such that $\A\subseteq \Nr_{\alpha}\B$, 
$A$ generates $\B$, but $\Ig^{\B}(J\cap A)\neq \B$.
\item There exist $\A, \A' \in \RCA_{\alpha}$, $\B, \B' \in \CA_{\alpha+\omega}$ with 
embeddings $e_A:\A \to \Nr_\alpha \B$ and $e_{A'}:\A' \to \Nr_\alpha \B'$ 
such that $ \Sg^\B e_A(A) = \B$ and $ \Sg^{\B'} e_{A'}(A) = \B'$, and an isomorphism $ i : \A \longrightarrow\A'$ 
for which there exists no isomorphism $\bar{i} : \B \longrightarrow \B'$ such that $\bar{i} 
\circ e_A =e_{A'} \circ i$.
\end{enumroman}
\end{theorem}

Lemma \ref{neat} tells us where to find direct counterexamples, namely from common subalgebras of algebras in $\RCA_{\alpha}$ that do not amalgamate.
We note that theorem \ref{amal} was generalized to $\SC$'s $\QPA$ $\QPEA$'s  \cite{fail}; 
however it does not hold for $\PA$'s nor $\SA$'s, emphasizing the dichotomy between
those two paradigms.

Let $\K\in \{\SC, \QPEA, \QPA, \CA\}$.  Let $\alpha$ be infinite. let $\A\in \K_{\alpha}$. Then 
recall that $\Delta x$, {\it the dimension set of $x$}, is defined by $\{i\in \alpha: {\sf c}_ix\neq x\}$.
Now we set:
$$\LfK_{\alpha}=\{\A\in \K_{\alpha}: \Delta x \text { is finite for all } x\in A\}$$
$$\DcK_{\alpha}=\{\A\in \K_{\alpha}: \alpha\sim \Delta x \text { is infinite for all } x\in A\}$$
$$\SsK_{\alpha}=SP\{\A: \A \text { is simple }\}$$
$$\ReK_{\alpha}=\{\A\in \K_{\alpha}: (\forall \Gamma\subseteq_{\omega} \alpha)( \forall x\neq 0)(\exists i,j\in \alpha\sim \Gamma)( i\neq j \land {\sf s}_i^jx\neq 0)\}.$$
Here, and elsewhere throughout the paper $x\subseteq_{\omega}y$ denotes that $x$ is a finite subset of $y$.
It is known that $\LfK_{\alpha}\subseteq \DcK_{\alpha}$ 
and $\DcK_{\alpha}\cup \SsK_{\alpha}\subseteq \ReK_{\alpha}\subseteq \RK_{\alpha}$ 
We mention a recent result on neat embeddings very much related to the amalgamation property.
Consider the following class $\L\subseteq \K_{\alpha}$. $\A\in \L$ if for every finite sequence 
$\rho$ without repeating terms and with range included in $\alpha$, for every non-zero $x\in A$,
there is a function $h$ and $k<\alpha$ such that $h$ is an endomorphism of $\Rd^{\rho}\A$, $k\in \alpha\smallsetminus Rg\rho$, ${\sf c}_k\circ h=h$ 
and $h(x)\neq 0$.  This $L$ is defined in \cite{HMT2} Theorem 2.6.50 (iii) for cylindric algebas, but it makes perfect sense for all algebras considered
herein . The fact that $\ReK_{\alpha}\subseteq \L\subseteq \RK_{\alpha}$ is proved in \cite{HMT1} Theorem 2.6.50 for cylindric algebras. 
The proof adapts without much difficulty to $\K_{\alpha}$. In fact, the latter follows from the neat embedding Theorem, namely
$S\Nr_{\alpha}\K_{\alpha+\omega}=\RK_{\alpha}$, where $\Nr_{\alpha}\K_{\alpha+\omega}$ denotes the class of neat $\alpha$ reducts of algebras
in $\K_{\alpha}$. 
Now $\ReK_{\alpha}\subset \L$ properly. The following example is taken from \cite{HMT1} and adapted to the cases considered herein. 
If we take $\A$ to be the full set algebra in the space $^{\alpha}\alpha$, then ${\sf s}_k^l(Id\upharpoonright \alpha)=0$ for every $k,l<\alpha$.
Suppose that $\rho$ is a finite one to one sequence with $Rg\rho\subseteq \alpha$ and $X\subseteq {}^{\alpha}\alpha$, $X\neq 0$. Let
$k\in \alpha\setminus Rg\rho$ and choose $\tau\in {}^{\alpha}\alpha$ such that $k\notin Rg\rho$, 
$\tau\upharpoonright Rg\rho\subseteq Id$ and $\tau$ is one to one.
Let $$h(Y)=\{\phi\in {}^{\alpha}\alpha: \phi\circ \tau\in Y\}.$$
Then $h$ satisfies the conclusion in the definition of $\L$. 
It is asked in \cite{HMT1} whether $\L$ (in the $\CA$ case) coincides with the class of representable cylindric algebras.
It is proved in \cite{neet} that the class $\L$ has $AP$ with respect to $\RK_{\alpha}$. Since $\RK_{\alpha}$
fails to have $AP,$ it follows that $\L\neq \RK_{\alpha}$. In other words, all three inclusions 
$\ReK_{\alpha}\subset \L\subset \RK_{\alpha}$ are proper. 
This answers a question of Henkin Monk and Tarski \cite{HMT1} p.417, formulated as problem 2.13.
The latter is one of the very few questions that are open in \cite{HMT1}, possibly 
the only one.
\section{Technical innovations }

\subsection{Appendix: Some stability theory in connection to neat embeddings}

In the following theorem, when the condition of maximality is omitted, then we are led to a statement that is independent of $ZFC +\neg CH$, 
However, when we consider ultrafilters, then we can prove, and indeed only in $ZFC$:
\begin{athm}{Theorem}\label{ZF}
Let $\A\in S_c\Nr_n\CA_{\omega}$ be countable.
Let $\kappa<{}^{\omega}2$. Let $(X_i:i\in \kappa)$ be a family of non-principal ultrafilters of $\A$.
Then there exists a representation $f:\A\to \wp(^nX)$ such that $\bigcap_{x\in X_i}f(x)=\emptyset$
for all $i\in \kappa$.
\end{athm}
\begin{demo}{Proof}
Assume that $\A$ is countable with $\A\in S_c\Nr_n\CA_{\omega}.$ Then $\A\subseteq\Nr_n\D$ with $\D\in \CA_{\omega}$.
Let $\B=\Sg^{\D}A$. Then $\B\subseteq \D$, $B$ is countable and $\B\in \Lf_{\omega}.$
Futhermore we have $\prod X_i=0$ in $\B$. We shall construct a representation of $\B$ preserving the given set of meets.
Say that an ultrafilter $F$ in $\A$ is realized in the representation $f:\B\to \wp(^{\omega}M)$ if $\bigcap_{x\in F}f(x)\neq \emptyset.$
We first construct two representations of $\B$ such that if $F$ is an ultrafilter in $\Nr_n\B$ that is realized in both representations, then $F$ is 
necessarily principal, that is $\prod F$ is an atom generating $F$.
We construct two ultrafilter $T$ and $S$ of $\B$ such that
\begin{equation}\label{t5}
\begin{split}
(\forall k<\alpha)(\forall x\in A)({\sf c}_kx\in T\implies (\exists l\notin \Delta x) {\sf s}_k^lx\in T)\\ 
(\forall k<\alpha)(\forall x\in A)({\sf c}_kx\in S\implies (\exists l\notin \Delta x) {\sf s}_k^lx\in S) 
\end{split}
\end{equation}
\begin{equation}\label{t6}
\begin{split}
\forall \tau_1, \tau_2\in {}^{\omega}\omega^{(Id)}( G_1=\{a\in \Nr_n\B: {\sf s}_{\tau_1}a\in T\},
G_2=\{a\in \Nr_n\B: s_{\tau_1}a\in S\})\\\implies G_1\neq G_2 \text { or $G_1$ is principal.}
\end{split}
\end{equation}
Note that $G_1$ and $G_2$ are indeed ultrafilters in $\Nr_n\B$.
We construct $S$ and $T$ as a union of a chain. We carry out various tasks as we build the chains.
The tasks are as in \ref{t5}, \ref{t6}, as well as
(***) for all $a\in A$ either $a\in T$ or $-a\in T$, and same for $S$. We let $S_0=T_0=\{1\}$.
There are countably many tasks. Metaphorically we hire countably many experts and give them one task each.
We partition $\omega$ into infinitely many sets and we assign one of these tasks to each expert. 
When $T_{i-1}$ and $S_{i-1}$ have been chosen and $i$ is in the set assigned to some expert $E$, then $E$ will construct
$T_i$ and $S_i$.
For consider the expert who handles task (***). Let $X$ be her subset of $\omega$. Let her list as $(a_i: i\in X)$ all elements of $X$.
When $T_{i-1}$ has been chosen with $i\in X$, she should consider whether $T_{i-1}\cup \{a_i\}$ is consistent. If it is she puts
$T_i=T_{i-1}\cup \{a_i\}$. If not she puts $T_i=T_{i-1}\cup \{-a_i\}$. Same for $S_i$.
Next consider the expert who deals with the tasks in \ref{t5}. She waits until she is gets a set $T_{i-1}$ which contains ${\sf c}_ka$. 
Every time this happens she chooses $l\notin \Delta a$
which is not used in $T_{i-1}$, and she puts $T_i=T_{i-1}\cup \{{\sf s}_k^la\}$. Same for $S_i$.
Now finally consider the tasks in \ref{t6}. Suppose that $X$ contains $i$ , and $S_{i-1}$ and $T_{i-1}$ have been chosen.
Let $e=\bigwedge S_{i-1}$ and $f=\bigwedge T_{i-1}$. We have two cases.
If $e$ is an atom in $\Nr_n\B$ then the ultrafilter $F$ containg $e$ is principal so our expert can put $S_i=S_{i-1}$ and $T_i=T_{i-1}$.
If not, then let $F_1$ , $F_2$ be distinct ultrafilters containing $e$. Let $G$ be an ultrafilter containing $f$. Say $F_1$ is different from $G$.
Let $\theta$ be in $F_1-G$. Then put $S_i=S_{-1}\cup \{\theta\}$ and $T_i=T_{i-1}\cup \{-\theta\}.$ 
It is not hard to check that the canonical models corresponding to $S$ and $T$ are as required.
In the above proof, each expert has to make sure that the theories $T$ and $S$ have some property $P$.
The proof shows that the expert can make $T$ and $S$ have $P$, provided that she is allowed to choose $T_i, S_i$ for infinitely many $i$.
We can express this in terms of a two player game, call it $G(P,X)$, where $X$ is any infinite countble subset of $\omega$
whose complement is also infinite. 
The players have to pick the pairs $(T_i,S_i)$ in turn, player $\exists$ makes 
the choice of $T_i$ if and only if
$i\in X$. Player $\exists$ wins if $T$ and $S$ have the property $P$, otherwise $\forall$ wins. 
We say that $P$ is enforceable if $\exists$ has a winninig strategy for 
this game.
Using the terminology of Hodges: The property 'every maximal type which is realized in both $M_1$ and $M_2$ is isolated' 
is enforceable.
There is no difficulty in stretching the above idea to make the experts build three, four or any finite number of models 
which overlap only at principal types. 
With a pinch of diagonalisation
we can extend the number to $\omega$. To push it still further to $^{\omega}2$ neeeds a new idea. 
This is one of those many places in model theory where we get continuum many models for the same price as two.
Adopt, basically,  the argument above, 
allowing the experts to introduce a new chain of theories which is a duplicate copy of one of the chains being 
constructed. In this case, the construction will take the form of a tree. 
Each branch $\beta$ will give rise to a chain $T_0\subseteq T_1\ldots $
of conditions, write $T_{\beta}$ for the ultrfailter containing this chain.  By splitting the tree often enough, 
the experts can ensure that there are continuum many branches
and hence continuum many models in the end. There is one expert whose job is to make sure that \ref{t6} is enforcable for each pair of branches.
But she can do this task, for at each step she has to act the number of branches is still finite.
Qouting Hodges again ``Let $R$ be an enforceable property of ordered pairs of $L$ structures. 
Let $P$ be the property which an indexed famly $(B_j:j\in J)$ of $L$ structures has iff for all $j\neq k$ in $J$ $(B_j,B_k)$ has propery $R$. 
Then $P$ is enforceable." 

Before we embark on the details we fix some terminology.
In what follows we write $\wp(^\alpha M)$ for the full cylindric set algebra
$(\wp(^{\alpha}M), \cup, \cap,\sim {\sf c}_i, {\sf d}_{ij}),$ i.e., set algebras are identified notationally with their universe. 
We shall need to modify the definition of neat reducts. Let $\A\in \CA_{\alpha}$ and $I$ be a subset of $\alpha$ (not necessarily an initial segment),
then $\Nr_{I}\A=\{x\in \A: {\sf c}_ix=x \text{ for all } i\notin I\}$. In this case $\Nr_I\A$ is only a Boolean algebra.
Let $\B\in \Lf_{\lambda}$. In our present case $\lambda=\omega$, but we shall deal with a more general  case
when the algebra in question need not be countable, so we denote the dimension by $\lambda$, which in turn, might not be countable..
A model of $\B$ is  a non-zero homomorphism $f:\B\to \wp(^{\lambda}M).$
We write $(f,M)$ for such a model. We need an exact algebraic formulation of the notion of types.
If $(f,M)$ is a model of $\B$ and $s\in {}^{\lambda}M$ 
and $I$ is a finite subset of $\lambda$, then $type_{f,I}(s)=\{a\in \Nr_I\B : s\in  f(a)\}$. Such a set is called a type. A type 
is therefore a Boolean ultrafilter of $\Nr_I\B$.  We may write just $type_l(s)$ without reference to $f$ for $type_{f,l}(s)$.
Note that $type_{f,l}(s)$ depends only on the values of $s$ on $I$. 
That is, fixing $f$,  if $s_1\upharpoonright I=s_2\upharpoonright I$ then $type_{f,I}(s_1)=type_{f,I}(s_2).$ 
Accordingly if $I$ is a finite subset of $\lambda$ and $s$ is a finite sequence defined only on $I$, we define $type_{f,l}(s)$ to be
$type_{f,l}(\bar{s})$ where $\bar{s}$ is any extension of $s$ to $\lambda$.
If $\bar{a}$ is a finite sequence we may write, by an abuse of notation, simply $type(\bar{a})$ for $type_{f,l}(\bar{a})$
without specifyng the arity of $\bar{a}$ which will be clear from
context. Also, everytime we do this $I$ and $f$ will be clear from context, or else their specification is immaterial.
For $p,q\subseteq \B$, we write $p\vdash q$, if $p\models x$ for all $x\in q$. That is 
for every model $(f,M)$ of $\B$, and $s\in {}^{\lambda}M,$ 
if $s\in \bigcap_{y\in p} f(y)$ then $s\in f(x)$.
A set $\Gamma\subseteq \B$ is consistent, if it has the finite intersection property. This is equivalent (by the completeness theorem for first order logic) 
to the fact that
there exists a model $(f,M)$ of $\B$ and $s\in {}^{\lambda}M$ such that $s\in \bigcap_{x\in \Gamma}f(x).$
Now let $\A=\Nr_n\B$.
(*) We shall define a family $(f_i,M_i)$ of models of $\B$, with $i<{}^{\lambda}2,$ such that $|M_i|=\lambda$, $f_i(d)\neq 0,$ and
such that if $s_1\in {}^{\lambda}M_1$ and $s_2\in {}^{\lambda}M_2$, $M_1$ $M_2$ are distinct, 
and $I\subseteq n$, if $type_{f_1,I}(s_1)=type_{f_2,I}(s_2)$, then there exists a $q\subseteq type_{f_2,I}(s_2)$ $|q|<\lambda$
and $q\vdash type_{f_1,I}(s_1)$. That is if a type is realized in two distinct models, then it is necessarily isolated.

From (*) the required will easily follow from the following reasoning:
We refer to $(f_i, M_i)$ simply by $M_i$. Let $\K=\{M_i:i<{}^{\lambda}2\}$. Recall that $(f_i,M_i)$ realizes $F\subseteq \A$ if 
$\bigcap_{x\in F} f_i(x)\neq \emptyset$.
We say that $M_i$ omits $F$ if the representation $(f_i, M_i)$
does not realize $F$. 
Now, let $\{F_i:i<\kappa\}$ be the given non-principal ultrafilters. Recall that $\kappa<{}^{\lambda}2$.
Let $\K_0=\{M\in \K: M \text { omits }F_0\}$. If $F_0$ is realized in two distinct models then it would principal, 
so this does not happen. Then $|\K_0|={}^{\lambda}2$.
Let $\K_1=\{M\in \K_0: M \text { omits } F_1\}$. Then, by the same reasoning, $|\K_1|={}^{\lambda}2$. 
In this way, we can define, by induction, a chain $(\K_i:i<\kappa)$ of decreasing sets of models
such that each for all $i<\kappa$, $|\K_i|= {}^{\lambda}2$, $\K_i$ omits $X_i$ and $\K_j\subseteq \K_i$ whenever $i<j$. 
That is, let $\K_{\delta}=\bigcap_{i<\delta}\K_i$ at limits and at successors, set $\K_{i+1}=\{M\in \K_i: M \text { omits } F_{i+1}\}.$
Note that at the limit cases, since $|\delta|<\kappa$, then $|K_{\delta}|={}^{\lambda}2$. Since $\kappa<{}^{\lambda}2$,
then by the same token, there is a model (representation) $M$ in the intersection, i.e. in $\bigcap_{j<\kappa}\K_j$ which clearly 
satisfies
the required. That is we obtain a set $M$ having cardinality $\lambda$ and $g:\B\to \wp(^\lambda M)$ 
such that for all $i\in \kappa$, we have $\bigcap_{x\in F_i}g(x)=\emptyset$.

Define $f:\A\to \wp(^nM)$ by $f(a)=\{s\upharpoonright n: s\in g(a)\}$. Then clearly $f$ is as required.   
Now, to implement (*)  we distinguish between several cases. Throughout the proof a boolean subalgebra $\Delta$ of $\B$ such that
$\A\subseteq \Delta$ is fixed. In particular, $|\Delta|=\lambda$.

In constructing our desired representation, we will make use of the following fact that can be proved using the Tarsi Vaught 
test for elementary substructures.
Let $\A\in \Lf_{\alpha}$, $\alpha$ infinite. Assume that $\Gamma\subseteq \A$ is consistent and satisfies that for all $x\in A$ and 
$k<\alpha$, whenever ${\sf c}_kx\in A$, then there is a $\lambda\notin \Delta x$ such that ${\sf s}_{\lambda}^kx\in \Gamma$.
Note that the ultrafilters constructed in Theorem satisfied this property, which is a form of elimination of quantifiers.
Let $f:\A\to \wp^{\alpha}M$ be a representation that relizes $\Gamma$. Let $s\in \bigcap_{x\in \Gamma}f(x)$. Then the map
$g:\A\to \wp^{\alpha}Range s$ defined by $g(a)=f(a)\cap {}^{\alpha}Range s$ is a homomorphism. 

Now for the details. We assume that $\lambda=\omega$.
We have $\A=\Nr_n\B$  and $\B\in \Lf_{\omega}$ is countable because we can assume that
$A$ generates $\B$.
Let $\{a_i:i<\omega\}$ be a list of $B$ such that $\Delta a_i\subseteq 2i$. 
The proof basically consists of defining by induction on $n<\omega$, a set 
$$J_n=\{\Gamma_{\eta}: \eta\in {}^n2, \Gamma_{\eta} \text { is a finite consistent set }\}.$$
By a consistent set, we understand a set $\Gamma$ that has the finite intersection property.
In forcing terminology, the elements of $J_n$ are are called conditions. We use two player games as above. 
For $n=0$, let $J_0=\emptyset$. Assume that $J_n$ is defined. 
Then for $\forall$ splits the conditions, he picks each and every $u\in {}^n2$ and $l\in \{0,1\}$. 
Now $\exists$ has to respond with a set of conditions, namely the set $J_{n+1}$. She has to define $\Gamma_{\mu-l}$ for every $\mu\in {}^n2$
and $l=\{0,1\}$.
She lists the set $\{(\eta_i^i, \eta_2^i, \bar{z}_1^i, \bar{z}_2^i): i<j\}$ of quadruples 
$(\eta_1, \eta_2, \bar{z_1}, \bar{z_2})$ such that $\eta_1, \eta_2\in {}^n2$ are distinct and 
$z_1$ and $z_2$ are finite sequences from $2n=\{j: j< 2n\}.$ Notice that $j<\omega$.
Now she plays a side game: for all $i\leq j$ she chooses conditions  $\Gamma_{\eta}^i$, $\eta\in {}^n2$, such that
for all $i$, for all $x\in \Gamma_{\eta}^i,$ $\Delta x\subseteq 2n$, and
for $l<m\leq j$ we have $\Gamma_{\eta}^l\subseteq \Gamma_{\eta}^m$.
Let $\Gamma_{\eta}^0=\Gamma_{\eta}$.
Now assume that
she defined $\Gamma_{\eta}^i$ for all $\eta<{}^n2$.
She distinguishes between two subcases:

(1) There is an element $a\in \Delta$  such that
$$\Gamma_{{\eta}_1}^i\cup \{a\} \text { and } \Gamma_{{\eta}_2}^i\cup \{\neg a\}$$
are consistent.
Then, she lets
$$\Gamma_{{\eta}_1}^{i+1}=\Gamma_{\eta_1}^i\cup \{a\}$$
$$\Gamma_{{\eta}_2}^{i+1}=\Gamma_{\eta_2}^i\cup \{\neg a\}.$$
$$\Gamma_{\eta}^{i+1}=\Gamma_{\eta}^i$$ for all other $\eta\in {}^n2$.

(2) There is no such $a$, then she sets
$$\Gamma_{\eta}^{i+1}=\Gamma_{\eta}^i$$ for all $\eta\in {}^n2$.
Now she has defined $\Gamma_{\eta}^i$ for all $\eta<{}^n2$ and all $i\leq j$.
At these stages of the game she makes a finite number of succesive choices, 
but to preserve the form of the game she will only mention the final choice to player $\forall$, and the rest she will keep secret.
The final output she declares, given $\eta\in {}^n2$ and $l\in \{0,1\}$, is
$$\Gamma_{{\eta}\widehat{}<l>}=\Gamma_{\eta}^j\cup \{-{\sf c}_ka_n\lor {\sf s}^k_{2n}a_n: k<2n\}\cup \{\neg{\sf d}_{l,2n+1}: l<2n\}$$
(This is the spliting of the tree as defined in the above theorem).

For each $\mu\in {}^{\lambda}2$, let $\Gamma_{\mu}=\bigcup_{\beta<\lambda}\Gamma_{\mu\upharpoonright \beta}$.
Then $\Gamma_{\mu}$ is consistent. Hence there a model $M_{\mu}$ and an assignment $s^{\mu}=(s_i^{\mu}:i<\lambda)\in {}^{\lambda}M_{\mu}$ that satisfies it.
That is there exists $f_{\mu}:\A\to \wp(^{\lambda}M_{\mu})$ such that $(s_i^{\mu}:i<\lambda)\in \bigcap_{x\in \Gamma_\mu} f_{\mu}(x).$
Define $g_{\mu}(x)=f(x)\cap {}^{\lambda}Range(s^{\mu})$. Then the the family of models $g_{\mu}:\B\to \wp(^{\lambda}Range(s^{\mu}))$, 
with $\mu\in {}^{\lambda}2$
is as desired.
\end{demo}
Building models by games is easy when the context is countable. Models of cardinality $\omega_1$ are not much harder, they can be reached 
as limits of chains of countable models.
For larger cardinals it can get very rough.
To build them by (transfinite) games usually strong set theoretic assumptions are invoked, and the combinatorics involved are quite challenging.
In our next theorem, we use transfinite induction to prove the existence of representations for uncountable algebras.
The techniques we use come from stability theory, they are essentially due to Shelah. We will be constructing uncountable structures by 
approximations of smaller size,
and this involves climbing up to uncountable cardinals through the ordinals below them. 
Climbing up a limit ordinal $\alpha$ means finding an unbounded subset of $\alpha,$ i.e. a set $X\subseteq \alpha$
such that for every $i<\alpha$, there is a $j\in X$ such that $i\leq j$. In what follows we collect useful facts about unbounded sets.
The cofinality of a limit ordinal $\alpha$, $cf(\alpha)$, is the least ordinal $\beta$ such that $\alpha$ has an unbounded subset or order type $\beta$. 
Infinite cardinal of the form $cf(\alpha)$ are called regular cardinals, in fact every regular cardinal is its own cofinality.
Infinite cradinals which are regular are called singular. Successor cardinals are regular. Let $\lambda$ be an uncountable regular cardinal. If $X$ is a subset of 
$\lambda$, a limit point of $X$ below $\lambda$ is a limit ordinal $\delta<\lambda$ such that $X\cap \delta$ is unbounded in $\lambda$. We call $X$ closed
if it contains all its limit points below $\lambda$. Subsets of $\lambda$ which are both closed and unbounded are called clubs.
A set $X\subseteq \lambda$ is called fat if it contains a club. It is called thin if $\lambda\sim X$ is fat. A subset $S$ of $\lambda$ is called stationery
if it is not thin. $S$ is stationery if it intersect every club, hence it necessarily unbounded. All clubs are stationery but the converse is false.
In our next theorem we use techniques of Shelah from stability theory \cite{Shelah2}.
\begin{athm}{Theorem}\label{stability}
Let $\A\in S_c\Nr_n\CA_{\omega}$ be infinite such that $|A|=\lambda$, $\lambda$ is a a regular cardinal.
Let $\kappa<{}^{\lambda}2$. Let $(X_i:i\in \kappa)$ be a family of non-principal ultrafilters of $\A$.
Then there exists a representation $f:\A\to \wp(^nX)$ such that $\bigcap_{x\in X_i}f(x)=\emptyset$
for all $i\in \kappa$.
\end{athm}
\begin{demo}{Proof}
Assume that $\A\subseteq_c \Nr_{n}\D$, $\D\in \CA_{\omega}$. Let $\B=\Sg^{\D}A$. Then $\B\in \Lf_{\omega}$ and
$\A\subseteq \Nr_n\B$. We can assume that $\B\in \Lf_{\lambda}$. For if $\lambda>\omega$, 
then there is a $\B'\in \Lf_{\lambda}$ such $\B=\Nr_{\omega}\B'$ and $B$ generates $\B'$. Hence $\A\subseteq_c \Nr_{n}\B'$, 
so simply replace $\B$ by $\B'$. 
Throughout the proof a boolean subalgebra $\Delta$ of $\B$ such that
$\A\subseteq \Delta$ is fixed. In particular, $|\Delta|=\lambda$.

{\bf Case 1}

We assume that $\lambda>\omega$ and $\lambda^+<\chi$.
Now for all $a\in \B$ there is a successor $i<\lambda$ such that $\Delta a\subseteq 2i$. Accordingly,
let $(a_i : i \text { a successor} <\lambda)$ be an enumeraton of $\B$, where $\Delta a_i\subseteq 2i$.
Let $\{\bar{z}_{\alpha}:\alpha<\lambda\}$ be a list of all finite $\bar{z}\subseteq \lambda$ such that for such $\bar{z}$, the set
$\{\alpha: \bar{z}_{\alpha}=z\}$ is stationary subset of $\lambda$. We can assume that $\bar{z}_{\alpha}\subseteq 1+\alpha$.
Let $\B=\bigcup_{\alpha<\lambda}A_{\alpha}$ where $|A_{\alpha}|<\lambda$
and $A_{\alpha}$ is increasing and continous. That is, at limits,  $A_{\delta}=\bigcup_{i<\delta}A_i$.
We can assume, without loss of generality,  that for $x\in A_{\alpha}$, we have  $\Delta x\subseteq \alpha$.
We shall define by induction on $\alpha<\lambda$, for each $\eta\in {}^{\alpha}2$ a consistent set $\Gamma_{\eta}$, such that
$\Delta x\subseteq \{i: i<2\alpha\}$ for all $x\in \Gamma_{\mu}$, and $|\Gamma_{\mu}|<\lambda$.
We have several subcases:
\begin{enumarab}
\item $\alpha=0,$ set $\Gamma=\emptyset$

\item $\alpha$ is a limit ordinal. 
For each $\mu\in {}^{\alpha}2$, set $\Gamma_{\mu}=\bigcup_{\beta<\alpha}\Gamma_{\mu\upharpoonright \beta}$

In case $\alpha$ is a successor ordinal, we distinguish between three subcases:

\item $\alpha=\beta+1$, $\beta$ successor. We assume, inductively, 
 that $\Gamma_{\mu}$ is defined for all $\mu\in {}^{\beta}2$. Let $(a^i: i<i(0))$ be a list of $A_{\beta}$ 
with $\Delta a^i\subseteq \{i: i<2\alpha\}$. This is possible since $\alpha=\beta+1$, and $\Delta x\subseteq \beta$ for all $x\in A_{\beta}$.
For each $\mu\in {}^{\beta}2$
we define by induction on $i<i(0)$ $\theta_{\mu}^i$ such that
$\Gamma_{\mu}\cup \{\theta_{\mu}^j: j\leq i\}$ is consistent and $\theta_{\mu}^i\in \Delta$.
Let $i$ be given. Assume that everything is defined for $j<i$. 

If there is 
$\theta\in \Delta$, $\Delta \theta\subseteq 2\alpha$, such that $\Gamma_{\mu}\cup \{\theta_{\mu}^j:j<i\}$ is consistent, and $a^i\vdash \neg \theta$,
choose $\theta_{\mu}^i=\theta$.

If not,  
choose any $\theta_{\mu}^i\in \Delta$ such that $\Delta\theta_{\mu}^i \subseteq 2\alpha$ and
$\Gamma_{\mu}\cup \{\theta_{\mu}^i: j\leq i\}$ is consistent. This is possible since $i<i(0)<\lambda$, $|\Gamma_{\mu}|<\lambda.$
and $|\Delta|=\lambda$.
Now let for 
$\mu\in {}^{\beta}2,$ 
$$\Gamma_{\mu}^1=\Gamma_{\mu}\cup \{\theta_{\mu}^j: j<i(0)\}\cup \{-{\sf c}_ka_{\beta}\lor {\sf s}^k_{2\beta}a_{\beta}: k<2\beta\}.$$ 

For $\eta\in {}^{\beta}2$ and $l\in \{0,1\}$,
$$\Gamma_{{\eta}\ \widehat{}\\<l>}=\Gamma_{\eta}^1\cup \{-{\sf d}_{j, 2\beta+1}: j<2\beta\}.$$

 For $\beta$ limit, we distinguish betwen two subcases:

\item $\alpha=\beta+1,$ $\beta$ limit, and for no $\mu<\beta$, $\bar{z}_{\mu}=\bar{z}_{\beta}$.
We assume inductively that $\Gamma_{\mu}$ is defined for all $\mu\in {}^{\beta}2$.
For $\mu\in {}^{\beta}2$, let 
$$p_{\mu}=\{{\sf c}_{(\Gamma)}\bigwedge X: X\subseteq_{\omega} \Gamma_{\mu}, \Gamma\cap \bar{z}_{\beta}=\emptyset\}.$$
If there is a finitary type $q\subseteq \Nr_I\A$, where $I=Range \bar{z}_{\beta}$, $|q|<\lambda$, such that  $p_{\mu}\cup q$ is consistent, and for every 
$\theta\in \Delta$ , $|\Delta \theta|=|\bar{z}_{\beta}|$, 
we have $p_{\mu}\cup q\vdash {\sf s}_{z_{\beta}}\theta$ or $p_{\mu}\cup q\vdash \neg {\sf s}_{z_{\beta}}\theta$, 
choose such $q_{\mu}$, and if there is no such $q_{\mu}$,
let $q_{\mu}=\emptyset$. 
For $l=0,1$, let
$$\Gamma_{{\mu}\ \ \widehat{}\ \ <l>}=\Gamma_{\mu}\cup q_{\mu}.$$

\item  $\alpha=\beta+1$, $\beta$ limit and for some $\mu<\beta$ $\bar{z}_{\mu}=\bar{z}_{\beta}$.
Assume that $\Gamma_{\mu}$ is defined for all $\mu\in {}^{\beta}2$. For $\mu\in {}^{\beta}2$, 
define $p_{\mu}$ as in the previous case. Then $|p_{\mu}|<\lambda$. Let
$$S_{\beta}=\{\mu\in {}^{\beta}2: \exists \theta\in \Delta \text { and both } p_{\mu}\cup \{{\sf s}_{\bar{z}_{\beta}}\theta\}$$
$$\text { and }
p_{\mu}\cup \{-{\sf s}_{\bar{z}_{\beta}}\theta \}\text { are consistent }\}.$$ 
By the previous case, as $|p_{\mu}|<\lambda$, for each $q\subseteq \Nr_{I}\B$ with $I=Range(\bar{z}_{\beta})$,
such that $|q|<\lambda$, $\mu\in S_{\beta}$, if
$p_{\mu}\cup q$ is consistent, then for some $\theta\in \Delta$, both 
$p_{\mu}\cup q\cup \{{\sf s}_{\bar{z}_{\beta}}\theta\}$ and $p_{\mu}\cup q\cup \{\neg {\sf s}_{\bar{z}_{\beta}}\theta\}$ are consistent. 
For $a\in \B$ set $a^0=a$ and $a^1=\neg a$. Let $\mu\in S_{\beta}$. Then  we can find
$\theta_{\mu}^{\rho}\in \Delta$ such that for every $v\in {}^{\beta}2$,
$p_{\mu}\cup \{({\sf s}_{\bar{z}_{\beta}}\theta_{\mu}^{v\upharpoonright \mu})^{v[\mu]}:\mu<\beta\}$ is consistent.
We can assume that if $\mu_1, \mu_2\in S_{\beta}$ and for every $\theta\in \Delta$, $p_{\mu_1}\vdash {\sf s}_{\bar{z}_{\beta}}\theta$ 
iff $p_{\mu_2}\vdash {\sf s}_{\bar{z}_{\beta}}\theta$, then for every $v\in {}^{<\beta}2,$
$\theta_{\mu_1}^v=\theta_{\mu_2}^v$.
Now for $\mu\in S_{\beta}$, $l=0,1$, let
$$\Gamma_{\mu\widehat{}\\ <l>}=\Gamma_{\mu}\cup \{({\sf s}_{\bar{z}_{\beta}}\theta_{\mu}^{v\upharpoonright \mu})^{v[\mu]}:\mu<\beta\}.$$
For $\mu\in {}^{\beta}2-S_{\beta},$ 
let $\Gamma_{\mu\widehat{}\\<l>}=\Gamma_u$.
\end{enumarab}
For each $\mu\in {}^{\lambda}2$ we have
$$\Gamma_{\mu}=\bigcup_{\alpha<\lambda}\Gamma_{\mu\upharpoonright \alpha}$$ 
is consistent.

Hence there a model $M_{\mu}$ and an assignment $s_{\mu}=(s_{\mu}^i:i<\lambda)\in {}^{\lambda}M_{\mu}$ that satisfies it.
That is there exists $f_{\mu}:\A\to \wp(^{\lambda}M_{\mu})$ such that $(s_{\mu}^i:i<\lambda)\in \bigcap_{x\in \Gamma_\mu} f_{\mu}(x).$
Define $g_{\mu}(x)=f(x)\cap {}^{\lambda}Range(s_{\mu})$. Then consider the family of models $g_{\mu}:\B\to \wp(^{\lambda}Range(s_{\mu}))$, 
with $\mu\in {}^{\lambda}2.$

Now suppose $\bar{z}=(i(0),\ldots i(k))$, $\mu_1\neq \mu_2\in {}^{\lambda}2$, and let $s_l=(s_{\mu_l}^{i(0)},\ldots s_{\mu_l}^{i(k)})$.
Let $I=\{i(0),\ldots, i(k)\}$. Suppose that $I\subseteq n$. Assume that $type_{\Delta,I}(s_1)=type_{\Delta,I}(s_2)$.
Suppose for contradiction that there is no $q\subseteq type_{\Delta,I}(s_1)$ with $q<\lambda$ and $q\vdash type_{\Delta,I}(s_l)$.
Choose $\alpha_0$ such that $\mu_1\upharpoonright \alpha_0\neq \mu_2\upharpoonright \alpha_0$ and for
some $i<\alpha_0,$ $\bar{z}_i=\bar{z}$. Then if $\alpha_0<\beta<\lambda$ and $\beta$ is a limit, and $\bar{z}_{\beta}=\bar{z}$,
then $\bar{z_i}=\bar{z_{\beta}}$ and so $\mu_1\upharpoonright \beta$
and $\mu_2\upharpoonright \beta\in S_{\beta}$. This is so, because there is no $q$ such that $q\vdash type_{\Delta}(s_l)$, so
 $\exists \theta\in \Delta \text { and both } p_{\mu_l\upharpoonright \beta}\cup \{{\sf s}_{\bar{z}_{\beta}}\theta\}
\text { and }
p_{\mu_l\upharpoonright \beta}\cup \{-{\sf s}_{\bar{z}_{\beta}}\theta \}\text { are consistent }.$
Now there is a club $W\subseteq \lambda$ such that for $\beta\in W$, $l=1,2$, 
$\Gamma_{\mu_{l}\upharpoonright\beta}\subseteq A_{\beta}$. Furthermore, we can assume that for $a\in A_{\beta}$
either $\Gamma_{\mu_{l}\upharpoonright\beta}\vdash a$ or $\Gamma_{\mu_{l}\upharpoonright\beta}\vdash \neg a$.
We can also assume that for each $\beta\in W$ , $\beta$ is a limit, $\beta>\alpha_0$.
As $\{\beta: \bar{z}_{\beta}=\bar{z}\}$ is a stationary subset of $\lambda$, it intersects $W$, thus there is a $\beta\in W$ such that
$\bar{z}_{\beta}=\bar{z}$. Now for every $\theta\in \Delta$, $p_{\mu_{1}\upharpoonright\beta}
\vdash {\sf s}_{\bar{z}_{\beta}}\theta$ 
iff $p_{\mu_{2}\upharpoonright \beta}\vdash {\sf s}_{\bar{z}_{\beta}}\theta$.
For if not, then by symmetry we can assume that $p_{\mu_{1}\upharpoonright\beta}\vdash {\sf s}_{\bar{z}_{\beta}}\theta$ 
but not  $p_{\mu_{2}\upharpoonright \beta}\vdash {\sf s}_{\bar{z}_{\beta}}\theta$, then for some $\psi\in p_{\mu_{1}\upharpoonright\beta}$, 
$\psi\vdash {\sf s}_{\bar{z}_{\beta}}\theta,$
$p_{\mu_{2}\upharpoonright \beta}\cup \{\neg \psi\}$ is consistent. For some $i$, $\alpha_0<i<\beta$, $\psi\in A_i$, so by subcase (ii)
for some $\theta_1\in \Delta$, $\psi\vdash \theta_1$, $\neg \theta_1\in p_{\mu_{1}\upharpoonright i+1}$ hence 
$\neg \theta\in p_{\mu_{2}\upharpoonright\beta}$.
But $\psi\in p_{\mu_{1}\upharpoonright b}$, we have $\theta\in p_{\mu_{1}\upharpoonright \beta}$ 
which is a contradiction to $type_{\Delta}(s_1)=type_{\Delta}(s_2)$.
We have proved that for every $\theta\in \Delta$, $p_{\mu_{1}\upharpoonright\beta}
\vdash {\sf s}_{\bar{z}_{\beta}}\theta$ 
iff $p_{\mu_{2}\upharpoonright \beta}\vdash {\sf s}_{\bar{z}_{\beta}}\theta$. 
So $\theta_{\mu_1}^v=\theta_{\mu_2}^v$ for every $v\in {}^{<\beta}2$, and ${\sf s}_{\bar{z}_{\beta}}[\theta_{\mu_l}^{\mu_{l}\upharpoonright v}]^{\mu_l(v)}$
$\in \Gamma_{{\mu_l}\upharpoonright \beta+1}\subseteq \Gamma_{\mu_l}$ where $v=min\{v: \mu_1(v)\neq \mu_2(v)\}$, which is a contradiction.
We have proved

(**) if $type_{\Delta,I}(s_1)=type_{\Delta,I}(s_2)$, then 
there is  $q\subseteq type_{\Delta,I}(s_1)$ with $|q|<\lambda$ and $q\vdash type_{\Delta,I}(s_l)$

Now for each $\mu\in {}^{\lambda}2$, let $g(\mu)$ be the set $v\in {}^{\lambda}2$ such that for some $\bar{a}\in M_v$
$\bar{b}\in M_{\mu}$, $type_{\Delta}(\bar{a})=type_{\Delta}(\bar{b})$ but for no $q\subseteq type(\bar{a})$, $|q|<\lambda$ does
$q\vdash type_{\Delta}(\bar{a})$. Then by (**) we have $|g(\mu)|\leq \lambda$. Since $2^{\lambda}>\lambda^+$, 
there is a $U\subseteq {}^{\lambda}2$, $|U|=2^{\lambda}$ 
such that $\mu\in U$ implies $\mu\notin g(v)$. Then $\{M_{\alpha}: \alpha\in U\}$ is as required.

{\bf Case 3}

Assume that $\lambda$ is uncountable and $\lambda^+=\chi$. We define the required representations $(f_i,M_i)$ $i<\chi$ by induction on $i$.
Assume we have defined $(f_j,M_j)$ for all $j<i$.
Let $$\Psi_i=\{type(\bar{a}): \bar{a}\in M_j: j<i\}.$$
We want a representation $(f_i,M_i)$ with $|M_i|=\lambda$ and such that
if $type(\bar{a})\in \Psi_i$ $ \bar{a}\in M_i$, then for some $q\subseteq type(\bar{a})$, $|q|<\lambda$ and $q\vdash type(\bar{a}).$
Assume that no such representation exists. Then for every representation $(f,M)$ of $\A$, $|M|=\lambda$ 
there is a type $p\in \Psi_i$ that is realized in $M$, but for which there is no such $q.$
Let $\{p_j:j<\lambda\}$ be the set of all such types. Let $\{y_k:k<\lambda\}$ be a list of all finite sequences from $\lambda$, 
each sequence appearing $\lambda$ times. Let $(a_i: i<\lambda)$ be a list of
$\B$ such that $y_i, \Delta a_i\subseteq i$ for all $i<\lambda$. 
We define by induction on $i<\lambda$ a consistent set of formulas $\Gamma_i$, such that
for all $x\in \Gamma_i$, we have $\Delta x\subseteq 2i$, $|\Gamma_i|<|i|^+ +\omega$. Let $\Gamma_0=\emptyset$ and 
at limits, $\Gamma_{\delta}=\bigcup_{i<\delta}\Gamma_i$.
Assume $\Gamma_i$ is defined.
Let $$\Gamma_i^1=\Gamma_i\cup \{-{\sf c}_ka^i\lor {\sf s}^k_{2i}a_i: k< 2i\}\cup \{\neg {\sf d}_{2i+1, j}: j\leq 2i\}.$$
Let $$\Gamma_i^2=\{\bigwedge X: X\subseteq_{\omega} \Gamma_1\}.$$
Let $$r_i=\{{\sf c}_{(\Gamma)}a: \Delta a=y_i\cup \Gamma: a\in \Gamma_i^2\}$$
Then $|r_i|<\lambda$, is an  type in $y_i$. Hence when $|\Delta y_i|=m_j$ $j<\alpha$, we have, by assumption, not $r_i\vdash p_j$.
Suppose $y_i$ is the $j$th appearance of $y_i$. If $j\geq \alpha$, or $|\Delta y_i|\neq m_j$, let
$\Gamma_{i+1}=\Gamma_i^2$, otherwise for some $\theta_i\in p_j$, $r_i\cup \{\neg \theta_i\}$ is consistent, so let
$\Gamma_{i+1}=\Gamma_i^2\cup \{\neg\theta_i\}.$
Let $\Gamma_{\lambda}=\bigcup_{i<\lambda} \Gamma_i$. Then $\Gamma_{\lambda}$ is consistent. 
Let $(f,M)$ be a representation of $\B$ such that
$\bigcap_{x\in \Gamma_{\lambda}} f(x)$ is non-empty. Let $s$ be an element of this intersection. 
Define $g(x)=f(x)\cap {}^{\lambda}Range(s)$. Note that $|Range s|=\lambda$. Then $g:\B\to \wp(^{\lambda}Range(s))$, 
is a representation that  omits all the $p_i$'s which is a contradiction.

\subsection{Quasipolyadic equality algebras as opposed to cylindric algebras}

Quoting Henkin Monk and Tarski in \cite{HMT2} p. 266-267:
``Quasi-polyadic algebras: These are like polyadic algebras, except that ${\sf s}_{\tau}$ is allowed only for finite transformations, 
and ${\sf c}_{(\Gamma)}$ only for finite
$\Gamma$. Their theory has not been much developed, but they form an interesting stage between cylindric and polyadic algebras"

We have generalized a lot of deep results proven originally for $\CA$'s to $\QPEA$'s.
That might tempt us to jump to the conclusion that these two classes are very close.
Here we show that for $\alpha$ infinite $\RQEA_{\alpha}$ is essentially different that $\RCA_{\alpha}$.
It is known that the class $\RCA_\alpha$ of representable cylindric algebras for $\alpha>2$ is not 
axiomatizable by a set of universal formulas containing 
finitely many variables \cite{An97}, 
same for $\RQEA_\alpha$ \cite{ST}. (A proof of the latter result for the infinite dimensional case is only sketched in \cite{ST}, 
and it seems to us that there are some serious gaps in this sketch).
A striking result of Andreka \cite{An97} is that for finite $\alpha>2$ the class
$\RQEA_\alpha$ is not finitely axiomatizable over $\RCA_\alpha.$ (To the best of our knowledge this result does not appear in print)
This already proves that for finite dimensions, the operations of substitutions give a lot; they cannot be captured
in a ``finitary" way. The analogous result for infinite ordinals is unknown. 
In this paper, we address the infinite dimensional case.
We do not recover Andreka's result in its strongest form, but we prove a necessary condition for the class $\RQEA_{\omega}$ 
to be non-finitely axiomatizable over $\RCA_{\omega}$.
We will show that there is an $\A\in \QPEA_{\omega}$ such that its cylindric reduct $\Rd_{ca}\A$ is representable, while $\A$ itself is not representable.
This means that the finitely many polyadic axiom schemas do not define $\RQPEA_{\omega}$ over $\RCA_\omega$. 
(In principal, there could be another finite schema that defines the quasi-polyadic operations).
This result is joint with Andr\'eka and N\'emeti. Indeed our construction is based on 
an unpublished construction of Andr\'eka and N\'emeti \cite{AN} proving the same result for finite $\alpha>3$.
(Some parts are identical to parts in \cite{AN}, but we include all the details. One reason is for the conveniance of the reader.
Second reason is that \cite{AN} is not published. )
This latter result in \cite{AN} is surpassed by Andr\'eka's result
mentioned above. In our treatment of cylindric algebras and quasi-polyadic equality algebras we follow 
\cite{HMT1}, \cite{HMT2}.

\begin{theorem} \label{q}There exists a $\A\in \QPEA_{\omega}$ such that $\Rd_{ca}\A\in \RCA_{\omega}$, but $\A$ is not representable
\end{theorem}
Proving the analogous result for polyadic equality algebras is easy since for any $\PEA_{\omega}$ its cylindric reduct is representable 
and there are easy examples of non representable
$\PEA_{\omega}$'s. But for $\QPEA_{\omega}$ the proof is much more intricate.
Our example will be constructed from a {\it weak} set algebra. A cylindric weak set algebra is an algebra whose unit is a weak space,
i.e. a set of the form $^{\alpha}U^{(p)}=\{s\in {}^{\alpha}U: |\{i\in \alpha: s_i\neq p_i\}|<\omega\}$ where $p$ is a fixed sequence in $^{\alpha}U$.
The operations of a weak set algebra with unit $V$ are the boolean operations of union, intersection and complementation
with respect to $V$, and cylindrifications and diagonal elements are defined like in set algebras but relativized to $V$.
We shall need to characterize abstractly (countable) quasipolyadic equality {\it weak} set algebras 
where we require that the algebra is also closed under finite substitutions.
This was done for cylindric algebras by Andreka, Nemeti and Thompson \cite{ANth}. It turns out, that in the countable case, 
weak set algebras coincide with the the class of weakly subdirect
indecomposable algebras for both $\CA$'s and $\QPEA$'s. 
This follows from the facts that subdirect indecomposability and its weak version are defined for general algebras via congruences, congruences correspond to ideals,
and that for $\A\in \QPEA_{\alpha}$, $I$ is a quasi-polyadic ideal of $\A$ if and only if it is a cylindric ideal of $\Rd_{ca}\A$. 
This ultimately makes the abstract characterization of
weak set algebras for countable quasi-polyadic algebras coincide with that of (countable) cylindric algebras.
Now let ${\bf WQEAs}_{\alpha}$ denote the class of quasipolyadic equality weak set algebras.
Then we have $\RQEA_{\alpha}={\bf SPWQEAs}_{\alpha}$. Here $\bf SP$ denotes the operation of forming subdirect products.
This is proved exactly like the cylindric case.
Next, we give the definition of subdirect indecomposability and its weak version relative to congruences in general algebras.

\begin{definition} 
\begin{enumroman}
\item An algebra $\A$ is weakly subdirectly indecomposable if $|A|\geq 2$ and if the formulas $R,S\in Co\A$
and $R\cap S=Id\upharpoonright A$ always imply that  $R=Id\upharpoonright A$ or $S=Id\upharpoonright A$.
\item An algebra $\A$ is subdirectly indecomposable if $|A|\geq 2$ and if for every system $R$ of relations satisfying 
$R\in {}^ICo\A$
and $\bigcap_{i\in I}R_i=Id\upharpoonright A$, there is an $i$ such that $R_i$ coincides with the identity relation.
\end{enumroman}
\end{definition}
We shall need to specify ideals in quasipolyadic equality algebras.
Ideals are congruence classes containing the least element. From now on $\alpha$ will denote an infinite ordinal and $FT_{\alpha}$ 
denotes the set of finite transformations on $\alpha$.
$x\subseteq_{\omega}y$ denotes that $x$ is a finite subset of $y$ and $Sb_{\omega}\alpha$ 
denotes the set of all $x$ such that $x\subseteq_{\omega}\alpha$. 

\begin{definition}Let $\A\in \QPEA_{\alpha}$. A subset $I$ of $\A$
in an ideal if the following conditions are satisfied:

\begin{enumroman}

\item $0\in I,$

\item If $x,y\in I$, then $x+y\in I,$

\item If $x\in I$ and $y\leq x$ then $y\in I,$

\item For all $\Gamma\subseteq_{\omega} \alpha$ and $\tau\in FT_{\alpha}$ if $x\in I$
then ${\sf c}_{(\Gamma)}x$ and ${\sf s}_{\tau}x\in I$.

\end{enumroman}
\end{definition}
If $X\subseteq \A\in \QPEA_{\alpha}$, then $\Ig^{\A}X$ is the ideal generated by $X$.

\begin{lemma} Let $\A\in \QPEA_{\alpha}$ and $X\in A$.
Then 
$\Ig^{\A}X=\{y\in A: y\leq {\sf c}_{(\Gamma)}(x_0+\ldots x_{k-1})\}: 
\text{ for some  } x\in {}^kX, \text { and }\Gamma\subseteq_{\omega} \alpha\}.$
\end{lemma}
\begin{demo}{Proof}
Let $H$ denote the set of elements on the right hand side.
It is easy to check $H\subseteq \Ig^{\A}X$.
Conversely, assume that $y\in H,$ $\Gamma\subseteq \omega.$ 
It is clear that ${\sf c}_{(\Gamma)}y\in H$.
$H$ is closed under substitutions, since for any finite transformation $\tau$, any $x\in A$ there exists finite $\Gamma\subseteq \omega$
such that ${\sf s}_{\tau}x\leq {\sf c}_{(\Gamma)}x$.
Now let $z, y\in H$. Assume that $z\leq {\sf c}_{(\Gamma)}(x_0+\ldots x_{k-1})$
and $y\leq {\sf c}_{(\Delta)}(y_0+\ldots y_{l-1}),$
then $$z+y\leq {\sf c}_{(\Gamma\cup \Delta)}(x_0+\ldots x_{k-1}+ y_0\ldots +y_{l-1}).$$
The Lemma is proved. 
\end{demo}
It follows from \cite{HMT2} 2.3.8 that if $\A\in \QPEA_{\alpha}$ and $I$ is a cylindric ideal of $\Rd_{ca}\A$ then $I$ is an ideal of $A$.
Therefore $\A$ is (weakly) subdirectly indecomposable if and only if $\Rd_{ca}\A$ is (weakly) subdirectly indecomposable.
Now we prove the analogue of a result of Thompson for quasi-polyadic equality algebras. 
The proof is the same as that given by Andr\'eka, N\'emeti and Thompson in \cite{ANth} theorem 3,  
but for the sake of completeness (and because the proof is short)
we include the proof adapted to the quasi-polyadic equality (present) case. $\bold IK$ denotes the set of all isomorphic images of algebras
in $K$. We now have:

\begin{lemma}\label{weak} Let $\A\in \RQPEA_{\alpha}$ be countable. Then (i) and (ii) are equivalent
\begin{enumroman}
\item $\A\in \bold I{\bf WQEAs}_{\alpha}$
\item $\A$ is weakly subdirectly indecomposable.
\end{enumroman}
\end{lemma}
\begin{demo}{Proof} We shall only need that $(ii)\implies (i)$. So assume that that $\A$ is weakly subdirectly indecomposable quasipolyadic algebra
of dimension $\alpha$. Then by \cite{HMT2} 2.4.46 which works for quasipolyadic algebras, we have that
$$(*) \ \ \ \ (\forall x,y\in A\sim \{0\})(\exists \Delta \subseteq_{\omega}\alpha)x\cdot {\sf c}_{(\Delta)}y\neq 0.$$
Let $a:\omega\to A\sim \{0\}$ be any enumeration of $A\sim \{0\}$. We define $\Gamma:(A\sim \{0\})\to Sb_{\omega}\alpha$ step by step, so that
$$(**)\ \ \ \ \ \ b_n=\prod\{{\sf c}_{(\Gamma a_m)}a_m: m<n\}\neq 0\text { for all } n\in \omega, n\neq 0.$$
Let $\Gamma(a_0)=0$. Let $n\in \omega$, $n>0$, and assume that $\Gamma(a_m)$ has been defined for all $m<n$ such that
$b_n\neq 0$ holds. By $(*)$, there is a $\Delta\subseteq_{\omega} \alpha$ such that $b_n\cdot {\sf c}_{(\Delta)}a_n\neq 0$. Set $\Gamma(a_n)=\Delta$.
Then clearly $b_{n+1}=b_n\cdot {\sf c}_{(\Delta)}a_n\neq 0$. Since
$A\sim \{0\}=\{a_n: n\in \omega\}$, the function $\Gamma$ is defined.
By $(**)$ $\Gamma :\A \to Sb_{\omega}\alpha$ satisfies
$$(***) \ \ \ \ (\forall A_0\subseteq_{\omega} (A\sim \{0\}))\prod\{{\sf c}_{(\Gamma a)}a: a\in A_0\}\neq 0.$$
Then there is a maximal proper ideal of $Bl\A$ such that $m\supseteq \{-{\sf c}_{(\Gamma a)}a: a\in A\}$.
Let $\Cm\A$ be the canonical embedding algebra of $\A$. $\Cm\A$ is defined like the $CA$ case \cite{HMT1} definition 2.7.3. In particular, it has domain 
$\wp(M)$, which we denote by $Em\A$,
where
$M$ is the set of maximal Boolean ideals of $\A$. Substitutions are defined on $Em(A)$ as follows: 
$${\sf s}_{\tau}X=\bigcup_{I\in X}\{J\in M: J\subseteq {\sf s}_{\tau}I\}.$$
 Let $z=\{m\}$. Then $z\in Em\A$ and $0\neq z\leq em({\sf c}_{(\Gamma a)}a)$ for all $a\in A$. Here $em$ is the map that embeds
$\A$ into $\Cm\A$; $em(x)=\{I\in M:x \notin M\}$.
Let
$I=\{y\in Em\A: (\forall \Gamma\subseteq_{\omega}\alpha){\sf c}_{(\Gamma)}y\cdot z=0\}$. Then $I$ is an ideal of $\Cm\A$ and $I\cap em(A)=\{0\}$.
Let $\B=\Cm\A/I$. Then $\B\in \RQPEA_{\alpha}$ and $\A$ is embeddable in $\B$. Here we are using that if $\A\in \RQPEA_{\alpha}$, then so is $\Cm\A$.
The proof of this is identical to the $\CA$ case. Also $\B$ is subdirectly indecomposable by \cite{HMT1} 2.4.44. 
By \cite{HMT2} 3.1.86 $\A$ is isomorphic to a weak set algebra.
Though 2.4.44 in \cite{HMT1} and 3.1.86 in \cite{HMT2} are formulated for $\CA$'s they are true for $\QPEA$'s.
\end{demo}

The following corollary which we shall need  is now immediate

\begin{corollary} Let $\A\in \RQEA_{\alpha}$ be countable such that $\Rd_{ca}\A$ is weakly subdirectly indecomposable
(equivalently isomorphic to a cylindric weak set algebra).
Then $\A\in \bold I{\bf WQEAs}_{\alpha}$.
\end{corollary}

\begin{corollary} There exists a countable $\A\in \QEA_{\omega}$ that is weakly subdirectly irreducible but not representable.
\end{corollary}

\subsection{Proof of theorem \ref{q}}

Let $U=\N$. Let $Z\in {}^{\omega}\wp(\N)$ be defined by $Z_0=Z_1=3=\{0,1,2\}$ and $Z_i=\{2i-1, 2i\}$ for $i>1$.
Let $p:\omega\to \omega$ be defined by $p(i)=2i$. Let $V={}^{\omega}U^{(p)}=\{s\in {}^{\omega}U: |\{i\in \omega: s_i\neq 2i\}|<\omega\}$.
We will work inside the weak set algebra with universe $\wp(V)$ and cylindrifications and diagonal elements for 
$i,j<\omega$ defined for $X\subseteq V$ by:
$${\sf c}_iX=\{s\in V: \exists t\in X,  t(j)=s(j) \ \ \forall j\neq i\}$$
and
$${\sf d}_{ij}=\{s\in V: s_i=s_j\}.$$
Let $${\sf P}Z=\{s\in V: (\forall i \in \omega)s_i\in Z_i\}.$$
Let $$t=\{s\in {}^{\omega\sim 2}U: |\{i\in \omega\sim 2: s_i\neq 2i\}|<\omega,  (\forall i>2)s_i\in Z_i\}.$$
Let
$$X=\{s\in t: |\{i\in \omega\sim 2: s(i)\neq 2i\}| \text { is even }\},$$
$$Y=\{s\in t: |\{i\in \omega\sim 2: s(i)\neq 2i\}| \text { is odd }\},$$
$$R=\{(u,v): u\in 3, v=u+1(mod3)\},$$
$$B=\{(u,v): u\in 3, v=u+2(mod3)\},$$
and 
$$a=\{s\in {\sf P}Z: (s\upharpoonright 2\in R \text { and } s\upharpoonright \omega\sim 2\in X)\text { or }
(s\upharpoonright 2\in B\text { and }s\upharpoonright \omega\sim 2\in Y\}.$$
Let $Eq(\omega)$ be the set of all equivalence relations on $\omega$.
For $E\in Eq(\omega)$, let $e(E)=\{s\in V: ker s=E\}$. Note that $e(E)$ may be empty. Let
$$d={\sf P}Z\cap {\sf d}_{01}.$$
$\pi(\omega)=\{\tau\in FT_{\omega}: \tau \text { is a bijection }\}.$
For $\tau\in FT_{\omega}$ and $X\subseteq V,$ recall that the substitution (unary) operation ${\sf S}_{\tau}$ is defined by 
$${\sf S}_{\tau}X=\{s\in V: s\circ \tau\in X\}.$$
Let
$$P'=\{{\sf S}_{\tau}a: \tau \in \pi(\omega)\}, \\ \ \ P=P'\cup \{{\sf S}_{\delta}d: \delta\in \pi(\omega)\}.$$
More concisely, 
$$P=\{{\sf S}_{\tau}x: \tau \in \pi(\omega),\ \  x\in \{a,d\}\}.$$
For $W\in {}^{\omega}RgZ^{(Z)}$, let
$${\sf P}W=\{s\in V: (\forall i\in \omega) s_i\in W_i\}.$$
Let
$$T=\{{\sf P}W\cdot e(E): W\in {}^{\omega}RgZ^{(Z)},  (\forall \delta\in \pi(\omega))W\neq Z\circ \delta, E\in Eq(\omega)\},$$
$$At =P\cup T,$$
and
$$A=\{\bigcup X: X\subseteq At\}.$$

\begin{athm}{Claim 1} $A$ is a subuniverse of the full cylindric weak set algebra 
$$\langle \wp(V), + ,\cdot,  -, {\sf c}_i, {\sf d}_{ij}\rangle_{i,j\in \omega}.$$ 
Furthermore $\A$ is atomic and $At \A=At\sim \{0\}$. 
\end{athm}
Notice that the boolean operations of the algebra are denoted by $+$, $\cdot $, $-$ standing for Boolean join (union), Boolean meet (intersection)
and complementation, 
respectively.

{\bf Proof of Claim 1}. Let $b = {\sf P} Z  \sim {\sf d}_{01}$. Then 
\begin{enumarab}
\item $a\cdot {\sf S}_{[0,1]} a = 0 $, $a + {\sf S}_{[0,1]} a = b $, $(\forall i \in \omega) {\sf
c}_i a = {\sf c}_i {\sf S}_{[0,1]} a ={\sf c}_i b.$

It is not difficult to check that (1) holds. One can check first $ B
= {\sf S}_{[0,1]} R$, $B\cdot R = 0, B + R = {}^2 3 - {\sf d}_{01}, (\forall i\in 2){\sf c}_i R =  {\sf c}_i B =  {\sf c}_i {}^2 3, X\cdot Y = 0, 
X \cup Y = t,$ and $(\forall i \in \omega\sim 2){\sf c}_i X = {\sf c}_i Y = {\sf c}_i t$.
From (1) we
immediately get

\item ${\sf P}Z = a + {\sf S}_{[0,1]} a + d$. For, ${\sf P} Z = {\sf P} Z \sim {\sf d}_{01} +
{\sf P} Z\cdot  {\sf d}_{01} = b + d.$\\

\item ${\sf S}_{\delta} {\sf P}W = {\sf P}(W \circ \delta^{-1})$  for every $
\delta \in \pi (\omega)$ and $W \in {}^\omega (Rg Z)^{(Z)}$.\\

Indeed, we have $ s \in{\sf S}_{\delta} {\sf P}W$ iff $ s \circ  \delta \in {\sf P}W$ iff $ s
\circ \delta_i \in W_i  \quad \forall i \in \omega $ iff $ s_j \in
W_{\delta^{-1}_j} \quad \forall j \in \omega $ iff $s \in {\sf P}(W \circ
\delta^{-1})$.

\item  ${\sf P}W \in A$ for every $W \in {}^\omega (Rg Z)^{(Z)}$.

Assume $W = Z \circ \delta^{-1}$ for some $ \delta \in \pi(\omega).$
Then ${\sf P}W = {\sf S}_\delta {\sf P}Z = {\sf S}_\delta a + {\sf S}_{\delta
\circ [0,1]} a + {\sf S}_\delta d \in A$ by (3) and (2). Assume $ W
\neq Z \circ \delta,\quad  \forall \delta \in \pi(\omega).$ Then by
$ V= \sum \{ e (E) : E \in Eq(\omega)\}$ we have
${\sf P}W = \sum \{ {\sf P}W\cdot e(E) : E \in Eq(\omega) \} \in A$.\\

\item $(\forall x, y\in At)( x \neq y \Rightarrow x\cdot y = 0)$ and $
{}V = \sum At$.

If $E\neq E'$,$ E, E' \in Eq(\omega)$ then $e(E) \cap e(E') = 0$ and
if $W\neq W'$, $W,W' \in {}^\omega Rg Z^{(Z)} $ then ${\sf P}W \cap {\sf P}W' = 0$.
Thus the elements of $T$ are disjoint from each other and from the
elements of $P$ since $(\forall x\in P)  (\exists \delta \in \pi(\omega))x \subseteq
{\sf S}_\delta {\sf P}Z = {\sf P}(Z\circ \delta^{-1})$ by
(3). Let $\delta, \delta' \in \pi(\omega)$. Clearly ${\sf S}_\delta
a \cdot {\sf S}_{\delta'} d = 0$ since ${\sf S}_\delta' a \subseteq
\prod \{ - {\sf d}_{ij} : i < j < \omega \}$ while ${\sf
S}_{\delta'} d \subseteq {\sf d}_{\delta'0 \delta'1}.$ Let $ y \in
\{ a , d\}$ and assume $ \delta' \neq \delta$. If $\delta' \neq
\delta\circ [0,1]$ then \footnote{For, we show $\delta \neq \delta'$ and $ Z
\circ \delta ^{-1} = Z \circ \delta'^{-1} $ imply $ \delta\ = \delta
\circ [0,1]$. Let $k \in \omega \sim 2$ and $ j = \delta
k$. Then $\delta^{-1}_j = k \notin 2$, hence $Z \delta^{-1}_j + Z
\delta'^{-1}_j$ implies $k = \delta^{-1}_j = \delta'^{-1}_j$, i.e.,
$\delta, k = j$. We have seen $ \delta \upharpoonright \omega
\sim 2 \subseteq \delta'$. By this and by $\delta \neq
\delta'$ we have $ \delta 0 = \delta' 1$ and $ \delta 1 = \delta'
0$. Thus $ \delta = \delta' \circ [0,1]$.}$ Z \circ \delta'^{-1} \neq Z \circ \delta^{-1}$ hence ${\sf P}(Z \circ
\delta^{-1}) \cap {\sf P}(Z \circ\delta'^{-1})=0$, thus ${\sf S}_{\delta
}y \cdot {\sf S}_{\delta'}y= 0$ since $  {\sf S}_{\sigma} y \subseteq {\sf S}_{\sigma}{\sf P}Z =
{\sf P}(Z\circ \sigma^{-1})\quad \forall \sigma \in \pi(\omega)$ by (3).
If $\delta' = \delta\circ[0,1]$ then $ {\sf S}_{\delta}a\cdot {\sf
S}_{\delta'}a = {\sf S}_{\delta}(a\cdot {\sf S}_{[0,1]}a) = 0$ by (1)
and ${\sf S}_{\delta} d = {\sf S}_{\delta} {\sf S}_{[0,1]} d = {\sf
S}_{\delta'} d$. Thus all the elements of $At$ are disjoint from
each other. By $ U = \bigcup RgZ$ we have $V = \sum \{ {\sf P}W
: W \in {}^\omega Rg Z^{(Z)} \} \subseteq \sum At$ by (4). Thus $V= \sum At$.\\

\item  $A$ is closed under the boolean operations.

For, (6) is an immediate corollary of (5) and the definition of $A$.
\item 
Let $\M$ denote the minimal subalgebra of $\wp(V)$, i.e.,
$\M = \Sg^{(\wp(V))}0.$ Then
$\M \subseteq A$.

Let $i < j < \omega$. Then $ {\sf d}_{ij} = \sum \{ e(E) : (i, j)
\in E, E \in Eq(\omega) \} = \sum (\{{\sf P}W\cdot  e(E) : W \in {}^\omega Rg
Z^{(Z)}$, $ W \neq Z\circ \delta \quad \forall \delta \in \pi (\omega)$,
$(i,j) \in E \in Eq(\omega) \} \bigcup \{ {\sf S}_{\delta} d  :
\delta \in \pi(\omega), \{ \delta0, \delta1\} = \{i, j\} \} ) \in A$.
Let $ k <\omega$. Then $ {\sf c}_{(k)} \bar{d}(k \times k) \in
\{ 0, V \} \subseteq A$ by (5). Thus by \cite{HMT1}
[2.2.24], and (6) we have $\M = \Sg^{(\wp(V))} \{ {\sf
d}_{ij}
: i < j < \omega \} \subseteq A$.\\

\item ${\sf P}W \in A$ for every $ W \in {}^\omega ( Rg Z \cup \{U\})^{(Z)}.$

Let $ \Im = \{ i \in \omega : W_i \neq U \}$. Then $ {\sf P}W = \sum \{
{\sf P}W' : W' \in {}^\omega Rg Z, W' \upharpoonright \Im \subseteq W
\}$ by $ U = \bigcup Rg Z$. Thus ${\sf P}W \in A$ by (4).\\

\item ${\sf S}_{\tau}{\sf P}W\in A$ for every $W\in {}^{\omega}RgZ^{(Z)}.$

For if ${\sf S}_{\tau}{\sf P}W=0$, then we are done. Assume that ${\sf S}_{\tau}{\sf P}W\neq 0$.
Let $z\in {\sf S}_{\tau}{\sf P}W$ be arbitrary. Let $\eta\in {}^{\omega}\omega$ such that $z_i\in Z_{\eta(i)}.$ Such an $\eta$ exists by $U=\bigcup RgZ$.
Now we have
$$(*)\ \ \  (\forall i\in \omega)W_i=W_{\eta\tau(i)}.$$ 
since $(\forall i\in \omega)z\tau(i)\in W_i\cap W_{\eta\tau(i)}$ by $z\in {\sf S}_{\tau}{\sf P}W$
and by the definition of $\eta$, hence $W_i=W_{\eta\tau(i)}$ since the elements of $RgZ$ are disjoint from each other.
Let $sup\tau=\{i\in \omega: \tau(i)\neq i\}$. Let $W'\in {}^{\alpha}RgZ^{(Z)}$ be defined by
$(\forall i\in \omega\sim sup\tau)W_i'=W_i$ and for all $(\forall i\in sup\tau)W_i'=W_{\eta(i)}$.
Then $${\sf S}_{\tau}{\sf P}W=\{s\in V: (\forall i\in \omega)s\tau(i)\in W_i\}={\sf P}W'.$$

\item $ {\sf S}_\tau x \in A$ for every $ x \in A$.

It is enough to show (10) for $ x \in At$ since $ {\sf S}_\tau$ is
additive. If $ \tau, \delta \in \pi(\omega)$ then ${\sf S}_\tau {\sf
S}_\delta a = {\sf S}_{\tau \circ \delta} a \in P \subseteq A$ since
$\tau \circ \delta \in  \pi(\omega)$. If $\tau \in {}FT_\omega
\sim  \pi(\omega)$ then ${\sf S}_\tau {\sf S}_\delta a = 0
\in A$. Note that ${\sf S}_{\delta}d=P(Z\circ \delta^{-1})\cdot {\sf d}_{\delta 0, \delta 1}$. By (9) and the above, to finish the proof (10), it is enough to show $
{\sf S}_\tau g \in A$ for all $g$ of the form ${\sf P}W\cdot e(E)$ since ${\sf S}_\tau $ is
a boolean homomorphism.  Let $g={\sf P}W\cdot e(E)$. Then by 
$$e(E)=\prod\{{\sf d}_{ij}:(i,j)\in E\}\cdot \prod \{-{\sf d}_{ij}: (i,j)\notin E\},$$
there exists a finite $K\subseteq \{{\sf d}_{ij}:  i<j<\omega\}\cup\{-{\sf d}_{ij}: i<j<\omega\}$
such that $g={\sf P}W\cdot \prod K$. Here we are using that there exists $n\in \omega$ such that ${\sf P}W\subseteq -{\sf d}_{ij}$ for all
$n\leq i<j$ since the elements of $RgZ$ are disjoint from each other and $W\in {}^{\omega}RgZ^{(Z)}$. 
The rest follows from (9), the fact that ${\sf S}_{\tau}$ is a Boolean homomorphism and  
that  $ {\sf S}_\tau {\sf d}_{ij} = {\sf d}_{\tau i \tau j} \in A$.\\

\item ${\sf c}_i x \in A$ for every $x \in A$ and $ i \in \omega$.

It is enough to show (11) for $ x \in At$ since ${\sf c}_i$ is
additive. Now ${\sf c}_i {\sf S}_\delta a = {\sf S}_\delta {\sf
c}_{\delta_i} a$. Indeed let $j = \delta_i$. Then ${\sf c}_j a = {\sf c}_j
b = {\sf P}Z(j / U) \cdot \gamma$ where $ \gamma = 1$ if $ j \in 2$ and $
\gamma = - {\sf d}_{01}$ if $ j \in \omega \sim 2$. Thus $
{\sf c}_i {\sf S}_\delta a \in A$ by (10), (8) and (7). Let $ x \in
At \sim P'$. Then $ x = {\sf P}W\cdot  \prod K$ for some $ W \in
{}^\omega (Rg Z \cup \{ U\})^{(Z)}$ and $ K \subseteq_{\omega} \{ {\sf d}_{ij} : i<
j < \omega \} \cup \{ - {\sf d}_{ij} : i< j < \omega \}$.
Assume ${\sf P}W\cdot  \prod K \neq 0$. We will show $ {\sf c}_i ({\sf P}W\cdot  \prod K )
= {\sf c}_i {\sf P}W \cdot  {\sf c}_i \prod K $. Let $ \Gamma = \{ j \in \omega
: {\sf d}_{ij} \in K \} $ and $ \Omega = \{ j \in \omega : - {\sf
d}_{ij} \in K \}.$ It is enough to show $  {\sf c}_i {\sf P}W \cdot {\sf c}_i
\prod K \subseteq {\sf c}_i ({\sf P}W\cdot  \prod K )$. Let $ s \in {\sf c}_i
{\sf P}W \cdot  {\sf c}_i \prod K$. Assume $ \Gamma \neq 0$. Let $ j \in
\Gamma$. Then $ s(i / s_j)\in  {\sf P}W\cdot \prod K$ since $ W_i = W_j$ by
${\sf P}W\cdot  \prod K \neq 0$. Assume $\Gamma = 0$. Let $ \Delta = \{ j \in
\Omega : W_j = W_i \}$. Then $|\Delta | < |W_i|$ by ${\sf P}W\cdot \prod K \neq
0$. Let $u \in W_i \sim \{ s_i : j \in \Delta \}$. Then $
s(i /u) \in {\sf P}W \cdot  \prod K$. Thus  $ {\sf c}_i ({\sf P}W\cdot  \prod K ) = {\sf
c}_i {\sf P}W \cdot  {\sf c}_i \prod K  = {\sf P}W(i /U) \cdot  {\sf c}_i \prod K \in A $
by (8) and (7).

By (6), (7) and (11) we have proved $ A \in Su \wp(V)$. 
 ($A$ is a subuniverse of $\wp(V)).$ By
(5) then we have $At \A = At \sim \{0\}$

{\bf The construction of $\B \in \QPEA_\omega$:}\\

Let $ \tau, \delta \in FT_{\omega}$. We say that $``\tau,
\delta$ transpose" iff $( \delta0-\delta1).(\tau \delta 0 - \tau
\delta 1)$ is negative.

Now we first define $ {\sf s}_\sigma : At \rightarrow A$ for every $ \sigma
\in FT_{\omega}$. 
\[ {\sf s}_\sigma({\sf S}_\delta a )  =
\begin{cases}
{\sf S}_{\sigma \circ \delta \circ [0,1] } a& \textrm{if}~~~``
\sigma, \delta
~~~\textrm{transpose}" \\
{\sf S}_{\sigma \circ \delta} a & \textrm{otherwise}
\end{cases}\]

$$ {\sf s}_\sigma (x) = {\sf S}_\sigma x \quad \textrm{if} \quad x \in At
\sim P'.$$
Then we set:
$${\sf s}_\sigma (\sum X) = \sum \{ {\sf s}_\sigma (x) : x \in X \} \quad \textrm{for}
\quad X \subseteq At.$$
We shall first prove that ${\sf s}_{\sigma}:A\to A$.

\item From the definition of ${\sf s}_\sigma$ we immediately get
${\sf s}_\sigma {\sf S}_\delta a \in \{ {\sf S}_{\sigma \circ
\delta} a, {\sf S}_{\sigma \circ \delta \circ [0,1] } a \}$ for
$\delta \in
\pi (\omega )$.\\

\item ${\sf s}_\sigma x = {\sf S}_\sigma x$ for $ \sigma \in
FT_{\omega} \sim \pi (\omega )$ and $ x \in At$.

If $ \tau \in FT_{\omega} \sim \pi (\omega )$ then $ {\sf S}_\tau a = 0$,
hence $ {\sf s}_\sigma {\sf S}_\delta a = 0 = {\sf S}_\sigma {\sf
S}_\delta a$ by (12). For $ x \in At \sim P'$ we have
${\sf s}_\sigma x = {\sf S}_\sigma x$ by definition.\\

\item  ${\sf s}_\sigma : At \rightarrow At$, is a bijection for $\sigma\in \pi(\omega)$

By (12) we have $ {\sf s}_\sigma : P' \rightarrow P'$. Assume
$\delta \neq \delta'$, $\delta, \delta' \in \pi(\omega)$. If $\delta
\neq \delta' \circ [0, 1]$ then $\{ \sigma \circ \delta, \sigma
\circ \delta \circ [0, 1] \} \cap \{ \sigma \circ \delta', \sigma
\circ \delta' \circ [0, 1] \} = 0$, hence ${\sf s}_\sigma {\sf
S}_\delta a \neq {\sf s}_\sigma {\sf S}_{\delta'} a$. Assume $
\delta = \delta' \circ [0, 1]$. In this case $``\sigma, \delta$
transpose" iff $``\sigma, \delta'$ transpose", hence $ {\sf
s}_\sigma  {\sf S}_\delta a = {\sf S}_{\sigma \circ \delta \circ [0,
1]} a \neq {\sf S}_{\sigma \circ \delta' \circ [0, 1]} a = {\sf
s}_\sigma {\sf S}_{\delta'} a $, by \footnote{This follows from the
proof of (5).}
$ [ \tau \neq \tau' \Rightarrow {\sf S}_\tau a \neq {\sf S}_{\tau'}
a]\quad \forall \tau, \tau' \in \pi(\omega)$. We have seen that ${\sf
s}_\sigma : P' \rightarrow P'$. Let $ \tau = \sigma^{-1} \circ
\delta$. Define $ \tau' = \tau \circ [0, 1]$ if $``\sigma, \tau$
transpose", $\tau' = \tau$ otherwise. Then $``\sigma, \tau$
transpose" iff $``\sigma, {\tau'}$ transpose", hence ${\sf s}_\sigma
{\sf S}_{\tau'} a = {\sf S}_{\sigma \circ \tau} a = {\sf S}_{\delta}
a$. Thus $ {\sf s}_\sigma : P' \rightarrow P'$ is onto. By $
{\sf s}_\sigma {\sf S}_\delta d = {\sf S}_\sigma {\sf S}_\delta d$
then we have $ {\sf s}_\sigma : (P \sim P')
\rightarrow (P \sim P')$. Next we show ${\sf
s}_\sigma : T \rightarrow T$ is a bijection.
Let $E \in Eq(\omega)$. Define $E(\tau) = \{ (\tau i, \tau j) : (i,
j) \in E \}$ for any $ \tau \in FT_{\omega}$. Then it is not
difficult to check that by $\sigma \in \pi (\omega)$ we have
$E(\sigma) \in Eq(\omega)$ and $ ( ker ( s \circ \sigma) = E$ iff $
Ker s = E(\sigma))$. Thus $ {\sf S}_\sigma e(E) = e(E(\sigma))$. Now
$ {\sf s}_\sigma({\sf P}W\cdot  e(E)) = {\sf S}_\sigma ({\sf P}W\cdot e(E)) = {\sf
S}_\sigma {\sf P}W\cdot  {\sf S}_\sigma e(E) = {\sf P}(W \circ \sigma^{-1})\cdot 
e(E(\sigma)) \in T$ if $W \neq Z \circ \delta $ for any $ \delta \in
\pi(\omega)$. If $ W \neq W'$ or $E \neq E'$ then $W\circ
\sigma^{-1} \neq W' \circ \sigma^{-1}$ or $ E(\sigma ) \neq
E'(\sigma)$, thus ${\sf s}_\sigma : T \rightarrow T$ is one to one. The fact that ${\sf
s}_\sigma({\sf  P}(W \circ  \sigma^{-1}) \cdot e(E( \sigma^{-1}))) = {\sf P}W\cdot 
e(E)$ shows that ${\sf s}_\sigma : T
\rightarrow T$ is onto.\\

Now we have proved that $ {\sf s}_\sigma : A \rightarrow A.$ Define 
$$\B = \langle A, +,
\cdot , -, 0, 1, {\sf c}_{i}, {\sf s}_{\tau}, {\sf d}_{ij}\rangle_{i,j\in 
\omega, \tau \in FT_{\omega}}. $$

\begin{athm}{Claim 2} $\B \in \QPEA_\omega$
\end{athm}

We shall proceed via several steps.

\item  ${\sf s}_\tau$ is a boolean homomorphism on $\A$, for any $\tau
\in FT_{\omega}$.

If $\tau \in \pi (\omega)$ then (15) follows from (14) and from the
definition of ${\sf s}_\sigma$. If $\tau \in FT_{\omega}
\sim \pi (\omega)$ then (15) follows from (13).\\

\item  ${\sf s}_\tau {\sf s}_\sigma {\sf S}_\delta  a = {\sf s}_{\tau
\circ \sigma} {\sf S}_\delta a$ for $\delta \in \pi (\omega)$,
$\tau, \sigma \in
FT_{\omega}$.\\

Assume $\delta 0 < \delta 1$.\\
Case 1: $ \tau \sigma \delta 0 < \tau \sigma \delta 1$. Then $ {\sf
s}_{ \tau \circ \sigma } {\sf S}_\delta  a = {\sf S}_{\tau \circ
\sigma \circ\delta}  a$ since $`` \tau \circ \sigma, \delta$ do not
transpose ". If $\sigma \delta 0 < \sigma \delta 1$ then $`` \sigma,
\delta$ do not transpose " and  $``\tau , \sigma \circ \delta$ do
not transpose ", hence $ {\sf s}_\sigma {\sf S}_\delta a = {\sf
S}_{\sigma \circ\delta} a$ and ${\sf s}_\tau {\sf S}_{\sigma
\circ\delta}a = {\sf S}_{\tau \circ\sigma \circ\delta} a$ and we are
done. Similarly, if $\sigma \delta 0 > \sigma \delta 1$ then ${\sf
s}_\sigma {\sf S}_\delta a= {\sf S}_{\sigma \circ \delta \circ
[0,1]} a$ and $ {\sf s}_\tau {\sf S}_{\sigma \circ \delta \circ
[0,1]} a = {\sf S}_{\tau \circ (\sigma \circ \delta \circ
[0,1])\circ [0,1]} a = {\sf S}_{\tau
\circ \sigma \circ \delta } a$ and we are done.\\

Case 2: $ \tau \sigma \delta 0 > \tau \sigma \delta 1$. Then $ {\sf
s}_{ \tau \circ \sigma } {\sf S}_\delta  a = {\sf S}_{\tau \circ
\sigma \circ\delta \circ [0,1]}  a$. If $  \sigma \delta 0 <  \sigma
\delta 1$ then ${\sf s}_\sigma {\sf S}_\delta a = {\sf S}_{\sigma
\circ \delta } a$ and ${\sf s}_\tau {\sf S}_{\sigma \circ \delta } a
= {\sf S}_{\tau \circ \sigma \circ \delta \circ [0,1]} a$ and we are
done. If $\sigma \delta 0 >  \sigma \delta 1$ then ${\sf s}_\sigma
{\sf S}_\delta a = {\sf S}_{\sigma \circ \delta \circ [0,1]} a$ and
${\sf s}_\tau {\sf S}_{\sigma \circ \delta \circ [0,1]} a = {\sf
S}_{\tau \circ\sigma
\circ \delta \circ [0,1]} a$.\\

The case $\delta 0 > \delta 1$ is completely analogous, hence we
omit it.\\

\item $ {\sf s}_\sigma ({\sf S}_\delta a + {\sf S}_{\sigma \circ
[0,1]} a) = {\sf S}_\sigma ({\sf S}_\delta a + {\sf S}_{\sigma \circ
[0,1]} a)$.

$`` \sigma, \delta$ transpose " iff  $``\sigma, \delta \circ [0,1]$
transpose ". Hence $\{ {\sf s}_\sigma {\sf S}_\delta a, {\sf
s}_\sigma {\sf S}_{\delta \circ [0,1]} a \} = \{ {\sf S}_{\sigma
\circ \delta} a, {\sf S}_{\sigma \circ \delta \circ [0,1]} a \}$ by
the definition of
$ {\sf s}_\sigma$.\\

\item $ {\sf s}_\tau {\sf s}_\sigma x = {\sf s}_{\tau \circ \sigma}
x$ for every $ \tau, \sigma \in FT_{\omega}$ and $x \in A$.

It is enough to show (18) for $x \in At$. For $x \in P'$, (18) is
true by (16). Let $ x \in At \sim P'$. Then $ {\sf
s}_\sigma x = {\sf S}_\sigma x$ by definition. Now $ {\sf S}_\sigma
x = {\sf P}W\cdot  \prod K$ for some $ W \in {}^\omega (Rg Z \cup \{U\})^{(Z)}$ and $
K \subseteq_{\omega} \{ {\sf d}_{ij} : i< j < \omega \} \cup \{ - {\sf
d}_{ij} : i< j < \omega \}$, by the proof of (10). Assume ${\sf
S}_\delta a \subseteq {\sf S}_\sigma x$ for some $ \delta \in \pi
(\omega)$. We will show that then $ {\sf S}_{\delta \circ [0,1]} a
\subseteq {\sf S}_\sigma x$, too. $ {\sf S}_\delta a \cup {\sf
S}_{\delta \circ [0,1]} a \subseteq {\sf P}Z \circ \delta^{-1} \cdot
e(Id_\omega)$, thus ${\sf S}_\delta a \leq {\sf S}_\sigma x$ implies
$[{\sf P}Z \circ \delta^{-1}\cdot e(Id_\omega) ] \cap {\sf P}W\cdot  \prod K \neq 0$.
But then ${\sf P}Z \circ \delta^{-1}\cdot  e(Id_\omega) \subseteq {\sf P}W\cdot  \prod
K$, thus $ {\sf S}_{\delta \circ [0,1]} a \subseteq {\sf S}_\sigma
x$, too. Thus ${\sf s}_\tau {\sf s}_\sigma x = {\sf S}_\tau {\sf
S}_\sigma x = {\sf S}_{\tau \circ \sigma } x = {\sf s}_{\tau \circ
\sigma } x$ by (17) and by the
definition of ${\sf s}_\tau, {\sf s}_\sigma, {\sf s}_{\tau \circ \sigma }$.\\

\item  $ {\sf c}_{(\Gamma)} {\sf S}_\delta a = {\sf c}_{(\Gamma)}
({\sf S}_\delta a + {\sf S}_{\delta \circ [0,1]} a) $ if $ \Gamma\subseteq_{\omega}\omega$, $\Gamma
\neq 0$.

Let $ i \in \omega$ be arbitrary. Then $ {\sf c}_i a =  {\sf c}_i (a
+ {\sf S}_{ [0,1]} a )$ holds by (1). Thus ${\sf c}_i {\sf S}_\delta
a = {\sf S}_\delta {\sf c}_{\delta_i} a = {\sf S}_\delta {\sf
c}_{\delta_i} ( a + {\sf S}_{ [0,1]} a) = {\sf c}_i  ( {\sf S}_\delta a + {\sf S}_{\delta \circ [0,1]} a)$.\\

\item $ {\sf s}_\sigma {\sf c}_{(\Gamma)} x = {\sf S}_\sigma {\sf
c}_{(\Gamma)} x$ for every $ x \in A$ if $ \Gamma\subseteq_{\omega}\omega,\Gamma \neq 0$.

Let $x \in A$ be arbitrary. Then $ {\sf c}_{(\Gamma)} x = \sum X$
for some $X \subseteq At$, by (11). Assume  ${\sf S}_\delta a \in
X$. Then $ {\sf c}_{(\Gamma)} {\sf S}_\delta a \subseteq  {\sf
c}_{(\Gamma)} x$, hence $  {\sf S}_{\delta \circ [0,1]} a \subseteq
{\sf c}_{(\Gamma)} x$ by (19). Therefore $ {\sf S}_{\delta \circ
[0,1]} a \in X$, too. Now $ {\sf s}_\sigma {\sf c}_{(\Gamma)} x =
\sum \{
{\sf s}_\sigma y : y \in X \}$ and (17) finish the proof of (20).\\

\item  $ \sigma \upharpoonright (\omega \sim \Gamma) = \tau
\upharpoonright (\omega \sim \Gamma) \Rightarrow {\sf
s}_\sigma {\sf c}_{(\Gamma)} x = {\sf s}_\tau {\sf c}_{(\Gamma)} x$
for every $x \in A$, $ \sigma, \tau \in FT_{\omega},$ and
$\Gamma \subseteq_{\omega} \omega$.

(21) follows from (20).\\

\item $ {\sf c}_{(\Gamma)} {\sf s}_\sigma x = {\sf c}_{(\Gamma)} {\sf
S}_\sigma x$, for every $x \in A$, $ \sigma \in FT_{\omega},$
if $ \Gamma\subseteq_{\omega}\omega, \Gamma \neq 0$.

It is enough to check (22) for $x \in P'$. Let $\delta \in \pi
(\omega )$. Then $ {\sf s}_\sigma  {\sf S}_\delta a \in \{  {\sf
S}_{\sigma \circ \delta} a, {\sf S}_{\sigma \circ \delta \circ
[0,1]} a \}$ by definition of ${\sf s}_\sigma$ and ${\sf
c}_{(\Gamma)} {\sf S}_{\sigma \circ \delta \circ [0,1]} a = {\sf
c}_{(\Gamma)} {\sf S}_{\sigma \circ \delta } a = {\sf c}_{(\Gamma)}
{\sf S}_\sigma
{\sf S}_\delta a $ by (19).\\

\item $ \tau \upharpoonright (\tau^{-1}\Gamma)$ is one - one then $ {\sf
c}_{(\Gamma)} {\sf s}_\tau x =  {\sf s}_\tau {\sf c}_{(\Delta)} x$
where $\Delta = \tau^{-1}\Gamma$.

If $\Gamma = 0$ then $\Delta = 0$ and we are done. If $\Gamma \neq
0$ and $\Delta = 0$ then $\pi \notin \pi(\omega)$ hence we are done
by (13). Assume $\Gamma \neq 0, \Delta \neq 0$. Then we are done by
(22) and (20).\\

Now we are ready to show $\B \in \QPEA_\omega$. We have to show that
$(1-15)$ in definition of polyadic equlaity algebras in \cite{HMT2} are
satisfied in $\B$. $(1 - 6) + 13$ are satisfied since $\Rd_{ca}
\B \in Ws_\omega$. $7$ holds because $``Id_\omega, \delta$ don't
transpose" $\forall \delta \in FT_{\omega}$. $8, 11, 12$ hold by (18), (22), (23) respectively. $9-10$ are
satisfied by (15). $14$ holds by (13) and $15$ holds since
${\sf s}_\tau {\sf d}_{ij} = {\sf S}_\tau {\sf d}_{ij}$ by
definition of ${\sf s}_\tau$.

We finally show:

\begin{athm}{Claim 3} $\B \notin \RQPEA_\omega$.
\end{athm}
\begin{demo}{Proof} 
Assume $\B \in \RQPEA_\omega$. Then by theorem \ref{weak} $\B$
is isomorphic to some weak set algebra $\C$ since $\Rd_{ca}\B$ is weakly subdirectly indecomposable. Let $U'$ be the base of $\C$.
The unit of $\C$ is of the form $^{\alpha}U'^{(p)}$ for some sequence $p$. 
Let $h : \B \twoheadrightarrow \C$ be an isomorphism. Let $x=Z_0\times U\times U\times U\times Z_5\times Z_6\ldots$.
That is $x=\{s\in V: s_0\in Z_0: (\forall i>4)(s_i\in Z_i)\}.$ 
Then $ x \in A$ by (8), and
${\sf c}_ix=x$ for $i\in \{1,2,3\}$. So ${\sf c}_ih(x)=h(x)$ for $i\in \{1,2,3\}$, thus $ h(x) = Z'\times U'\times U'\times U'\times \ldots $.
for some $Z' \subseteq U'$. Let $\bar{x}= \prod \{ {\sf
s}_{[0,i]} x : i \in 4 \}$. Then $\bar{x}=Z_0\times Z_0\times Z_0\times Z_0\times Z_5\times Z_6\ldots.$ 
For a relation $R$, recall that $\bar{d}(R)=\prod_{(i,j)\in R\sim Id}-{\sf d}_{ij}$.
Then we have $\bar{x}\cdot  \bar{d}(3 \times 3) \neq
0$ and $ \bar{x}\cdot \bar{d}(4 \times 4) = 0$ imply the same for
$h(x)$, therefore $|Z'| = 3$.

Let $b' = h(b), a' = h(a), g = {\sf S}_{[0,1]} a, g' = h(g).$ Then $
b \leq x \cdot {\sf s}_{[0,1]}x - {\sf d}_{01}$ hence $ b' \subseteq h(x)
\cdot {\sf S}_{[0,1]} h(x) - {\sf d}_{01}$, thus $$ \forall s \in
b'\quad (s_0, s_1) \in {}^2 Z' \sim {\sf d}_{01}\quad \textrm{and}
\quad |Z'|= 3.\qquad (\star)$$

In $\A$ we have $ a+ g = b \neq 0, a \cdot g = 0, {\sf s}_{[0,1]} a = a,
{\sf s}_{[0,1]} g = g$ and $ {\sf c}_i a = {\sf c}_i g = {\sf c}_i b
\quad \forall i \in 2$.

Therefore
$$ (*) \quad a' + g' = b' \neq 0, a' \cdot  g' = 0$$
$$ (**) \quad {\sf S}_{[0,1]} a' = a' , {\sf S}_{[0,1]} g' = g' \quad \textrm{and} $$
$$ (***) \quad {\sf c}_i a' = {\sf c}_i g' = {\sf c}_i b' \quad
\forall i \in 2$$

Let $q \in b'$ be arbitrary. $q^{01}_{uv}$ is the function $q'$ that agrees with $q$ everywhere except that
$q'(0)=u$ and $q'(1)=v$. Define $$\bar{a} = \{ (u, v) : q^{01}_{uv}
\in a' \}$$ and 
$$\bar{g} = \{(u, v) : q^{01}_{uv} \in g' \}.$$ Then by $(*)-
(\star)$ we have
$$ (*)'\quad \bar{a} + \bar{g} = {}^2 Z' \sim {\sf d}_{01}, \bar{a}\cdot 
\bar{g} = 0,$$
$$ (**)'\quad {\sf S}_{[0,1]} \bar{a} = \bar{a} , {\sf S}_{[0,1]} \bar{g} = \bar{g}
 \quad \textrm{and} $$
$$ (***)' \quad {\sf c}_0 \bar{a} = {\sf c}_0 \bar{g} = {\sf c}_0 {}^2 Z'.$$

We show that $(*)' - (***)' $ together with $|Z'| = 3$ is
impossible. By $ (***)'$ we have $ Rg \bar{a} = Rg \bar{g} = Z'$,
hence $|\bar{a}| \geq 3$ and  $|\bar{g}| \geq 3$. By $(*')$ we have
then $ |\bar{a}| =  |\bar{g}| = 3$ by $ \bar{a}\cdot   \bar{g} = 0$ and
$|{}^2 Z'_1 \sim {\sf d}_{01}| = 6 $. But by $(**)'$ and $ \bar{a} \leq
-  {\sf d}_{01}$ we have $|\bar{a}| \geq 4$, contradiction.
\end{demo}
\end{enumarab}
Claims 1-3 prove Theorem \ref{q}.

\end{demo}
With an axiomatization of the finite dimensional representable algebras at hand, we can obtain a 
recursive axiomatization of the class $\RQPEA_{\alpha}$ for infinite $\alpha$, cf. \cite{HHbook} 
corollary 8.13.

\begin{theorem} Let $\alpha\geq \omega$ be an ordinal. Then $\RQPEA_{\alpha}$ is axiomatized by the set
$$\Sigma=\QPEA_{\alpha}\cup\{(\epsilon_n^d)^{\sigma}; d,n<\omega, d\geq 3, \sigma:d\to \alpha \text { is one to one }\}$$
where the $\epsilon_n^d$ are as in theorem \ref{polyadic}
\end{theorem}

\subsection{Algebras not closed under Dedekind completions}

In this section we construct an atomic representable polyadic algebra, such that the diagonal free reduct of its completion is not representable. 
We have obtained this result in theorem \ref{com} above, but here we present a simpler proof, that does not depend on the probabilistic graphs of Erdos.
The proof also substantialy simplifies Hodkinson's proof in \cite{Hod97}, although the technique used is also model-theoretic.
Also Hodkinson proves his result only for $\CA$'s; our proof covers more algebras like $\PA$'s and $\Df$'s. 
The base of the algebra, we construct,  will be a certain graph $M$ that is constructed as a limit of certain labelled graphs.
This $M$ is the heart and soul of our proof. From now on, $n$ is a finite ordinal $>2$.
Our notation is mostly standard. 
An ordinal is the set of all smaller ordinals; so for $n<\omega$, $n=\{0,1,\ldots n-1\}$. 
Maps are regarded 
formally as sets of ordered pairs. Thus, if $\theta$ is a map, we write 
$|\theta|$ for the cardinality of the set that
is $\theta$. We write $Dom(\theta)$, $Range(\theta)$ for the domain and range of $\theta$ respectively.
We write $Id_X$ for the identity map on $X$. $\wp(X)$ denotes the power set of $X$.

We write $\bar{a}, \bar{x}$ for sequences. A sequence (or tuple) $\bar{a}$ of elements of a 
set $X$, of length $n$, is formally
an element of the set $^nX$. We write $a_i$ for the $i$th element of this sequence, 
and $Range(\bar{a})$ for 
$\{a_0, \ldots a_{n-1}\}$. 
We may write $\bar{a}$ as $(a_0, \ldots a_{n-1})$.
If $\theta:X\to Y$ is a map, we write $\theta(\bar{a})$ for the sequence 
$(\theta(a_0)\ldots \theta(a_{n-1})\}\in {}^nY$.
If $\bar{a}, \bar{b}$ are $n$ sequences, we write $(\bar{a}\mapsto \bar{b})$ for the map $\{(a_i, b_i): i<n\}$. For 
$i<n$, we write $\bar{a}=_i\bar{b}$ if $a_j=b_j$ for all $j<n$ with $j\neq i.$
Fix finite $N\geq n(n-1)/2$. 
Throughout $\G$ will denote the graph $\G=(\N,E)$ with nodes $\N$ and $i,l$ is an edge i.e $(i,l)\in E$ if 
$0<|i-l|<N$.

\begin{definition}
A \textit{labelled graph} is an undirected graph $\Gamma$ such that
every edge ( \textit{unordered} pair of distinct nodes ) of $\Gamma$
is labelled by a unique label from $(\G \cup \{\rho\}) \times n$, where
$\rho \notin \G$ is a new element. The colour of $(\rho, i)$ is
defined to be $i$. The \textit{colour} of $(a, i)$ for $a \in \G$ is
$i$.
\end{definition}

We will write $ \Gamma (x, y)$ for the label of an edge $ (x, y)$ in
the labelled graph $\Gamma$. Note that these may not always be
defined: for example, $ \Gamma (x, x)$ is not.\\
If $\Gamma$ is a labelled graph, and $ D \subseteq \Gamma$, we write
$ \Gamma \upharpoonright D$ for the induced subgraph of $\Gamma$ on
the set $D$ (it inherits the edges and colours of $\Gamma$, on its
domain $D$). We write $\triangle \subseteq \Gamma$ if $\triangle$ is
an induced subgraph of $\Gamma$ in this sense.

\begin{definition}
Let $ \Gamma, \triangle$ be labelled graphs, and $\theta : \Gamma
\rightarrow \triangle$ be a map. $\theta$ is said to be a
\textit{labelled graph embedding}, or simple an \textit{embedding},
if it is injective and preserves all edges, and all colours, where
defined, in both directions. An \textit{isomorphism} is a bijective
embedding.
\end{definition}

Now we define a class $\GG$ of certain labelled graphs.

\begin{definition}
The class $\GG$ consists of all complete labelled graphs $\Gamma$ (possibly
the empty graph) such that for all distinct $ x, y, z \in \Gamma$,
writing $ (a, i) = \Gamma (y, x)$, $ (b, j) = \Gamma (y, z)$, $ (c,
l) = \Gamma (x, z)$, we have:\\
\begin{enumarab}
\item $| \{ i, j, l \} > 1 $, or
\item $ a, b, c \in \G$ and $ \{ a, b, c \} $ has at least one edge
of $\G$, or
\item exactly one of $a, b, c$ -- say, $a$ -- is $\rho$, and $bc$ is
an edge of $\G$, or
\item two or more of $a, b, c$ are $\rho$.
\end{enumarab}
Clearly, $\GG$ is closed under isomorphism and under induced
subgraphs.
\end{definition}

\begin{theorem}
There is a countable labelled graph $M\in \GG$ with the following
property:\\
$\bullet$ If $\triangle \subseteq \triangle' \in \GG$, $|\triangle'|
\leq n$, and $\theta : \triangle \rightarrow M$ is an embedding,
then $\theta$ extends to an embedding $\theta' : \triangle'
\rightarrow M$.
\end{theorem}

\begin{demo}{Proof} Two players, $\forall$ and $\exists$, play a game to build a
labelled graph $M$. They play by choosing a chain $\Gamma_0
\subseteq \Gamma_1 \subseteq\ldots $ of finite graphs in $\GG$; the
union of
the chain will be the graph $M.$
There are $\omega$ rounds. In each round, $\forall$ and $\exists$ do
the following. Let $ \Gamma \in \GG$ be the graph constructed up to
this point in the game. $\forall$ chooses $\triangle \in \GG$ of
size $< n$, and an embedding $\theta : \triangle \rightarrow
\Gamma$. He then chooses an extension $ \triangle \subseteq
\triangle^+ \in \GG$, where $| \triangle^+ \backslash \triangle |
\leq 1$. These choices, $ (\triangle, \theta, \triangle^+)$,
constitute his move. $\exists$ must respond with an extension $
\Gamma \subseteq \Gamma^+ \in \GG$ such that $\theta $ extends to an
embedding $\theta^+ : \triangle^+ \rightarrow \Gamma^+$. Her
response ends the round.
The starting graph $\Gamma_0 \in \GG$ is arbitrary but we will take
it to be the empty graph in $\GG$.
We claim that $\exists$ never gets stuck -- she can always find a suitable
extension $\Gamma^+ \in \GG$.  Let $\Gamma \in \GG$ be the graph built at some stage, and let
$\forall$ choose the graphs $ \triangle \subseteq \triangle^+ \in
\GG$ and the embedding $\theta : \triangle \rightarrow \Gamma$.
Thus, his move is $ (\triangle, \theta, \triangle^+)$.
We now describe $\exists$'s response. If $\Gamma$ is empty, she may
simply plays $\triangle^+$, and if $\triangle = \triangle^+$, she
plays $\Gamma$. Otherwise, let $ F = rng(\theta) \subseteq \Gamma$.
(So $|F| < n$.) Since $\triangle$ and $\Gamma \upharpoonright F$ are
isomorphic labelled graphs (via $\theta$), and $\GG$ is closed under
isomorphism, we may assume with no loss of generality that $\forall$
actually played $ ( \Gamma \upharpoonright F, Id_F, \triangle^+)$,
where $\Gamma \upharpoonright F \subseteq \triangle^+ \in \GG$,
$\triangle^+ \backslash F = \{\delta\}$, and $\delta \notin \Gamma$.
We may view $\forall$'s move as building a labelled graph $ \Gamma^*
\supseteq \Gamma$, whose nodes are those of $\Gamma$ together with
$\delta$, and whose edges are the edges of $\Gamma$ together with
edges from $\delta$ to every node of $F$. The labelled graph
structure on $\Gamma^*$ is given by\\
$\bullet$ $\Gamma$ is an induced subgraph of $\Gamma^*$ (i.e., $
\Gamma \subseteq \Gamma^*$)\\
$\bullet$ $\Gamma^* \upharpoonright ( F \cup \{\delta\} ) =
\triangle^+$.
Now $ \exists$ must extend $ \Gamma^*$ to a complete
graph on the same node and complete the colouring yielding a graph
$ \Gamma^+ \in \GG$. Thus, she has to define the colour $
\Gamma^+(\beta, \delta)$ for all nodes $ \beta \in \Gamma \backslash
F$, in such a way as to meet the conditions of definition 1. She
does this as follows. The set of colours of the labels in $ \{
\triangle^+(\delta, \phi) : \phi \in F \} $ has cardinality at most
$ n - 1 $. Let  $ i < n$ be a "colour"not in this set. $ \exists$
labels $(\delta, \beta) $ by $(\rho, i)$ for every $ \beta \in
\Gamma \backslash F$. This completes the definition of $ \Gamma^+$.
\\
It remains to check that this strategy works--that the conditions
from the definition of $\GG$ are met. But this is not so hard.

Now there are only countably many
finite graphs in $ \GG$ up to isomorphism, and each of the graphs
built during the game is finite. Hence $\forall$ may arrange to play
every possible $(\triangle, \theta, \triangle^+)$ (up to
isomorphism) at some round in the game. Suppose he does this, and
let $M$ be the union of the graphs played in the game. 
\end{demo}

We want to view $M$ as a classical structure, and for that we recall some rather elementary notions from model theory.
Recall the definition of the $n$-variable infinitary language
$L^n_{\infty \omega}$. We use variables $ x_0,\ldots x_{n-1}$. The
atomic formulas are $x_i = x_j$ for any $i, j < n$, and $R(\bar{x})$
for any $k$-ary $ R \in L$ and any $k$-tuple $\bar{x}$ of variables
taken form $ x_0,\ldots x_{n-1}$. If $\phi$ is an $L^n_{\infty
\omega}$-formula then so are $\neg \phi$ and $\exists x_i \phi$ for
$i < n$; and if $\Phi$ is a set of $L^n_{\infty \omega}$-formulas
then $\bigwedge \Phi$ and $\bigvee \Phi $ are also $L^n_{\infty
\omega}$-formulas. Of course, we write $\bigwedge \{ \phi, \psi \}$
as $\phi  \wedge \psi$, etc.
The logic $L^n_{\infty \omega}$ is given semantics in a model $A$ in the
usual way, defining $ A \models \phi(\bar{a})$ for an $n$-tuple
$\bar{a}$ of elements of $A$ by induction on the formula $\phi$.

Let $L^n$ denote the first-order fragment of $L^n_{\infty \omega}$

\begin{definition}
An $n$-\textit{back-and-forth system} on $A$ is a set $\Theta$ of
one-to-one partial maps : $ A \rightarrow A$ such that:\\
\begin{enumerate}
\item if $\theta \in \Theta$ then $|\theta| \leq n$
\item if $ \theta' \subseteq \theta \in \Theta$ then $\theta' \in \Theta$
\item if $\theta \in \Theta$, $|\theta| \leq n$, and $ a \in A$,
then there is $ \theta' \supseteq \theta$ in $\Theta$ with $ a \in
Dom(\theta')$ (forth)
\item if $\theta \in \Theta$, $|\theta| \leq n$, and $ a \in A$,
then there is $ \theta' \supseteq \theta$ in $\Theta$ with $ a \in
Range(\theta')$ (back).

\end{enumerate}
\end{definition}

 Recall that a \textit{partial isomorphism} of $A$ is a
partial map $ \theta : A \rightarrow A$ that preserves all
quantifier-free $L$-formulas.

\begin{theorem} 
Let $\Theta$ be an $n$-back-and-forth system
of partial isomorphism on $A$, let $\bar{a}, \bar{b} \in {}^{n}A$,
and suppose that $ \theta = ( \bar{a} \mapsto \bar{b})$ is a map in
$\Theta$. Then $ A \models \phi(\bar{a})$ iff $ A \models
\phi(\bar{b})$, for any formula $\phi$ of $L^n_{\infty \omega}$.
\end{theorem}
\begin{demo}{Proof} By induction on the structure of $\phi$. 
\end{demo}
Suppose that $W \subseteq {}^{n}A$ is a given non-empty set. We can
relativise quantifiers to $W$, giving a new semantics $\models_W$
for $L^n_{\infty \omega}$, which has been intensively studied in
recent times. (see,e.g. \cite{AJN}.) If $\bar{a} \in W$:
\begin{itemize}
\item for atomic $\phi$, $A\models_W \phi(\bar{a})$ 
iff $A \models \phi(\bar{a})$

\item the boolean clauses are as expected

\item for $ i < n, A \models_W \exists x_i \phi(\bar{a})$ iff $A \models_W
\phi(\bar{a}')$ for some $ \bar{a}' \in W$ with $\bar{a}' \equiv_i
\bar{a}$.
\end{itemize}

\begin{theorem}\label{atom} If $W$ is $L^n_{\infty \omega}$ definable, $\Theta$ is an
 $n$-\textit{back-and-forth} system
of partial isomorphisms on $A$, $\bar{a}, \bar{b} \in W$, and $
\bar{a} \mapsto \bar{b} \in \Theta$, then $ A \models \phi(\bar{a})$
iff $ A \models \phi(\bar{b})$ for any formula $\phi$ of
$L^n_{\infty \omega}$.
\end{theorem}
\begin{demo}{Proof} Assume that $W$ is definable by the $L^n_{\infty \omega}$
formula $\psi$, so that $W = \{ \bar{a} \in {}^{n}A:A\models \psi(a)\}$. We may
relativise the quantifiers of $L^n_{\infty \omega}$-formulas to
$\psi$. For each $L^n_{\infty
\omega}$-formula $\phi$ we obtain a relativised one, $\phi^\psi$, by
induction, the main clause in the definition being:
\begin{itemize}
\item $( \exists x_i \phi)^\psi = \exists x_i ( \psi \wedge
\phi^\psi)$.
\end{itemize}
 Then clearly, $ A \models_W \phi(\bar{a})$ iff $ A \models
 \phi^\psi(\bar{a})$, for all $ \bar{a} \in W$.
 \end{demo}

\begin{definition} Let $L^+$ be the signature consisting of the binary
relation symbols $(a, i)$, for each $a \in \G \cup \{ \rho \}$ and
$ i < n$. Let $L = L^+ \setminus \{ (\rho, i) : i < n \}$. From now
on, the logics $L^n, L^n_{\infty \omega}$ are taken in this
signature.
\end{definition}

We may regard any non-empty labelled graph equally as an
$L^+$-structure, in the obvious way. The $n$-homogeneity built into
$M$ by its construction would suggest that the set of all partial
isomorphisms of $M$ of cardinality at most $n$ forms an
$n$-back-and-forth system. This is indeed true, but we can go
further.

\begin{definition}
Let $\chi$ be a permutation of the set $\omega \cup \{ \rho\}$. Let
$ \Gamma, \triangle \in \GG$ have the same size, and let $ \theta :
\Gamma \rightarrow \triangle$ be a bijection. We say that $\theta$
is a $\chi$-\textit{isomorphism} from $\Gamma$ to $\triangle$ if for
each distinct $ x, y \in \Gamma$,
\begin{itemize}
\item If $\Gamma ( x, y) = (a, j)$ with $a\in \N$, then there exist unique $l\in \N$ and $r$ with $0\leq r<N$ such that
$a=Nl+r$. 
\begin{equation*}
\triangle( \theta(x),\theta(y)) =
\begin{cases}
( N\chi(i)+r, j), & \hbox{if $\chi(i) \neq \rho$} \\
(\rho, j),  & \hbox{otherwise.} 
\end{cases}
\end{equation*}
\end{itemize}

\begin{itemize}
\item If $\Gamma ( x, y) = ( \rho, j)$, then
\begin{equation*}
\triangle( \theta(x),\theta(y)) \in
\begin{cases}
\{( N\chi(\rho)+s, j): 0\leq s < N \}, & \hbox{if $\chi(\rho) \neq \rho$} \\
\{(\rho, j)\},  & \hbox{otherwise.} \end{cases}
\end{equation*}
\end{itemize}
\end{definition}

\begin{definition}
For any permutation $\chi$ of $\omega \cup \{\rho\}$, $\Theta^\chi$
is the set of partial one-to-one maps from $M$ to $M$ of size at
most $n$ that are $\chi$-isomorphisms on their domains. We write
$\Theta$ for $\Theta^{Id_{\omega \cup \{\rho\}}}$.
\end{definition}

\begin{lemma}
For any permutation $\chi$ of $\omega \cup \{\rho\}$, $\Theta^\chi$
is an $n$-back-and-forth system on $M$.
\end{lemma}
\begin{demo}{Proof}
Clearly, $\Theta^\chi$ is closed under restrictions. We check the
``forth" property. Let $\theta \in \Theta^\chi$ have size $t < n$.
Enumerate $ dom(\theta)$, $rng(\theta)$ respectively as $ \{ a_0,
\ldots, a_{t-1} \}$, $ \{ b_0,\ldots b_{t-1} \}$, with $\theta(a_i)
= b_i$ for $i < t$. Let $a_t \in M$ be arbitrary, let $b_t \notin M$
be a new element, and define a complete labelled graph $\triangle
\supseteq M \upharpoonright \{ b_0,\ldots, b_{t-1} \}$ with nodes
$\{ b_0,\ldots, b_{t} \}$ as follows.\\

Choose distinct "nodes"$ e_s < N$ for each $s < t$, such that no
$(e_s, j)$ labels any edge in $M \upharpoonright \{ b_0,\dots,
b_{t-1} \}$. This is possible because $N \geq n(n-1)/2$, which
bounds the number of edges in $\triangle$. We can now define the
colour of edges $(b_s, b_t)$ of $\triangle$ for $s = 0,\ldots, t-1$.

\begin{itemize}
\item If $M ( a_s, a_t) = ( Ni+r, j)$, for some $i\in \N$ and $0\leq r<N$, then
\begin{equation*}
\triangle( b_s, b_t) =
\begin{cases}
( N\chi(i)+r, j), & \hbox{if $\chi(i) \neq \rho$} \\
\{(\rho, j)\},  & \hbox{otherwise.} \end{cases}
\end{equation*}
\end{itemize}

\begin{itemize}
\item If $M ( a_s, a_t) = ( \rho, j)$, then assuming that $e_s=Ni+r,$ $i\in \N$ and $0\leq r<N$,
\begin{equation*}
\triangle( b_s, b_t) =
\begin{cases}
( N\chi(\rho)+r, j), & \hbox{if $\chi(\rho) \neq \rho$} \\
\{(\rho, j)\},  & \hbox{otherwise.} \end{cases}
\end{equation*}
\end{itemize}

This completes the definition of $\triangle$. It is easy to check 
that $\triangle \in \GG$. Hence, there is a graph embedding $ \phi : \triangle \rightarrow M$
extending the map $ Id_{\{ b_0,\ldots b_{t-1} \}}$. Note that
$\phi(b_t) \notin rng(\theta)$. So the map $\theta^+ = \theta \cup
\{(a_t, \phi(b_t))\}$ is injective, and it is easily seen to be a
$\chi$-isomorphism in $\Theta^\chi$ and defined on $a_t$.
The converse,``back" property is similarly proved ( or by symmetry,
using the fact that the inverse of maps in $\Theta$ are
$\chi^{-1}$-isomorphisms).
\end{demo}

But we can also derive a connection between classical and
relativised semantics in $M$, over the following set $W$:\\

\begin{definition}
Let $W = \{ \bar{a} \in {}^n M : M \models ( \bigwedge_{i < j < n,
l < n} \neg (\rho, l)(x_i, x_j))(\bar{a}) \}.$
\end{definition}

$W$ is simply the set of tuples $\bar{a}$ in ${}^nM$ such that the
edges between the elements of $\bar{a}$ don't have a label involving
$\rho$. Their labels are all of the form $(Ni+r, j)$. We can replace $\rho$-labels by suitable $(a, j)$-labels
within an $n$-back-and-forth system. Thus, we may arrange that the
system maps a tuple $\bar{b} \in {}^n M \backslash W$ to a tuple
$\bar{c} \in W$ and this will preserve any formula
containing no relation symbols $(a, j)$ that are ``moved" by the
system. The next proposition uses this idea to show that the
classical and $W$-relativised semantics agree.

\begin{theorem}
$M \models_W \varphi(\bar{a})$ iff $M \models \varphi(\bar{a})$, for
all $\bar{a} \in W$ and all $L^n$-formulas $\varphi$.
\end{theorem}
\begin{demo}{Proof}
The proof is by induction on $\varphi$. If $\varphi$ is atomic, the
result is clear; and the boolean cases are simple.
Let $i < n$ and consider $\exists x_i \varphi$. If $M \models_W
\exists x_i \varphi(\bar{a})$, then there is $\bar{b} \in W$ with
$\bar{b} =_i \bar{a}$ and $M \models_W \varphi(\bar{b})$.
Inductively, $M \models \varphi(\bar{b})$, so clearly, $M \models_W
\exists x_i \varphi(\bar{a})$.
For the (more interesting) converse, suppose that $M \models_W
\exists x_i \varphi(\bar{a})$. Then there is $ \bar{b} \in {}^n M$
with $\bar{b} =_i \bar{a}$ and $M \models \varphi(\bar{b})$. Take
$L_{\varphi, \bar{b}}$ to be any finite subsignature of $L$
containing all the symbols from $L$ that occur in $\varphi$ or as a
label in $M \upharpoonright rng(\bar{b})$. (Here we use the fact
that $\varphi$ is first-order. The result may fail for infinitary
formulas with infinite signature.) Choose a permutation $\chi$ of
$\omega \cup \{\rho\}$ fixing any $i'$ such that some $(i'N+r, j)$
occurs in $L_{\varphi, \bar{b}}$ for some $r<N$, and moving $\rho$.
Let $\theta = Id_{\{a_m : m \neq i\}}$. Take any distinct $l, m \in
n \setminus \{i\}$. If $M(a_l, a_m) = (i'N+r, j)$, then $M( b_l,
b_m) = (i'N+r, j)$ because $ \bar{a} = _i \bar{b}$, so $(i'N+r, j)
\in L_{\varphi, \bar{b}}$ by definition of $L_{\varphi, \bar{b}}$.
So, $\chi(i') = i'$ by definition of $\chi$. Also, $M(a_l, a_m) \neq
( \rho, j)$(any $j$) because $\bar{a} \in W$. It now follows that
$\theta$ is a $\chi$-isomorphism on its domain, so that $ \theta \in
\Theta^\chi$.
Extend $\theta $ to $\theta' \in \Theta^\chi$ defined on $b_i$,
using the ``forth" property of $ \Theta^\chi$. Let $
\bar{c} = \theta'(\bar{b})$. Now by choice of of $\chi$, no labels
on edges of the subgraph of $M$ with domain $rng(\bar{c})$ involve
$\rho$. Hence, $\bar{c} \in W$.
Moreover, each map in $ \Theta^\chi$ is evidently a partial
isomorphism of the reduct of $M$ to the signature $L_{\varphi,
\bar{b}}$. Now $\varphi$ is an $L_{\varphi, \bar{b}}$-formula. Hence
we
have $M \models \varphi(\bar{a})$ iff $M \models \varphi(\bar{c})$.
So $M \models \varphi(\bar{c})$. Inductively, $M \models_W
\varphi(\bar{c})$. Since $ \bar{c} =_i \bar{a}$, we have $M
\models_W \exists x_i \varphi(\bar{a})$ by definition of the
relativised semantics. This completes the induction.
\end{demo}

We can now extract form the labelled graph $M$  a
relativised set algebra $\A$, which will turn out to be
representable atomic polyadic algebra.
\begin{definition}
\begin{enumerate}
\item For an $L^n_{\infty \omega}$-formula $\varphi $, we define
$\varphi^W$ to be the set $\{ \bar{a} \in W : M \models_W \varphi
(\bar{a}) \}$. 

\item We define $\A$ to be the relativised set algebra with domain
$$\{\varphi^W : \varphi \,\ \textrm {a first-order} \;\ L^n-
\textrm{formula} \}$$  and unit $W$, endowed with the algebraic
operations ${\sf d}_{ij}, {\sf c}_i, $ ect., in the standard way .

\end{enumerate}

Note that $\A$ is indeed closed under the operations and so is a
bona fide relativised set algebra. For, reading off from the
definitions of the standard operations and the relativised
semantics, we see that for all $L^n$-formulas $\varphi,\psi,$
\begin{itemize}
\item $-^{\A}(\varphi^W)=(\neg\varphi)^W$
\item $\varphi^W \cdot ^{\A} \psi^W = (\varphi \wedge \psi)^W$
\item ${\sf d}^{\A}_{ij}=(x_i = x_j)^W \,\ \textrm{for all} \;\ i,j<n.$
\item ${\sf c}_i^{\A}(\varphi^W)=(\exists x_i \varphi)^W \,\ \textrm{for all} \,\ i <
n.$\\ 
For a formula $\phi$ and $i,j<n$ , $\phi[x_i,x_j]$ stands for the formula obtained from $\phi$ by interchanging the free occurences of $x_i$ 
and $x_j$. Then we have:
\item ${\sf p}_{ij}^{\A}(\varphi^W)=\varphi[x_i,x_j]^W$
\end{itemize}

\end{definition}

\begin{theorem}
$\A$ is a representable (countable) atomic polyadic algebra
\end{theorem}
\begin{demo}{Proof}
Let $\cal S$ be the polyadic set algebra with domain  $\wp ({}^{n} M )$ and
unit $ {}^{n} M $. Then, the map $h : \A
\longrightarrow S$ given by $h:\varphi ^W \longmapsto \{ \bar{a}\in
{}^{n} M: M \models \varphi (\bar{a})\}$ can be checked to be well -
defined and one-one. It clearly respects the polyadic operations. So it is a representation of $\A.$
A formula $ \alpha$  of  $L^n$ is said to be $MCA$
('maximal conjunction of atomic formulas') if (i) $M \models \exists
x_0\ldots x_{n-1} \alpha $ and (ii) $\alpha$ is of the form
$$\bigwedge_{i \neq j < n} \alpha_{ij}(x_i, x_j),$$
where for each $i,j,\alpha_{ij}$ is either $x_i=x_i$ or $R(x_i,x_j)$
for some binary relation symbol $R$ of $L$. 
The rough idea is that a
formula $\alpha$  being $MCA$ says that the set it defines in ${}^n M$
is nonempty, and that if $M \models \alpha (\bar{a})$ then the graph
$M \upharpoonright rng (\bar{a})$ is determined up to isomorphism
and has no edge whose label is of the form $(\rho,i)$. Hence, any two
tuples satisfying $\alpha$ are isomorphic and one is mapped to the
other by the $n$-back-and-forth system $\Theta$. By theorem \ref{atom}, no
$L^n_{\infty \omega}$- formula can distinguish them. So $\alpha$
defines an atom of $\A$ --- it is literally indivisible. Since the
$MCA$ - formulas clearly `cover' $W$, the atoms defined by them are
dense in $\A$. So $\A$ is atomic, as required. This, informally, is
the content of what follows.
Let $\varphi$ be any $L^n_{\infty\omega}$-formula, and $\alpha$ any
$MCA$-formula. If $\varphi^W \cap \alpha^W \neq \emptyset $, then
$\alpha^W \subseteq \varphi^W $.
Indeed, take $\bar{a} \in  \varphi^W \cap \alpha^W$. Let $\bar{a} \in
\alpha^W$ be arbitrary. Clearly, the map $( \bar{a} \mapsto
\bar{b})$ is in $\Theta$. Also, $W$ is
$L^n_{\infty\omega}$-definable in $M$, since we have
$$ W = \{
\bar{a} \in {}^n M : M \models (\bigwedge_{i < j< n} (x_i = x_j \vee
\bigvee_{R \in L} R(x_i, x_j)))(\bar{a})\}.$$
We have $M \models_W \varphi (\bar{a})$
 iff $M \models_W \varphi (\bar{b})$. Since $M \models_W \varphi (\bar{a})$, we have
$M \models_W \varphi (\bar{b})$. Since $\bar{b} $ was arbitrary, we
see that $\alpha^W \subseteq \varphi^W$.
Let $$F = \{ \alpha^W : \alpha \,\ \textrm{an $MCA$},
L^n-\textrm{formula}\} \subseteq \A.$$
Evidently, $W = \bigcup F$. We claim that 
$\A$ is an atomic algebra, with $F$ as its set of atoms.
First, we show that any non-empty element $\varphi^W$ of $\A$ contains an
element of $F$. Take $\bar{a} \in W$ with $M \models_W \varphi
(\bar{a})$. Since $\bar{a} \in W$, there is an $MCA$-formula $\alpha$
such that $M \models_W \alpha(\bar{a})$. Then $\alpha^W
\subseteq \varphi^W $. By definition, if $\alpha$ is an $MCA$ formula
then $ \alpha^W$ is non-empty. If $ \varphi$ is
an $L^n$-formula and $\emptyset \neq \varphi^W \subseteq \alpha^W $,
then $\varphi^W = \alpha^W$. It follows that each $\alpha^W$ (for
$MCA$ $\alpha$) is an atom of $\A$.
\end{demo}

Define $\C$ to be the complex algebra over $At\A$, the atom structure of $\A$.
Then $\C$ is the completion of $\A$. The domain of $\C$ is $\wp(At\A)$. The diagonal ${\sf d}_{ij}$ is interpreted as the set of all $S\in At\A$ with $a_i=a_j$ for some $\bar{a}\in S$.
The cylindrification ${\sf c}_i$ is interpreted by ${\sf c}_iX=\{S\in At\A: S\subseteq c_i^{\A}(S')\text { for some } S'\in X\}$, for $X\subseteq At\A$.
Finally ${\sf p}_{ij}X=\{S\in At\A: S\subseteq {\sf p}_{ij}^{\A}(S')\text { for some } S'\in X\}.$
Let $\cal D$ be the relativized set algebra with domain $\{\phi^W: \phi\text { an $L_{\infty\omega}^n$ formula }\}$,  unit $W$
and operations defined like those of $\cal A$.

\begin{theorem} ${\C}\cong \D$, via the map $X\mapsto \bigcup X$.
\end{theorem}
In the following, we assume familiarity with the definition of relation algebra atom structures by listing the consistent triples.
We also assume familiarity with the notion of basic matrices over a relation algebra atom structure, and that of 
$n$ dimensional cylindric bases \cite{Mad}. 
Though $\C$ can be represented as a {\it relativized} set algebra, we have:
\begin{theorem} $\Rd_{ca}\cal C$ is not representable.
\end{theorem} 

\begin{demo}{Proof} We define a relation algebra atom structure $\alpha(\G)$ of the form
$(\{1'\}\cup (\G\times n), R_{1'}, \breve{R}, R_;)$.
The only identity atom is $1'$. All atoms are self converse, 
so $\breve{R}=\{(a, a): a \text { an atom }\}.$
The colour of an atom $(a,i)\in \G\times n$ is $i$. The identity $1'$ has no colour. A triple $(a,b,c)$ 
of atoms in $\alpha(\G)$ is consistent if
$R;(a,b,c)$ holds. Then the consistent triples are $(a,b,c)$ where

\begin{itemize}

\item one of $a,b,c$ is $1'$ and the other two are equal, or

\item none of $a,b,c$ is $1'$ and they do not all have the same colour, or

\item $a=(a', i), b=(b', i)$ and $c=(c', i)$ for some $i<n$ and 
$a',b',c'\in \G$, and there exists at least one graph edge
of $G$ in $\{a', b', c'\}$.

\end{itemize}
$\alpha(\G)$ can be checked to be a relation atom structure. 
The atom structure of $\Rd_{ca}\A$ is isomorphic (as a cylindric algebra
atom structure) to the atom structure ${\cal M}_n$ of all n-dimensional basic
matrices over the relation algebra atom structure $\alpha(\G)$.
Indeed, for each  $m  \in {\cal M}_n, \,\ \textrm{let} \,\ \alpha_m
= \bigwedge_{i,j<n}  \alpha_{ij}. $ Here $ \alpha_{ij}$ is $x_i =
x_j$ if $ m_{ij} = 1$' and $R(x_i, x_j)$ otherwise, where $R =
m_{ij} \in L$. Then the map $(m \mapsto
\alpha^W_m)_{m \in {\cal M}_n}$ is a well - defined isomorphism of
$n$-dimensional cylindric algebra atom structures.
We shall prove that $\Cm\alpha(\G)$ is not representable. 
Hence the full complex cylindric algebra over the set of $n$ by $n$ basic matrices
- which is isomorphic to $\cal C$ is not representable either, for we have a relation algebra
embedding of $\Cm\alpha(\G)$ onto $\Ra\Cm\M_n$.
Assume for contradiction  that $g: \Cm\alpha(\G)\to \B$ 
is an embedding into a proper relation set algebra $\B$
with base set $X$. Each $h(a)$ ($a\in \Cm\alpha(\G)$) is a binary relation on $X$, and
$h$ respects the relation algebra operations. 
For $Y\subseteq \N$ and $s<n$, set 
$$[Y,s]=\{(l,s): l\in Y\}.$$
For $r\in \{0, \ldots N-1\},$ $N\N+r$ denotes the set $\{Nq+r: q\in \N\}.$
Let $$J=\{1', [N\N+r, s]: r<N,  s<n\}.$$
Then $\sum J=1$ in $\Cm\alpha(\G).$
As $J$ is finite, we have for any $x,y\in X$ there is a $P\in J$ with
$(x,y)\in h(P)$.
Since $\Cm\alpha(\G)$ is infinite then $X$ is infinite. 
By Ramsey's Theorem, there are distinct
$x_i\in X$ $(i<\omega)$ 
and $P\in J$ such that $(x_i, x_j)\in h(P)$ for all $i<j<\omega.$ Clearly $P\neq 1'$. 
Also $(P;P)\cdot P\neq 0$. 
This follows from that if $x_0,x_1, x_2\in X$, 
$a,b,c\in \Cm\alpha(\G)$, $(x_0,x_1)\in h(a)$, $(x_1,x_2)\in h(b)$, and 
$(x_0, x_2)\in h(c)$, then $(a;b)\cdot c\neq 0$. 
A non -zero element $a$ of $\Cm\alpha(\G)$ is monochromatic, if $a\leq 1'$,
or $a\leq [\N,s]$ for some $s<n$. 
Now  $P$ is monochromatic, it follows from the definition of $\alpha$ that
$(P;P)\cdot P=0$. This contradiction shows that 
$\Cm\alpha(\G)$ is not representable.
\end{demo}
\begin{theorem} $\Rd_{df}\C$ is not representable.
\end{theorem}
\begin{demo}{Proof} Assume that $\Rd_{df}\C$ is representable, via the isomorphism $h$, as a set algebra 
${\cal D}\subseteq P(\prod U_i:i<n)$. We show that $\Rd_{ca}\cal C$ is representable, which is a contradiction.
We can asume that $U_0=\ldots =U_{n-1}=U$ and ${\sf d}_{ij}^U\subseteq h({\sf d}_{ij})$ \cite{HMT2} 5.1.48.
Define $R$ on $U$ as follows: Let $i,j<n$ be distinct, then
$$R=\{(u,v)\in U\times U: s(i)=u \text { and } s(j)=v \text { for some } s\in h({\sf d}_{ij})\}$$
Then $R$ is independent of the choice of $i$ and $j$ and is an equivalence relation on $U$ \cite{HMT2} 5.1.49
Let
$$E=\{x\in C: (\forall s,t\in {}^nU)[(\forall i<n)((s(i), t(i))\in R$$
$$\implies (s\in h(x)\longleftrightarrow t\in h(x))]\}.$$
Then $\{x\in C: \Delta x\neq n\}\subseteq E$ and $E$ is a subset $\C$ that is closed under cylindrifications, 
complementation, intersections and contains the diagonal elements \cite{HMT2} 5.1.50.
Since $E$ contains $\alpha^W$ for all atomic formulas $\alpha$ as we have only binary relation symbols, it follows that $E=C$. 
Then we can factor $U$ by $R$  so that $\C$ 
can be embedded into 
$(\wp^{n}(U/R), {\sf c}_i)_{i<n}$ via the isomorphism $f$ given by
$$f(x)=\{(s(i)/R: i<n)\in {}^n(U/R): s\in h(x)\}.$$
Moreover, as easily checked, diagonals are preserved, 
that is $$f({\sf d}_{ij})=\{s\in {}^{n}(U/R): s_i=s_j\}.$$

\end{demo} 

\begin{corollary}\label{complete} Let $n\geq 3$. Then the classes $\{\RCA_n, \RPEA_n, \RPA_n ,\RDf_n\}$ are not closed under completions
\end{corollary}
\begin{demo}{Proof}
Let $K\in \{\CA, \PA, \PEA, \Df\}$. Then $\Rd_{K}\C$ is the completion of $\Rd_{K}\A$. The latter is representable, while the former is not.
\end{demo}

Taking the boolean reducts of $\A$ and $\C$ as given, their cylindric structure, 
is determined by the way cylindrifications are defined on atoms, i.e by their atom structure.
Now they have the {\it same} atom structure.
The difficulty in finding representations for cylindric algebras arise from the cylindrifications
and diagonal elements. By Stone's theorem, it is easy to represent the Boolean part.
So one might be tempted to think that these difficulties can be pinned down to the atom structure in case of atomic algebras.
That is representability  of an atomic algebra would depend on its atom structure, but this is not the case.
The underlying reason that $\A$ is representable  while its completion is not, is that $\C$ has more elements. 
This would have to be mirrored property in a true representation  of $\C$.
For certain algebras $\A$ deadlocks occur when one tries to find suitable genuine relations for the extra elements of $\C$.
$\A$ has few relations so a representation of it can sweep potential problems under the carpet. 
Adding the new relations in $\C$ brings the problem 
to the surface.

$\Ra\CA_n$ stands for the class of relation algebra reducts of $\CA_n$.
The full complex algebra of an atom structure $S$
is denoted by $\Cm S$, and the term algebra by $\Tm S.$
$S$ could be a relation atom structure or a cylindric atom structure.
In \cite{MS} it is proved that exists a cylindric atom structure $H$ such that $\Tm H$ is representable
while $\Cm H\notin {\bf S}\Nr_3\CA_6.$
Indeed, let $S$ be a relation atom structure such that $\Tm S$ is representable while $\Cm S\notin \RA_6$.
Such an atom structure exists \cite{HHbook} Lemmas 17.34-17.36. It follows that $\Cm S\notin {\bf S}\Ra\CA_6$.
Let $H$ be the set of $3$ by $3$ atomic networks over $S$.
Then by it is not so hard to show that $\Tm H\in \RCA_3$. We claim that $\Cm H\notin {\bf S}\Nr_3\CA_6$. For assume not, i.e. assume that
$\Cm H\in {\bf S}\Nr_3\CA_6.$
Then $\Cm S$ is embeddable in $\Ra\Cm H.$ But then the latter is in ${\bf S}\Ra\CA_6$
and so is $\Cm S$, which is not the case.
   
We note that the proof adopted herein does not generalize to $n>3$. 
\begin{itemize}
\item For $n=4$ the above proof fails as illustrated by the following example.
Let $\P$ be the pentagol relation algebra introduced in \cite{Mad} p. 369. Then $\P$ is a finite $\RA$ that is representable. 
Let $H$ be the set of all $4$ by $4$ atomic matrices. 
Then $\Cm H=\Tm H$ is a finite $\CA_4$ that is not representable as shown by Maddux \cite{Mad} p. 389.
\item If $n>4$ and $H$ is the set of all $n$ by $n$ atomic matrices over an $\RA$, then $H$ may not be an $n$-dimensional cylindric basis in 
the first place. That is $\Tm H$ may not be a $\CA_n$, let alone being an $\RCA_n$.  
 \end{itemize}
However, this result is generalized to higher dimensions in \cite{Moh2}.
That is it is proved in \cite{Moh2} that for every $n\geq 3$, and $k\geq 2$, the class $S\Nr_n\CA_{n+k}$
is not closed under completions.

\subsection{The Omitting Types Theorem fails in $\L_n$}

The ultimate purpose of algebraic logic is to solve problems in logic. 
Here we give an application of results on completions, or rather the non-existence thereof,
to omitting types of finite variable fragments. 
We work in usual first order logic $(FOL)$. For a formula $\phi$ and a first order structure $\M$
in the language of $\phi$ we write $\phi^M$ to denote the set of all assignments that satisfy $\phi$ in $M$., i.e
$$\phi^{\M}=\{s\in {}^{\omega}M: \M\models \phi[s]\}.$$
For example if $\M=(\N,<)$ and $\phi$ is the formula $x_1<x_2$ then a sequence
$s\in {}^{\omega}\N$ is in $\phi^{\M}$ iff $s_1<s_2$. 
Let $\Gamma$ be a set of formulas ($\Gamma$ may contain free variables). We say that $\Gamma$ is realized in $\M$
if $\bigcap_{\phi\in \Gamma}\phi^{\M}\neq \emptyset$. Let $\phi$ be a formula and $T$ be a theory.
We say that $\phi$ ensures $\Gamma$ in $T$ if $T\models \phi\to \mu$ for all $\mu\in \Gamma$.

The classical Henkin-Orey omitting types theorem, $OTT$ for short, states that if 
$T$ is a consistent theory in a countable language 
$\L$ and $\Gamma(x_1\ldots x_n)\subseteq \L$ is realized in every model of $T$, then there is a formula $\phi\in \L$ such that
$\phi$ ensures $\Gamma$ in $T$. The formula $\phi$ is called a $T$-witness for $\Gamma$.
Now the problem of resourse sensitivity can be applied to $OTT$ in the following sense. 
Can we always guarantee that the witness uses the same number of variables as $T$ and $\Gamma$, or do we need 
extra variables? If we do need extra variables, is there perhaps an upper bound on the number of extra variables needed?
In other words, let $\L_n$ denotes the set of formulas of $\L$ which are built up using only $n$ variables. 
The question is: If $T\cup \Gamma\subseteq \L_n$,
is there any guarantee that the witness stays in $\L_n$, or do we occasionally have to step outside $\L_n$?

Assume that $T\subseteq \L_n$. We say that $T$ is $n$ complete iff for all sentences $\phi\in \L_n$ 
we have either $T\models \phi$ or $T\models \neg \phi$. We say that $T$
is $n$ atomic iff for all $\phi\in \L_n$, there is $\psi\in \L_n$ such that $T\models \psi\to \phi$ and for all $\eta\in \L_n$ either $T\models \psi\to \eta$
or $T\models \psi\to \neg \eta$

\begin{theorem}\label{o} Assume that $\L$ is a countable first order language containing a binary relation symbol. For $n>2$ and 
$k\geq 0$,  there are a consistent $n$ complete
and $n$ atomic theory $T$ using only $n$ variables, and a set $\Gamma(x_1)$ using only $3$ varaibles
(and only one free variable) such that $\Gamma$ is realized in all models of $T$ but each $T$-witness for $T$ uses
more that $n+k$ variables
\end{theorem}

Theorem \ref{o} is proved using algebraic logic in \cite{ANT}, 
where the following refinement of the above construction in theorem \ref{complete} 
is proved.

\begin{theorem}\label{oo}
Suppose that $n$ is a finite ordinal with $n>2$ and $k\geq 0$.
There is a countable 
symmetric integral representable 
relation algebra ${\R}$
such 
\begin{enumroman}
\item Its completion, i.e. the complex algebra of its atom structure is 
not representable, so $\R$ is representable but not completely representable 
\item $\R$ is generated by a single element.
\item The (countable) set $\B_n{\R}$ of all $n$ by $n$ basic matrices over $\R$ 
constitutes an $n$-dimensional cylindric basis. 
Thus $\B_n{\R}$ is a cylindric atom structure 
and the full complex algebra $\Cm(\B_n{\R})$ 
with universe the power set of $\B_n{\R}$
is an $n$-dimensional cylindric algebra 
\item The {\it term algebra} over the atom structure 
$\B_n{\R}$, which is the countable subalgebra of $\Cm(\B_n{\R})$ 
generated by the countable set of 
$n$ by $n$ basic matrices, $\Tm(B_n \R)$ for short,
is a countable representable $\CA_n$, but $\Cm(\B_n)$ is not representable. 
\item Hence $\C$ is a simple, atomic representable but not completely representable $\CA_n$
\item $\C$ is generated by a single $2$ dimensional element $g$, the relation algebraic reduct
of $\C$ does not have a complete representation and is also generated by $g$ as a relation algebra, and 
$\C$ is a sub-neat reduct of some simple rep $\D\in \CA_{n+k}$ such that the relation algebraic reducts of $\C$ and $\D$
coincide.  
\end{enumroman}
\end{theorem}
\begin{demo}{Sketch of proof} \cite{ANT}. In fact, the proof is like Theorem \ref{OTT}.
\end{demo}

Now we give a proof of Theorem \ref{o} modulo Theorem \ref{oo}.

\begin{demo}{Proof of Theorem \ref{o}} let $g, \C$ and $\D$ be as in theorem \ref{oo} (vi).
Then $g$ generates $\C$ and $g$ is $2$ dimensional in $\C$. We can write up a theory $T\subseteq \L_n$ such that for any model $\M$ we have
$$\M=(M,G)\models T \text { iff } \C_n(\M)\cong \C\text { and $G$ corresponds to $g$ via this isomorphism }$$
now $T\subseteq L_n$, $T$ is consistent and $n$ complete and $n$ atomic because $C$ is simple and atomic.
We now specify $\Gamma(x,y)$/ For $a\in At$, let $\tau_a$ be a relation algebraic term such that $\tau_a(g)=a$ in 
$R$, the relation algebra reduct of $\C$.
For each $\tau_a$ there is a formula $\mu_a(x,y)$ such that $\tau_a(g)=\mu_a^{\M}$. Define $\Gamma(x,y)=\{\neg \mu_a: a\in At\}$.
We will show that $\Gamma$ is as required. First we show that $\Gamma$ is realized in every model of $T$. Let $\M\models T$. 
Then
$\C_n(\M)\cong \C$, hence $\M$ gives a representation of $\R$ becuase $\R$ is the relation algebraic reduct of $\C_n(\M)$. 
But $\R$ has no complete representation, which means that 
$X=\bigcup\{\mu_a^{\M}: a\in At\}\subset M\times M$, i.e proper subset, so let $u,v)\in M\times M\sim X$. 
This means that $\Gamma$ is realized in $(u,v)$ in $\M$,
We have seen that $\Gamma$ is realized in each model of $T$. assume that that $\phi\in \L_{n+k}$ such that
$T\models \exists \bar{x}\phi$. We may assume that $\phi$ has only two free variables, say $x,y$. 
Take the representable $\D\in \CA_{n+k}$ from thm 2 (iv). recall that $g\in C\subseteq D$ and $\D$ is simple. 
Let $\M=(M,g)$ where $M$ is the base set of 
$\D.$ then $\M\models T$ because $\C$ is a subreduct of $\D$ generated by $g$. by $T\models \exists\bar{x}\phi$, we have 
$\phi^{\M}\neq \emptyset$.
Also $\phi^{\M}\in \D$ and is $2$ dimensional, hence $\phi^{\M}\in R$, since $\R$ is the relation algebraic reduct of 
$\D$, as well. But $\R$ is atomic hence $\phi^{\M}\cap \mu_a\neq \emptyset$ for some $a\in At$. this shows that 
it is not the case that $\M\models \phi\to \neg\mu_a$ where $\neg \mu_a\in \Gamma$, thus
$\phi$ is not a $T$-witness for $\Gamma.$  
Now we modify $T$, $\Gamma$ so that $\Gamma$ uses only one free variable. 
We use the technique of so-called partial pairing functions. Let $g,\C,\D$ 
be as in Theorem \ref{oo} (iv) with $\D\in \CA_{2n+2k}$. We may assume that $g$ is disjoint from the identity $1'$ 
because $1'$ is an atom in the relation algebraic reduct of $\C$. let $U$ be the base set of $\C$. 
We may assume that $U$ and $U\times U$ are disjoint. Let $M=U\cup (U\times U)$, let 
$G=g\cup \{(u,(u,v)): u,v\in U\}\cup \{((u,v),v): u,v\in U\}\cup \{((u,v), (u,v)): u,v\in U\}$
and let $\M=(M,G)$. from $G$ we can define $U\times U$ as $\{x: G(x,x)\}$ and from $U\times U$ and $G$ we can define the projection functions between 
$U\times U$ and $U$, and $g$. All these definitions use only $3$ variables. Thus for all $t\geq 3$ 
for all $\phi(x,y)\in \L_t$ there is a $\psi(x)\in \L_t$
such that $\psi^{\M}=\{(u,v)\in U\times U: \phi^{(U,g)}(u,v)\}$. For any $a\in At$ let $\psi_a(x)$ be the formula
corresponding to $\mu_a(x,y)$ this way. Conversely for any $\psi\in \L_t$ there is a $\phi\in \L_{2t}$ 
such that the projection of $\psi^{\M}$ to $U$ is $\phi^{(U,g)}$.
Now define $T$ as the $\L_n$ theory of $\M$, and set $\Gamma(x)=\{\neg \psi_a(x): a\in At\}.$
Then it can be easily checked that $\Gamma$ and $T$ are as required.
\end{demo}

\subsection{Independence of OTT}

In this final section, we use iterated forcing to prove independence of statements involving existence of representations for algebras
enjoying a complete neat embedding property. By examples \ref{countable}, \ref{OTT}, 
we have that the condition of countability and being in $S_c\Nr_n\CA_{\omega}$ cannot be dispensed with in theorem \ref{OT}.
\begin{definition} 
\begin{enumroman}
\item Let $\kappa$ be a cardinal. Let $OTT(\kappa)$ be the following statement.
$\A\in S_c\Nr_n\CA_{\omega}$ is countable and for $i\in \kappa$, $X_i\subseteq A$ are such that $\prod X_i=0$, then for all $a\neq 0$, there exists
a set algebra $\C$ with countable base, $f:\A\to C$ such that $f(a)\neq 0$ and for all $i\in \kappa$, $\bigcap_{x\in X_i} f(x)=0.$
\item Let $OTT$ be the statement that 
$$(\forall k<{}^{\omega}2\implies OTT(\kappa))$$
\item Let $OTT_m(\kappa)$ be the statement obtained from $OTT(\kappa)$ by replacing $X_i$ with ``nonprincipal ultrafilter $F_i$"
and $OTT_m$ be the statement 
$$(\forall k<{}^{\omega}2\implies OTT_m(\kappa))$$
\end{enumroman}
\end{definition}
\begin{theorem}\label{OT} 
\begin{enumroman}
\item $OTT$ is independent from $ZFC+\neg CH$. In fact for any regular cardinal $\kappa>\omega_1$, 
there is a model of $ZFC$ in which $\kappa=2^{\omega}$ 
and $OTT$ holds. Conversely, there is a model of $ZFC$ in which $\omega_3=2^{\omega}$ and $OTT(\omega_2)$ is false.
\item $OTT_m$ is provable in $ZFC$
\end{enumroman}
\end{theorem}
\begin{demo}{Proof} We proved (ii) above. 
$OTT$ is equivalent to $MA$ restricted to countable partially ordered sets.
We first show that for any $n>1$,
there is a countable transitive model of $ZFC$ such that $M\models OTT\land {}^{\omega}2=\omega_n$.
Let $M$ be a model of $ZFC$ with $^{\omega}2=\omega_n$.
It is not hard to show that, assuming $ZFC$ consistent,  such models (violating $CH$) exist. 
We say that $(\A,J)\in M$ is a counterexample to $OTT$ if $\A$ is a countable boolean algebra, $J$ is a family of subsets of $A$
such that $|J|\leq \omega_{n-1}^M$, and $\sum X=1$ for all $X\in J$, but there is no ultrafilter $F$ of $\A$ such that $F\cap X\neq \emptyset$
forall $X\in J$. We want to get a model, where there is no counterexamples, so we are going to adjoin infinitely many generic sets 
to kill potential counterexamples, using sophisticated iteration techniques of Solovay. In such a model $OTT$ holds for the following reasoning.
For let $\A\in S_c\Nr_n\CA_{\omega}$ and $(X_i:i<\omega_{n-1})$ be non-principal types. Then $\A=\Nr_n\B$, $\B\in \Lf_{\omega}$ is countable.
To construct the desired representation, wo we are searching for an ultrafilter $F$ that preserves the following joins and meets: 
$$(\forall j<\alpha)(\forall x\in B)({\sf c}_jx=\sum_{i\in \alpha\smallsetminus \Delta x}
{\sf s}_i^jx.)$$
$$(\forall\tau\in V)(\forall  i\in \omega_{n-1})\prod{}^{\A}X_{i,\tau}=0$$
where $$X_{i,\tau}=\{{\sf s}_{\tau}x: x\in X_i\}.$$
Such joins and meets can be easily transformed to a an equivalent set of joins $(X_i:i<\omega_{n-1})$ such that $\sum X_i=1$. 
An ultrafilter preserving the set of new joins is one that preserves the original sets of meets and joins, meaning that for all $j<\alpha$ if ${\sf c}_jx\in F$
then ${\sf s}_i^jx\in F$ for some $i\notin \Delta x$, and for ll $\tau$ there exists $i$ such that $X_{i,\tau}$ is not included in $F$.
We follow closely the 
treatment carried out in \cite{Handbook} 
proving independence of $MA$. We also follow the notation adopted therein, often without warning.
Let $P_0$ be any normalized $ccc$ (satisifying countable chain condition, i.e has no uncounable antichains)
partially ordered set with underlying set $\omega.$ We  define an increasing sequence 
$P_{\eta}$ $\eta\leq \omega_n$ of $PO$ sets each of cardinality $\leq \omega_n$ by transfinite induction.
At successor ordinls we use product forcing, and at limit ordinal we use direct limits.
At the same time we construct retractions $h_{\eta}^{\mu}$ of $P_{\mu}$ to $P_{\eta}$ for $\eta<\mu\leq \kappa$, such that
$h_{\eta}^{\xi}=h_{\eta}^{\mu}h_{\mu}^{\xi}$ to control the construction. Let $a, b$ be functions with domain $\kappa\sim \{0\}$ 
and for all $\eta<\kappa$ $a(\eta)<\eta$, $b(\eta)<\kappa$. Assume, too, that 
for all $\alpha,\beta<\kappa$,  there exists $\eta<\kappa$ such that $a(\eta)=\alpha$, $b(\eta)=\beta$. 
Suppose everything is constructed up to and including $\mu$. We continue the construction at $\mu+1$. Let $\alpha=a(\mu)$, $\beta=b(\mu)$.
Let $t(P_{\alpha})\subseteq P_{\alpha}\times \kappa\times \omega\times \omega$, such that
$P_{\alpha}$ forces $t(P_{\alpha})$ is a function with domain $\omega_n$ and range the set of all relations of partially ordered sets 
with underlying set $\omega$.  
Let $v$ be a $P_{\alpha}$ term such that $P_{\alpha}\models_f v=u(\beta^*).$
Let $\bar{R_{\mu}}$ be a $P_{\mu}$ term such that $P_{\mu}\models $ 
$v$ is $ccc\land \bar{R_{\mu}}=v$ or $v$ is not $ccc$ and $\bar{R_{\mu}}$ is the order relation of 
$P_0^*$.
Let $P_{\mu+1}$ be the forcing product $P_{\mu}\otimes \bar{R_{\mu}}$. Let $h$ be the natural restriction of $P_{\mu+1}$ to $P_{\mu}$.
Let $h_{\eta}^{\mu+1}=h$ for $\eta=\mu$ and $h_{\eta}^{\mu}h$ for $\eta<\mu$. At a limit $\eta\leq \kappa$, let $P_{\eta}$ be the union of of the 
$P_{\xi}$ for $\xi<\eta$, assuming that these have been defined. For $\xi<\eta$, let $h_{\xi}^{\eta}$ be the union of the maps $h_{\xi}^{\mu}$ for 
$\xi<\mu<\eta$.  Then, it can be checked that $|P_{\eta}|\leq \omega_n$ for all $\eta\leq \omega_n$ \cite{Handbook}.
Now we have constructed a $P=P_k\in M$ such that $M\models $$P$ is $ccc$. (The limit case follows from \cite{Handbook} 
lemma 6.5 p.448.)
Let $G$ be $M$ generic subset of $P$, and $N=M[G]$.
Then $N$ and $M$ have the same cardinals, so that $N\models {}^{\omega}2={}^{\omega_j}2=\kappa$ for all $0<j<n$. 
We show that $N\models OTT$.
Let $\A\in S_c\Nr_n\CA_{\omega}$ and $J$ be a family of subsets of $\A$ such that $|J|< \omega_n^N$ 
and $\sum X=1$ for all $X\in J$ and suppose that $(\A,J)\in N$. We want to find an ultrafilter $F$ intersecting the elements in $J$.
Assume that  $|J|=\omega_{n-1}$. The other cases are treated analogously. 
For each $\eta<\kappa$, $G_{\eta}=G\cap |P_{\eta}|$ is an $M$ generic subset of $P_{\eta}$.
Let $M_{\eta}=M[G_{\eta}]$. Then $M\subseteq M_{\eta}\subseteq N$ and all three models have the same cardinals.
Now we show that $(\A,J)\in M_{\alpha}$ for some $\alpha$. Take $t\in \wp^{M}(P\times \mu\times \mu)$ with $I_G(t)$ being the order relation on $\A$.
If $\sigma, \tau<\omega_{n-1}$ and $\sigma\leq \tau$, then there is a $p\in G$ with
$(p,\sigma,\tau)\in t$. Let $f(\sigma,\tau)$ be one such $p$. If it is not the case that $\sigma\leq \tau$, then there is a $p\in G$ such that for no
$q\leq p$ it is the case that $(q,\sigma,\tau)\in t$. Since $\kappa>\omega_{n-1}$ is regular, all $f(\sigma,\tau)\in G_{\alpha}$
for some $\alpha<\kappa$. So the order relation of $Q$ is equal to $I_{G_{\alpha}(t)}\in M_{\alpha}$ for some 
$\alpha<\kappa$. To handle $F$ fix a surjection $f\in N$ from $\omega_{n-1}$ to $J$ and apply a similar argument to 
$E=\{(\sigma,\tau))\in \omega_{n-1}\times \omega_{n-1}: \sigma\in f(\tau)\}$, to show that $E$ hence $J$ belongs to $M_{\alpha}$.
Finally fix $\alpha<\kappa$ such that $(\A,J)\in M_{\alpha}$.
Now the order on $\A$ in $M_{\alpha}$ is a  $PO$ set with underlying set $\omega$.
If $u=t_{n-1}(P_{\alpha})$ is the term such that $I_{G_{\alpha}}(u)$ is a surjection from $\kappa$ onto the set of all order relations 
of such $PO$ sets, then the order relation of $\A$ is $I_{G_{\alpha}}(u)(\beta)$ for some $\beta<\kappa$.
Let $v$ be the $P_{\alpha}$ term such that $P_{\alpha}\models v=u(\beta^*)$, so that
$I_G(v)$ is the order relation of $\A$. Take $\mu<\kappa$ such that $a(\mu)=\alpha$, $b(\mu)=\beta$.
Now in the construction of $P_{\mu+1}$ given above $\bar{R_{\mu}}$ was a term so chosen such that for any $M$ generic subset $H$ of
$P_{\mu}$ we have $I_H(\bar{R_{\mu}})=I_H(v)$ if the latter is $ccc$ inside $M[H]$.
Now we have $I_{G_{\mu}}(R_{\mu})=I_G(v)=I_{G_{\alpha}}(v)$ is the order relation on $\A$.
Finally consider the $M$ generic $G_{\mu+1}$ of $P_{\mu+1}=P_{\mu}\otimes \bar{R_{\mu}}.$
Then
$$K=\{I_{G_{\mu}}(r): \exists q\in G_{\mu}((q,r)\in G_{\mu+1})\}\in M_{\mu+1}\subseteq N$$
is an $M_{\mu}$ generic subset of the $PO$ set with order relation of $I_{G}(\bar{R})$, i.e of $\A$.
$K$ is  $F$ generic and $K$ can be extended to the desired ultrafilter.
One can construct, using standard iteration techniques,  a model of the stronger $MA$, in fact in such a model $2^{\omega}=\kappa$, 
where $\kappa$ is any regular cardinal $>\omega_1$ and $OTT$ holds.
We are done.\footnote{It is proved by Miller that $covK$ has uncountable cofinality. 
So if we start with a ground model that satisfies $GCH$, and we let $P$ the notion of forcing that 
adds $\kappa$ reals then we get in $M[G]$ $2^{\omega}=covK>\omega_1$.}

Now we prove that the negation of $OTT$ is also consistent. Let $covK$ be the least cardinal $\lambda$ 
such that the real line can be covered by $\lambda$ nowhere dense sets.
Then in theorem \ref{covK}, it is proved that $(\forall \lambda<covK\implies OTT(\kappa))$ is provable in $ZFC$.
\footnote{ Note that $covK$ is the least cardinal such that the Baire category Theorem fails. It is also
the largest cardinal such that $MA$ restricted to countable Boolean algebras holds.
To see this, we have the irrationals are homeomorphic to the space $\omega^{\omega}$. 
The topology on the Baire space is generated by sets of the form 
$N_s=\{g\in {}^{\omega}\omega: s\subseteq g\}$ for $s\in {}^{<\omega}\omega$. Dense open subsets $D$ of $^{\omega}\omega$ correspond to dense
subsets of $^{\omega}\omega$, i.e $\{s\in ^{<\omega}\omega: N_s\subseteq D\}$. Therefore $MA_{covK}(^{<\omega}\omega)$ holds.
Let $P$ be any countable partail order. If there is a condition $p\in P$ such that every two extensions of $P$ are compatible, then
$\{q\in P: q\leq p \text { or } p\leq q\}$ is a $P$ filter meeting every dense subset of $P$. If there is no such $P$, then for every element of $P$
there exists an infinite, maximal set of incompatible extensions. Since $P$ is countable, then one can inductively define an order
preserving embedding of $^{<\omega}\omega$ onto a dense subset of $P$. By a well known theorem we are done.}
Let $\kappa$ be a regular cardinal. The following notion of forcing which we denote by $C(\kappa)$
adjoins $\kappa$ real numbers called Cohen reals. $C(\kappa)$ be the set of all functions $p$ such that
$Dom(p)$ is a finite subset of $\kappa\times \omega$
and $Range(p)\subseteq \{0,1\}$.
Let $I=\kappa\times \omega$. Let $\Psi={}^I\{0,1\}$.
Let $T$ be the set of $0,1$ functions with $dom(t)\subseteq I$. Let $S$ be the $\sigma$-algebra generated by the sets $S_t$, $t\in T$, where 
$S_t=\{f\in \Psi: t\subseteq f\}$. The product measure on $S$ is the unique $\sigma$-additive measure such that
each $S_t$ has measure $1/2^{|t|}$. Let $B=S/I$ where $I$ is the ideal of measure $0$ sets.
Let $B(\kappa)$ be the measure algebra on $^{\kappa}2$ and $C(\kappa)=\{p:D\to 2, D\subseteq_{\omega} k\}$. 
The following models can be easily destilled from the work of Miller \cite{Miller}.  The model for $covK=\omega_2$ and $^{\omega}2=\omega_3$
is $M[G][H]$ where $G$ is $B(\omega_3)$ generic over $M$ and $H$ is $C(\omega_2)$ generic over $M[G]$. 
It is easy to see that $covK\geq \omega_2.$
To see that $covK\leq \omega_2$ notice that no real is Cohen over $M[H]\cap {}^{\omega}2$.
Now there is alo a model in which $covK=\omega_1$ and $^{\omega}2=\omega_3$. This is $M[H][G]$ where $H$ is $C(\omega_3)$ generic over
$M$ and $G$ is $B^{M[H]}(\omega_2)$ generic over $M[H]$.
(The above are two step iterations). 
There are other known models which give the same results, that are scattered all over the literature. 
A model of $ZFC$ in which $covK=\omega_1$ and $\omega_2={}^{\omega}2$
can be  defined by an $\omega_1$ iteration of random models over a model of $2^{\omega}=\omega_2$.
Another model is of $ZFC$ for which $covK=\omega_2$ and $^{\omega}2=\omega_3$ is defined as follows.
Start with $M\models 2^{\omega}=\omega_3\land \exists D |D|=\omega_1$ $\forall f\in \omega^{\omega} \exists g\in D \forall nf(n)<g(n)$.
(There exists an $\omega_1$ scale).
Let $N$ be an iterated random real extension of $M$. 
Then in $N$ $covK\geq \omega_2$ because the iteration has length $\omega_2$.
Also the set of $\omega_2$ 
Cohen reals added by the iteration is not meager, hence $covK=\omega_2$.
Now we show that $OTT(covK)$ is false, by which we will be done. 
We go back to $OTT$ for $FOL$. We work in a countable language $L$ where the variables available are of order type $\omega$. The variables are $v_0, v_1, \ldots v_n\ldots$.
For $n\leq \omega$, $Fm_n$ is the set of all formulas with free variables in $\{v_i:i<n\}$. An $n$ type is a $p\subseteq Fm_n$.
Fix a theory $T$ in $L$. An $n$ type $p$ is principal if there exists $\psi$ such that $T\models \psi\implies p$. Otherwise it is non-principal.
Recall that the classical Henkin-Orey omitting types theorem says that countably many non-principal types can be omitted. 
$S^n(T)$ is the set of complete types in the variables $\{v_i:i<n\}$ which are consistent with $T$. Algebraically, these are the ultrafilters in the algebra
$\B_n=\Fm_n/T$. Note that $\B_n\in \Nr_n\CA_{\omega}$.
$S^n(T)$ is the stone space of $\B_n$. $S^n(T)$ is a Boolean space with a countable basis of open sets
$$[\phi(\bar{x})]=\{p\in S^{n}T: \phi(\bar{x})\in p\}.$$
There is a correspondance between types in the variables $\{x_i:i<n\}$ and closed sets in $S^n(T)$, the closed set associated with the type $p$ is
$\bigcap_{\phi\in p}[\phi(\bar{x})]$. $p(x)$ is non-principal iff $\bigcap_{\phi\in p}[\phi(\bar{x}]$ is nowhere dense.
Let $n>1$. We adapt an example in \cite{OTT}. Fix $n>1$. 
Let $T$ be a theory such that if $T'$ is a finite and complete extension of $T$, then in $S^n(T')$ the isolated points are not dense.
It is easy to construct such theories, for any fixed $n$. 
Let $X$ be the space $S^0(T)$ of all complete $0$ types which are consistent with $T$. For an ordinal $\alpha$, let $X^{(\alpha)}$ 
be the $\alpha$-iterated
Cantor-Bendixon derivative of $X$. The language is countable, there is some $\alpha<\omega_1$ 
such that $X^{(\alpha)}=X^{(\alpha+1)}$
and $X\setminus X^{\alpha}$ is countable. $X^{\alpha}$ is a perfect set and therefore it is homeomorphic to the Cantor space $^{\omega}2$ or it is empty.
Recall that $covK$ is the least cardinal such that  the real line (equivalently any Polish space without isolated points) can be covered by $covK$ many closed nowhere
dense sets. Then, clearly $\omega<covK\leq {}^{\omega}2$. Martin's axiom (by the above) implies that $covK={}^{\omega}2$ but it is consistent that 
$covK<{}^{\omega}2$, as also illustrated above. 
We associate a set $P_{\infty}$ of $\leq covK$ many types with $X^{\alpha}$. Assume that $X^{\alpha}$ is non-empty, since it is a 
closed set in $X$, 
there is some extension $T_{\infty}$ of $T$
such that in $X$ 
$$X^{\alpha}=\bigcap_{\sigma\in T^{\infty}}[\sigma].$$
Hence the space $S^0(T_{\infty})$ is homeomorphic to $X^{\alpha}$ and to $^{\omega}2$. Then there are $Y_{\beta} (\beta<covK)$ 
closed nowhere dense sets in $S^0(T_{\infty})$
such that
$$S^0(T_{\infty})=\bigcup_{\beta<covK}Y_{\beta}.$$
Let $\beta<covK$. Since $Y_{\beta}$ is closed, there is a $0$ type $p_{\beta}$ such that in $S^{0}(T_{\infty})$
$$Y_{\beta}=\bigcap_{\sigma\in p_{\beta}}[\sigma].$$
As $Y_{\beta}$ is nowhere dense $p_{\beta}$ is non principal in $T_{\infty}$. Assuming, without loss, that $T_{\infty}\subseteq p_{\beta}$ we get that 
$p_{\beta}$ is non 
principal in $T$.
Set
$$P_{\infty}=\{p_{\beta}:\beta<covK\}.$$
Let us consider the $0$ types in $X\setminus X^{\alpha}$. These are complete consistent extensions of $T$.
For every $T'\in X\setminus X^{\alpha}$ we shall define a set $P_{T'}$ of $\leq covK$ many $n$ types that are not omitted in $T'$.
If $T'$ is not a finite extension of $T$, set $P_{T'}=\{T'\}$. Otherwise, in $S^n(T')$ the isolated types are not dense. 
Hence there is some non-empty $Y\subseteq
S^n(T')$ clopen and perfect. Now we can cover $Y$ with a family of $covK$ many closed nowhere dense sets of $n$ types. 
Since $Y$ is clo-open in $S^{n}(T')$, these sets are closed nowhere dense sets in $S^{n}(T')$, 
so we obtain a family of $covK$ many non principal $n$ types that cannot be omitted. We may assume that 
$T'\subseteq p$ for every $p\in P_T'$ and therefore every type in $P_T'$ is non principal in $T$.
Define
$$P=P_{\infty}\cup\bigcup\{P_{T'}: T'\in X\setminus X^{\alpha}\}.$$  
Now $P$ is a family of non-principal types $|P|=covK$ that cannot be omitted.
Let $\A=\Fm/T$ and for $p\in P$ let $X_p=\{\phi/T:\phi\in p\}$. Then $X_i\subseteq \Nr_n\A$, and $\prod X_i=0$.
However for any $0\neq a$, there is no set algebra $\C$ with countable base 
$M$ and $g:\A\to \C$ such that $g(a)\neq 0$ and $\bigcap_{x\in X_i}f(x)=\emptyset$.
Now let $\B=\Nr_n\A$. Let $a\neq 0$. Assume, seeking a contradiction, that there exists 
$f:\B\to \D'$ such that $f(a)\neq 0$ and $\bigcap_{x\in X_i} f(x)=\emptyset$. We can assume that
$B$ generates $\A$ and that $\D'=\Nr_n\D$ where $\D\in \Lf_{\omega}$. Let $g=\Sg^{\A\times \D}f$. We will show  
that $g$ is a one to one function with domain $\A$ that preserves the $X_i$'s which is impossible (Note that by definition $g$ is a homomorphism). 
We have
$$Dom g=Dom\Sg^{\A\times \D}f=\Sg^{\A}Domf=\Sg^{\A}\Nr_{n}\A=\A.$$
By symmetry it is enough to show that $g$ is a function.  We first prove the following (*)
 $$ \text { If } (a,b)\in g\text { and }  {\sf c}_k(a,b)=(a,b)\text { for all } k\in \omega\sim n, \text { then } f(a)=b.$$
Indeed,
$$(a,b)\in \Nr_{n}\Sg^{\A\times \D}f=\Sg^{\Nr_{n}(\A\times \D)}f=\Sg^{\Nr_{n}\A\times \Nr_{n}\D}f=f.$$
Here we are using that $\A\times \D\in \Lf_{\omega}$, so that  $\Nr_{n}\Sg^{\A\times \D}f=\Sg^{\Nr_{n}(\A\times \D)}f.$
Now suppose that $(x,y), (x,z)\in g$.
Let $k\in \omega\sim n.$ Let $\Delta$ denote symmetric difference. Then
$$(0, {\sf c}_k(y\Delta z))=({\sf c}_k0, {\sf c}_k(y\Delta z))={\sf c}_k(0,y\Delta z)={\sf c}_k((x,y)\Delta(x,z))\in g.$$
Also,
$${\sf c}_k(0, {\sf c}_k(y\Delta z))=(0,{\sf c}_k(y\Delta z)).$$ 
Thus by (*) we have  $$f(0)={\sf c}_k(y\Delta z) \text { for any } k\in \omega\sim n.$$
Hence ${\sf c}_k(y\Delta z)=0$ and so $y=z$.
We conclude that there exists a countable $\B\in \Nr_n\CA_{\omega}$ and $(X_i:i<covK)$ such that $\prod X_i=0$ but there is no representation that
preserves the $X_i$'s. In more detail. Give any $a\in \B$, if $a$ is non zero, $\C$ is a set algebra with countable base 
and $f:\B\to \C$ is a homomorphism such 
that $f(a)\neq 0$, then there exists $i<covK$, such that $\bigcap_{x\in X_i} f(x)\neq \emptyset.$
Therefore $OTT$ is false in a model of $ZFC+\neg CH$.
\footnote{Another model in which $covK=\omega_2$ and $^{\omega}2=\omega_3$ is due to Bukovski \cite{Miller} 
who does it by starting with a model of $2^{\omega}=\omega_3$, and then doing an $\omega_2$ iteration.
At each step of the iteration he does an $\omega_3$ iteration making $MA$ true. 
Alternatively we could start with a model of $2^{\omega}=\omega_3$ and then do an $\omega_2$ iteration with $\bold D$ 
where $\bold D=\{(n,f): n<\omega, f\in {}^{\omega}\omega\}$
the order for forcing an eventually dominant real.}
\end{demo}
 
Finally we point out that $covK$ can be defined to be  
the least cardinal such that the Baire category theorem for compact Hausdorff spaces fails
or the largest cardinal such that $MA(countable)$ holds or
the largest cardinal $\kappa$ such  $<\kappa$ many non principal types can be omitted.

\subsection{Neat Embeddings, Monk's result yet once again, and games}

The results in the previous section adressed the class $S_c\Nr_n\CA_{\omega}$ of algebras
having a complete neat embedding property. Neat reducts have been a central notion in algebraic logic
since the beginnings. Indeed, the consecutive problems 2.11, 2.12, 2.13  in the monograph \cite{HMT1} are on neat reducts.
Problem 2.12 is solved by Hirsch Hodkinson and Maddux \cite{HHM}. The authors of \cite{HHM} show that the 
sequence $\langle S\Nr_n\CA_{n+k}: k\in \omega\rangle$ is strictly decreasing for $\omega>n>2$ with respect to inclusion.
(Recall that we generalized this rsult to quasipolyadic equality algebras).
The infinite dimensional case is settled by Pigozzi as reported in \cite{HMT1}.
The main result in \cite{HHM} strengthes Monk's classical result that for every finite $n>2$ and any $k\in \omega$, 
$\RCA_n\subset S\Nr_n\CA_{n+k}.$ Taking $\A_k\in S\Nr_n\CA_{n+k}\sim \RCA_n$, and forming the ultraproduct 
$\prod \A_k/F$ relative to a non-principal ultrafilter
on $\omega$, the resulting structure
will be representable, showing that $\RCA_n$, though, elementary (indeed a variety) is not finitely axiomatizable.
Problem 2.13 is solved above (see the paragraph after theorem \ref{amal}). 
Problem 2.11 which is relevant to our later discussion asks: For which pair of ordinals $\alpha<\beta$ 
is the class $\Nr_{\alpha}\CA_{\beta}$ closed under
forming subalgebras  and homomorphic images?
N\'emeti proves that for any $1<\alpha<\beta$ the class $\Nr_{\alpha}\CA_{\beta}$ though closed under forming homomorphic images
and products 
is not a variety, i.e., it is not closed under forming subalgebras \cite{N83}. The next natural question is whether this class is elementary,
and in this particular case, since the class of neat reducts is closed under ultraproducts, 
 this amounts to asking whether it is  closed under elementary subalgebras?
In \cite{IGPL} it is proved that for any $1<\alpha<\beta$, the class $\Nr_{\alpha}\CA_{\beta}$ is not elementary answering problem 4.4 in \cite{HMT2}. 
In \cite{AU}, it is shown that this class cannot be characterized by any
$L_{\infty \omega}$ sentence. 
In this section we will be concerned with the class $\Nr_n\CA_{\omega}$ when $n$ is finite.
Note that $\Nr_n\CA_{\omega}\subseteq S_c\Nr_n\CA_{\omega}$.
We know that $\Nr_n\CA_{\omega}$ is closed under products and homomorphic images, thus under ultraproducts. However, for $n>1$,
it is {\it not}  closed under
elementary subalgebras, equivalently, under ultraroots. (For $n\leq 1, \Nr_n\CA_{\omega}=\RCA_n=\CA_n$; so this is a degenerate case which we 
ignore). For a class $K$, $ELK$ denotes the elementary closure of $K$, 
that is the least elementary class containing $K$. $UpK$ denotes the class of all ultraproducts of members of $K$
and $UrK$ denotes the class of all ultraroots of members of $K$.
Recall that, by the celebrated Shelah - Keisler theorem,  $ElK=UpUrK$. 

\begin{theorem} Let $n>1$. Then the class $\Nr_n\CA_{\omega}$ is pseudo-elementary, but is not elementary.
Furthermore, $El\Nr_n\CA_{\omega}\subset \RCA_n$, $EL\Nr_n\CA_{\omega}$ is recursively enumerable, and for $n>2$ is 
not finitely axiomatizable.

\end{theorem}

\begin{demo}{Proof} The class $\Nr_n\CA_{\omega}$ is not elementary \cite{IGPL}. 
To show that it is pseudo-elementary, we use a three sorted defining theory, with one sort for a cylindric algebra of dimension $n$ 
$(c)$, the second sort for the Boolean reduct of a cylindric algebra $(b)$
and the thirs sort for a set of dimensions $(\delta)$. We use superscripts $n,b,\delta$ for variables 
and functions to indicate that the variable, or the returned value of the function, 
is of the sort of the cylindric algebra of dimension $n$, the Boolean part of the cylindric algebra or the dimension set, respectively.
The signature includes dimension sort constants $i^{\delta}$ for each $i<\omega$ to represent the dimensions.
The defining theory for $\Nr_n\CA_{\omega}$ incudes sentences demanding that the consatnts $i^{\delta}$ for $i<\omega$ 
are distinct and that the last two sorts define
a cylindric algenra of dimension $\omega$. For example the sentence
$$\forall x^{\delta}, y^{\delta}, z^{\delta}(d^b(x^{\delta}, y^{\delta})=c^b(z^{\delta}, d^b(x^{\delta}, z^{\delta}). d^{b}(z^{\delta}, y^{\delta})))$$
represents the cylindric algebra axiom ${\sf d}_{ij}={\sf c}_k({\sf d}_{ik}.{\sf d}_{kj})$ for all $i,j,k<\omega$.
We have have a function $I^b$ from sort $c$ to sort $b$ and sentences requiring that $I^b$ be injective and to respect the $n$ dimensional 
cylindric operations as follows: for all $x^r$
$$I^b({\sf d}_{ij})=d^b(i^{\delta}, j^{\delta})$$
$$I^b({\sf c}_i x^r)= {\sf c}_i^b(I^b(x)).$$
Finally we require that $I^b$ maps onto the set of $n$ dimensional elements
$$\forall y^b((\forall z^{\delta}(z^{\delta}\neq 0^{\delta},\ldots (n-1)^{\delta}\rightarrow c^b(z^{\delta}, y^b)=y^b))\leftrightarrow \exists x^r(y^b=I^b(x^r))).$$
For ${\A}\in \CA_n,$ $\Rd_3{\A}$ denotes the $\CA_3$
obtained from $\A$ by discarding all operations indexed by indices in $n\sim 3.$
$\Df_n$ denotes the class of diagonal free cylindric algebras.
$\Rd_{df}{\A}$ denotes the $\Df_n$ obtained from $\A$
by deleting all diagonal elements. To prove the non-finite axiomatizability result we use Monk's algebras.
For $3\leq n,i<\omega$, with $n-1\leq i, {\C}_{n,i}$ denotes
the $\CA_n$ associated with the cylindric atom structure as defined on p. 95 of 
\cite{HMT1}.
Then by \cite[3.2.79]{HMT1}
for $3\leq n$, and $j<\omega$, 
$\Rd_3{\C}_{n,n+j}$ can be neatly embedded in a 
$\CA_{3+j+1}$.  (1)
By \cite[3.2.84]{HMT1}) we have for every $j\in \omega$, 
there is an $3\leq n$ such that $\Rd_{df}\Rd_{3}{\cal C}_{n,n+j}$
is a non-representable $\Df_3.$ (2)
Now suppose $m\in \omega$. By (2), 
choose $j\in \omega\sim 3$ so that $\Rd_{df}\Rd_3{\C}_{j,j+m+n-4}$
is a non-representable $\Df_3$. By (1) we have
$\Rd_{df}\Rd_3{\C}_{j,j+m+n-4}\subseteq \Nr_3{\B_m}$, for some 
${\B}\in \CA_{n+m}.$
Put ${\A}_m=\Nr_n\B_m$.  
$\Rd_{df}{\A}_m$ is not representable, 
a friotri, ${\A}_m\notin \RCA_{n},$ for else its $\Df$ reduct would be 
representable. Therefore $\A_m\notin EL\Nr_n\CA_{\omega}$.
Now let $\C_m$ be an algebra similar to $\CA_{\omega}$'s such that $\B_m=\Rd_{n+m}\C_m$.
Then $\A_m=\Nr_n\C_m$. Let $F$ be a non-principal ultrafilter on $\omega$. Then
$$\prod_{m\in \omega}\A_m/F=\prod_{m\in \omega}(\Nr_n\C_m)/F=\Nr_n(\prod_{m\in \omega}\C_m/F)$$
But $\prod_{m\in \omega}\C_m/F\in \CA_{\omega}$. Hence $\CA_n\sim El\Nr_n\CA_{\omega}$ is not closed under ultraproducts.
It follows that the latter class is not finitely axiomatizable.
In \cite{IGPL} it is proved that for $1<\alpha<\beta$, $El\Nr_{\alpha}\CA_{\beta}\subset S\Nr_{\alpha}\CA_{\beta}$.
\end{demo}
From the above proof it follows that
\begin{corollary} Let $K$ be any class such that $\Nr_n\CA_{\omega}\subseteq K\subseteq \RCA_n$. Then $ELK$ is not finitely axiomatizable
\end{corollary}

For $n>2$ the addition of finitely many 
{\it first order definable operations} does not remedy the non-finite axiomatizability
result for $\RCA_n$, as proved by Biro. First order definable operations are those operations that can be defined using spare
dimensions, and hence the notion of neat reducts are appropriate for handing them. A non-trivial question 
that relates to the Finitization problem, and involves the class $\Nr_n\CA_{\omega}$ in an essential way, is whether
we can expand the signature of cylindric algebras by extra natural operations on $n$-ary relations so that
if $\A\in \Cs_n$ and is closed under these operations then this forces $\A$ to be in the class $\Nr_n\CA_{\omega}.$
(For example, the polyadic operations are not enough.) The class $\Nr_n\CA_{\omega}$ contains 
all first order definable operations, so the question can be reformulated
as to whether one can capture all first order definable operations using a {\it finite} set of operations.
Next we characterize the class $\Nr_n\CA_{\omega}$ using games. 
Since games go deeper into the analysis, they could shed light on the possible choice of such operations. 
For that, we need some preparations. We use ``cylindric algebra" games that are analogues
to certain ``relation algebra" games used by Robin Hirsch in \cite{R}. In \cite{R} Robin Hirsch studies quite extensively the 
class $\Ra\CA_n$ of relation algebra reducts of cylindric algebras of dimension $n$. This class was studied by many authors, to mention a few, 
Maddux, Simon and Nemeti. References for their work can be found in the most recent reference \cite{R}.
Our treatment in this part 
follows very closely \cite{R}.

\begin{definition}\label{def:string} 
Let $n$ be an ordinal. An $s$ word is a finite string of substitutions $({\sf s}_i^j)$, 
a $c$ word is a finite string of cylindrifications $({\sf c}_k)$.
An $sc$ word is a finite string of substitutions and cylindrifications
Any $sc$ word $w$ induces a partial map $\hat{w}:n\to n$
by
\begin{itemize}

\item $\hat{\epsilon}=Id$

\item $\widehat{w_j^i}=\hat{w}\circ [i|j]$

\item $\widehat{w{\sf c}_i}= \hat{w}\upharpoonright(n\sim \{i\}$ 

\end{itemize}
\end{definition}
If $\bar a\in {}^{<n-1}n$, we write ${\sf s}_{\bar a}$, or more frequently 
${\sf s}_{a_0\ldots a_{k-1}}$, where $k=|\bar a|$,
for an an arbitary chosen $sc$ word $w$
such that $\hat{w}=\bar a.$ 
$w$  exists and does not 
depend on $w$ by \cite[definition~5.23 ~lemma 13.29]{HHbook}. 
We can, and will assume \cite[Lemma 13.29]{HHbook} 
that $w=s{\sf c}_{n-1}{\sf c}_n.$
[In the notation of \cite[definition~5.23,~lemma~13.29]{HHbook}, 
$\widehat{s_{ijk}}$ for example is the function $n\to n$ taking $0$ to $i,$
$1$ to $j$ and $2$ to $k$, and fixing all $l\in n\setminus\set{i, j,k}$.]
Let $\delta$ be a map. Then $\delta[i\to d]$ is defined as follows. $\delta[i\to d](x)=\delta(x)$
if $x\neq i$ and $\delta[i\to d](i)=d$. We write $\delta_i^j$ for $\delta[i\to \delta_j]$.

\begin{definition}
From now on let $2\leq n<\omega.$ Let $\C$ be an atomic $\CA_{n}$. 
An \emph{atomic  network} over $\C$ is a map
$$N: {}^{n}\Delta\to At\cal C$$ 
such that the following hold for each $i,j<n$, $\delta\in {}^{n}\Delta$
and $d\in \Delta$:
\begin{itemize}
\item $N(\delta^i_j)\leq {\sf d}_{ij}$
\item $N(\delta[i\to d])\leq {\sf c}_iN(\delta)$ 
\end{itemize}
\end{definition}
Note than $N$ can be viewed as a hypergraph with set of nodes $\Delta$ and 
each hyperedge in ${}^{\mu}\Delta$ is labelled with an atom from $\C$.
We call such hyperedges atomic hyperedges.
We write $\nodes(N)$ for $\Delta.$ But it can happen 
let $N$ stand for the set of nodes 
as well as for the function and the network itself. Context will help.

Define $x\sim y$ if there exists $\bar{z}$ such that $N(x,y,\bar{z})\leq {\sf d}_{01}$.
Define an equivalence relation
$\sim$ over the set of all finite sequences over $\nodes(N)$ by $\bar
x\sim\bar y$ iff $|\bar x|=|\bar y|$ and $x_i\sim y_i$ for all
$i<|\bar x|$.

(3) A \emph{ hypernetwork} $N=(N^a, N^h)$ over $\cal C$ 
consists of a network $N^a$
together with a labelling function for hyperlabels $N^h:\;\;^{<
\omega}\!\nodes(N)\to\Lambda$ (some arbitrary set of hyperlabels $\Lambda$)
such that for $\bar x, \bar y\in\; ^{< \omega}\!\nodes(N)$ 
\begin{enumerate}
\renewcommand{\theenumi}{\Roman{enumi}}
\setcounter{enumi}3
\item\label{net:hyper} $\bar x\sim\bar y \Rightarrow N^h(\bar x)=N^h(\bar y)$. 
\end{enumerate}
If $|\bar x|=k\in nats$ and $N^h(\bar x)=\lambda$ then we say that $\lambda$ is
a $k$-ary hyperlabel. $(\bar x)$ is referred to a a $k$-ary hyperedge, or simply a hyperedge.
(Note that we have atomic hyperedges and hyperedges) 
When there is no risk of ambiguity we may drop the superscripts $a,
h$. 

The following notation is defined for hypernetworks, but applies
equally to networks.  

(4) If $N$ is a hypernetwork and $S$ is any set then
$N\restr S$ is the $n$-dimensional hypernetwork defined by restricting
$N$ to the set of nodes $S\cap\nodes(N)$.  For hypernetworks $M, N$ if
there is a set $S$ such that $M=N\restr S$ then we write $M\subseteq
N$.  If $N_0\subseteq N_1\subseteq \ldots $ is a nested sequence of
hypernetworks then we let the \emph{limit} $N=\bigcup_{i<\omega}N_i$  be
the hypernetwork defined by
$\nodes(N)=\bigcup_{i<\omega}\nodes(N_i)$,\/ $N^a(x_0,\ldots x_{n-1})= 
N_i^a(x_0,\ldots x_{n-1})$ if
$x_0\ldots x_{\mu-1}\in\nodes(N_i)$, and $N^h(\bar x)=N_i^h(\bar x)$ if $\rng(\bar
x)\subseteq\nodes(N_i)$.  This is well-defined since the hypernetworks
are nested and since hyperedges $\bar x\in\;^{<\omega}\nodes(N)$ are
only finitely long.

For hypernetworks $M, N$ and any set $S$, we write $M\equiv^SN$
if $N\restr S=M\restr S$.  For hypernetworks $M, N$, 
and any set $S$, we write $M\equiv_SN$ 
if the symmetric difference $\Delta(\nodes(M), \nodes(N))\subseteq S$ and
$M\equiv^{(\nodes(M)\cup\nodes(N))\setminus S}N$. We write $M\equiv_kN$ for
$M\equiv_{\set k}N$.

Let $N$ be a network and let $\theta$ be any function.  The network
$N\theta$ is a complete labelled graph with nodes
$\theta^{-1}(\nodes(N))=\set{x\in\dom(\theta):\theta(x)\in\nodes(N)}$,
and labelling defined by 
$(N\theta)(i_0,\ldots i_{\mu-1}) = N(\theta(i_0), \theta(i_1), \theta(i_{\mu-1}))$,
for $i_0, \ldots i_{\mu-1}\in\theta^{-1}(\nodes(N))$.  Similarly, for a hypernetwork
$N=(N^a, N^h)$, we define $N\theta$ to be the hypernetwork
$(N^a\theta, N^h\theta)$ with hyperlabelling defined by
$N^h\theta(x_0, x_1, \ldots) = N^h(\theta(x_0), \theta(x_1), \ldots)$
for $(x_0, x_1,\ldots) \in \;^{<\omega}\!\theta^{-1}(\nodes(N))$.

Let $M, N$ be hypernetworks.  A \emph{partial isomorphism}
$\theta:M\to N$ is a partial map $\theta:\nodes(M)\to\nodes(N)$ such
that for any $
i_i\ldots i_{\mu-1}\in\dom(\theta)\subseteq\nodes(M)$ we have $M^a(i_1,\ldots i_{\mu-1})= 
N^a(\theta(i), \ldots\theta(i_{\mu-1}))$
and for any finite sequence $\bar x\in\;^{<\omega}\!\dom(\theta)$ we
have $M^h(\bar x) = 
N^h\theta(\bar x)$.  
If $M=N$ we may call $\theta$ a partial isomorphism of $N$.

\begin{definition}\label{def:games} Let $2\leq n<\omega$. For any $\CA_{n}$  
atom structure $\alpha$, and $n\leq m\leq
\omega$, we define two-player games $F_{n}^m(\alpha),$ \; and 
$H_{n}(\alpha)$,
each with $\omega$ rounds, 
and for $m<\omega$ we define $H_{m,n}(\alpha)$ with $n$ rounds.

\begin{itemize}
\item 
Let $m\leq \omega$.  
In a play of $F_{n}^m(\alpha)$ the two players construct a sequence of
networks $N_0, N_1,\ldots$ where $\nodes(N_i)$ is a finite subset of
$m=\set{j:j<m}$, for each $i$.  In the initial round of this game \pa\
picks any atom $a\in\alpha$ and \pe\ must play a finite network $N_0$ with
$\nodes(N_0)\subseteq  n$, 
such that $N_0(\bar{d}) = a$ 
for some $\bar{d}\in{}^{\mu}\nodes(N_0)$.
In a subsequent round of a play of $F_{n}^m(\alpha)$ \pa\ can pick a
previously played network $N$ an index $\l<n$, a ``face" 
$F=\langle f_0,\ldots f_{n-2} \rangle \in{}^{n-2}\nodes(N),\; k\in
m\setminus\set{f_0,\ldots f_{n-2}}$, and an atom $b\in\alpha$ such that 
$b\leq {\sf c}_lN(f_0,\ldots f_i, x,\ldots f_{n-2}).$  
(the choice of $x$ here is arbitrary, 
as the second part of the definition of an atomic network together with the fact
that $\cyl i(\cyl i x)=\cyl ix$ ensures that the right hand side does not depend on $x$).
This move is called a \emph{cylindrifier move} and is denoted
$(N, \langle f_0, \ldots f_{\mu-2}\rangle, k, b, l)$ or simply $(N, F,k, b, l)$.
In order to make a legal response, \pe\ must play a
network $M\supseteq N$ such that 
$M(f_0,\ldots f_{i-1}, k, f_i,\ldots f_{n-2}))=b$ 
and $\nodes(M)=\nodes(N)\cup\set k$.

\pe\ wins $F_{n}^m(\alpha)$ if she responds with a legal move in each of the
$\omega$ rounds.  If she fails to make a legal response in any
round then \pa\ wins.

\item
Fix some hyperlabel $\lambda_0$.  $H_{n}(\alpha)$ is  a 
game the play of which consists of a sequence of
$\lambda_0$-neat hypernetworks 
$N_0, N_1,\ldots$ where $\nodes(N_i)$
is a finite subset of $\omega$, for each $i<\omega$.  
In the initial round \pa\ picks $a\in\alpha$ and \pe\ must play
a $\lambda_0$-neat hypernetwork $N_0$ with nodes contained in
$\mu$ and $N_0(\bar d)=a$ for some nodes $\bar{d}\in {}^{\mu}N_0$.  
At a later stage
\pa\ can make any cylindrifier move $(N, F,k, b, l)$ by picking a
previously played hypernetwork $N$ and $F\in {}^{n-2}\nodes(N), \;l<n,  
k\in\omega\setminus\nodes(N)$ 
and $b\leq {\sf c}_lN(f_0, f_{l-1}, x, f_{n-2})$.  
[In $H_{n}$ we
require that \pa\ chooses $k$ as a `new node', i.e. not in
$\nodes(N)$, whereas in $F_{n}^m$ for finite $m$ it was necessary to allow
\pa\ to `reuse old nodes'. This makes the game easior as far as $\forall$ is concerned.) 
For a legal response, \pe\ must play a
$\lambda_0$-neat hypernetwork $M\equiv_k N$ where
$\nodes(M)=\nodes(N)\cup\set k$ and 
$M(f_0, f_{i-1}, k, f_{n-2})=b$.
Alternatively, \pa\ can play a \emph{transformation move} by picking a
previously played hypernetwork $N$ and a partial, finite surjection
$\theta:\omega\to\nodes(N)$, this move is denoted $(N, \theta)$.  \pe\
must respond with $N\theta$.  Finally, \pa\ can play an
\emph{amalgamation move} by picking previously played hypernetworks
$M, N$ such that $M\equiv^{\nodes(M)\cap\nodes(N)}N$ and
$\nodes(M)\cap\nodes(N)\neq \emptyset$.  
This move is denoted $(M,
N)$.  To make a legal response, \pe\ must play a $\lambda_0$-neat
hypernetwork $L$ extending $M$ and $N$, where
$\nodes(L)=\nodes(M)\cup\nodes(N)$.

Again, \pe\ wins $H_n(\alpha)$ if she responds legally in each of the
$\omega$ rounds, otherwise \pa\ wins. 

\item For $m< \omega$ the game $H_{m,n}(\alpha)$ is similar to $H_n(\alpha)$ but
play ends after $m$ rounds, so a play of $H_{m,n}(\alpha)$ could be
\[N_0, N_1, \ldots, N_m\]
If \pe\ responds legally in each of these
$m$ rounds she wins, otherwise \pa\ wins.
\end{itemize}

\end{definition}

\begin{definition}\label{def:hat}
For $m\geq 5$ and $\c C\in\CA_m$, if $\A\subseteq\Nr_n(\C)$ is an
atomic cylindric algebra and $N$ is an $\A$-network then we define
$\widehat N\in\C$ by
\[\widehat N =
 \prod_{i_0,\ldots i_{n-1}\in\nodes(N)}{\sf s}_{i_0, \ldots i_{n-1}}N(i_0\ldots i_{n-1})\]
$\widehat N\in\C$ depends
implicitly on $\C$.
\end{definition}
We write $\A\subseteq_c \B$ if $\A\in S_c\{\B\}$. 
\begin{lemma}\label{lem:atoms2}
Let $n<m$ and let $\A$ be an atomic $\CA_n$, 
$\A\subseteq_c\Nr_n\C$
for some $\C\in\CA_m$.  For all $x\in\C\setminus\set0$ and all $i_0, \ldots i_{n-1} < m$ there is $a\in\At(\A)$ such that
${\sf s}_{i_0\ldots i_{n-1}}a\;.\; x\neq 0$.
\end{lemma}
\begin{proof}
We can assume, see definition  \ref{def:string}, 
that ${\sf s}_{i_0,\ldots i_{n-1}}$ consists only of substitutions, since ${\sf c}_{m}\ldots {\sf c}_{m-1}\ldots 
{\sf c}_nx=x$ 
for every $x\in \A$.We have ${\sf s}^i_j$ is a
completely additive operator (any $i, j$), hence ${\sf s}_{i_0,\ldots i_{\mu-1}}$ 
is too  (see definition~\ref{def:string}).
So $\sum\set{{\sf s}_{i_0\ldots i_{n-1}}a:a\in\At(\A)}={\sf s}_{i_0\ldots i_{n-1}}
\sum\At(\A)={\sf s}_{i_0\ldots i_{n-1}}1=1$,
for any $i_0,\ldots i_{n-1}<n$.  Let $x\in\C\setminus\set0$.  It is impossible
that ${\sf s}_{i_0\ldots i_{n-1}}\;.\;x=0$ for all $a\in\At(\c A)$ because this would
imply that $1-x$ was an upper bound for $\set{{\sf s}_{i_0\ldots i_{n-1}}a:
a\in\At(\A)}$, contradicting $\sum\set{{\sf s}_{i_0\ldots i_{n-1}}a :a\in\At(\c A)}=1$.
\end{proof}

\begin{lemma}\label{lem:hat}
Let $n<m$ and let $\A\subseteq_c\Nr_{n}\C$ be an
atomic $\CA_n$
\begin{enumerate}
\item For any $x\in\C\setminus\set0$ and any
finite set $I\subseteq m$ there is a network $N$ such that
$\nodes(N)=I$ and $x\;.\;\widehat N\neq 0$.
\item
For any networks $M, N$ if 
$\widehat M\;.\;\widehat N\neq 0$ then $M\equiv^{\nodes(M)\cap\nodes(N)}N$.
\end{enumerate}
\end{lemma}

\begin{proof}
The proof of the first part is based on repeated use of
lemma~\ref{lem:atoms2}. We define the edge labelling of $N$ one edge
at a time. Initially no hyperedges are labelled.  Suppose
$E\subseteq\nodes(N)\times\nodes(N)\ldots  \times\nodes(N)$ is the set of labelled hyper
edges of $N$ (initially $E=\emptyset$) and 
$x\;.\;\prod_{\bar c \in E}{\sf s}_{\bar c}N(\bar c)\neq 0$.  Pick $\bar d$ such that $\bar d\not\in E$.  
By lemma~\ref{lem:atoms2} there is $a\in\At(\c A)$ such that
$x\;.\;\prod_{\bar c\in E}{\sf s}_{\bar c}N(\bar c)\;.\;{\sf s}_{\bar d}a\neq 0$.  
Include the edge $\bar d$ in $E$.  Eventually, all edges will be
labelled, so we obtain a completely labelled graph $N$ with $\widehat
N\neq 0$.  
it is easily checked that $N$ is a network.
For the second part, if it is not true that
$M\equiv^{\nodes(M)\cap\nodes(N)}N$ then there are is 
$\bar c \in^{n-1}\nodes(M)\cap\nodes(N)$ such that $M(\bar c )\neq N(\bar c)$.  
Since edges are labelled by atoms we have $M(\bar c)\cdot N(\bar c)=0,$ 
so $0={\sf s}_{\bar c}0={\sf s}_{\bar c}M(\bar c)\;.\; {\sf s}_{\bar c}N(\bar c)\geq \widehat M\;.\;\widehat N$.
\end{proof}

\begin{lemma}\label{lem:khat}  Let
Let $m>n$. Let $\C\in\CA_m$ and let $\A\subseteq\Nr_{n}(\C)$ be atomic.
Let $N$ be a network over $\c A$ and $i,j <n$.
\begin{enumerate}
\item\label{it:-i}
If $i\not\in\nodes(N)$ then ${\sf c}_i\widehat N=\widehat N$.

\item \label{it:-j} $\widehat{N Id_{-j}}\geq \widehat N$.

\item\label{it:ij} If $i\not\in\nodes(N)$ and $j\in\nodes(N)$ then
$\widehat N\neq 0 \rightarrow \widehat{N[i/j]}\neq 0$.
where $N[i/j]=N\circ [i|j]$

\item\label{it:theta} If $\theta$ is any partial, finite map $n\to n$
and if $\nodes(N)$ is a proper subset of $n$,
then $\widehat N\neq 0\rightarrow \widehat{N\theta}\neq 0$.
\end{enumerate}
\end{lemma}
\begin{proof}
The first part is easy.
The second
part is by definition of $\;\widehat{\;}$. For the third part suppose
$\widehat N\neq 0$.  Since $i\not\in\nodes(N)$, by part~\ref{it:-i},
we have ${\sf c}_i\widehat N=\widehat N$.  By cylindric algebra axioms it
follows that $\widehat N\;.\;d_{ij}\neq 0$.  By lemma~\ref{lem:hat}
there is a network $M$ where $\nodes(M)=\nodes(N)\cup\set i$ such that
$\widehat M\;.\widehat N\;.\;d_{ij}\neq 0$.  By lemma~\ref{lem:hat} we
have $M\supseteq N$ and $M(i, j)\leq 1'$.  It follows that $M=N[i/j]$.
Hence $\widehat{N[i/j]}\neq 0$.
For the final part 
(cf. \cite[lemma~13.29]{HHbook}), since there is 
$k\in n\setminus\nodes(N)$, \/ $\theta$ can be
expressed as a product $\sigma_0\sigma_1\ldots\sigma_t$ of maps such
that, for $s\leq t$, we have either $\sigma_s=Id_{-i}$ for some $i<n$
or $\sigma_s=[i/j]$ for some $i, j<n$ and where
$i\not\in\nodes(N\sigma_0\ldots\sigma_{s-1})$.
Now apply parts~\ref{it:-j} and \ref{it:ij} of the lemma.
\end{proof}

We now prove two Theorems relating neat embeddings
to the games we defined:

\begin{theorem}\label{thm:n}
Let $n<m$, and let $\A$ be a $\CA_m$.  
If $\A\in{\bf S_c}\Nr_{n}\CA_m, $
then \pe\ has a \ws\ in $F^m(\At\A)$. 
\end{theorem}
\begin{proof}
If $\A\subseteq\Nr_n\C$ for some $\C\in\CA_m$ then \pe\ always
plays hypernetworks $N$ with $\nodes(N)\subseteq n$ such that
$\widehat N\neq 0$. In more detail, in the initial round , let $\forall$ play $a\in \At \cal A$.
$\exists$ play a network $N$ with $N(0, \ldots n-1)=a$. Then $\widehat N=a\neq 0$.
At a later stage suppose $\forall$ plays the cylindrifier move 
$(N, \langle f_0, \ldots f_{\mu-2}\rangle, k, b, l)$ 
by picking a
previously played hypernetwork $N$ and $f_i\in \nodes(N), \;l<\mu,  k\notin \{f_i: i<n-2\}$, 
and $b\leq {\sf c}_lN(f_0,\ldots  f_{i-1}, x, f_{n-2})$.
Let $\bar a=\langle f_0\ldots f_{l-1}, k\ldots f_{n-2}\rangle.$
Then ${\sf c}_k\widehat N\cdot {\sf s}_{\bar a}b\neq 0$.
By \ref{lem:atoms2} there is a network  $M$ such that
$\widehat{M}.\widehat{{\sf c}_kN}\cdot {\sf s}_{\bar a}b\neq 0$. Hence 
$M(f_0, k, f_{n-2})=b.$
\end{proof}

\begin{theorem}\label{thm:RaC}
Let $\alpha$ be a countable 
$\CA_n$ atom structure.  If \pe\ has a \ws\ in $H_n(\alpha)$ then
there is a representable cylindric algebra $\C$ of
dimension $\omega$ such that $\Nr_n\c C$ is atomic 
and $\At \Nr_n\C\cong\alpha$.
\end{theorem}

\begin{proof} 
Suppose \pe\ has a \ws\ in $H(\alpha)$. Fix some $a\in\alpha$. We can define a
nested sequence $N_0\subseteq N_1\ldots$ of hypernetworks
where $N_0$ is \pe's response to the initial \pa-move $a$, requiring that
\begin{enumerate}
\item If $N_r$ is in the sequence and 
and $b\leq {\sf c}_lN_r(\langle f_0, f_{n-2}\rangle\ldots , x, f_{n-2})$.  
then there is $s\geq r$ and $d\in\nodes(N_s)$ such 
that $N_s(f_0, f_{i-1}, d, f_{n-2})=b$.
\item If $N_r$ is in the sequence and $\theta$ is any partial
isomorphism of $N_r$ then there is $s\geq r$ and a
partial isomorphism $\theta^+$ of $N_s$ extending $\theta$ such that
$\rng(\theta^+)\supseteq\nodes(N_r)$.
\end{enumerate}
Since $\alpha$ is countable there are countably many requirements  to extend. 
Since the sequence of networks is nested , these requirements
to extend
remain in all subsequent rounds. So that we can schedule these requirements
to extend so that eventually, every requirement gets dealt with.
If we are required to find $k$ and $N_{r+1}\supset N_r$
such that 
$N_{r+1}(f_0, k, f_{n-2})=b$ then let $k\in \omega\setminus \nodes(N_r)$
be least possible for definiteness, 
and let $N_{r+1}$ be \pe's response using her \ws, 
to the \pa move $N_r, (f_0,\ldots f_{n-1}), k, b, l).$
For an extension of type 2, let $\tau$ be a partial isomorphism of $N_r$
and let $\theta$ be any finite surjection onto a partial isomorphism of $N_r$ such that 
$dom(\theta)\cap nodes(N_r)= dom\tau$. \pe's response to \pa's move $(N_r, \theta)$ is necessarily 
$N\theta.$ Let $N_{r+1}$ be her response , using her wining strategy, to the subsequent \pa 
move $(N_r, N_r\theta).$ 

Now let $N_a$ be the limit of this sequence.
This limit is well-defined since the hypernetworks are nested.  Note,
for $b\in\alpha$, that
\begin{equation}\label{eq:sim}( \exists i_0, \ldots I_{\mu-1}\in\nodes(N_a),\; N_a(i_0\ldots , i_{\mu-1}) = b)\iff b\sim a
\end{equation}

Let $\theta$ be any finite partial isomorphism of $N_a$ and let $X$ be
any finite subset of $\nodes(N_a)$.  Since $\theta, X$ are finite, there is
$i<\omega$ such that $\nodes(N_i)\supseteq X\cup\dom(\theta)$. There
is a bijection $\theta^+\supseteq\theta$ onto $\nodes(N_i)$ and $j\geq
i$ such that $N_j\supseteq N_i, N_i\theta^+$.  Then $\theta^+$ is a
partial isomorphism of $N_j$ and $\rng(\theta^+)=\nodes(N_i)\supseteq
X$.  Hence, if $\theta$ is any finite partial isomorphism of $N_a$ and
$X$ is any finite subset of $\nodes(N_a)$ then
\begin{equation}\label{eq:theta}
\exists \mbox{ a partial isomorphism $\theta^+\supseteq \theta$ of $N_a$
 where $\rng(\theta^+)\supseteq X$}
\end{equation}
and by considering its inverse we can extend a partial isomorphism so
as to include an arbitrary finite subset of $\nodes(N_a)$ within its
domain.
Let $L$ be the signature with one $\mu$ -ary predicate symbol ($b$) for
each $b\in\alpha$, and one $k$-ary predicate symbol ($\lambda$) for
each $k$-ary hyperlabel $\lambda$.  [Notational point: if $\lambda$ is
$k$-ary and $l$-ary for $k\neq l$ then make one $k$-ary predicate
symbol $\lambda$ and one $l$-ary predicate symbol $\lambda'$, so that
every predicate symbol has a unique arity.]  The set of variables for
$L$-formulas is $\set{x_i:i<\omega}$. We also have equality.  
Pick $f_a\in\;^\omega\!\nodes(N_a)$.  Let
$U_a=\set{f\in\;^\omega\!\nodes(N_a):\set{i<\omega:g(i)\neq
f_a(i)}\mbox{ is finite}}$.


We can make $U_a$ into the base of an $L$-structure $\c N_a$ and
evaluate $L$-formulas at $f\in U_a$ as follow.  For $b\in\alpha,\;
l_0, \ldots l_{\mu-1}, i_0 \ldots, i_{k-1}<\omega$, \/ $k$-ary hyperlabels $\lambda$,
and all $L$-formulas $\phi, \psi$, let
\begin{eqnarray*}
\c N_a, f\models b(x_{l_0}\ldots  x_{n-1})&\iff&N_a(f(l_0),\ldots  f(l_{n-1}))=b\\
\c N_a, f\models\lambda(x_{i_0}, \ldots,x_{i_{k-1}})&\iff&  N_a(f(i_0), \ldots,f(i_{k-1}))=\lambda\\
\c N_a, f\models\neg\phi&\iff&\c N_a, f\not\models\phi\\
\c N_a, f\models (\phi\vee\psi)&\iff&\c N_a,  f\models\phi\mbox{ or }\c N_a, f\models\psi\\
\c N_a, f\models\exists x_i\phi&\iff& \c N_a, f[i/m]\models\phi, \mbox{ some }m\in\nodes(N_a)
\end{eqnarray*}
For any $L$-formula $\phi$, write $\phi^{\c N_a}$ for
$\set{f\in\;^\omega\!\nodes(N_a): \c N_a, f\models\phi}$.  Let
$Form^{\c N_a} = \set{\phi^{\c N_a}:\phi\mbox{ is an $L$-formula}}$ 
and define a cylindric algebra
\[\c D_a=(Form^{\c N_a},  \cup, \sim, {\sf D}_{ij}, {\sf C}_i, i, j<\omega)\]
where ${\sf D}_{ij}=(x_i= x_j)^{\c N_a},\; {\sf C}_i(\phi^{\c N_a})=(\exists
x_i\phi)^{\c N_a}$.  Observe that $\top^{\c N_a}=U_a,\; (\phi\vee\psi)^{\c N_a}=\phi^{\c
N_a}\cup\psi^{\c N_a}$, etc. Note also that $\c D$ is a subalgebra of the
$\omega$-dimensional cylindric set algebra on the base $\nodes(N_a)$,
hence $\c D\in\RCA_\omega$.

Let $\phi(x_{i_0}, x_{i_1}, \ldots, x_{i_k})$ be an arbitrary
$L$-formula using only variables belonging to $\set{x_{i_0}, \ldots,
x_{i_k}}$.  Let $f, g\in U_a$ (some $a\in \alpha$) and suppose
is a partial isomorphism of $N_a$.  We can prove by induction over the
quantifier depth of $\phi$ and using (\ref{eq:theta}), that
\begin{equation}
\c N_a, f\models\phi\iff \c N_a,
g\models\phi\label{eq:bf}\end{equation} Let $\c C=\prod_{a\in
\alpha}\c \c D_a$.  Then 
$\c C\in\RCA_\omega$.  An element $x$ of $\c C$ has the form
$(x_a:a\in\alpha)$, where $x_a\in\c D_a$.  For $b\in\alpha$ let
$\pi_b:\c C\to\c \c D_b$ be the projection defined by
$\pi_b(x_a:a\in\alpha) = x_b$.  Conversely, let $\iota_a:\c D_a\to \c
C$ be the embedding defined by $\iota_a(y)=(x_b:b\in\alpha)$, where
$x_a=y$ and $x_b=0$ for $b\neq a$.  Evidently $\pi_b(\iota_b(y))=y$
for $y\in\c D_b$ and $\pi_b(\iota_a(y))=0$ if $a\neq b$.

Suppose $x\in\Nr_{\mu}\c C\setminus\set0$.  Since $x\neq 0$, 
it must have a non-zero component  $\pi_a(x)\in\c D_a$, for some $a\in \alpha$.  
Say $\emptyset\neq\phi(x_{i_0}, \ldots, x_{i_k})^{\c
 D_a}= \pi_a(x)$ for some $L$-formula $\phi(x_{i_0},\ldots, x_{i_k})$.  We
 have $\phi(x_{i_0},\ldots, x_{i_k})^{\c D_a}\in\Nr_{\mu}\c D_a)$.  Pick
 $f\in \phi(x_{i_0},\ldots, x_{i_k})^{\c D_a}$ and let $b=N_a(f(0),
 f(1), \ldots f_{n-1})\in\alpha$.  We will show that 
$b(x_0, x_1, \ldots x_{n-1})^{\c D_a}\subseteq
 \phi(x_{i_0},\ldots, x_{i_k})^{\c D_a}$.  Take any $g\in
b(x_0, x_1\ldots x_{n-1})^{\c D_a}$, 
so $N_a(g(0), g(1)\ldots g(n-1))=b$.  The map $\set{(f(0),
g(0)), (f(1), g(1))\ldots (f(n-1), g(n-1))}$ is 
a partial isomorphism of $N_a$.  By
 (\ref{eq:theta}) this extends to a finite partial isomorphism
 $\theta$ of $N_a$ whose domain includes $f(i_0), \ldots, f(i_k)$. Let
 $g'\in U_a$ be defined by
\[ g'(i) =\left\{\begin{array}{ll}\theta(i)&\mbox{if }i\in\dom(\theta)\\
g(i)&\mbox{otherwise}\end{array}\right.\] By (\ref{eq:bf}), $\c N_a,
g'\models\phi(x_{i_0}, \ldots, x_{i_k})$. Observe that
$g'(0)=\theta(0)=g(0)$ and similarly $g'(n-1)=g(n-1)$, so $g$ is identical
to $g'$ over $\mu$ and it differs from $g'$ on only a finite
set of coordinates.  Since $\phi(x_{i_0}, \ldots, x_{i_k})^{\c
\ D_a}\in\Nr_{\mu}(\c C)$ we deduce $\c N_a, g \models \phi(x_{i_0}, \ldots,
x_{i_k})$, so $g\in\phi(x_{i_0}, \ldots, x_{i_k})^{\c D_a}$.  This
proves that $b(x_0, x_1\ldots x_{\mu-1})^{\c D_a}\subseteq\phi(x_{i_0},\ldots,
x_{i_k})^{\c D_a}=\pi_a(x)$, and so $\iota_a(b(x_0, x_1,\ldots x_{n-1})^{\c \ D_a})\leq
\iota_a(\phi(x_{i_0},\ldots, x_{i_k})^{\c D_a})\leq x\in\c
C\setminus\set0$.  Hence every non-zero element $x$ of $\Nr_{n}\c C$ 
is above
a non-zero element $\iota_a(b(x_0, x_1\ldots n_1)^{\c D_a})$ (some $a, b\in
\alpha$) and these latter elements are the atoms of $\Nr_{n}\c C$.  So
$\Nr_{n}\c C$ is atomic and $\alpha\cong\At\Nr_{n}\c C$ --- the isomorphism
is $b \mapsto (b(x_0, x_1,\dots x_{n-1})^{\c D_a}:a\in A)$.
\end{proof}
In \cite{e}, we use such games to show that for $n\geq 3$, there is a representable $\A\in \CA_n$ 
with atom structure $\alpha$ such that $\forall$ can win the game $F^{n+2}(\alpha)$.
However \pe\ has a \ws\ in $H_n(\alpha)$, for any $n<\omega$.
It will follow that there a countable cylindric algebra $\c A'$ such that $\c A'\equiv\c
A$ and \pe\ has a \ws\ in $H(\c A')$.
So let $K$ be any class such that $\Nr_n\CA_{\omega}\subseteq K\subseteq S_c\Nr_n\CA_{n+2}$.
$\c A'$ must belong to $\Nr_n(\RCA_\omega)$, hence $\c A'\in K$.  But $\c A\not\in K$
and $\c A\preceq\c A'$. Thus $K$ is not elementary. From this it easily follows that the class of completely representable cylindric algebras
is not elementary, and that the class $\Nr_n\CA_{n+k}$ for any $k\geq 0$ is not elementary either. 
Furthermore the constructions works for many variants of cylindric algebras
like Halmos' polyadic equality algebras and Pinter's substitution algebras.

\begin{theorem}\label{r} Let $3\leq n<\omega$. Then the following hold:
\begin{enumroman}
\item Any $K$ such that $\Nr_n\CA_{\omega}\subseteq K\subseteq S_c\Nr_n\CA_{n+2}$ is not elementary.
\item The inclusions $\Nr_n\CA_{\omega}\subseteq S_c\Nr_n\CA_{\omega}\subseteq S\Nr_n\CA_{\omega}$ are all proper
\end{enumroman}
\end{theorem}
\begin{demo}{Proof} (i) is already mentioned. While for (ii), for the first inclusion \cite{IGPL}, and for the second \cite{HH97b}.
\end{demo}

Robin Hirsch prove the analagous result of theorem \ref{r} (i) for relation algebras $(\RA$) \cite{R}. 
 For $\RA$'s we do have a $NET$ to the effect that $\RRA=S\Ra\CA_{\omega}=S\Ra\RCA_{\omega}$.
If a representable relation algebra $\A$ generates at most one $\RCA_{\omega}$ then $\A\in APbase(\RRA)$.
This is another way of saying that an $\RA$ has the $UNEP$.
In particular, $\QRA\subseteq APbase(\RRA)$. $\QRA$ defined in e.g \cite{TG} p. 242 is the class of relation algebras with quasi-projections. 
In fact, we have $\QRA\subseteq SUPAPbase(\RRA)$. A recent reference dealing with representability of $\QRA$'s 
via a Neat Embedding Theorem for $\CA$'s
is \cite{Simon}.
So for $\RA$'s, $\QRA$ is a ``natural" class such that each of its members has $SNEP$ and $UNEP$.
$\QRA$ lies at the heart of `finitizing" set theory \cite{TG}. The $\CA$ analogue of this class is the class of directed cylindric algebras 
invented by N\'emeti, and studied by Andras Simon and Gabor Sagi \cite{Sagi}.
The representability of such algebras, providing a solution to the finite dimensional version of $FP$ in certain non well founded set theories, 
can be also proved using a $NET$.
Furthermore for such algebras neat reducts commute with forming subalgebras ( that is if $X\subseteq \Nr_n\A$, then $\Sg^{\Nr_n\A}X=\Nr_n\Sg^{\A}X$),
hence this class has $SUPAP$.
In \cite{Amer} the $NET$ of Henkin is likened to his completeness proof; therefore it is not a coincidence that interpolation results and omitting types
for variants of first order logic turn out closely linked to appropriate variations on the $NET$.
Indeed one theme of this paper is to deepen and highlight this connection. 
An algebra $\A$ is representable if it neatly embeds into an algebra in $\omega$ extra dimensions, 
for a class of algebras to have the amalgamation property
its members should embed neatly into $\omega$ extra dimensions in a unique way, for a class of algebras to have super amalgamation its members
should embed uniquely and strongly into $\omega$ extra dimensions ; finally for atomic countable 
algebras to be completely representable they should embed completely into 
algebras in $\omega$ extra dimensions.
We end, this article, by remarking that for different solutions to the Finitizability problem, resorting to a $NET$, like $\QRA$, Nemeti's directed $\CA$'s
Sain's algebras, when an algebra is forced to neatly embed into one in $\omega$ extra dimensions, then it does so, strongly, 
uniquely and completely! This also happens for $\PA$'s. 
In other words, for such algebras the inclusions in theorem \ref{r} (ii) are {\it not} proper. Thats essentially why such classes have $SUPAP$ and
their atomic algebras are completely representable. We do not think that this is a coincidence, but further
research is needed to clarify this point.


\begin{thebibliography}{100}

\bibitem{Amer} M. Amer, T. Sayed Ahmed.
{\it Polyadic and cylindric algebras of sentences.}  Mathematical logic quarterly 
{\bf 52}(5)(2006) p.44-49.


\bibitem{AH1991} I. H. Anellis and N. Houser, \textit{Nineteenth century roots of
algebraic logic and universal algebra}. In H. Andreka, J. D. Monk,
and I. N\'emeti, (eds.), Algebraic logic, Vol. \textbf{54} of
Colloquia Math. Soc. Janos Bolyai. North-Holland, Amsterdam, (1991),
p.1-36.


\bibitem{An97} H. Andr\'eka, {\it Complexity of equations valid in algebras of relations}. 
Annals of Pure and Applied logic, {\bf  89}(1997), 149 -- 209.

\bibitem{AN} Andr\'eka, H., N\'emeti I, {\it On a problem of Johnson}
Unpublished manuscript.



\bibitem{AJN} H. Andr\'eka, J.Van Benthem and I. N\'emeti, 
{\it Modal languages and bounded fragments
of predicate logic.} Journal of Philosophical Logic, {\bf 27}(1998), 217--274.


\bibitem{AGMNS} H. Andr\'eka, S. Givant,  S. Mikulas, I. N\'emeti and A.Simon, 
{\it Notions of density
that imply representability in algebraic logic.} 
Annals of pure and applied logic, {\bf 91}(1998), 93--190.


\bibitem{ACNS} H. Andr\'eka, C. Comer, C.,  J. Madar\'asz, J., I. N\'emeti,I. and T. Sayed Ahmed 
{\it Epimorphisms in cylindric algebras.}
Algebra Universalis, {\bf 61}(3-4)(2009) p. 261-281


\bibitem{AGN77} H. Andr\'eka, T. Gregely, and I. N\'emeti, 
{\it On universal algebraic constructions of logics.} 
Studia Logica, {\bf 36}(1977), 9--47.

\bibitem{AMN91} H. Andr\'eka, J.D.Monk. and I. N\'emeti,  (editors) {\it Algebraic Logic.}
North-Holland, Amsterdam, (1991).


\bibitem{ANth} H. Andr\'eka, I. Nem\'eti I, J. Thompson  {\it Weak cylindric set algebras and weak subdirect indecomposability}
Journal of Symbolic Logic, {\bf 55}(2) 1990.

\bibitem{ANT} H. Andr\'eka, I. N\'emeti, T. Sayed Ahmed, 
{\it Omitting types for finite variable fragments and complete representations of algebras.}
Journal of Symbolic Logic {\bf 73}(1) (2008) p.65-89

\bibitem{ANS} H. Andr\'eka I. Nemeti I, T. Sayed Ahmed, {\it Omitting types for finite variable fragments and complete representations for algebras.}
Journal of Symbolic Logic {\bf 73}(1)(2008)  p.65-89 




\bibitem{AS98} H. Andr\'eka and T. Sayed Ahmed,
{\it Omitting types in logics with finitely many variables}. 
Abstract. Bulletin of Symbolic Logic.
{\bf 5}(1), (1999) p. 88.

\bibitem {AT} H. Andr\'eka, 
and R.J. Thompson {\it A stone-type representation theorem for algebras of relation
of higher rank.} 
Transactions of the American Mathematical Society, 
{\bf 309}(2)  (1988),  671--682.

\bibitem {Bi92} B. Bir\'o.  {\it Non-finite axiomatizability results in algebraic logic.}
Journal of Symbolic Logic,  {\bf 57}(3)(1992), 832--843.

\bibitem{BP89} W.J. Blok, and D. Pigozzi,  {\it Algebraizable logics}. 
Memoirs of American  Mathematical Society, {\bf 77} 
(396), (1989).

\bibitem{Handbook} J. Burgess {\it Forcing}  Chapter in  {\it Handbook of Mathemmatical Logic} Edited by Barwise. J.


\bibitem{C} S.D. Comer {\it Classes without the amalgamation property} 
Pacific journal of Mathematics 28 (2) (1969) p.309-318. 

\bibitem{Comer} S. D. Comer, {\it A Sheaf theoretic duality theory for cylindric algebras}.
Transactions American Mathematical Society, {\bf 169}(1985), 75--87.

\bibitem{Co} S. D. Comer, {\it The representation of $3$ dimensional cylindric algebras.}
In \cite{AMN91}, 146--172. 


\bibitem{C} W. Craig  {\it Logic in algebraic form.} 
North Holland, Amsterdam (1974). 204 pages.



\bibitem{DM63} A. Daigneault and J.D. Monk, 
{\it Representation Theory for Polyadic algebras}. Fund.Math. {\bf 52} (1963), 151--176.

\bibitem{F98} M. Ferenczi, {\it On representability of  cylindric algebras}. 
Abstracts of papers presented to the American Mathematical Society, 
{\bf 13}(3), (1992) p. 336.

\bibitem{Fer2} M. Ferenczi, {\it Finitary polyadic algebras from cylindric algebras.} Studia Logica {\bf 87}(1)(2007) p.1-11

\bibitem{Fer3} M. Ferenczi,  {\it On cylindric algebras satisfying the merry-go-round properties}
Logic Journal of IGPL, {\bf 15}(2) (2007), p. 183-199

\bibitem{00Fer}M. Ferenczi, {\it On representability of neatly
embeddable cylindric algebras} Journal of Applied  Non-classical Logic, {\bf 10} 3-4(2000), p. 1-11

\bibitem{08Fer}M. Ferenczi, {\it On the representability of neatly
embeddable $\CA's$ by cylindric set algebras}, to appear

\bibitem{09Fer} M. Ferenczi, {\it On conservative extensions in logics with
infinitary predicates}, to appear

\bibitem{F} D, H, Fremlin, {\it Consequences of $MA$}. Cambridge University press. 
(1984)


\bibitem{Ga81} D. Gabbay, {\it An irreflexitivity lemma with applications to 
axiomatizations of conditions
in linear frames} in U. Monnich (editor) , {\it Aspects of Philosophical Logic}
Reidel, Dordrecht, (1981).



\bibitem{Giv} S. Givant  and H. Andr\'eka {\it Groups and algebras of binary relations}
Bulletin of Symbolic Logic, {\bf  8}(2002), 38--64.

\bibitem{Goldblatt89} R. Goldblatt,{\it Varieties of Complex algebras.} Annals of Pure
and Applied Logic, {\bf 38}(1989), 173--241.

\bibitem{Goldblatt2000} R. Goldblatt, {\it Algebraic Polymodal Logic: A survey}
Logic Journal of the IGPL, {\bf 8}(4), (2000) 393--450.

\bibitem {GHV} R. Goldblatt, I. Hodkinson, I, and Y. Venema,  
{\it Erdos graphs resolve Fine's canonicity problem.} 
Bulletin of Symbolic Logic, {\bf 10}(2)(2004), p.186-208


\bibitem {HenMonk} L. Henkin and J.D. Monk 
{\it Cylindric algebras and related structures},  Proceedings of
the Tarski Symposium, American Mathematical 
Society, {\bf 25} (1974), 105--121.

\bibitem{HMT1} L. Henkin, J.D. Monk and A. Tarski {\it Cylindric Algebras Part I}. 
North Holland, (1971.)

\bibitem{HMT2} L. Henkin, J.D. Monk, and A. Tarski {\it Cylindric Algebras Part II}.
North Holland, (1985).

\bibitem{HMTAN81} L. Henkin, J. D. Monk, A. Tarski, H. Andreka, and I. N\'emeti,
\textit{Cylindric Set Algebras}. Lecture Notes in Mathematics, Vol.
\textbf{883}, Springer-Verlag, Berlin, (1981), p.vi + 323.


\bibitem{R} Hirsch R. {\it Relation algebra reducts of cylindric algebras and complete representations}
Journal of Symbolic Logic, {\bf 72}(2) (2007) p.673-703.


\bibitem{HH97a} R. Hirsch and I. Hodkinson, 
{\it Step by step-building representations in algebraic logic}. 
Journal of Symbolic Logic, {\bf  62}(1) (1997), 225--279.

\bibitem{HH97b} R. Hirsch and I. Hodkinson, 
{\it Complete representations in algebraic logic}. Journal of Symbolic Logic, 
{\bf  62}(3) (1997), 816--847.



\bibitem {HHII} R. Hirsch I. Hodkinson {\it Relation algebras from cylindric algebras,II} 
Annals of Pure and Applied Logic, {\bf 112} (2001), 267--297. 

\bibitem{HH2000} R. Hirsch and I. Hodkinson
{\it Strongly representable atom structures.} 
Proceedings of the American Mathematical Society {\bf 130} (2002),
1819--1831.

\bibitem{finite} R. Hirsch,  I. Hodkinson {\it Representability  is not decidable for finite relation algebras} 
Trans of Amer. Math. Soc. {\bf 353}(3) (2002)p. 1403-1425

\bibitem{un} R. Hirsch, I. Hodkinson {\it Representation is not decidable for finite relation algebras} Trans. of Amer. Math. Soc. {\bf 353}(4)
(2002) p. 1403-1425

\bibitem{HH2000} R. Hirsch and I. Hodkinson,
{\it Strongly representable atom structures.} 
Proceedings of the American Mathematical Society {\bf 130} (2002)
p.1819-1831

\bibitem{HH} R. Hirsch and I. Hodkinson {\it Synthesizing axioms by games.}
a CD-ROM of essays dedicated to Johan van Benthem 
on occasion of his 50th birthday (1999).

\bibitem{HHbook} R. Hirsch I. Hodkinson {\it Relation algebras by games.}
(2002) Studies in Logic and the Foundations of Mathematics. Volume 147. (2002)

\bibitem{HHstrong} R. Hirsch I. Hodkinson , {\it Strongly representable atom structures of cylindric algebras}
Journal of Symbolic Logic {\bf 74}(3) (2009) p. 811-828


\bibitem{HHK} R. Hirsch , I. Hodkinson, A. Kurusz {\it On modal logics between $K\times K\times K$ and $S5\times S5\times S5$}
Journal of Symbolic Logic {\bf 67}(1) (2002) p. 221-234 

\bibitem{K5} R. Hirsch, I.Hodkinson, A. Kurucz {\it On modal logics betwen $K\times K\times K$ and $S5\times S5 \times S5$}. 
Journal of Symbolic Logic {\bf 67}(1)(2002) p.221-234

\bibitem{HHM}  R. Hirsch, I. Hodkinson, R. Maddux,  
{\it Relation algebra reducts of cylindric algebras
and an application to proof theory.} Journal of Symbolic Logic 
{\bf 67}(1) (2002), 197--213.

\bibitem{HHM2000} R. Hirsch, I. Hodkinson, R. Maddux, 
{\it On provability with finitely many variables},
Bulletin of Symbolic Logic, vol 8 (2002), no 3, p.329-347.




\bibitem{HV} I. Hodkinson, Y. Venema 
{\it Canonical varieties with no canonical axiomatizations} Trans. Amer. Math Society. {\bf 357} 
(2002) p.4579-4605

\bibitem{Hod97} I. Hodkinson,  
{\it Atom structures of cylindric algebras and relation algebras.}
Annals of pure and applied logic, {\bf 89}(1997), 117--148. 

\bibitem{HM} I. Hodkinson I, S. Mikulas {\it Axiomatizability of reducts of cylindric algebras and relation algebras} Algebra Universalis
{\bf 43}(2003) p.127-156 


\bibitem{Ho} W. Hodges, {\it Model Theory}, volume 42 of Encyclopedia of 
mathematics and its applications. Cambridge University Press. (1993).
 





\bibitem{Jonsson} B. Jonsson B, 
{\it The theory of binary relations} In \cite{AMN91} 
245--292.

\bibitem{JonTar48} B.Jonsson and A.Tarski 
{\it Representation problems for relation algebras}
Bull. Amer.Math.Soc. {\bf 54} (1948), 80-92

\bibitem{K} A. Kurucz {\it Arrow logic and infinite countng} Studial Logica {\bf 65} (2000) 199-222.

\bibitem{Kb} A. Kurucz {\it  Weakly associative relation algebras and projection elements} Preprint (2000)
\bibitem{Lyndon50} I. Lyndon,  {\it The representation of relational algebras.} 
Annals of  Mathematics, {\bf 51}(3)(1950), 707--729. 

\bibitem{Lyndon56} I. Lyndon,   {\it The representation of relational algebras, II.} 
Annals of  Mathematics, {\bf 63}(3) (1956), 294--307. 

\bibitem{Lyndon61}I. Lyndon,  {\it Relation algebras and Projective Geometries.}
Michigan Mathematics Journal, {\bf 8}(1961), 207--210. 


\bibitem{Makkai} I. Makkai  {\it On $PC_{\Delta}$ classes in the theory of models}
Matemtikai Kutato Intezetenek Kozlemenyei 9 .159-194, (1964).

\bibitem{Mad} J.Mad\'arasz {\it Logic and Relativity (in the light of definability theory)}.
Ph.D thesis , Budapest, February 26, 2002.

\bibitem{Ma} J. Mad\'arasz {\it Hereditary non-finite axiomatizability of relation algebras
and their variants.} Manuscript consists of 3 TEX pages
and 18 handwritten pages.

\bibitem{AUU} J. Madarasz and T. Sayed Ahmed 
{\it Amalgamation, interpolation and epimorphisms} 
Algebra Universalis {\bf 56} 2 (2007) p. 179-210.


\bibitem{MSone} J. Mad\'arasz J. and T. Sayed Ahmed 
{\it Neat reducts and amalgamation in retrospect, a survey of results and some
methods. Part 1: Results on neat reducts} Logic Journal of IGPL  
{\bf 17}(4)(2009) p.429-483

\bibitem{MStwo} J. Mad\'arasz and T. Sayed Ahmed,
{\it Neat reducts and amalgamation in retrospect, a survey of results and some
methods. Part 2: Results on amalgamation} Logic Journal of IGPL (2009)  
{\bf 17}(6) (2009) 755-802


\bibitem{Madd80} R. Maddux {\it The equational theory of $CA_3$ is undecidable.}
Journal of Symblic Logic {\bf 45}(2)(1980), 311-316 .

\bibitem{Maddux} R. Maddux, {\it Canonical Relativized Cylindric Set Algebras}
Proceedings of the American Mathematical Society, {\bf 107},
(2)(1989), 465--478.



\bibitem{Maddux} R. Maddux, {\it Canonical Relativized Cylindric Set Algebras}
Proceedings of the American Mathematical Society, {\bf 107},
(2)(1989), 465--478.


\bibitem{Maddux89} R. Maddux {\it Non-finite axiomatizability results for cylindric and
relational algebras.} The Journal of Symbolic Logic, {\bf 54}(3) (1989), 951--974.

\bibitem{Madd} R. Maddux, {\it Finitary algebraic logic} 
Zeitschrift f\"ur mathematische Logik und Grundlagen der 
Mathematik {\bf 35} (1989), p.321--332



\bibitem{MadduxCA} R. Maddux  {\it  A relation algebra which is not a cylindric reduct}
Algebra Universalis vol 27, 1990 pp. 279-288.


\bibitem{Mad91c} R. Maddux, {\it The neat embedding property and the number
of variables required in proofs} Proc. Amer.Math Soc {\bf 112}(1991), 195--202.

\bibitem{Mad1991} R. Maddux, \textit{The Origin of Relation Algebras in the Development
and Axiomatization of the Calculus of Relations}. Studia Logica,
Vol. \textbf{50}, (3/4), (1991), p.421-455.


\bibitem{Maddux} R. Maddux {\it Introductory course on relation algebras, finite-dimensional cylindric algebras, and their interconnections}
In {\it Algebraic Logic} editors Andreka H, Monk J.D., Nemeti I., North Holland, (1991)p. 361-392

\bibitem{Mad} R. Maddux   {\it Introductory course on relation algebras}
In \cite{AMN91} p.361-392
 

\bibitem{Madd92} R. Maddux, 
{\it Relation algebras of every dimension} Jornal of 
Symbolic Logic, {\bf  57}(4) (1992), 1213--1229 

\bibitem{Madd1} R. Maddux, {\it Finitary algebraic logic, II} Mathematical Logic Quarterly
{\bf 39} (1993), 566--569, 

\bibitem{Madd94} R. Maddux {\it Undecidable semiassociative relation algebras}
Journal of Symbolic Logic {\bf 59}(1993),  398-418.

\bibitem{Mak} L. Maksimova, L. 
{\it Amalgamation and interpolation in normal modal logics}.
Studia Logica {\bf 50}(1991) p.457-471. 


\bibitem{Mc} R. McKenzie, {\it The representation of relation algebras.} PhD thesis
University of Colorado at Boulder, (1966) 



\bibitem{Marx95} M. Marx, 
{\it Algebraic relativization and arrow logic}. PhD  thesis, University
of Amsterdam, (1995).  

\bibitem{Marx99} M. Marx, {\it Relativized relation algebras.}
Algebra Universalis, {\bf 41}(1999), 23--45.

\bibitem {MLM}M. Marx ,L.Polos L, M. Masuch,  editors 
{\it Arrow logic and Multi-modal logic} (1996). 
Studies in Logic Language and Information.
CSLI Publications, Centre for the Study of Language and Information. 

\bibitem{MS} D. Martin, R.M. Solovay R.M, 
{\it Internal Cohen extensions}. Ann. Mathematical Logic
{\bf 2} (1970) 143-178.

\bibitem{M80} A. Miller {\it Covering $2^{\omega}$ with $\omega_1$ disjoint closed sets.} 
The Kleene Symposuim (proceedings, Madison, Wisconsin, 1978). Studies
in Logic and the Foundation of Mathematics, vol 101, North-Holland,
Amsterdam, (1980) 415--421.

\bibitem{Miller} A. Miller, {\it Some properties of measure and category}
Transactions of the American Mathematical Society,  {\bf 266} (1981),. 
93-113

\bibitem{M} A. Miller  {\it 
Characterization of the least cardinal for which the Baire Category Theorem 
fails} Proceedings of the American Mathematical Society. {\bf 86} (1982).

\bibitem{Monk} J. D. Monk. {\it Studies in cylindric algebra} Doctoral dissertation , University of California (1961).


\bibitem{Monk64} J.D. Monk
{\it On representable relation algebras.} Michigan Mathematics
Journal {\bf 11} (1964), 207--210.

\bibitem{Monk65} J.D.Monk 
{\it Model-theoretic methods and results in the theory of cylindric algebras} in {\bf The
Theory of Models}, Addision, Henkin, Tarski, ed., North-Holland, Amsterdam, (1965),
238--250. 

\bibitem{M69} J.D. Monk. 
{\it Non-finitizability of classes of representable cylindric algebras}. 
Journal of Symbolic Logic.
{\bf 34}(1969), 331--343. 

\bibitem{Monk}J.D.Monk {\it On an algebra
of sets of finite sequences.} The Journal of Symbolic Logic 
{\bf 35}(1970), 19--28.

\bibitem{M71} J.D. Monk, {\it Provability with finitely many variables}
Proceedings of the American Mathematical Society, {\bf 27} (1971), 352--353.

\bibitem {M74} J.D. Monk,{\it Connections between 
combinatorial theory and algebraic logic}
In Studies in Math, Math Assoc Amer{\bf 9} (1974), 58--91.

\bibitem{N83} I. N\'emeti, {\it The Class of Neat Reducts of Cylindric Algebras
is Not a Variety But is closed w.r.t. HP}. Notre Dame Journal of
Formal logic, {\bf 24}(3) (1983), pp 399-409. 


\bibitem{N96} I. N\'emeti, {\it Algebraization of quantifier logics, an introductory overview}. 
Math.Inst.Budapest, Preprint, No 13-1996. 
A shortened version appeared in Studia Logica {\bf 50}(1991),465--569. 
\bibitem{N97} I. N\'emeti  
{\it Strong representability of fork algebras, a set theoretic foundation}
Logic Journal of IGPL. {\bf 5}(1) (1997), 8--28.


\bibitem{NS} I. N\'emeti, A.Simon {\it Relation algebras from cylindric and polyadic algebras}
Logic Journal of IGPL {\bf 5} (1997), p.575-588.


\bibitem{P} D. Pigozzi
{\it Amalgamation, congruence extension, and interpolation properties in algebras.}
Algebra Universalis.
{\bf 1}(1971), p.269 - 349.

\bibitem{SaGy96} I. Sain, and V. Gyuris, 
{\it Finite Schematizable Algebraic Logic.} Preprint, Math.Inst.
Hng.Acad.Sci., Budapest, (1996). 



\bibitem{Sa98} I. Sain.  {\it Searching for a finitizable algebraization of first order logic}. 
Logic Journal of IGPL. Oxford University Press. {\bf 8}(4) (2000), 495--589.

\bibitem{SN96} I.Sain and I. N\'emeti.{\it Fork algebras in usual and in non-well founded
set theories (An overview)} Preprint of the mathematical institute of the Hungarian 
academy of Sciences. 

\bibitem{ST} I. Sain, R. Thompson 
{\it Strictly finite schema axiomatization of quasi-polyadic algebras}. In 
\cite{AMN91} p.539 - 571.




\bibitem{Sagi99}  G. S\'agi,  {\it On the Finitizability Problem of Algebraic Logic.}
Ph.D dissertation. Budapest (1999.) 

\bibitem{Sagi2000} G. Sagi, {A completeness theorem for higher order logics.}
Journal of Symbolic Logic {\bf 65}(3)( 2000), 857--884.


\bibitem{Sagi2000} G. Sagi, {\it A completeness theorem for higher order logics}
Journal of Symbolic Logic {\bf 65}(3)(2000)  p.857-884.

\bibitem{SF} G. Sagi, M. Ferenszi, M, {\it On some developments in the representation theory of cylindric- like algebras}
Algebra Universalis, {\bf 55}(2-3)(2006), p.345-353


\bibitem{SS99} G. Sagi,  T. Sayed Ahmed {\it N\'emeti's directed cylindric algebras 
have the strong amalgamation property}. Manuscript. 1999

\bibitem{SSS} G. Sagi, T. Sayed Ahmed and I. Sain  
{\it The Finitizability Problem in algebraic logic, a survey}
Manuscript, 2000.

\bibitem{Shelah} Sagi, G, Shelah S., {\it Weak and strong interpolation for algebraic logics.}
Journal of Symbolic Logic, {\bf 71}(2006), p.104-118.




\bibitem{ma}  T. Sayed Ahmed,  {\it Algebras of sentences of Logic}.
Masters Thesis, Cairo University. (1998).

\bibitem {San} T. Sayed Ahmed  {\it The class of neat reducts is not elementary.}
Bulletin of the Journal of  Symbolic Logic, {\bf 5} (3) (1999), 407-408.


\bibitem {FC} T. Sayed Ahmed,  {\it On a stronger version of the Finitizability problem.}
Presented in Conference in Algebra (in honour of the 70th birthday of Ervin Fried). August
17-21, 1999. Alfred R\'enyi Institute of Mathematics, Budapest, Hungary.
Electronically available at http:// www.renyi.hu/~rabbit/.  

\bibitem{BL} T. Sayed Ahmed  {\it On neat reducts and amalgamation}
Bulletin of Symbolic Logic, {\bf 7} (1), (2001), p.83



\bibitem{IGPL} T. Sayed Ahmed {\it The class of neat reducts is not elementary.}
Logic Journal of IGPL,  {\bf 9} (2001) p. 31-65 
electronically available at http://www.math-inst.hu/pub/algebraic-logic.

\bibitem{FM} T. Sayed Ahmed {\it The class of $2$-dimensional neat reducts of 
polyadic algebras is not elementary}. 
Fundementa  Mathematicea, {\bf 172}, (2002), p.61-81.

\bibitem{MQ} T. Sayed Ahmed {\it A Model-theoretic Solution to a problem of Tarski}
Mathematical Logic Quaterly, {\bf 48}, issue 3. March 2002, p. 343-355 

\bibitem{th} T. Sayed Ahmed  {\it  Topics in Algebraic Logic} P.hD thesis. Cairo university 
(2002)

\bibitem {SSL}  T. Sayed Ahmed  
{\it Martin's axiom, omitting types and complete representations in algebraic
logic} .   Studia Logica  {\bf 72} (2002), p.1-25

\bibitem{Tarski} T. Sayed Ahmed 
{\it A confirmation of a conjecture of Tarski} Bulletin section of
logic {\bf 32} (3) (2003), p.103-105

\bibitem{Notre}  T. Sayed Ahmed  {\it Neat embeddings, interpolation, and 
omitting types, an overview.}
Notre Dame Journal of formal logic, {\bf 44} (3)(2003), p.157-173

\bibitem{Bull3} T. Sayed Ahmed  
{\it Omitting types for finite variable fragments of first order logic}
Bulletin section of logic {\bf 32}(3)(2003) p.115-120

\bibitem{AU} T. Sayed Ahmed  {\it On Amalgamation of Reducts of Polyadic Algebras.}
Algebra Universalis  {\bf 51} (2004), p.301-359.

\bibitem{Bull5} T. Sayed Ahmed,  
{\it A sufficient and necessary condition for omitting types}
Bulletin section of logic {\bf 34}(1)(2005)  p.23-28

\bibitem{IGPL2} T. Sayed Ahmed, {\it Independence results in algebraic logic}
Logic Journal of IGPL {\bf 14}(1) (2005) p.87-96 (2005)

\bibitem{IGPL3} T. Sayed Ahmed, 
{\it Amalgamation Theorems in Algebraic Logic, an overview}
Logic Journal of IGPL, {\bf 13} (2005), 277-286.

\bibitem{Bull5b} T. Sayed Ahmed, {\it An independence result in algebraic logic}
Bulletin section of logic {\bf 34}(1)(2005) p. 29-36


\bibitem{Am} T. Sayed Ahmed, {\it On amalgamation of algebras of logic}
Studia Logica {\bf 81} (2005), p.61-77.

\bibitem{Bulletin} T. Sayed Ahmed, {\it Algebraic Logic, where does it stand today?}
Bulletin of Symbolic Logic. {\bf 11} (4) (2005), p.465-516.

\bibitem{amal}  T. Sayed Ahmed, 
{\it Amalgamation Theorems in Algebraic Logic, an overview.}
Logic Journal of IGPL, {\bf 13} (2005), p. 277-286.


\bibitem{Qs} T. Sayed Ahmed,
{\it The class of infinite dimensional neat reducts of quasi-polyadic
algebras is not axiomatizable} Mathematical Logic quaterly  {\bf 52}
(1) (2006), p.106-112

\bibitem{OT} T. Sayed Ahmed,
{\it Omitting types for algebraizable extensions of first order
logic} Journal of Applied non-classical logics {\bf 15} (4) (2006),
p.465-487



\bibitem{IGPL4} T.Sayed Ahmed {\it Some results on amalgamation in algebraic logic}
Logic Journal of IGPL (2006) {\bf 14} p. 623-627


\bibitem {Bull6} T. Sayed Ahmed {\it Algebras of sentences}, Bulletin of Section of Logic 
{\bf 35}, no 1,  p.1-10 (2006)

\bibitem{IGPL5} T. Sayed Ahmed {\it On neat reducts and amalgamation} 
Logic Journal of IGPL {\bf 15}(1) (2007) p.33-39.

\bibitem{IGPL6} T. Sayed Ahmed 
{\it An interpolation theorem for first order logic with infinitary predicates} 
Logic journal of IGPL {\bf 15}(1) (2007) p.21-32

\bibitem{note} T. Sayed Ahmed  {\it A note on neat reducts}, 
Studia Logica  {\bf 85}(2) (2007), p. 139-151.


\bibitem{TT} T. Sayed Ahmed  {\it On complete representability of reducts of Polyadic algebras} Studia Logica 
{\bf 89}(3) (2008) p.325-332


\bibitem{weak} Sayed Ahmed T., {\it Weakly representable atom structures that are not strongly representable, with an application to first order logic.}
Mathematical Logic Quarterly. {\bf 3}(2008) p. 294-306




\bibitem{Completions} T. Sayed Ahmed, {\it A simple construction of representable relation algebras with non representable completions} 
Mathematical  Logic Quarterly {\bf 55}(3) (2009) 237-244


\bibitem{JANCL} T. Sayed Ahmed. {\it On a Theorem of Vaught for first order logic with finitely many variables} 
Journal of Applied Non-classical Logic {\bf 19}(1) (2009) p. 97-112.

\bibitem{sub} T. Sayed Ahmed. {\it A note on substitutions in cylindric algebras} Mathematical Logic Quarterly
{\bf 55}(3)(2009) p. 280-287

\bibitem{neatMLQ} T. Sayed Ahmed {\it On neat embeddings of cylindric algebras}
Mathematical Logic Quarterly {\bf 55}(6)(2009)p.666-668

\bibitem{neet1} T. Sayed Ahmed  {\it On neat embedding of algebraisations of first order logic}
Journal of Algebra, number theory, advances and applications {\bf 1}(2) 2009 p. 113-125

\bibitem{neet2}  T. Sayed Ahmed {\it The amalgamation property, and a problem of Henkin Monk and Tarski}
Journal of Algebra, number theory, advances and applications {\bf 1}(2) 2009 p. 127-141



\bibitem{super} T. Sayed Ahmed {\it The class of polyadic algebras has the superamalgamation property}
Mathematical Logic Quarterly {\bf 56}(1)(2010)p.103-112 

\bibitem{fail} T. Sayed Ahmed {\it Varieties of algebras without the amalgamation property} Logic Journal of IGPl, to appear.

\bibitem{AU} T. Sayed Ahmed {\it Some results on neat reducts} Algebra universalis, to appear.
\bibitem{e} T. Sayed Ahmed {\it Non elementary classes in algebraic logic} Submitted
\bibitem{STR} T. Sayed Ahmed 
{\it The class of strongly representable atom structures of $RCA_3$ is not elementary.
In particular, $RCA_3$ is not single-persistent.}
Manuscript 


\bibitem{complete} T. Sayed Ahmed {\it On algebras not closed under completions}
Submitted to Archive of Mathematical Logic


\bibitem{p} T. Sayed Ahmed {\it Complexity of equational axiomatizations of polyadic algebras of relations}
Submitted to Logic Journal of IGPL

\bibitem{ca} T.Sayed Ahmed {\it $\RCA_n$ is barely canonical} Submitted to Mathematical Logic Quarterly


\bibitem{neet} T.Sayed Ahmed {\it The neat Embedding Problem for algebras other than cylindric algebras}
Submitted to Mathematical Logic Quarterly

\bibitem{universal} T. Sayed Ahmed {\it Amalgamation in Universal Algebraic Logic} Submitted to Math. Stud.  
Hung.


\bibitem{Stability} T. Sayed Ahmed {\it Some Stability Theory in connection to neat embeddings} Manuscript

\bibitem{SL} T. Sayed Ahmed, I. N\'emeti, {\it  On neat reducts of algebras of logic.}
Studia Logica, {\bf 62} (2) (2001), p.229-262.


\bibitem{Bull6b} T. Sayed Ahmed B. Samir {\it Neat embeddings and amalgamation}
Bulletin section of logic {\bf 3} 5/4 (2006) p. 164-172

\bibitem {IGPL7} T. Sayed Ahmed, B.Samir
{\it A Neat embedding theorem for expansions of cylindric algebras}
Logic journal of IGPL {\bf 15} (2007) p. 41-51

\bibitem {IGPL7c} T. Sayed Ahmed, T B. Samir B.,
{\it A Neat embedding theorem for expansions of cylindric algebras.}
Logic Journal of IGPL {\bf 15} (2007) p. 41-51.

\bibitem{OTT} T. Sayed Ahmed, and B. Samir, {\it Omitting types for first order logic with infinitary predicates}
Mathematical Logic Quaterly  {\bf 53}(6) (2007) p.564-576.



\bibitem{Basim} T. Sayed Ahmed, T B. Samir  {\it The class $S\Nr_3\CA_k$ is not closed under completions}
Logic Journal of IGPL {\bf 16} (2008) p.427-429.

\bibitem{MS}T. Sayed Ahmed,T. M. Khaled {\it On complete representations in algebras of logic}
Logic journal of IGPL {\bf 17}(3)(2009)p. 267-272

\bibitem{Moh2} T. Sayed Ahmed and M. Khaled {\it Classes of algebras not closed under completions}
Bulletin section of Logic
{\bf  38} (1-2)( 2009) p. 29-44

\bibitem{Moh3} T. Sayed Ahmed and M. Khaled {\it Omitting types algebraically via cylindric algebras} 
International Journal of Algebra. {\bf 3}(8) (2009) p. 377-390

\bibitem{Sagi} G. Sagi, {\it On the Finitization problem in Algebraic logic}
PhD dissertation. (1999)

\bibitem{Fer4} G. Sagi, M. Ferenszi {\it On some developments in the representation theory of cylindric- like algebras}
Algebra Universalis, {\bf 55}(2-3)(2006) p.345-353




\bibitem{Shelah} G. Sagi, S. Shelah, {\it Weak and strong interpolation for algebraic logics.}
Journal of Symbolic Logic, {\bf 71}(2006), p.104-118.


\bibitem{Shelah2} S. Shelah. {\it Classification Theory } 
Second Edition North Holland P.C., Amsterdam 1990.

\bibitem{Sim93} A. Simon {\it What the Finitizability problem is not.} 
Algebraic methods in Logic
and Computer Science, Banach Centre Publications, {\bf 28}(1996) 95--116.

\bibitem{Sim97} A. Simon,  {\it Non representable algebras of relations.}
Ph.D Dissertation. Budapest 1997.

\bibitem{Simon} A. Simon, {\it Connections between quasi-projective relation algebras and cylindric algebras}
Algebra universalis {\bf 56}(3-4)(2007), p. 263-302



\bibitem{Tar1935} A. Tarski, \textit{Grundz¨uge der Systemenkalk¨uls. Erster Teil}. Fundamenta
Mathematica, Vol. \textbf{25}, (1935), p.503-526. English
translation in [A. Tarski, Logic, Semantics, Metamathematics. Papers
from 1923 to 1938, edited by J. Corcoran, Hackett Pub. Co.,
Indianapolis, Indiana, second edition, (1983)]: Foundations of the
calculus of systems, p.342-383.




\bibitem{TG} A. Tarski and S. Givant
{\it A formalization of set theory without variables.} 
AMS Colloquium Publications {\bf 41}, (1987). 

\bibitem{Venemap} Y.  Venema. {\it A Modal Logic of Quantification and Substitution} Logic Journal of IGPL 2(1) (1994) 31-45.

\bibitem {V98} Y.Venema, {\it Rectangular games.} Journal of Symbolic Logic
{\bf 63}(4)(1998), 1549--1564


\bibitem{Venema} Y. Venema 
{\it Atom structures and Sahlqvist equations.}
Algebra Universalis, {\bf 38} (1997),  p.185 - 199.

\bibitem{Venemac} Y. Venema, {\it Cylindric modal logic}
Journal of Symbolic Logic.  {\bf 60}(2) (1995),  591--623.


\end{thebibliography}
\end{document}